\documentclass[11pt,reqno]{amsart}
\usepackage{amsmath,amsthm,amssymb}
\usepackage{yhmath}
\usepackage{mathtools}
\usepackage{enumerate}
\usepackage{amsfonts}
\usepackage{eucal}
\usepackage{tikz-cd}
\usepackage[all]{xy}
\CompileMatrices
\usepackage{hyperref}
\definecolor{webgreen}{rgb}{0,.5,0}
\definecolor{webbrown}{rgb}{.6,0,0}
\definecolor{ocre}{RGB}{52,177,201}
	\definecolor{royalblue(web)}{rgb}{0.25, 0.41, 0.88}
\hypersetup{
 colorlinks=true,linktocpage=true,pdfstartpage=1,
 pdfstartview=FitV,breaklinks=true,pdfpagemode=UseNone,
 pageanchor=true,pdfpagemode=UseOutlines,
 plainpages=false,bookmarksnumbered,
 bookmarksopen=true,bookmarksopenlevel=1,
 hypertexnames=true,pdfhighlight=/O,
 urlcolor=black,linkcolor=royalblue(web),
 citecolor=webgreen,
 hyperfootnotes=false,pdfpagelabels,
 pdfcreator={pdfLaTeX},
 pdfproducer={LaTeXwithArsClassica}
 }
\usepackage{enumitem}
\usepackage[normalem]{ulem}
\usepackage{slashed}
\usepackage{scalerel}
\usepackage{galois}
\usepackage{esvect}
\usepackage{imakeidx,etoolbox}
\usepackage[usestackEOL]{stackengine}
\makeindex[title=Index,columns=2]

\newcommand{\defoperator}[1]{\expandafter\def\csname#1\endcsname{\operatorname{#1}}}
\newcommand{\deffrak}[1]{\expandafter\def\csname#1\endcsname{\mathfrak{#1}}}
\newcommand{\defbb}[1]{\expandafter\def\csname#1#1\endcsname{\mathbb{#1}}}
\newcommand{\dbar}{\bar{\partial}}

\defbb{Z}
\defbb{R}
\defbb{C}
\defbb{H}
\defbb{V}
\defbb{P}
\defoperator{U}
\defoperator{SU}
\deffrak{g}
\deffrak{t}
\defoperator{Im}
\defoperator{Re}
\defoperator{Id}
\defoperator{id}
\defoperator{End}
\defoperator{Hom}
\defoperator{Aut}
\defoperator{Iso}
\defoperator{Diff}
\defoperator{ad}
\defoperator{pr}
\defoperator{Ad}
\defoperator{At}
\defoperator{rank}
\defoperator{ker}
\defoperator{coker}
\defoperator{im}
\defoperator{tr}
\defoperator{Str}
\defoperator{ind}
\defoperator{diag}
\defoperator{ord}
\defoperator{Spec}
\defoperator{Proj}
\defoperator{Ext}
\defoperator{Gal}
\defoperator{Jac}
\defoperator{Gr}
\defoperator{Ric}
\defoperator{Isom}
\defoperator{DF}

\defoperator{GL}
\defoperator{SL}
\defoperator{PSL}

\defoperator{PSU}
\defoperator{SO}
\defoperator{Spin}
\defoperator{grad}
\defoperator{Alb}
\defoperator{Stab}

\makeatletter
\newcommand*\rel@kern[1]{\kern#1\dimexpr\macc@kerna}
\newcommand*\widebar[1]{%
  \begingroup
  \def\mathaccent##1##2{%
    \rel@kern{0.8}%
    \overline{\rel@kern{-0.8}\macc@nucleus\rel@kern{0.2}}%
    \rel@kern{-0.2}%
  }%
  \macc@depth\@ne
  \let\math@bgroup\@empty \let\math@egroup\macc@set@skewchar
  \mathsurround\z@ \frozen@everymath{\mathgroup\macc@group\relax}%
  \macc@set@skewchar\relax
  \let\mathaccentV\macc@nested@a
  \macc@nested@a\relax111{#1}%
  \endgroup
}
\makeatother

\newcommand{\ph}{\mathord{\vcenter{\hbox{\scaleobj{0.6}{\bullet}}}}}

\newcommand{\I}{\mathrm{i}}
\newcommand{\E}{\mathrm{e}}

\DeclareFontFamily{OT1}{rsfs}{}
\DeclareFontShape{OT1}{rsfs}{n}{it}{<-> rsfs10}{}
\DeclareMathAlphabet{\mathscr}{OT1}{rsfs}{n}{it}
\renewcommand*{\setminus}{-}
\DeclareEmphSequence{\bfseries}

\makeatletter
\newcommand*{\Rom}[1]{\expandafter\@slowromancap\romannumeral #1@}
\makeatother

\theoremstyle{plain}
  \newtheorem{theorem}{Theorem}
  \numberwithin{theorem}{section}
	
  \newtheorem{proposition}[theorem]{Proposition}
  \newtheorem{lemma}[theorem]{Lemma}
  \newtheorem{corollary}[theorem]{Corollary}
\theoremstyle{definition}

  \newtheorem{definition}[theorem]{Definition}
  
  \newtheorem{assumption}[theorem]{Assumption}
  \newtheorem{convention}[theorem]{Convention}

  \newtheorem{example}[theorem]{Example}
  \newtheorem{remark}[theorem]{Remark}

\numberwithin{equation}{section}

\newenvironment{enumerate1}[1][\arabic]
{
  \begin{enumerate}[label=(#1*),leftmargin=1.7\parindent]
}%
{
  \end{enumerate}
}

\newcommand{\D}{\mathrm{d}}
\newcommand{\G}{\mathbb{G}}

\usepackage{graphicx}

\title[BKN identity and nonlinear Hodge correspondence for morphisms]{Nonlinear Bochner--Kodaira--Nakano identity and nonlinear Hodge correspondence for morphisms}
\author{Nianzi Li}
\email{lnz@mail.tsinghua.edu.cn}
\address{Yau Mathematical Sciences Center, Tsinghua University, Beijing 100084, China}

\author{Mao Sheng}
\email{msheng@tsinghua.edu.cn}
\address{Yau Mathematical Sciences Center, Tsinghua University, Beijing 100084, China \& Yanqi Lake Beijing Institute of Mathematical Sciences and Applications, Beijing 101408, China}

\subjclass[2020]{53C07, 53C55, 32L05, 32Q15, 58E20}
\keywords{nonabelian Hodge theory, Bochner--Kodaira--Nakano identity, flat sections, Higgs sections}

\begin{document}
\begin{abstract}
We establish a nonlinear Bochner--Kodaira--Nakano identity for sections of complex fiber bundles. As applications, we obtain a nonlinear Bochner-type vanishing theorem for holomorphic sections and a generalization of Yau's Schwarz lemma. After incorporating a nonlinear Higgs field, we derive the $D'$-$D''$ and $D^c$-$D$ identities. Under natural Hamiltonian and compactness assumptions, these identities imply that a section of a nonlinear harmonic bundle is flat if and only if it is a Higgs section with vanishing degree. We extend this correspondence first to sub-fibrations and then, via the graph construction, to morphisms.
\end{abstract}

\maketitle
\tableofcontents

\section{Introduction}\label{sec:introduction}
In \cite[\S2--4]{LS26}, we studied complex fiber bundles as nonlinear generalizations of complex vector bundles and developed the corresponding notions of $\dbar$-operator, connection, curvature, and fiberwise K\"ahler metrics. The present paper continues this program from the viewpoint of sections.

A linear connection on a complex vector bundle $E\to S$ is specified by a $\CC$-linear first-order operator $D_E:A^0(S,E)\to A^1(S,E)$ satisfying the Leibniz rule. Its curvature is defined as the composite of $D_E$ with its natural extension $D^1_E: A^1(S,E)\to A^2(S,E)$. In the nonlinear setting, this becomes more interesting. For an Ehresmann connection $\nabla^\RR:f^*TS^\RR\to TX^\RR$ on a smooth fiber bundle $f:X\to S$, it is classically known how to make a correct definition of covariant derivative of smooth sections using $\nabla^{\RR}$. We continue to develop the differential calculus for sections of \emph{complex} fiber bundles. Let $(f:X\to S,T_{X/S})$ be a complex fiber bundle over a complex manifold $S$ of dimension $n\ge 1$, equipped with a $\dbar$-operator and an almost connection $\partial$ (\cite[\S2]{LS26}). For a smooth section $u:S\to X$, taking the appropriate vertical $(1,0)$-components of $\D u$ defines $\dbar u\in A^{0,1}(S,u^*T_{X/S})$ and $\partial u\in A^{1,0}(S,u^*T_{X/S})$. If $\dbar$ is integrable, then $\dbar u=0$ precisely when $u$ is holomorphic. After choosing a suitable fiberwise linear connection, we construct an induced connection $\nabla^E$ on $E:=u^*T_{X/S}$, so that these first-order derivatives can themselves be differentiated. Under the lifting and relative holomorphicity hypotheses of Proposition \ref{prop:curv_second_deriv}, the second differentiation recovers the restriction of the nonlinear curvature \cite[\S2]{LS26} along the section, making a perfect analogy to the linear case.

We establish the following nonlinear Bochner--Kodaira--Nakano identity for sections of complex fiber bundles.
\begin{theorem}[Proposition \ref{prop:pointwise_identity} and Corollary \ref{cor:energy_id}]\label{thm:intro_nonlinear_BKN}
Let $(S,\omega_S)$ be a Hermitian manifold of dimension $n$, and let $(f,T_{X/S})$ be a complex fiber bundle equipped with a $\dbar$-operator $\dbar$, a fiberwise K\"ahler metric $\omega_{X/S}$, and a symplectic connection $\nabla^{1,0}$ induced by a real $(1,1)$-form $\omega_X$ on $X$ (with respect to the almost complex structure determined by $\dbar$), where $\omega_X$ restricts to $\omega_{X/S}$ on each fiber. Write $\partial$ for the induced almost connection and $\omega_X^H$ for the horizontal part of $\omega_X$. Then every smooth section $u:S\to X$ satisfies
\begin{equation}\label{eq:intro_pointwise_id}
|\partial u|^2-|\dbar u|^2=\Lambda_{\omega_S}\bigl(u^*\omega_X-u^*\omega_X^H\bigr).
\end{equation}
If $S$ is compact, then
\begin{equation}\label{eq:intro_energy_id}
\|\partial u\|_{L^2}^2-\|\dbar u\|_{L^2}^2=n\bigl(\deg_{\omega_X}(u)-\deg_{\omega_X^H}(u)\bigr),
\end{equation}
where $\deg_{\omega_X}(u):=\int_Su^*\omega_X\wedge\omega_S^{n-1}$ and $\deg_{\omega_X^H}(u):=\int_Su^*\omega_X^H\wedge\omega_S^{n-1}$.
\end{theorem}
\eqref{eq:intro_pointwise_id} generalizes the classical identity for maps between K\"ahler manifolds \cite{lichnerowicz69applications}. In the vector bundle case, \eqref{eq:intro_energy_id} is equivalent to the classical Bochner--Kodaira--Nakano identity on bundle-valued $0$-forms. Theorem \ref{thm:intro_nonlinear_BKN} immediately gives the following nonlinear vanishing theorem for holomorphic sections, which generalizes the classical Bochner-type vanishing theorem \cite[Ch.~3]{{Kobayashi87}} for holomorphic sections of holomorphic vector bundles.
\begin{theorem}[Lemma \ref{lem:potential_identity}, Theorems \ref{thm:potential_rigidity} and \ref{thm:degree_obstruction}]\label{thm:intro_vanishing_degree}
Let $(S,\omega_S)$ be a compact connected Hermitian manifold of dimension $n$, let $(f,T_{X/S})$ be a complex fiber bundle equipped with an integrable $\dbar$-operator $\dbar$, and set $\Delta_{\omega_S}:=\I\Lambda_{\omega_S}\partial_S\dbar_S$.
\begin{enumerate}[label=(\roman*)]
\item Suppose that $\phi\in C^\infty(X,\RR)$ is such that $\omega_X:=\I\,\partial_X\dbar_X\phi$ restricts to a K\"ahler form on each fiber, let $H$ be the symplectic connection induced by $\omega_X$, and set $q_\phi:=\Lambda_{\omega_S}\omega_X^H$. Then every holomorphic section $u:S\to X$ satisfies
\[
\Delta_{\omega_S}(\phi\comp u)=|\partial u|^2+q_\phi\comp u.
\]
If $q_\phi\geq0$, every holomorphic section is $H$-horizontal and has image in $q_\phi^{-1}(0)$. Conversely, every smooth $H$-horizontal section is holomorphic. In particular, if $q_\phi>0$ everywhere, then $f$ admits no holomorphic section.
\item Let $\omega_X$ be a real $(1,1)$-form on $X$ which is K\"ahler on every fiber and induces a symplectic connection $H$, and set $q_{\omega_X}:=\Lambda_{\omega_S}\omega_X^H$. Then every holomorphic section $u$ satisfies
\[
\|\partial u\|_{L^2}^2+\int_S(q_{\omega_X}\comp u)\,\omega_S^n=n\,\deg_{\omega_X}(u).
\]
If $q_{\omega_X}\geq0$, then $\deg_{\omega_X}(u)\geq0$, with equality if and only if $u$ is $H$-horizontal and $u(S)\subset q_{\omega_X}^{-1}(0)$.
\end{enumerate}
\end{theorem}

Let $(S,\omega_S)$ be a connected Hermitian manifold of dimension $n$, let $(Y,\omega_Y)$ be a Hermitian manifold, let $\rho:\pi_1(S)\to\Aut(Y,\omega_Y)$ be a representation by holomorphic isometries, and let $f:X:=\widetilde S\times_\rho Y\to S$ be the associated flat holomorphic bundle. The fiberwise form and the flat horizontal distribution determine a real $(1,1)$-form $\omega_X$ with $\omega_X^H=0$. Let $F_{TY}$ (resp. $F_{T^*S}$) denote the Chern curvature of $(TY,h_Y)$ (resp. $(T^*S,h_{T^*S})$). Taking $\rho$ to be trivial, the following result recovers Yau's Schwarz lemma \cite[Th.~2]{yau78schwarz}.
\begin{theorem}[Theorem \ref{thm:relative_yau_schwarz}]\label{thm:intro_relative_schwarz}
Assume that $(S,\omega_S)$ is complete K\"ahler and that $-\I\Lambda_{\omega_S}F_{T^*S}\geq-\mu\id_{T^*S}$ for some $\mu\geq0$. Let $u:S\to X$ be a holomorphic section. If $-2\langle F_{TY}(v,\bar v)w,w\rangle\geq\lambda|v|^2|w|^2$ for some $\lambda>0$ and all $y\in Y$, $v,w\in T_yY$, then
\[
|\partial u|^2\leq\lambda^{-1}\mu,\qquad 0\leq u^*\omega_X\leq\lambda^{-1}\mu\,\omega_S.
\]
In particular, if $\mu=0$, then $u$ is flat.
\end{theorem}

We next consider nonlinear Higgs fields on a complex fiber bundle. The Simpson--Uhlenbeck--Yau correspondence \cite[Cor.~1.3]{Si1} (also called the nonabelian Hodge correspondence in the literature) is an equivalence of categories between the category of semisimple flat bundles over a compact complex manifold of K\"ahler type and the category of polystable Higgs bundles with vanishing rational Chern classes. The essential case of zero Higgs field was treated in \cite{uhlenbeck_yau1986existence}. At the level of morphisms, full faithfulness follows by applying the natural identification between the global flat sections of a semisimple flat bundle and the global Higgs sections of the corresponding Higgs bundle to the relevant Hom-bundle. This is an elementary but important consequence of the K\"ahler identities for harmonic bundles, see \cite[Lem.~1.2~and~\S2]{Si1}.

In \cite[\S5--6]{LS26}, we initiated the study of a nonlinear Hodge correspondence, to generalize the Simpson--Uhlenbeck--Yau correspondence from vector bundles to fiber bundles. It is worth mentioning that a nonlinear Hodge correspondence in positive characteristic was established by the second-named author in \cite{She25b}. We obtained the following half nonlinear Hodge correspondence.
\begin{theorem}[{\cite[Th.~1.11]{LS26}}]\label{thm:half_nhc}
Let $S$ be a compact complex manifold of K\"ahler type. Then there is a faithful functor from the category of nonlinear flat bundles reductive of K\"ahler type over $S$, with morphisms being admissible morphisms as in \cite[Def.~5.20]{LS26}, to the category of nonlinear Higgs bundles over $S$, with morphisms being Higgs morphisms as in \cite[Prop.~5.22]{LS26}.
\end{theorem}
The former category is generally \emph{not} a full subcategory of the category of nonlinear flat bundles over $S$, since a fiberwise holomorphic flat morphism need not be admissible (Example \ref{ex:noncompact_flat_morphism_not_higgs}). Note that the former category contains the trivial object $({\rm id}: S\to S, \nabla)$, where $\nabla: {\rm id}^*TS\to TS$ is the tautological identity map, and the functor maps it to the trivial nonlinear Higgs bundle $({\rm id}: S\to S, \dbar, 0)$. This raises the natural problem of characterizing the image of the functor. In this paper, we address this problem for morphisms. We consider the larger class of fiberwise holomorphic flat morphisms, without assuming admissibility, and study their relation to Higgs morphisms. Let $(f: X\to S, \nabla)$ (resp.\ $(f, \dbar, \theta)$) be a nonlinear flat bundle (resp.\ Higgs bundle). We call the set of morphisms (without imposing admissibility) from the trivial flat (resp.\ Higgs) object to $(f,\nabla)$ (resp.\ $(f, \dbar, \theta)$) the set of \emph{flat sections} (resp.\ \emph{Higgs sections}) of $f$. Since a typical fiber of $f$ need not be a vector space, neither set of sections carries a natural vector-space structure. Such a structure is crucial to the classical argument in the linear case. This poses a fundamental difficulty in our investigation of this problem. Before describing our methods for overcoming this and other difficulties, we draw attention to the following algebro-geometric consequence.
\begin{theorem}[Example \ref{ex:degree_vertical_tangent}, Corollary \ref{cor:equiv_higgs_flat}, and Theorem \ref{thm:hE_harmonic}]\label{thm:main_fano_bundles}
Let $(f,\nabla)$ be a nonlinear flat bundle reductive of K\"ahler type and $(f,\dbar,\theta)$ be the corresponding nonlinear Higgs bundle. Suppose the typical fiber of $f$ is a K-polystable Fano manifold. For each flat \textup(resp.\ Higgs\textup) section $u: S\to X$ of $f$, the vector bundle $E=u^*T_{X/S}$ is equipped with the induced flat linear connection $\nabla_D^E$ \textup(resp.\ Dolbeault operator $\dbar_E$ and Higgs field $\Theta_u$\textup), which is the linearization of $\nabla$ \textup(resp.\ $(\dbar,\theta)$\textup) along $u$. The set of flat sections of $(f,\nabla)$ consists precisely of the Higgs sections of $(f,\dbar,\theta)$ satisfying $\deg_{\omega_S}E:=c_1(E)\cdot[\omega_S]^{n-1}=0$, where $\omega_S$ is any fixed K\"ahler metric on $S$. Moreover, when $u$ is simultaneously flat and Higgs, $(E,\nabla_D^E)$ and $(E,\dbar_E,-\Theta_u)$ correspond to each other under the Simpson--Uhlenbeck--Yau correspondence.
\end{theorem}

To formulate the nonlinear Higgs-flat correspondence for sections, first let $\theta$ be an almost Higgs field and define its action $\theta u$ on a section $u$ by its restriction along $u$. Our convention is \[
D''u=\dbar u-\theta u,\qquad D'u=\partial u-\bar\theta u.
\] Then a Higgs section is a section satisfying $D''u=0$, equivalently $\dbar u=0$ and $\theta u=0$ by type decomposition, while a flat section is a section satisfying $Du=0$ for $D=D'+D''$. In the case of vector bundles, under the canonical identification between the vertical tangent bundle along a section and the underlying vector bundle, the usual notions of Higgs sections and flat sections are recovered. The construction of the linearization of nonlinear flat (resp.\ Higgs) bundles along flat (resp.\ Higgs) sections is also carried out in this step. The theory becomes richer when fiberwise K\"ahler metrics are taken into account.

Now let $(S,\omega_S)$ be a Hermitian manifold of dimension $n$, and let $(f,T_{X/S})$ be a complex fiber bundle with connected fibers. Suppose that $f$ is equipped with a $\dbar$-operator satisfying the lifting condition (\cite[Def.~2.3,~Lem.~2.7]{LS26}), a fiberwise K\"ahler metric $\omega_{X/S}$, a K\"ahler connection $\nabla^{1,0}$ (\cite[Def.~4.5]{LS26}), and a relatively holomorphic almost Higgs field $\theta$ (\cite[Def.~5.1]{LS26}). Assume that $\omega_{X/S}$ is $\theta$-adapted, so that $\theta\in A^{1,0}(S,\mathfrak{k}_{X/S}^\CC)$ (\cite[Def.~5.5]{LS26}), and write $\bar\theta:=\bar\theta_{\omega_{X/S}}$ for the conjugate Higgs field. Set
\[
D''=\dbar+\theta,\qquad D'=\partial+\bar\theta,\qquad D=D'+D'',\qquad D^c=D''-D'.
\]
Their induced actions on sections are the ones displayed above. Let $\omega_X$ be a real $(1,1)$-form on $X$ restricting to $\omega_{X/S}$ on the fibers and inducing $\nabla^{1,0}$, and let $\nabla^\RR$ be the underlying real connection. The following notion of nonlinear harmonic bundle is slightly more flexible than the one introduced in \cite[Def.~1.13~and~\S5.2]{LS26}, where the Chern connection rather than the K\"ahler connection is used.
\begin{definition}
Notation as above. A nonlinear harmonic bundle is a quadruple $(f,\omega_{X/S},D'',D)$ with vanishing pseudo-curvature $G_{D''}$ and vanishing curvature $F_D$.
\end{definition}
By definition, a nonlinear harmonic bundle gives $f$ simultaneously the structure of a nonlinear flat bundle and the structure of a nonlinear Higgs bundle over $S$.
\begin{remark}
The notion in \cite{LS26} arises from a nonlinear version of the Simpson mechanism. See Remark \ref{rmk:relation_to_old_harmonic} for the relation between that notion and the one used here.
\end{remark}
We impose the following \emph{four assumptions}: for each $s\in S$, each element in $\mathfrak{k}_s^\mathrm{R}$ (consisting of the real part of holomorphic vector fields in $\mathfrak{k}_s$, which is the fiber of $\mathfrak{k}_{X/S}$ at $s$) is Hamiltonian; there is a smooth fiberwise equivariant comoment map $\mu^*$ satisfying $\D_{X_s}\mu^*(\xi)=-\iota_\xi\omega_s$ and $\mu^*([\xi,\eta])=\omega_s(\xi,\eta)$; the comoment map is parallel with respect to $\nabla^\RR$; and $\omega_X$ is closed and $\mu^*F_{\nabla^\RR}$ is defined, where $F_{\nabla^\RR}$ is the curvature of $\nabla^\RR$. By the minimal coupling \cite[Eq.~(1.12)]{GLS96},
\[
\omega_X^H=-\mu^*F_{\nabla^\RR}+f^*\zeta
\]
for a real $2$-form $\zeta$ on $S$, where $\omega_X^H$ is the horizontal part of $\omega_X$.

Our main result in the next step is the $D'$-$D''$ and $D^c$-$D$ identities derived from \eqref{eq:intro_pointwise_id}. See Section \ref{sec:kahler_id} for more general statements.

\begin{theorem}[Theorems \ref{thm:D'_D''u_identity} and \ref{thm:Dc_Du_identity}]\label{thm:intro_kahler_identities}
In the setup above, assume that the four assumptions just stated hold. Then for every smooth section $u:S\to X$,
\[
|D'u|^2-|D''u|^2=\Lambda_{\omega_S}\bigl(u^*\omega_X+u^*\mu^*\bigl[F_D^{1,1}-G_{D'}^{1,1}-G_{D''}^{1,1}\bigr]_\RR\bigr)-\Lambda_{\omega_S}\zeta.
\]
Let $K_{v^\RR}:=(\bar\theta(\bar v)-\theta(v))^\RR$ for $v^\RR:=v+\bar v\in TS^\RR$ and define $\alpha\in A^1(X,\RR)$ by $\alpha(W):=\mu_{f(x)}^*(K_{\D f_x(W)})(x)$ for $W\in T_xX^\RR$. Then
\[
|D^cu|^2-|Du|^2=\Lambda_{\omega_S}\bigl(-2\D(u^*\alpha)+2u^*\mu^*\bigl([G_{D'}^{1,1}-G_{D''}^{1,1}]_\RR\bigr)\bigr).
\]
Here $F_D^{1,1}$ is the $(1,1)$-curvature of $D$, $G_{D''}^{1,1}$ and $G_{D'}^{1,1}$ are the $(1,1)$-pseudo-curvature terms of $D''$ and $D'$, and $[\cdot]_\RR$ denotes the corresponding real form \textup(see \eqref{eq:realification_11}\textup).
\end{theorem}

For a smooth section $u$ of $f$, recall from Theorem \ref{thm:intro_nonlinear_BKN} that $\deg_{\omega_X}(u):=\int_Su^*\omega_X\wedge\omega_S^{n-1}$, and set $\deg_{\omega_X,\zeta}(u):=\int_S(u^*\omega_X-\zeta)\wedge\omega_S^{n-1}$. As a direct consequence of the preceding identities, we obtain the \emph{Higgs--flat correspondence for sections} of a nonlinear harmonic bundle as follows.
\begin{corollary}[Corollary \ref{cor:equiv_higgs_flat}]\label{cor:intro_section_correspondence}
Suppose that $(S,\omega_S)$ is compact semi-K\"ahler, i.e., $\D\, \omega_S^{n-1}=0$. Let $(f,\omega_{X/S},D'',D)$ be a nonlinear harmonic bundle over $S$ satisfying the four assumptions, and suppose that $\nabla^\RR$ preserves $\mathfrak{k}_{X/S}^{\mathrm R}$. Then a smooth section $u:S\to X$ is flat if and only if it is a Higgs section and $\deg_{\omega_X,\zeta}(u)=0$.
\end{corollary}

This recovers Simpson's statement for harmonic vector bundles. Indeed, with the standard choices of the relative form and comoment map in the vector bundle case, one has $\zeta=0$, and every section is homotopic to the zero section, so $\deg_{\omega_X,\zeta}(u)=0$ holds automatically. For general fiber bundles, vanishing of the degree is a nontrivial obstruction for a Higgs section to be flat, see Example \ref{ex:triv_bundle_sec}. The Hamiltonian hypothesis cannot in general be omitted, see Example \ref{ex:flat_not_higgs}.

Sections are special cases of sub-fibrations of a complex fiber bundle.
\begin{definition}
Let $(f: X\to S, T_{X/S})$ be a complex fiber bundle whose fibers are $m$-dimensional. A sub-fibration of relative dimension $0\le k\le m$ of $f$ is a pair $(\Sigma,u)$, where $\Sigma$ is a complex manifold of complex dimension $n+k$ and $u:\Sigma\to X$ is a smooth map, satisfying
\begin{enumerate}[label=(\alph*)]
\item the composition $\pi:=f\comp u:\Sigma\to S$ is a surjective holomorphic map whose fibers have complex dimension $k$;
\item if $k=0$, $\pi$ is locally semi-finite in the sense of \cite[\S2]{AJS04};
\item if $k\ge 1$, $\pi$ is a holomorphic submersion, and for every $s\in S$ the restriction $u_s:=u|_{\pi^{-1}(s)}:\pi^{-1}(s)\to X_s$ is a holomorphic immersion.
\end{enumerate}
\end{definition}
For a sub-fibration $(\Sigma, u)$ of $f$, set $F:=\D u(T_{\Sigma/S})\subset u^*T_{X/S}$ (resp.\ $F:=0$) when $k\ge 1$ (resp.\ $k=0$), and $N:=u^*T_{X/S}/F$. Our third step is to extend the Higgs--flat correspondence for sections to a correspondence for sub-fibrations. We call $(\Sigma,u)$ a \emph{Higgs sub-fibration} if $(D''u)_N=0$, and a \emph{flat sub-fibration} if $(Du)_N=0$ (see Definition \ref{def:higgs_flat_subfib}). These sub-objects are more general than subbundles in the vector bundle case, since their fibers need not be linear subspaces.

Fix a Hermitian form $\omega_\Sigma$ on $\Sigma$. In the setup of Theorem \ref{thm:intro_kahler_identities}, when $\Sigma$ is compact we define the \emph{combined degree} of $(\Sigma,u)$ by
\begin{align}
\deg_\zeta(\Sigma,u)&:=\deg(\Sigma,u)-\int_\Sigma\pi^*\zeta\wedge\omega_\Sigma^{n+k-1},\\
\deg(\Sigma,u)&:=\deg_{\omega_X}(\Sigma,u)-\deg_{F,\partial}(\Sigma,u)+\deg_{F,\dbar}(\Sigma,u)+\deg_{\theta,F}(\Sigma,u),\label{eq:comb_deg}
\end{align}
where the last three terms measure, respectively, the tangent component of $\partial u$, the tangent component of $\dbar u$, and the imbalance between $\theta$ and $\bar\theta$ along $F$ (see Definition \ref{def:deg_subfib}).

There is also a degree defined directly by fiber integration. Suppose that $k\geq0$, $\pi$ is proper, and $S$ is compact, and assume that $\pi$ is unramified when $k=0$. Let $\hat\omega_{u,\zeta}:=\pi_*\bigl(u^*(\omega_X-f^*\zeta)^{k+1}\bigr)/(k+1)$. For a Hermitian form $\omega_S$ on $S$, define the \emph{transgressed degree} by
\begin{equation}\label{eq:trans_deg}
\deg_{\omega_X,\zeta}^{\mathrm{tr}}(\Sigma,u):=\int_S\hat\omega_{u,\zeta}\wedge\omega_S^{n-1}.
\end{equation}
When $\omega_\Sigma$ is a split Hermitian metric as in Lemma \ref{lem:deg_theta_F_split}, $\deg_\zeta(\Sigma,u)$ and $\deg_{\omega_X,\zeta}^{\mathrm{tr}}(\Sigma,u)$ differ by a positive factor (Lemma \ref{lem:combined_transgressed_degree}).

The \emph{Higgs--flat correspondence for sub-fibrations} is as follows.
\begin{theorem}[Theorems \ref{thm:D'_D''u_identity_subfib}, \ref{thm:Dc_Du_identity_subfib}, and \ref{thm:Dc_Du_identity_subfib_transgression}]\label{thm:intro_subfibration_correspondence}
Let $(f,\omega_{X/S},D'',D)$ be a nonlinear harmonic bundle over a compact Hermitian manifold $S$ satisfying the four assumptions, and suppose that $\nabla^\RR$ preserves $\mathfrak{k}_{X/S}^{\mathrm R}$. Let $(\Sigma,u)$ be a sub-fibration such that $(\Sigma, \omega_{\Sigma})$ is compact semi-K\"ahler. Then
\begin{enumerate}[label=(\roman*)]
\item if $(\Sigma,u)$ is a flat sub-fibration, then $\deg_\zeta(\Sigma,u)=0$, and it is a Higgs sub-fibration;
\item if $(\Sigma,u)$ is a Higgs sub-fibration, then $\deg_\zeta(\Sigma,u)\ge 0$, and it is a flat sub-fibration if and only if $\deg_\zeta(\Sigma,u)=0$.
\end{enumerate}
\end{theorem}

The notion of transgressed degree generalizes that of the difference of slopes in the vector bundle case. Theorem \ref{thm:intro_subfibration_correspondence}\,(ii) may be viewed as a nonlinear analogue of the slope inequalities that underlie the polystability of a harmonic vector bundle. In future work, we will investigate the relation between the stability of a complex fiber bundle and the existence of a harmonic metric.

When $k=0$, all terms in \eqref{eq:comb_deg} involving $F$ vanish, and Theorem \ref{thm:intro_subfibration_correspondence} generalizes Corollary \ref{cor:intro_section_correspondence} to possibly ramified finite multi-sections. When $k\ge 1$, these extra terms may be nontrivial, see Example \ref{ex:higgs_subfib_conic}.

Our final step is to reduce the correspondence for morphisms to the case of sub-fibrations. In the linear theory, a morphism between harmonic vector bundles is regarded as a section of the harmonic vector bundle $\Hom(E_1,E_2)$. By analogy, we view a morphism $\mathsf{F}:X_1\to X_2$ over $S$ through its graph inside $X_1\times_S X_2$. Proposition \ref{prop:morphism_subfib_correspondence} identifies Higgs and fiberwise holomorphic flat morphisms with Higgs and flat graph sub-fibrations, respectively. The following result is the \emph{Higgs--flat correspondence for morphisms} between nonlinear harmonic bundles.

\begin{theorem}[Theorem \ref{thm:morphism_correspondence}]\label{thm:intro_morphism_correspondence}
Let $(S,\omega_S)$ be a compact Hermitian manifold. For $i=1,2$, let $(f_i:X_i\to S,\omega_{X_i/S},D_i'',D_i)$ be nonlinear harmonic bundles with connected fibers of complex dimension $m_i$, satisfying the four assumptions, and assume that the corresponding real connections preserve $\mathfrak{k}_{X_i/S}^{\mathrm R}$. Assume that $X_1$ is compact and carries a semi-K\"ahler metric $\omega_\Sigma$ for the complex structure determined by $\dbar_1$. Let $\mathsf{F}:X_1\to X_2$ be a fiberwise holomorphic smooth morphism over $S$. Let $\deg_\zeta(\mathsf F):=\deg(\mathsf F)-\int_{X_1}f_1^*(\zeta_1+\zeta_2)\wedge\omega_\Sigma^{n+m_1-1}$, where $\deg(\mathsf{F})$ is defined by applying \eqref{eq:comb_deg} to the graph of $\mathsf{F}$ in $X_1\times_SX_2$. Then
\begin{enumerate}[label=(\roman*)]
\item if $\mathsf{F}$ is a flat morphism, then $\deg_\zeta(\mathsf{F})=0$, and $\mathsf{F}$ is a Higgs morphism;
\item if $\mathsf{F}$ is a Higgs morphism, then $\deg_\zeta(\mathsf F)\geq0$, and $\mathsf{F}$ is a flat morphism if and only if $\deg_\zeta(\mathsf{F})=0$.
\end{enumerate}
\end{theorem}
Here, as in Definition \ref{def:Higgs_flat_morphism}, $\mathsf{F}$ is a \emph{Higgs morphism} if it is pseudo-holomorphic for the almost complex structures determined by $\dbar_1$ and $\dbar_2$ and satisfies $\D\mathsf{F}(\theta_1(v)|_x)=\theta_2(v)|_{\mathsf{F}(x)}$ for every $s\in S$, $v\in T_sS$, and $x\in X_{1,s}$; $\mathsf{F}$ is a \emph{flat morphism} if it maps the $D_1$-horizontal distribution into the $D_2$-horizontal distribution. In general, a Higgs morphism may not be a flat morphism (Example \ref{ex:higgs_morphism_not_flat}), and a fiberwise holomorphic flat morphism may not be a Higgs morphism when $X_1$ is noncompact (Example \ref{ex:noncompact_flat_morphism_not_higgs}).

We finally apply the theory to the characterization problem of the nonlinear Hodge correspondence for morphisms. Let $S$ be a compact connected complex manifold of K\"ahler type, and fix $s_0\in S$. Let $(f_i: X_i\to S,\nabla_i)$ ($i=1,2$) be two nonlinear flat bundles over $S$ reductive of K\"ahler type. Let $X_{i,s_0}$ be the typical fiber of $f_i$, and write $\rho_i:\pi_1(S,s_0)\to G_i\leq\Aut_0(X_{i,s_0})$ for the chosen reductive monodromy representation. By the definition of K\"ahler type (\cite[Def.~1.10]{LS26}), there is a model K\"ahler metric $\omega_{X_{i,s_0}}$ on $X_{i,s_0}$ such that the isometry subgroup $K_i={\rm Stab}_{G_i}(\omega_{X_{i,s_0}})$ is a compact real form of $G_i$.
\begin{theorem}[Theorem \ref{thm:morphism_correspondence}, {\cite[Props.~5.22 and 5.23]{LS26}}]\label{thm:mor_intro}
Let $(f_i: X_i\to S,\dbar_i,\theta_i)$ \textup($i=1,2$\textup) be the nonlinear Higgs bundle corresponding to the nonlinear flat bundle $(f_i,\nabla_i)$. Suppose that the action of $K_i$ on $(X_{i,s_0},\omega_{X_{i,s_0}})$ is Hamiltonian and $X_1$ is compact. Choose the relatively K\"ahler forms on $X_i$ by the construction of Remark \ref{rmk:moment_existence}, so that $\zeta_1=\zeta_2=0$, and fix a K\"ahler metric on $X_1$ to define $\deg(\mathsf F)$. Then there is a bijection between the set $\mathfrak{M}^{\mathrm{fl}}_{12}$ of fiberwise holomorphic flat morphisms $\mathsf{F}: X_1\to X_2$ and the set $\mathfrak{M}^{\mathrm{Hig}}_{12}$ of Higgs morphisms $\mathsf{F}$ with $\deg(\mathsf{F})=0$. The set $\mathfrak{M}^{\mathrm{fl}}_{12}$ contains the set $\mathfrak{M}^{\mathrm{adm}}_{12}$ of admissible morphisms from $f_1$ to $f_2$ in the category of nonlinear flat holomorphic bundles reductive of K\"ahler type. If both $X_{1,s_0}$ and $X_{2,s_0}$ are projective, then $\mathfrak{M}^{\mathrm{adm}}_{12}=\mathfrak{M}^{\mathrm{fl}}_{12}$.
\end{theorem}
The algebraicity condition of \cite[Prop.~5.23\,(ii)]{LS26} is satisfied by \cite[Rmk.~2.2]{GM18}, since the action of $K_i$ on $X_{i,s_0}$ is Hamiltonian.

A typical class of examples to which the above theorem applies consists of flat bundles with polarized cscK typical fibers $(X_{i,s_0},H_i,\omega_{X_{i,s_0}})$ and reductive monodromy representations in $G_i=\Aut_0(X_{i,s_0},H_i)$. The theorem treats fiberwise holomorphic flat morphisms without imposing admissibility.
\smallskip

The paper is organized as follows. Section \ref{sec:diff_op} develops the differential calculus of $\dbar$-operators, almost connections, Higgs fields, and linearizations along sections. Section \ref{sec:nonlin_BKN} establishes the nonlinear Bochner--Kodaira--Nakano identity and applies it to the vanishing theorem and the Schwarz lemma. Section \ref{sec:kahler_id} incorporates nonlinear Higgs fields and proves the $D'$--$D''$ and $D^c$--$D$ identities, from which the Higgs--flat correspondence for sections follows. Section \ref{sec:subfibrations} extends these results to sub-fibrations. Section \ref{sec:morphisms} applies the graph construction to obtain the correspondence for morphisms.
\smallskip

\noindent{\bf Acknowledgements.} The authors would like to thank Zhaofeng Yu for valuable discussions, which made us aware of the possibility of Theorem \ref{thm:hE_harmonic}. The first-named author is supported by the Shuimu Tsinghua Scholar Program. This work is supported by the Chinese Academy of Sciences Project for Young Scientists in Basic Research (Grant No. YSBR-032).

\section{Differential calculus of sections}\label{sec:diff_op}

In this section we interpret connections, $\dbar$-operators, and Higgs fields as first-order (or zeroth-order) differential operators acting on sections of a fiber bundle. This viewpoint generalizes the classical picture for vector bundles. A new feature that appears naturally in the nonlinear setting is the linearization along flat sections and Higgs sections.

\subsection{Covariant derivatives of sections}\label{subsec:cov_der}
Let $f:X\to S$ be a smooth fiber bundle over a smooth manifold $S$ (we will later specialize to $S$ complex).  Let $u:S\to X$ be a smooth section of $f$.  The differential of $u$ is a smooth bundle map $\D u: TS^\RR \to u^*TX^\RR$, where $TS^\RR$ and $TX^\RR$ denote the real tangent bundles of $S$ and $X$. Since $f\comp u=\id_S$, we have $\D f\comp \D u=\id_{TS^\RR}$, so $\D u$ is a splitting of the pulled-back exact sequence
\begin{equation}\label{eq:pullback_tangent_seq}
0\to u^*T_{X/S}^\RR \to u^*TX^\RR \xrightarrow{u^* \D f} TS^\RR \to 0.
\end{equation}

Now let $\nabla^\RR$ be a connection on $f$, i.e., a smooth splitting of the exact sequence
\[0\to T_{X/S}^{\RR}\to TX^{\RR}\xrightarrow{\D f} f^\ast TS^{\RR}\to 0.\]
  This splitting defines a projection $\mathrm{pr}^{\mathrm{v}}_{\nabla^\RR}: TX^\RR \to T_{X/S}^\RR$.

\begin{definition}[{\cite[p.~56]{sardanashvili2013advanced}}]\label{def:cov_deriv_real}
The \emph{covariant derivative} of the section $u$ with respect to $\nabla^\RR$ is the smooth section
\[
\nabla^\RR u := u^*\mathrm{pr}^{\mathrm{v}}_{\nabla^\RR}\comp \D u \;\in\; C^\infty\bigl(S,T^*S^\RR\otimes u^*T_{X/S}^\RR\bigr).
\]
More precisely, for $v\in T_sS^\RR$,
\begin{equation}\label{eq:cov_deriv_real_def}
(\nabla^\RR u)(v) = \D u(v) - \nabla^\RR(v)\big|_{u(s)}=:\nabla_v^\RR u \in (u^*T_{X/S}^\RR)_s,
\end{equation}
where $\nabla^\RR(v)|_{u(s)}$ denotes the horizontal lift of $v$ to $T_{u(s)}X^\RR$.

A section $u$ is called \emph{flat} (or \emph{horizontal}) with respect to $\nabla^\RR$ if $\nabla^\RR u=0$.  This means that for every $v\in T_s S^\RR$, the vector $\D u(v)$ is horizontal, i.e.,  for every smooth curve $\gamma$ in $S$, $u\comp \gamma$ is the horizontal lift of $\gamma$ through $u(\gamma(0))$.
\end{definition}

\begin{remark}
Unlike the vector bundle case, the space of sections of $f$ does not have a natural vector space structure, and there is no natural notion of ``multiplying a section by a function'' for a general fiber bundle, so $\nabla^\RR u$ does not satisfy a Leibniz rule.
\end{remark}

Now suppose $S$ is a complex manifold of dimension $n$, and $(f, T_{X/S})$ is a complex fiber bundle (\cite[Def.~2.1]{LS26}) with fibers of complex dimension $m$. The complexified differential $\D u^\CC: TS^\CC \to u^*TX^\CC$ decomposes as $\D u^\CC = \partial_S u + \dbar_S u$, where $\partial_S u := \D u^\CC|_{TS}$ and $\dbar_S u := \D u^\CC|_{\widebar{TS}}$ are the $(1,0)$- and $(0,1)$-components with respect to $S$, $TS^\CC=TS\oplus \widebar{TS}$ with $TS$ being the holomorphic tangent bundle of $S$.

Let $\nabla = \nabla^{1,0}+\nabla^{0,1}$ be a complex connection on $(f,T_{X/S})$, which induces a projection $\mathrm{pr}^{\mathrm{v}}_\nabla: TX^\CC\to T_{X/S}^\CC$.  The complex fiber bundle structure gives a further decomposition $T_{X/S}^\CC = T_{X/S}\oplus \widebar{T_{X/S}}$, with projections $\mathrm{pr}^{1,0}:T_{X/S}^\CC\to T_{X/S}$ and $\mathrm{pr}^{0,1}:T_{X/S}^\CC\to \widebar{T_{X/S}}$.

\begin{definition}\label{def:cov_deriv_complex}
The \emph{$(1,0)$-covariant derivative} of $u$ with respect to $\nabla^{1,0}$ is
\[
\nabla^{1,0} u := u^*\mathrm{pr}^{\mathrm{v}}_\nabla\comp\, \partial_S u \;\in\; C^\infty\bigl(S,\, T^*S\otimes u^*T_{X/S}^\CC\bigr).
\]
The \emph{$(0,1)$-covariant derivative} of $u$ with respect to $\nabla^{0,1}$ is
\[
\nabla^{0,1} u := u^*\mathrm{pr}^{\mathrm{v}}_\nabla\comp\, \dbar_S u \;\in\; C^\infty\bigl(S,\, \widebar{T^*S}\otimes u^*T_{X/S}^\CC\bigr).
\]
\end{definition}
If $\nabla$ is the complexification of a real connection $\nabla^\RR$, then $\nabla u = \nabla^{1,0} u + \nabla^{0,1} u$ is the complexification of $\nabla^\RR u$.

Choose adapted local coordinates $(s^1,\ldots,s^n,z^1,\ldots,z^m)$ on $X$, which means $\{s^i\}$ are holomorphic coordinates on $S$ and $\{z^\alpha\}$ are holomorphic coordinates along the fibers, and $f$ is locally given by $(s,z)\mapsto s$. Suppose $\nabla^{1,0}$ and $\nabla^{0,1}$ are locally given by $\nabla^{1,0}(\partial_i) = \partial_i + \Gamma_i^\alpha \partial_\alpha + \Gamma_i^{\bar{\beta}} \partial_{\bar{\beta}}$ and $\nabla^{0,1}(\partial_{\bar{i}}) = \partial_{\bar{i}} + \Gamma_{\bar{i}}^\alpha \partial_\alpha + \Gamma_{\bar{i}}^{\bar{\beta}} \partial_{\bar{\beta}}$, where $\partial_i=\frac{\partial}{\partial s^i},\partial_{\bar{i}}=\frac{\partial}{\partial \bar{s}^i}, \partial_\alpha=\frac{\partial}{\partial z^\alpha},\partial_{\bar{\beta}}=\frac{\partial}{\partial \bar{z}^\beta}$. Throughout, we use the Einstein summation convention.

\begin{lemma}\label{lem:cov_deriv_local}
Write the section $u$ locally as $s\mapsto (s, u^1(s),\ldots, u^m(s))$.  Then
\begin{align}
\nabla^{1,0} u &= \bigl(\partial_i u^\alpha - \Gamma_i^\alpha\comp u\bigr)\,\D s^i\otimes \partial_\alpha + \bigl(\partial_i \bar{u}^\beta - \Gamma_i^{\bar{\beta}}\comp u\bigr)\,\D s^i\otimes \partial_{\bar{\beta}},\label{eq:cov_deriv_10_local}\\
\nabla^{0,1} u &= \bigl(\partial_{\bar{i}} u^\alpha - \Gamma_{\bar{i}}^\alpha\comp u\bigr)\,\D \bar{s}^i\otimes \partial_\alpha + \bigl(\partial_{\bar{i}} \bar{u}^\beta - \Gamma_{\bar{i}}^{\bar{\beta}}\comp u\bigr)\,\D \bar{s}^i\otimes \partial_{\bar{\beta}}.\label{eq:cov_deriv_01_local}
\end{align}
\end{lemma}

\begin{proof}
At a point $s\in S$, we have
\[
\partial_S u(\partial_i)\big|_{u(s)} = \partial_i + (\partial_i u^\alpha)\partial_\alpha + (\partial_i \bar{u}^\beta)\partial_{\bar{\beta}} \;\in\; T_{u(s)}X^\CC.
\]
The horizontal lift of $\partial_i$ to $u(s)$ is
\[
\nabla^{1,0}(\partial_i)\big|_{u(s)} = \partial_i + \Gamma_i^\alpha(u(s))\,\partial_\alpha + \Gamma_i^{\bar{\beta}}(u(s))\,\partial_{\bar{\beta}}.
\]
Taking the difference,
\[
(\nabla^{1,0} u)(\partial_i) = (\partial_i u^\alpha - \Gamma_i^\alpha\comp u)\,\partial_\alpha + (\partial_i \bar{u}^\beta - \Gamma_i^{\bar{\beta}}\comp u)\,\partial_{\bar{\beta}}.
\]
The computation for $\nabla^{0,1} u$ is similar.
\end{proof}

\subsection{$\dbar$-operators and Higgs fields on sections}\label{subsec:dbar_sections}
Let $\dbar$ be a $\dbar$-operator (\cite[Def.~2.3]{LS26}) on the complex fiber bundle $(f,T_{X/S})$, and let $\nabla^{0,1}$ be any $(0,1)$-connection inducing $\dbar$.

\begin{definition}\label{def:dbar_on_section}
The \emph{$\dbar$-derivative} of a smooth section $u$ of $f$ is
\[
\dbar u := \mathrm{pr}^{1,0}\comp\, \nabla^{0,1} u \;\in\; C^\infty\bigl(S,\, \widebar{T^*S}\otimes u^*T_{X/S}\bigr),
\]
where $\mathrm{pr}^{1,0}:T_{X/S}^\CC\to T_{X/S}$ is the projection.  This is independent of the choice of a $(0,1)$-connection $\nabla^{0,1}$ inducing $\dbar$, since two such connections differ by an element of $C^\infty(X, f^*\widebar{T^*S}\otimes \widebar{T_{X/S}})$, which is annihilated by $\mathrm{pr}^{1,0}$.

In adapted local coordinates, with $\dbar$ given by $\partial_{\bar{i}} \mapsto [\partial_{\bar{i}} + \Gamma_{\bar{i}}^\alpha \partial_\alpha] \pmod{\widebar{T_{X/S}}}$, we have
\begin{equation}\label{eq:dbar_section_local}
\dbar u = \bigl(\partial_{\bar{i}} u^\alpha - \Gamma_{\bar{i}}^\alpha\comp u\bigr)\,\D \bar{s}^i\otimes \partial_\alpha.
\end{equation}
\end{definition}

\begin{remark}\label{rmk:dbar_hol_section}
For any smooth section $u$, $\dbar u=0$ if and only if $u$ is pseudo-holomorphic, i.e., $\D u\comp J_S=J\comp\D u$ for the almost complex structure $J$ on $X$ determined by $\dbar$ (\cite[Prop.~2.6]{LS26}). Indeed, in adapted local coordinates $\widebar{TX}$ is spanned by $\partial_{\bar\alpha}$ and $\partial_{\bar i}+\Gamma_{\bar i}^\alpha\partial_\alpha$, where $\dbar(\partial_{\bar i})=[\partial_{\bar i}+\Gamma_{\bar i}^\alpha\partial_\alpha]\bmod\widebar{T_{X/S}}$. $\D u(\partial_{\bar i})\big|_{u(s)}=\partial_{\bar i}+(\partial_{\bar i}u^\alpha)\,\partial_\alpha+(\partial_{\bar i}\bar u^\beta)\,\partial_{\bar\beta}$, which lies in $\widebar{TX}$ if and only if its $T_{X/S}$-component satisfies $\partial_{\bar i}u^\alpha=\Gamma_{\bar i}^\alpha\comp u$. By \eqref{eq:dbar_section_local}, this is exactly $\dbar u=0$. When $\dbar$ is integrable \cite[Prop.~2.6]{LS26}, we may choose adapted local holomorphic coordinates in which $\Gamma_{\bar{i}}^\alpha=0$.  In this case, $\dbar u = (\partial_{\bar{i}} u^\alpha)\,\D\bar{s}^i\otimes\partial_\alpha$, so $\dbar u=0$ if and only if $u$ is a holomorphic section.
\end{remark}

Similarly, we define the action of an almost connection \cite[Def.~2.16]{LS26} on a section as follows.

\begin{definition}\label{def:partial_on_section}
Let $\partial$ be an almost connection on $(f,T_{X/S})$, and let $\nabla^{1,0}$ be any $(1,0)$-connection inducing $\partial$.  The \emph{$\partial$-derivative} of a smooth section $u$ is
\[
\partial u := \mathrm{pr}^{1,0}\comp\, \nabla^{1,0} u \;\in\; C^\infty\bigl(S,\, T^*S\otimes u^*T_{X/S}\bigr),
\]
which is independent of the choice of $\nabla^{1,0}$ inducing $\partial$.  In adapted local coordinates,
\begin{equation}\label{eq:partial_section_local}
\partial u = \bigl(\partial_i u^\alpha - \Gamma_i^\alpha\comp u\bigr)\,\D s^i\otimes \partial_\alpha.
\end{equation}
\end{definition}
The following lemma is immediate from \eqref{eq:dbar_section_local} and \eqref{eq:partial_section_local}.
\begin{lemma}\label{lem:dbar_section_affine}
Let $(f,T_{X/S})$ be a complex fiber bundle and let $\dbar,\dbar_0$ be two $\dbar$-operators on it, so that their difference defines a tensor $\Phi:=\dbar-\dbar_0\in C^\infty(X,f^*\widebar{T^*S}\otimes T_{X/S})$. Then for every smooth section $u:S\to X$,
\begin{equation}\label{eq:dbar_section_affine}
\dbar u-\dbar_0 u=-u^*\Phi\quad\text{in }A^{0,1}(S,u^*T_{X/S}).
\end{equation}
The analogous identity holds for almost connections: if $\partial-\partial_0=\Psi\in C^\infty(X,f^*T^*S\otimes T_{X/S})$, then $\partial u-\partial_0 u=-u^*\Psi$.
\end{lemma}

\begin{definition}\label{def:higgs_on_section}
Let $\theta\in C^\infty(X, f^*T^*S\otimes T_{X/S})$ be an almost Higgs field on $(f,T_{X/S})$.  The \emph{restriction of $\theta$ to a section $u$} is the pullback
\[
\theta u := u^*\theta \;\in\; C^\infty\bigl(S,\, T^*S\otimes u^*T_{X/S}\bigr).
\]
Concretely, for $v\in T_sS$,
\[
(\theta u)(v) = \theta(v)\big|_{u(s)} \;\in\; (u^*T_{X/S})_s.
\]
In adapted local coordinates, if $\theta = \theta_i^\alpha\,\D s^i\otimes\partial_\alpha$, then
\begin{equation}\label{eq:higgs_section_local}
\theta u = (\theta_i^\alpha\comp u)\,\D s^i\otimes\partial_\alpha.
\end{equation}
\end{definition}

\begin{definition}\label{def:higgs_section}
Let $D'':=\dbar+\theta$ be the sum of a $\dbar$-operator and an almost Higgs field. A section $u$ is called a \emph{Higgs section} if $D'' u := \dbar u - \theta u = 0$ (here we use the minus sign because of Lemma \ref{lem:dbar_section_affine}), i.e., $\dbar u=\theta u=0$ by type decomposition.  This means that $u$ is a pseudo-holomorphic section (Remark \ref{rmk:dbar_hol_section}) and $\theta$ restricts to zero along $u$.
\end{definition}

\subsection{Second-order operators and curvatures}\label{subsec:second_order}
In the classical theory of vector bundles, the curvature of a connection $D$ is defined by $F_D=D\comp D$.  For a section $u$ of a complex fiber bundle, $\partial u\in C^\infty(S,T^*S\otimes u^*T_{X/S})$ and $\dbar u\in C^\infty(S,\widebar{T^*S}\otimes u^*T_{X/S})$, and to differentiate these again one needs a connection on the vector bundle $E:=u^*T_{X/S}\to S$.

\begin{lemma}\label{lem:dbar_bracket_vertical}
Let $(f,T_{X/S})$ be a complex fiber bundle with a $\dbar$-operator $\dbar$. Let $TX\subset TX^\CC$ be the $\I$-eigenbundle of the almost complex structure $J$ induced by $\dbar$ via \cite[Prop.~2.6]{LS26}. Then, for every $\widebar{Z}\in C^\infty(X,\widebar{TX})$ and every $W\in C^\infty(X,T_{X/S})$, $[\widebar{Z},W]\in C^\infty(X,T_{X/S}\oplus\widebar{TX})$.
\end{lemma}
\begin{proof}
We have $T_{X/S}\oplus \widebar{TX}=\ker\bigl(TX^\CC\xrightarrow{\D f^\CC}f^*TS^\CC\xrightarrow{\mathrm{pr}^{1,0}} f^*TS\bigr)$. It suffices to show that
$\mathrm{pr}^{1,0}\D f^\CC([\widebar{Z},W])=0$. Let $\alpha$ be a local $(1,0)$-form on $S$.  By Cartan's formula,
\begin{align*}
	(f^*\alpha)([\widebar{Z},W])&=\widebar{Z}\bigl((f^*\alpha)(W)\bigr)-W\bigl((f^*\alpha)(\widebar{Z})\bigr)-\D(f^*\alpha)(\widebar{Z},W)\\&=-f^*(\D\alpha)(\widebar{Z},W)	=-\D\alpha(\D f^\CC(\widebar{Z}),\D f^\CC(W))=0,
\end{align*}
since $\D f^\CC(W)=0$.  The claim follows.
\end{proof}

\begin{definition}\label{def:canonical_dbar_TXS}
Let $(f,T_{X/S})$, $\dbar$, and $TX$ be as in Lemma~\ref{lem:dbar_bracket_vertical}.  The \emph{canonical $\dbar$-operator on $T_{X/S}$ induced by $\dbar$} is the first-order differential operator
\[\dbar_{T_{X/S}}:C^\infty(X,T_{X/S})\longrightarrow
C^\infty(X,\widebar{T^*X}\otimes T_{X/S})\]
defined by
\begin{equation}\label{eq:canonical_dbar_TXS_def}
(\dbar_{T_{X/S}}W)(\widebar{Z}):=\mathrm{pr}_{T_{X/S}}\bigl([\widebar{Z},W]\bigr),\qquad\widebar{Z}\in C^\infty(X,\widebar{TX}),
\end{equation}
where $\mathrm{pr}_{T_{X/S}}:T_{X/S}\oplus \widebar{TX}\longrightarrow T_{X/S}$ is the projection.  Lemma~\ref{lem:dbar_bracket_vertical} shows that the right-hand side is well-defined.
\end{definition}

If $\dbar$ is represented in adapted local coordinates by $\dbar(\partial_{\bar{j}})=\bigl[\partial_{\bar{j}}+\Gamma_{\bar{j}}^\alpha\partial_\alpha\bigr]\pmod{\widebar{T_{X/S}}}$, then $\widebar{TX}$ is locally generated by $\partial_{\bar{\beta}}$ and $\widebar{H}_j:=\partial_{\bar{j}}+\Gamma_{\bar{j}}^\alpha\partial_\alpha$. For $W=W^\alpha\partial_\alpha$, the induced $\dbar$-operator is
\begin{equation}\label{eq:induced_dbar}
\dbar_{T_{X/S},\partial_{\bar{\beta}}}W=(\partial_{\bar{\beta}}W^\alpha)\partial_\alpha, \quad \dbar_{T_{X/S},\widebar{H}_j}W=\bigl(\partial_{\bar{j}}W^\alpha+\Gamma_{\bar{j}}^\gamma\partial_\gamma W^\alpha-W^\gamma\partial_\gamma\Gamma_{\bar{j}}^\alpha\bigr)\partial_\alpha.
\end{equation}

\begin{lemma}\label{lem:canonical_dbar_TXS_valid}
The operator $\dbar_{T_{X/S}}$ is a $\dbar$-operator on the complex vector bundle $T_{X/S}\to X$ in the usual sense.  If $\dbar$ is integrable, so that $f:X\to S$ is a holomorphic fibration \textup(\cite[Prop.~2.6]{LS26}\textup), then $\dbar_{T_{X/S}}$ is the natural Dolbeault operator of the
holomorphic vector bundle $T_{X/S}\to X$.
\end{lemma}
\begin{proof}
For $g\in C^\infty(X)$, $[g\widebar{Z},W]=g[\widebar{Z},W]-W(g)\widebar{Z}$. After applying $\mathrm{pr}_{T_{X/S}}$, the second term vanishes.  Hence
\[\dbar_{T_{X/S},g\widebar{Z}}W=g\,\dbar_{T_{X/S},\widebar{Z}}W.\]
Similarly, $[\widebar{Z},gW]=\widebar{Z}(g)W+g[\widebar{Z},W]$. Projection gives
\[\dbar_{T_{X/S},\widebar{Z}}(gW)=\widebar{Z}(g)W+g\,\dbar_{T_{X/S},\widebar{Z}}W.\]
Thus $\dbar_{T_{X/S}}$ is a $\dbar$-operator in the vector-bundle sense.

Now assume that $\dbar$ is integrable.  Then $X$ is a complex manifold, $f$ is holomorphic, and $T_{X/S}=\ker(\D f:TX\to f^*TS)$ is a holomorphic subbundle of the holomorphic vector bundle $TX$.  In adapted local holomorphic coordinates
$(s^i,z^\alpha)$, the bundle $T_{X/S}$ is spanned by
$\partial_\alpha$, and $\widebar{TX}$ is spanned by
$\partial_{\bar{i}}$ and $\partial_{\bar{\alpha}}$.  For
$W=W^\alpha\partial_\alpha$, the definition gives
\[
\dbar_{T_{X/S},\partial_{\bar{i}}}W
=
(\partial_{\bar{i}}W^\alpha)\partial_\alpha,
\qquad
\dbar_{T_{X/S},\partial_{\bar{\beta}}}W
=
(\partial_{\bar{\beta}}W^\alpha)\partial_\alpha.
\]
These are exactly the local formulas for the natural Dolbeault operator on the
holomorphic vector bundle $T_{X/S}$.  Hence the two operators coincide.
\end{proof}

We now construct a linear connection on the vertical tangent bundle
$T_{X/S} \to X$ of a complex fiber bundle $(f, T_{X/S})$ with a $\dbar$-operator $\dbar$. Let $\nabla^{1,0}$ be a pure $(1,0)$-connection with respect to $\dbar$, and put $\nabla^{0,1} := \widebar{\nabla^{1,0}}$. This determines a splitting of the complexified tangent bundle as $TX^\CC=T_{X/S}\oplus \widebar{T_{X/S}}\oplus \im(\nabla^{1,0})\oplus \im(\nabla^{0,1})$. In adapted local coordinates $(s^i,z^\alpha)$, the horizontal lifts take the form
\[H_i = \partial_i + \Gamma_i^\alpha\partial_\alpha + \Gamma_i^{\bar{\beta}}\partial_{\bar{\beta}},
\qquad
H_{\bar{i}} = \partial_{\bar{i}} + \Gamma_{\bar{i}}^\alpha\partial_\alpha + \Gamma_{\bar{i}}^{\bar{\beta}}\partial_{\bar{\beta}},\]
where $\Gamma_{\bar{i}}^\alpha = \overline{\Gamma_i^{\bar{\alpha}}}$ and $\Gamma_{\bar{i}}^{\bar{\beta}} = \widebar{\Gamma_i^\beta}$. To differentiate along the fibers, we require a choice of a \emph{fiberwise linear connection} $\nabla^{\mathrm{fib}}$ on $T_{X/S}\to X$, which is a smooth family of linear connections on the complex vector bundles $TX_s\to X_s$. Intrinsically, it is a differential operator
\[\nabla^{\mathrm{fib}}: C^\infty(X, T_{X/S}^\CC) \times C^\infty(X, T_{X/S}) \to C^\infty(X, T_{X/S}),\]
which is $C^\infty(X)$-linear in the first argument and satisfies the fiberwise Leibniz rule.

\begin{definition}\label{def:vert_conn}
The \emph{vertical connection} $\nabla^{\mathrm{v}}$ on the complex vector bundle $T_{X/S}\to X$ associated to $\nabla^{1,0}$ and $\nabla^{\mathrm{fib}}$ is defined as follows. For $W\in C^\infty(X,T_{X/S})$,
\begin{enumerate}[label=(\roman*)]
\item if $Z\in C^\infty(X, T_{X/S}^\CC)$ is a vertical vector field, then $\nabla^{\mathrm{v}}_Z W = \nabla^{\mathrm{fib}}_Z W$;
\item if $H$ is a horizontal vector field (in $\im(\nabla^{1,0})\oplus \im(\widebar{\nabla^{1,0}})$), then $\nabla^{\mathrm{v}}_{H}W = [H,W]^{T_{X/S}}$.
\end{enumerate}
Here $[\cdot,\cdot]^{T_{X/S}}$ denotes projection to $T_{X/S}$ with respect to the above splitting of $TX^\CC$.

Equivalently, let $A_{\gamma\beta}^\alpha$ and $A_{\bar{\gamma}\beta}^\alpha$ be the local connection coefficients of $\nabla^{\mathrm{fib}}$. For $W=W^\alpha(s,z)\partial_\alpha$,
\begin{align}
\nabla^{\mathrm{v}}_{\partial_\gamma}W &= \bigl(\partial_\gamma W^\alpha+A_{\gamma\beta}^\alpha W^\beta\bigr)\partial_\alpha,\label{eq:vert_conn_fiber_10}\\
\nabla^{\mathrm{v}}_{\partial_{\bar{\gamma}}}W &= \bigl(\partial_{\bar{\gamma}} W^\alpha+A_{\bar{\gamma}\beta}^\alpha W^\beta\bigr)\partial_\alpha,\label{eq:vert_conn_fiber_01}\\
\nabla^{\mathrm{v}}_{H_i}W &= \bigl(H_i(W^\alpha) -W^\gamma\partial_\gamma\Gamma_i^\alpha\bigr)\partial_\alpha,\label{eq:vert_conn_hor_10}\\
\nabla^{\mathrm{v}}_{H_{\bar{i}}}W &= \bigl(H_{\bar{i}}(W^\alpha) -W^\gamma\partial_\gamma\Gamma_{\bar{i}}^\alpha\bigr)\partial_\alpha.\label{eq:vert_conn_hor_01}
\end{align}
\end{definition}

\begin{lemma}\label{lem:vert_conn_well_def}
$\nabla^{\mathrm{v}}$ is a well-defined linear connection on $T_{X/S}\to X$. In
other words, for $g\in C^\infty(X)$, $Z\in C^\infty(X,TX^\CC)$, and
$W\in C^\infty(X,T_{X/S})$,
\[
\nabla^{\mathrm{v}}_{gZ}W = g\nabla^{\mathrm{v}}_ZW, \qquad \text{and} \qquad \nabla^{\mathrm{v}}_Z(gW) = Z(g)W+g\nabla^{\mathrm{v}}_ZW.
\]
\end{lemma}
\begin{proof}
The assertion holds for vertical directions by the definition of $\nabla^{\mathrm{fib}}$.
For horizontal directions, it suffices to check the local horizontal frames $H_i$
and $H_{\bar{i}}$. For example, $[gH_i,W] = g[H_i,W]-W(g)H_i$, and then $[gH_i,W]^{T_{X/S}} = g[H_i,W]^{T_{X/S}}$. Moreover,
\[
[H_i,gW]^{T_{X/S}} = \bigl(H_i(g)W+g[H_i,W]\bigr)^{T_{X/S}} = H_i(g)W+g[H_i,W]^{T_{X/S}}.
\]
The same argument applies exactly to $H_{\bar{i}}$.
\end{proof}

By Definitions \ref{def:canonical_dbar_TXS}, \ref{def:vert_conn}, and Lemma \ref{lem:dbar_bracket_vertical}, we obtain the following.
\begin{lemma}\label{lem:vert01_induced_dbar}
If the fiberwise connection $\nabla^{\mathrm{fib}}$ is chosen such that its $(0,1)$-part coincides with the canonical fiberwise $\dbar$-operator $\dbar^{\mathrm{fib}}$ (induced by the fiberwise holomorphic structure on $T_{X/S}$), then the $(0,1)$-part of the vertical connection $\nabla^{\mathrm{v}}$ is the canonical $\dbar$-operator on $T_{X/S}$ induced by $\dbar$.
\end{lemma}

Now suppose that $(f,T_{X/S})$ is equipped with a fiberwise
K\"ahler metric $\omega_{X/S}$, and let $h_{X/S}$ be the associated Hermitian metric on $T_{X/S}$. Recall that a $(1,0)$-connection $\nabla^{1,0}$ is a \emph{symplectic connection} (\cite[Def.~4.5]{LS26}) associated to $\dbar$ and $\omega_{X/S}$ if it is pure and the real horizontal lifts preserve $\omega_{X/S}$, which implies $\mathcal{L}_{H_i}\omega_{X/S} = 0$, where $\mathcal{L}$ denotes the Lie derivative.

\begin{lemma}\label{lem:vert_conn_metric}
Suppose $\nabla^{\mathrm{fib}}$ is metric-compatible with $h_{X/S}$, and $\nabla^{1,0}$ is a symplectic connection. Then $\nabla^{\mathrm{v}}$ preserves $h_{X/S}$, i.e., for every $Z\in C^\infty(X,TX^\CC)$ and $W_1,W_2\in C^\infty(X,T_{X/S})$,
\[Z\bigl(h_{X/S}(W_1,W_2)\bigr) = h_{X/S}(\nabla^{\mathrm{v}}_Z W_1,W_2) + h_{X/S}(W_1,\nabla^{\mathrm{v}}_{\widebar{Z}}W_2).\]
\end{lemma}
\begin{proof}
For vertical directions, this is exactly the metric compatibility of $\nabla^{\mathrm{fib}}$. For horizontal directions, we check $Z=H_i$. Writing $\omega_{X/S} = \I g_{\alpha\bar{\beta}} \D z^\alpha\wedge \D\bar z^\beta$, we have $h_{X/S}(\partial_\alpha,\partial_\beta)=2g_{\alpha\bar{\beta}}$. Since $\nabla^{1,0}$ is symplectic, we have $\mathcal{L}_{H_i}\omega_{X/S}=0$. Therefore,
\begin{align*}
H_i(\omega_{X/S}(\partial_\alpha,\partial_{\bar{\beta}})) &= (\mathcal{L}_{H_i}\omega_{X/S})(\partial_\alpha,\partial_{\bar{\beta}}) + \omega_{X/S}([H_i,\partial_\alpha], \partial_{\bar{\beta}}) + \omega_{X/S}(\partial_\alpha, [H_i, \partial_{\bar{\beta}}])\\
&= \omega_{X/S}(\nabla^{\mathrm{v}}_{H_i}\partial_\alpha, \partial_{\bar{\beta}}) + \omega_{X/S}(\partial_\alpha, \overline{\nabla^{\mathrm{v}}_{H_{\bar{i}}}\partial_\beta}).
\end{align*}
The case $Z=H_{\bar{i}}$ is similar.
\end{proof}

Lemmas \ref{lem:vert01_induced_dbar} and \ref{lem:vert_conn_metric} immediately imply the following.
\begin{corollary}\label{cor:vert_conn_is_chern}
Let $f:X\to S$ be a holomorphic fibration equipped with a fiberwise K\"ahler metric $\omega_{X/S}$ and a K\"ahler connection  $\nabla^{1,0}$ (\cite[Def.~4.5]{LS26}). Let $\nabla^{\mathrm{fib}}$ be the fiberwise Chern connection. Then $\nabla^{\mathrm{v}}$ is the Chern connection of the Hermitian holomorphic vector bundle $(T_{X/S},h_{X/S})\to X$, where $h_{X/S}$ is the Hermitian metric determined by $\omega_{X/S}$.
\end{corollary}

\begin{definition}\label{def:pullback_conn}
Let $u:S\to X$ be a smooth section of $f$. The \emph{pullback connection} on $E=u^*T_{X/S}$ is defined by
\begin{equation}\label{eq:pullback_conn_def}
\nabla^E_v V := (\nabla^{\mathrm{v}}_{\D u_s(v)} \widetilde{V})\big|_{u(s)},\qquad v\in T_sS^\CC,\; V\in C^\infty(S,E),
\end{equation}
where $U\ni s$ is a neighborhood and $\widetilde V$ is a smooth section of $T_{X/S}$ defined near $u(U)$ such that $\widetilde V\comp u=V|_U$. The value at $u(s)$ of the right-hand side is independent of the choice of $\widetilde V$, and $\nabla^E$ is a well-defined linear connection on $E\to S$.
\end{definition}

The connection $\nabla^E$ decomposes by type as $\nabla^E = \nabla^{E,1,0}+\nabla^{E,0,1}$.

\begin{lemma}\label{lem:pullback_dbar_holomorphic}
Let $f:X\to S$ be a holomorphic fibration with the canonical $\dbar$-operator $\dbar$. Let $\nabla^{1,0}$ be a pure $(1,0)$-connection, and $\nabla^{\mathrm{fib}}$ a fiberwise linear connection on $T_{X/S}\to X$ whose $(0,1)$-part agrees with $\dbar^{\mathrm{fib}}$. Let $u:S\to X$ be a holomorphic section, and $\nabla^E$ the pullback connection on $E=u^*T_{X/S}$ built from $\nabla^{1,0}$ and $\nabla^{\mathrm{fib}}$. Then
\begin{equation}\label{eq:pullback_dbar_holomorphic}
\nabla^{E,0,1} = \dbar_E,
\end{equation}
where $\dbar_E$ is the Dolbeault operator of the pullback holomorphic bundle $E$.
\end{lemma}

\begin{proof}
By Lemmas \ref{lem:canonical_dbar_TXS_valid} and \ref{lem:vert01_induced_dbar}, the $(0,1)$-part of $\nabla^{\mathrm{v}}$ coincides with the canonical $\dbar$-operator $\dbar_{T_{X/S}}$ on $T_{X/S}\to X$. For $v=v^{\bar j}\partial_{\bar j}\in\widebar{T_sS}$, $\D u(v)\in\widebar{TX}_{u(s)}$ since $u$ is holomorphic, so for any $V\in\Gamma(S,E)$ with extension $\widetilde{V}$,
\[\nabla^{E,0,1}_v V = \nabla^{\mathrm{v}}_{\D u(v)}\widetilde{V}\big|_{u(s)} = \dbar_{T_{X/S},\D u(v)}\widetilde{V}\big|_{u(s)}= v^{\bar j}(\partial_{\bar j}V^\alpha)\,\partial_\alpha|_{u(s)} = \dbar_{E,v} V.\]
This implies \eqref{eq:pullback_dbar_holomorphic}.
\end{proof}

If $h_{X/S}$ is a Hermitian metric on $T_{X/S}$, then the vector bundle $E:=u^*T_{X/S}$ over $S$ is equipped with the pullback Hermitian metric $h_E := u^*h_{X/S}$.
\begin{corollary}\label{cor:pullback_chern}
Let $(f,T_{X/S})$ be a complex fiber bundle with a fiberwise  K\"ahler metric $\omega_{X/S}$ and its associated Hermitian metric $h_{X/S}$ on $T_{X/S}$. Let $\dbar$ be a $\dbar$-operator on $(f,T_{X/S})$ satisfying the lifting condition, $\nabla^{1,0}$ a symplectic $(1,0)$-connection associated to $(\dbar,\omega_{X/S})$, and $\nabla^{\mathrm{fib}}$ the fiberwise Chern connection determined by $h_{X/S}$.
\begin{enumerate}[label=(\roman*)]
\item For any smooth section $u:S\to X$, the pullback connection $\nabla^E$ is a Hermitian connection on $(E,h_E)$.
\item If moreover $\dbar$ is integrable and $u:S\to X$ is a holomorphic section, then $\nabla^E$ is the Chern connection of the Hermitian holomorphic bundle $(E,h_E)$.
\end{enumerate}
\end{corollary}

\begin{proof}
(i) By Lemma \ref{lem:vert_conn_metric}, $\nabla^{\mathrm{v}}$ preserves $h_{X/S}$. For $v\in T_sS$ and $V_1,V_2\in\Gamma(S,E)$ with extensions $\widetilde{V}_1,\widetilde{V}_2$,
\begin{align*}
	v\bigl(h_E(V_1,V_2)\bigr)& = v\bigl(h_{X/S}(\widetilde{V}_1,\widetilde{V}_2)\comp u\bigr) = h_{X/S}(\nabla^{\mathrm{v}}_{\D u(v)}\widetilde{V}_1,\widetilde{V}_2)\big|_u + h_{X/S}(\widetilde{V}_1, \nabla^{\mathrm{v}}_{\widebar{{\D} u(v)}}\widetilde{V}_2)\big|_u \\&= h_E(\nabla^E_v V_1,V_2) + h_E(V_1,\nabla^E_{\bar{v}} V_2).
\end{align*}
This proves that $\nabla^E$ is compatible with $h_E$.

(ii) The fiberwise Chern connection of $h_{X/S}$ has $(0,1)$-part equal to $\dbar^{\mathrm{fib}}$, so Lemma~\ref{lem:pullback_dbar_holomorphic} applies and gives $\nabla^{E,0,1} = \dbar_E$. Combined with (i), $\nabla^E$ is the Chern connection of $(E,h_E)$.
\end{proof}

Having constructed the pullback connection $\nabla^E$ on $E=u^*T_{X/S}$, we can now extend $\nabla^E$ to act on $E$-valued differential forms on $S$ in the standard way. For $\sigma\in A^{p,q}(S,E)$,
\[
\nabla^{E,1,0}\sigma \in A^{p+1,q}(S,E),\qquad \nabla^{E,0,1}\sigma \in A^{p,q+1}(S,E).
\]
In particular, $\partial u\in A^{1,0}(S,E)$ and $\dbar u\in A^{0,1}(S,E)$ (Definitions \ref{def:partial_on_section} and \ref{def:dbar_on_section}) can be differentiated further.

For a complex fiber bundle with a $(1,0)$-connection $\nabla^{1,0}$ (which is equivalent to $D=\partial+\dbar$), the curvature of $D$ decomposes into three components $F_\partial^{2,0}$, $F_{D}^{1,1}$, and $F_{\dbar}^{0,2}$ (see \cite[\S 2]{LS26}). All three components are well-defined once $\dbar$ satisfies the lifting condition and $\nabla^{1,0}$ is relatively holomorphic.

\begin{proposition}\label{prop:curv_second_deriv}
Let $(f,T_{X/S})$ be a complex fiber bundle with a $\dbar$-operator $\dbar$ satisfying the lifting condition and a relatively holomorphic $(1,0)$-connection $\nabla^{1,0}$ pure with respect to $\dbar$. Let $\nabla^{\mathrm{fib}}$ be a fiberwise torsion-free linear connection on $T_{X/S}\to X$ whose $(0,1)$-part agrees with $\dbar^{\mathrm{fib}}$. Let $\nabla^E$ be the pullback connection on $E$ for a smooth section $u:S\to X$, built from $\nabla^{1,0}$ and $\nabla^{\mathrm{fib}}$. Then
\begin{align}
\nabla^{E,1,0}(\partial u) &= -u^*F_\partial^{2,0},\label{eq:curv_second_deriv_20}\\
\nabla^{E,0,1}(\partial u) + \nabla^{E,1,0}(\dbar u) &= -u^*F_{D}^{1,1},\label{eq:curv_second_deriv_11}\\
\nabla^{E,0,1}(\dbar u) &= -u^*F_{\dbar}^{0,2},\label{eq:curv_second_deriv_02}
\end{align}
where $u^*F$ denotes the restriction of the curvature tensor along $u$.
\end{proposition}

\begin{proof}
In adapted local coordinates, $\dbar u = \eta_{\bar{i}}^\alpha\,\D\bar{s}^i\otimes\partial_\alpha|_{u(s)}$ with $\eta_{\bar{i}}^\alpha=\partial_{\bar{i}}u^\alpha-\Gamma_{\bar{i}}^\alpha\comp u$. Writing $\dbar u$ as an $E$-valued $(0,1)$-form with components $\xi^\alpha=\eta_{\bar{i}}^\alpha\,\D\bar{s}^i$, we have
\[
\nabla^{E,0,1}(\dbar u) = \bigl(\dbar\xi^\beta - \xi^\alpha\wedge (\Gamma^{E,0,1})_\alpha^\beta\bigr)\otimes\partial_\beta|_{u},
\]
where $(\Gamma^{E,0,1})_\alpha^\beta=\Gamma_{\bar{j} \alpha}^\beta\,\D\bar{s}^j$ is the $(0,1)$-connection form of $\nabla^E$ in the frame $\{\partial_\alpha|_u\}$.

To compute $\Gamma_{\bar{j}\alpha}^\beta$, extend $\partial_\alpha|_u$ to the constant local vertical field $\partial_\alpha$. Writing
\[
\D u(\partial_{\bar{j}}) = H_{\bar{j}}|_{u(s)} + \eta_{\bar{j}}^\gamma\partial_\gamma + \bigl(\partial_{\bar{j}}\bar{u}^\gamma-\Gamma_{\bar{j}}^{\bar{\gamma}}\comp u\bigr)\partial_{\bar{\gamma}}
\]
and applying \eqref{eq:vert_conn_hor_01}, \eqref{eq:vert_conn_fiber_10}, and \eqref{eq:vert_conn_fiber_01}, together with the fact that the fiberwise linear connection satisfies $A_{\bar{\gamma}\alpha}^\beta=0$ in a local holomorphic frame, we obtain
\begin{equation}\label{eq:pullback_conn_form_01}
\Gamma_{\bar{j}\alpha}^\beta = \bigl(-\partial_\alpha\Gamma_{\bar{j}}^\beta + \eta_{\bar{j}}^\gamma A_{\gamma\alpha}^\beta\bigr)\big|_{u(s)}.
\end{equation}
Evaluating on $(\partial_{\bar{j}},\partial_{\bar{i}})$,
\[
\nabla^{E,0,1}(\dbar u)(\partial_{\bar{j}},\partial_{\bar{i}}) = \bigl(\partial_{\bar{j}}\eta_{\bar{i}}^\beta - \partial_{\bar{i}}\eta_{\bar{j}}^\beta - \eta_{\bar{j}}^\alpha\Gamma_{\bar{i}\alpha}^\beta + \eta_{\bar{i}}^\alpha\Gamma_{\bar{j}\alpha}^\beta\bigr)\partial_\beta\big|_{u(s)}.
\]
Substituting \eqref{eq:pullback_conn_form_01}, the expression $(\eta_{\bar{i}}^\alpha\eta_{\bar{j}}^\gamma - \eta_{\bar{j}}^\alpha\eta_{\bar{i}}^\gamma)A_{\gamma\alpha}^\beta$ vanishes because the fiberwise Christoffel symbols satisfy $A_{\gamma\alpha}^\beta = A_{\alpha\gamma}^\beta$ by torsion-freeness. The lifting condition $\partial_{\bar{\gamma}}\Gamma_{\bar{i}}^\beta=0$ and the chain rule give
\[
\partial_{\bar{j}}\eta_{\bar{i}}^\beta = \partial_{\bar{j}}\partial_{\bar{i}}u^\beta - \partial_{\bar{j}}\Gamma_{\bar{i}}^\beta|_u - \partial_\gamma\Gamma_{\bar{i}}^\beta|_u\,\partial_{\bar{j}}u^\gamma.
\]
Antisymmetrizing in $(i,j)$ eliminates the second-order derivative of $u$, giving
\[
\partial_{\bar{j}}\eta_{\bar{i}}^\beta - \partial_{\bar{i}}\eta_{\bar{j}}^\beta = -\partial_{\bar{j}}\Gamma_{\bar{i}}^\beta|_u + \partial_{\bar{i}}\Gamma_{\bar{j}}^\beta|_u - \partial_\gamma\Gamma_{\bar{i}}^\beta|_u\,\partial_{\bar{j}}u^\gamma + \partial_\gamma\Gamma_{\bar{j}}^\beta|_u\,\partial_{\bar{i}}u^\gamma.
\]
Therefore, we obtain
\[
\nabla^{E,0,1}(\dbar u)(\partial_{\bar{j}},\partial_{\bar{i}})^\beta = -\bigl(\partial_{\bar{j}}\Gamma_{\bar{i}}^\beta - \partial_{\bar{i}}\Gamma_{\bar{j}}^\beta + \Gamma_{\bar{j}}^\gamma\partial_\gamma\Gamma_{\bar{i}}^\beta - \Gamma_{\bar{i}}^\gamma\partial_\gamma\Gamma_{\bar{j}}^\beta\bigr)\big|_{u(s)},
\]
which is exactly $-F_{\dbar}^{0,2}(\partial_{\bar{j}},\partial_{\bar{i}})^\beta|_{u(s)}$. This proves \eqref{eq:curv_second_deriv_02}.

For \eqref{eq:curv_second_deriv_20}, the same computation with $(i,j)$-indices replacing $(\bar{i},\bar{j})$, relative holomorphicity $\partial_{\bar{\gamma}}\Gamma_i^\alpha=0$ replacing the lifting condition, and $\eta_i^\alpha:=\partial_i u^\alpha - \Gamma_i^\alpha\comp u$ replacing $\eta_{\bar{i}}^\alpha$ yields
\[
\nabla^{E,1,0}(\partial u)(\partial_i,\partial_j)^\beta = -\bigl(\partial_i\Gamma_j^\beta - \partial_j\Gamma_i^\beta + \Gamma_i^\gamma\partial_\gamma\Gamma_j^\beta - \Gamma_j^\gamma\partial_\gamma\Gamma_i^\beta\bigr)\big|_{u(s)} = -F_\partial^{2,0}(\partial_i,\partial_j)^\beta|_{u(s)}.
\]

For \eqref{eq:curv_second_deriv_11}, one computes $(\nabla^{E,0,1}(\partial u) + \nabla^{E,1,0}(\dbar u))(\partial_i,\partial_{\bar{j}})$ analogously, using the $(1,0)$-pullback connection form $\Gamma_{i \alpha }^\beta = (-\partial_\alpha\Gamma_i^\beta + \eta_i^\gamma A_{\gamma\alpha}^\beta)|_u$ and \eqref{eq:pullback_conn_form_01}. The $\eta_i\eta_{\bar{j}}$-quadratic terms cancel by the symmetry $A_{\gamma\alpha}^\beta=A_{\alpha\gamma}^\beta$. The terms involving $\partial_i u$ and $\partial_{\bar{j}}u$ also cancel pairwise. What remains is
\[
-\bigl(\partial_i\Gamma_{\bar{j}}^\beta - \partial_{\bar{j}}\Gamma_i^\beta + \Gamma_i^\gamma\partial_\gamma\Gamma_{\bar{j}}^\beta - \Gamma_{\bar{j}}^\gamma\partial_\gamma\Gamma_i^\beta\bigr)\big|_{u(s)} = -F_{D}^{1,1}(\partial_i,\partial_{\bar{j}})^\beta|_{u(s)}.
\]
The result follows.
\end{proof}

We now introduce the linearization of a relatively holomorphic almost Higgs field along a smooth section. This construction refines the restriction $\theta u\in A^{1,0}(S,E)$ to an $\End(E)$-valued object, and depends on an auxiliary fiberwise linear connection.

\begin{definition}\label{def:linearized_higgs}
Let $(f,T_{X/S})$ be a complex fiber bundle, and let $\nabla^{\mathrm{fib}}$ be a fiberwise linear connection on $T_{X/S}\to X$ whose $(0,1)$-part agrees with $\dbar^{\mathrm{fib}}$. Let $\theta\in A^{1,0}(S,f_*T_{X/S})$ be a relatively holomorphic almost Higgs field on $f$, and $u:S\to X$ a smooth section. The \emph{linearized almost Higgs field} $\Theta_u\in A^{1,0}(S,\End(E))$ is defined as follows. For each $s\in S$ and $v\in T_sS$, $\theta(v)|_{X_s}$ is a holomorphic vector field on $X_s$. Its fiberwise covariant derivative $\nabla^{\mathrm{fib}}|_{X_s}(\theta(v)|_{X_s})$ lies in $A^{1,0}(X_s,TX_s)$. Restricting to the point $u(s)\in X_s$,
\begin{equation}\label{eq:linearized_higgs_def}
\Theta_u(v)(\ph) := \nabla_{\ph}^{\mathrm{fib}}|_{X_s}\bigl(\theta(v)|_{X_s}\bigr)\big|_{u(s)} \in T_{u(s)}^*X_s\otimes T_{u(s)}X_s = \End(E_s).
\end{equation}
In adapted local coordinates, if $\theta = \theta_i^\alpha(s,z)\,\D s^i\otimes\partial_\alpha$ and $\nabla^{\mathrm{fib}}$ has $(1,0)$-coefficients $A_{\beta\gamma}^\alpha$ in the local frame $\{\partial_\alpha\}$, then
\begin{equation}\label{eq:linearized_higgs_local}
(\Theta_u)_{i,\beta}^\alpha = \bigl(\partial_\beta\theta_i^\alpha + A_{\beta\gamma}^\alpha\theta_i^\gamma\bigr)\big|_{u(s)}.
\end{equation}
\end{definition}

\begin{remark}\label{rmk:linearized_higgs_chern}
The linearized field $\Theta_u$ depends on the choice of $\nabla^{\mathrm{fib}}$: if $\widetilde\nabla^{\mathrm{fib}}-\nabla^{\mathrm{fib}} = B\in C^\infty(X,T_{X/S}^*\otimes\End(T_{X/S}))$ with $(0,1)$-part zero, then $(\widetilde\Theta_u(v) - \Theta_u(v))(\ph) = B|_{u(s)}(\ph)\bigl(\theta(v)|_{u(s)}\bigr)\in\End(E_s)$. This vanishes when $\theta u=0$, in particular when $u$ is a Higgs section.

$\Theta_u$ admits an equivalent description via the holomorphic jet bundle. For each $s\in S$ and $v\in T_sS$, the holomorphic vector field $\theta(v)|_{X_s}$ has a $1$-jet $j^1(\theta(v))\in H^0(X_s,J^1(TX_s))$ fitting in the exact sequence
\[
0\to T^*X_s\otimes TX_s \to J^1(TX_s) \to TX_s \to 0.
\]
A smooth splitting of this sequence (e.g., the one provided by the fiberwise Chern connection) extracts $\Theta_u(v)$ from $j^1(\theta(v))|_{u(s)}$. When $(\theta u)(v)=\theta(v)|_{u(s)}=0$, the element $\Theta_u(v)\in\End(E_s)$ is independent of the choice of splitting, as the jet projects to zero in $TX_s$ and hence already lies in $T^*X_s\otimes TX_s = \End(E_s)$.
\end{remark}

\begin{remark}\label{rmk:theta_u_not_holomorphic}
Even when $f$ is a holomorphic fibration, $\theta$ is a holomorphic Higgs field, and $u$ is a holomorphic section, $\Theta_u$ need not be holomorphic. In adapted local holomorphic coordinates, the fiberwise Christoffel symbols $A_{\beta\gamma}^\alpha$ of $\nabla^{\mathrm{fib}}$ may depend on $\bar z$, and
\begin{equation}\label{eq:theta_u_holom_obstruction}
\partial_{\bar j}(\Theta_u)_{i,\beta}^\alpha=\bigl(\partial_{\bar j}A_{\beta\gamma}^\alpha + \overline{\partial_ju^\epsilon}\,\partial_{\bar\epsilon}A_{\beta\gamma}^\alpha\bigr)\big|_{u(s)}\,\theta_i^\gamma(u(s)),
\end{equation}
which does not vanish in general, so $(E,\Theta_u)$ need not be a Higgs bundle. \eqref{eq:theta_u_holom_obstruction} vanishes, for example, if the fiberwise Christoffel symbols vanish identically or if $u$ is a Higgs section.
\end{remark}

\begin{proposition}\label{prop:pseudo_curv_second_deriv}
In the setup of Proposition \ref{prop:curv_second_deriv}, but dropping the hypotheses that $\nabla^{1,0}$ is relatively holomorphic and that $\nabla^{\mathrm{fib}}$ is fiberwise torsion-free, let $\theta\in A^{1,0}(S,f_*T_{X/S})$ be a relatively holomorphic almost Higgs field on $f$. Then
\begin{equation}\label{eq:pseudo_curv_second_deriv}
\nabla^{E,0,1}(\theta u) + \Theta_u\wedge \dbar u = u^*G_{D''}^{1,1},
\end{equation}
where $D''=\dbar+\theta$, $G_{D''}^{1,1}\in A^{1,1}(S,f_*T_{X/S})$ is the $(1,1)$-part of the pseudo-curvature (\cite[Eq.~(5.2)]{LS26}), and the wedge product is given by
\[(\Theta_u\wedge\dbar u)(v,w) := \Theta_u(v)\bigl(\dbar u(w)\bigr) - \Theta_u(w)\bigl(\dbar u(v)\bigr),\qquad v,w\in T_sS^\CC.\]
\end{proposition}
\begin{proof}
	Locally, $\theta u = (\theta_i^\alpha\comp u)\,\D s^i\otimes\partial_\alpha|_u$. Writing $\theta u$ as an $E$-valued $(1,0)$-form with components $\tau^\alpha := (\theta_i^\alpha\comp u)\,\D s^i$, the $(0,1)$-exterior covariant derivative is
\[
\nabla^{E,0,1}(\theta u) = \bigl(\dbar\tau^\beta + \Gamma_{\bar j \alpha}^\beta\,\D\bar s^j\wedge\tau^\alpha\bigr)\otimes\partial_\beta|_u,
\]
with $\Gamma_{\bar j \alpha}^\beta = (-\partial_\alpha\Gamma_{\bar j}^\beta + \eta_{\bar j}^\gamma A_{\gamma\alpha}^\beta)|_{u(s)}$ given by the local form of the pullback connection \eqref{eq:pullback_conn_form_01}. Evaluating on $(\partial_i,\partial_{\bar j})$,
\[
\nabla^{E,0,1}(\theta u)(\partial_i,\partial_{\bar j})^\beta = -\partial_{\bar j}(\theta_i^\beta\comp u) - \Gamma_{\bar j \alpha}^\beta\,(\theta_i^\alpha\comp u).
\]
Applying the chain rule and the relative holomorphicity of $\theta$ ($\partial_{\bar\gamma}\theta_i^\beta=0$),
\[
\partial_{\bar j}(\theta_i^\beta\comp u) = \partial_{\bar j}\theta_i^\beta\big|_u + \partial_\gamma\theta_i^\beta\big|_u\,\partial_{\bar j}u^\gamma.
\]
Substituting $\partial_{\bar j}u^\gamma = \eta_{\bar j}^\gamma + \Gamma_{\bar j}^\gamma|_u$,
\begin{align*}
\nabla^{E,0,1}(\theta u)(\partial_i,\partial_{\bar j})^\beta
&= \bigl(-\partial_{\bar j}\theta_i^\beta - \Gamma_{\bar j}^\gamma\partial_\gamma\theta_i^\beta + \theta_i^\alpha\partial_\alpha\Gamma_{\bar j}^\beta\bigr)\big|_u - \bigl(\partial_\gamma\theta_i^\beta + A_{\gamma\alpha}^\beta\theta_i^\alpha\bigr)\big|_u\,\eta_{\bar j}^\gamma\\
&= u^*G_{D''}^{1,1}(\partial_i,\partial_{\bar j})^\beta - (\Theta_u)_{i,\gamma}^\beta\,\eta_{\bar j}^\gamma,
\end{align*}
where we used the local formula \cite[Eq.~(5.2)]{LS26} and \eqref{eq:linearized_higgs_local}. This proves \eqref{eq:pseudo_curv_second_deriv}.
\end{proof}

\begin{corollary}\label{cor:higgs_bundle_curv_vanish}
Under the hypotheses of Proposition \ref{prop:pseudo_curv_second_deriv}, suppose further that $\dbar$
is integrable \textup(so that $f:X\to S$ is a holomorphic fibration\textup). Then $G_{D''}^{1,1}=0$ if and only if $\theta u\in A^{1,0}(U,E)$ is holomorphic for every open $U\subset S$ and every holomorphic section $u$ of $f|_U$. In particular, if $(f,\theta)$ is a Higgs bundle \textup(\cite[Def.~5.2]{LS26}\textup), then $\theta u$ is a holomorphic section of $\Omega_S^1\otimes E$ for every holomorphic section $u$.
\end{corollary}

\begin{proof}
For a holomorphic section $u$, $\dbar u=0$, so Proposition~\ref{prop:pseudo_curv_second_deriv} gives $\nabla^{E,0,1}(\theta u) = u^*G_{D''}^{1,1}$. By Lemma~\ref{lem:pullback_dbar_holomorphic}, $\nabla^{E,0,1} = \dbar_E$. The ``only if'' direction follows. Conversely, if $G_{D''}^{1,1}\ne 0$ at some $x_0\in X_{s_0}$, any local holomorphic section $u$ with $u(s_0)=x_0$ (exists since $f$ is a holomorphic submersion) satisfies $u^*G_{D''}^{1,1}\ne 0$ at $s_0$, so $\nabla^{E,0,1}(\theta u)\ne 0$.
\end{proof}

\begin{proposition}\label{prop:pseudo_curv_20}
Let $(f,T_{X/S})$ be a complex fiber bundle, $\theta\in A^{1,0}(S,f_*T_{X/S})$ a relatively holomorphic almost Higgs field, and $\nabla^{\mathrm{fib}}$ a fiberwise torsion-free linear connection on $T_{X/S}\to X$ whose $(0,1)$-part agrees with $\dbar^{\mathrm{fib}}$. Then for any smooth section $u:S\to X$,
\begin{equation}\label{eq:pseudo_curv_20}
\Theta_u\wedge(\theta u) = -u^*G_\theta^{2,0},
\end{equation}
where $G_\theta^{2,0}=\frac{1}{2}[\theta,\theta]\in A^{2,0}(S,f_*T_{X/S})$ is the $(2,0)$-part of the pseudo-curvature.
\end{proposition}
\begin{proof}
For $v,w\in T_sS$, we compute
\begin{align*}
(\Theta_u \wedge (\theta u))(v,w) &= \Theta_u(v)\bigl(\theta(w)|_{u(s)}\bigr) - \Theta_u(w)\bigl(\theta(v)|_{u(s)}\bigr) \\
&= \bigl(\nabla^{\mathrm{fib}}_{ \theta(w)}\bigl(\theta(v)|_{X_s}\bigr) - \nabla^{\mathrm{fib}}_{ \theta(v)}\bigl(\theta(w)|_{X_s}\bigr)\bigr)\big|_{u(s)}.
\end{align*}
Since $\nabla^{\mathrm{fib}}$ is fiberwise torsion-free,
\[
(\Theta_u \wedge (\theta u))(v,w) = [\theta(w), \theta(v)]\big|_{u(s)} = -[\theta(v), \theta(w)]\big|_{u(s)} = -G_{\theta}^{2,0}(v,w)\big|_{u(s)},
\]
which proves \eqref{eq:pseudo_curv_20}.
\end{proof}

\subsection{Vector bundles}\label{subsec:vb_diff_op_curv}
We now show that the constructions of Sections \ref{subsec:cov_der}--\ref{subsec:second_order} recover the classical theory of linear connections, Higgs fields, and curvatures on complex vector bundles. Throughout this subsection, $f:E\to S$ denotes a smooth complex vector bundle of rank $m$ over a complex manifold $S$ of dimension $n$. We equip $T_{E/S}\to E$ with the fiberwise trivial linear connection $\nabla^{\mathrm{fib}}$.

Since each fiber $E_s$ is a complex vector space, we have the canonical complex-linear isomorphism
\begin{equation}\label{eq:canonical_ident}
\iota_v:E_s\xrightarrow{\;\cong\;}T_{E/S,v},\quad
w\longmapsto (\gamma_w)_* \bigl(\tfrac{\partial}{\partial t}\big|_0\bigr),
\quad \gamma_w(t)=v+tw,\quad t\in\CC,
\end{equation}
for each $v\in E_s$. For a smooth section $u:S\to E$, this gives a canonical isomorphism of complex vector bundles
\begin{equation}\label{eq:pullback_ident_vb}
\iota_u: E \xrightarrow{\;\cong\;} u^*T_{E/S}, \qquad \iota_u|_{E_s} = \iota_{u(s)}.
\end{equation}

\begin{lemma}\label{lem:canonical_identification}
In a local frame $\{e_\alpha\}$ for $E$, the isomorphism \eqref{eq:canonical_ident} sends $e_\alpha|_s$ to $\partial_{z^\alpha}|_{u(s)}$, where $(s^i,z^\alpha)$ are the induced adapted local coordinates on $E$ (i.e., $z^\alpha$ is the $\alpha$-th coordinate with respect to the frame $\{e_\alpha\}$).
\end{lemma}
\begin{proof}
We have $\iota_{u(s)}(e_\alpha|_s) = \frac{\D}{\D t}|_{t=0}(u(s)+t\, e_\alpha|_s)$.  In the coordinates $(s^i,z^\alpha)$, the point $u(s)+t\,e_\alpha|_s$ has coordinates $(s, u^1(s),\ldots, u^\alpha(s)+t,\ldots, u^m(s))$.  Differentiating in $t$ gives $\partial_{z^\alpha}|_{u(s)}$.
\end{proof}

Let $D_E$ be a linear connection on $E$, with type decomposition $D_E = D_E^{1,0}+D_E^{0,1}$, and let $\nabla^\RR$ be the Ehresmann connection on $f:E\to S$ induced by $D_E$ as in \cite[Ex.~2.17]{LS26}. Let $\nabla^{1,0}$ be the $(1,0)$-part of the complexification of $\nabla^\RR$. In a smooth local frame $\{e_\alpha\}$ for $E$ with connection $1$-form $\omega_\alpha^\beta = A_{i \alpha}^\beta\,\D s^i+ B_{\bar{i} \alpha}^\beta \D \bar{s}^i$ defined by $D_E e_\alpha = \omega_\alpha^\beta\otimes e_\beta$, the local coefficients of $\nabla^{1,0}$ are
\begin{equation}\label{eq:vb_local_conn}
\Gamma_i^\alpha = -A_{i \beta}^\alpha\, z^\beta,\qquad \Gamma_{i}^{\bar \alpha} = -\overline{B_{\bar i \beta}^\alpha\, z^\beta}.
\end{equation}
Let $\dbar$ be the $\dbar$-operator induced by $\nabla^{0,1}=\widebar{\nabla^{1,0}}$.

\begin{lemma}\label{lem:vb_connection_compat}
	For any smooth section $u$ of $E$, the covariant derivative $\nabla^\RR u$ in the sense of Definition~\ref{def:cov_deriv_real} corresponds to $D_E u$ under \eqref{eq:pullback_ident_vb}:
\[
\iota_u^{-1}\comp \mathrm{pr}^{1,0}\comp (\nabla^\RR u) = D_E u \;\in\; A^1(S,E).
\]
More precisely,
\begin{equation}\label{eq:vb_10_compat}
\iota_u^{-1}\comp \mathrm{pr}^{1,0}\comp (\nabla^{1,0} u) = D_E^{1,0} u,\qquad \iota_u^{-1}\comp \mathrm{pr}^{1,0}\comp (\nabla^{0,1} u) = D_E^{0,1} u.
\end{equation}
\end{lemma}
\begin{proof}
Write $u=u^\alpha e_\alpha$. By the Leibniz rule, $D_E^{1,0}u = (\partial_i u^\beta + A_{i \alpha}^\beta u^\alpha)\,\D s^i\otimes e_\beta$. By Lemma~\ref{lem:cov_deriv_local} and \eqref{eq:vb_local_conn},
\[
\mathrm{pr}^{1,0}(\nabla^{1,0} u)(\partial_i) = \bigl(\partial_i u^\alpha + A_{i \beta}^\alpha u^\beta\bigr)\,\partial_\alpha\big|_{u(s)}.
\]
Under $\iota_u^{-1}$, which maps $\partial_\alpha|_{u(s)}$ to $e_\alpha|_s$ (Lemma~\ref{lem:canonical_identification}), this becomes $(\partial_i u^\alpha + A_{i\beta}^\alpha u^\beta)\, e_\alpha = (D_E^{1,0}u)(\partial_i)$. The argument for $D_E^{0,1}$ is analogous.  Summing over types gives $\iota_u^{-1}\comp\mathrm{pr}^{1,0}\comp(\nabla^\RR u) = D_E u$.
\end{proof}

Combining $\nabla^{1,0}$ with the fiberwise trivial $\nabla^{\mathrm{fib}}$, Definition \ref{def:vert_conn} produces the vertical connection $\nabla^{\mathrm{v}}$ on $T_{E/S}\to E$, and Definition~\ref{def:pullback_conn} produces the pullback connection $\nabla^{\hat{E}}$ on $\hat{E} := u^*T_{E/S}$ for any smooth section $u:S\to E$.
\begin{lemma}\label{lem:vb_pullback_conn_curv}
	Under the canonical identification $\iota_u: E\xrightarrow{\cong} \hat{E}$:
\begin{enumerate}[label=(\roman*)]
\item $\nabla^{\hat{E}}$ corresponds to $D_E$, i.e., for every $V\in C^\infty(S,E)$,
\[\iota_u^{-1}\bigl(\nabla^{\hat{E}}(\iota_u V)\bigr) = D_E V;\]
\item the total curvature appearing in Proposition~\ref{prop:curv_second_deriv} corresponds to $-F_{D_E}u$:
\begin{equation}\label{eq:vb_pullback_curv}
\iota_u^{-1}\bigl(u^*F_\partial^{2,0}+u^*F_D^{1,1}+u^*F_{\dbar}^{0,2}\bigr) = -F_{D_E}\,u\;\in\;A^2(S,E),
\end{equation}
where $F_{D_E}\in A^2(S,\End E)$ denotes the curvature of $D_E$.
\end{enumerate}
\end{lemma}

\begin{proof}
(i)  Since the fiberwise Christoffel symbols of $\nabla^{\mathrm{fib}}$ vanish in the frame $\{\partial_\alpha\}$, formula \eqref{eq:vert_conn_hor_10} gives, for $V=V^\beta(s)\partial_\beta\comp u$ regarded as a fiberwise constant extension,
\[
\nabla^{\mathrm{v}}_{H_i}V = \bigl(\partial_i V^\beta - V^\gamma\partial_\gamma\Gamma_i^\beta\bigr)\partial_\beta = \bigl(\partial_i V^\beta + A_{i \gamma}^\beta V^\gamma\bigr)\partial_\beta,
\]
where we used \eqref{eq:vb_local_conn}. Restricting to $u(s)$ and applying $\iota_u^{-1}$, this equals $D_E^{1,0}(V^\alpha e_\alpha)(\partial_i)$. The identification extends to $\hat{E}$-valued forms.

(ii) Summing the three identities of Proposition~\ref{prop:curv_second_deriv} gives
\[
\nabla^{\hat{E}}(\partial u + \dbar u) = -\bigl(u^*F_\partial^{2,0} + u^*F_D^{1,1} + u^*F_{\dbar}^{0,2}\bigr).
\]
By Lemma~\ref{lem:vb_connection_compat}, $\iota_u^{-1}(\partial u + \dbar u) = D_E u$, and by (i), $\iota_u^{-1}\comp\nabla^{\hat{E}} = D_E\comp\iota_u^{-1}$ on $\hat{E}$-valued forms. Therefore the left-hand side corresponds to $D_E(D_E u) = F_{D_E}u$, which yields \eqref{eq:vb_pullback_curv}.
\end{proof}
\begin{remark}
When $E$ is a holomorphic vector bundle equipped with a Hermitian metric $h$, $D_E$ is the Chern connection, and $u$ is a holomorphic section, Lemma \ref{lem:vb_pullback_conn_curv} (i) is just Corollary~\ref{cor:pullback_chern} (ii), since $\nabla^{1,0}$ in this setting is the K\"ahler connection of \cite[Ex.~4.17]{LS26} and the fiberwise trivial connection coincides with the fiberwise Chern connection.
\end{remark}

Let $\theta_E\in A^{1,0}(S,\End E)$, which we shall call an \emph{almost Higgs field} on $E$. Realizing $E = P\times_G \CC^m$, where $P$ is the smooth complex frame bundle of $E$ and $G=\GL(m,\CC)$, the infinitesimal action $\tau_0:\mathfrak{gl}(m,\CC)\to H^0(\CC^m,T\CC^m)$ of \cite[\S2.5]{LS26} is
\begin{equation}\label{eq:tau0_glm}
\tau_0(\xi)^\alpha(z) = -\xi_\beta^\alpha\, z^\beta,\qquad \xi=(\xi_\beta^\alpha)\in\mathfrak{gl}(m,\CC).
\end{equation}
This induces an injective morphism of sheaves
\begin{equation}\label{eq:morphism_tau_vb}
\tau:\End E = \ad P \hookrightarrow f_*T_{E/S},
\end{equation}
where $f_*T_{E/S}$ denotes the sheaf of smooth fiberwise holomorphic vertical vector fields on $E$. In a local frame $\{e_\alpha\}$, $\tau$ sends $\xi_\beta^\alpha(s)\,(e_\alpha\otimes e^\beta)$ to $-\xi_\beta^\alpha(s)\, z^\beta\, \partial_\alpha$. Given an almost Higgs field $\theta_E$, the corresponding nonlinear almost Higgs field on $f$ is
\begin{equation}\label{eq:nonlin_higgs_vb}
\tilde{\theta} := \tau(\theta_E)\in A^{1,0}(S, f_*T_{E/S}),
\end{equation}
which is relatively holomorphic. Locally, if $\theta_E = \theta_{i,\beta}^\alpha\,\D s^i\otimes (e_\alpha\otimes e^{\beta})$, then
\begin{equation}\label{eq:nonlin_higgs_vb_local}
\tilde{\theta}(\partial_i) = -\theta_{i,\beta}^\alpha(s)\, z^\beta\,\partial_\alpha.
\end{equation}

\begin{lemma}\label{lem:higgs_section_compat}
For any smooth section $u$ of $E$,
\[\iota_u^{-1}\comp(\tilde{\theta} u) = -\theta_E(u) \;\in\; A^{1,0}(S, E),\]
where $\theta_E(u)=\theta_{i,\beta}^\alpha u^\beta\,\D s^i\otimes e_\alpha$ is the classical action. In particular, $u$ is a Higgs section for $D'':=\dbar+\tilde\theta$ (Definition~\ref{def:higgs_section}) if and only if $\dbar_E u = 0$ and $\theta_E(u)=0$, where $\dbar_E$ is a $\dbar$-operator on $E$ which induces $\dbar$ as a $\dbar$-operator in the complex fiber bundle sense.
\end{lemma}

\begin{proof}
By Definition~\ref{def:higgs_on_section} and \eqref{eq:nonlin_higgs_vb_local},
\[(\tilde{\theta} u)(\partial_i) = \tilde{\theta}(\partial_i)\big|_{u(s)} = -\theta_{i,\beta}^\alpha(s)\, u^\beta(s)\,\partial_\alpha\big|_{u(s)}.\]
Under $\iota_u^{-1}$ (Lemma~\ref{lem:canonical_identification}), this becomes $-\theta_{i,\beta}^\alpha\,u^\beta\,e_\alpha = -(\theta_E(u))(\partial_i)$. The final assertion follows from this together with \eqref{eq:vb_10_compat}.
\end{proof}

\begin{lemma}\label{lem:vb_linearized_higgs}
Let $u:S\to E$ be a smooth section, and let $\Theta_u\in A^{1,0}(S,\End(\hat{E}))$ be the linearized almost Higgs field of Definition~\ref{def:linearized_higgs}, computed with respect to the fiberwise trivial $\nabla^{\mathrm{fib}}$. Then:
\begin{enumerate}[label=(\roman*)]
\item $\iota_u^{-1}\comp \Theta_u\comp \iota_u = -\theta_E$.
\item Set $D'':=\dbar+\tilde\theta$ and $D_E''=\dbar_E+\theta_E$, where $\dbar$ is induced by $\dbar_E$. Then
\begin{equation}\label{eq:vb_pseudo_curv_compat}
	\iota_u^{-1}(u^*G_{\dbar}^{0,2}+u^*G_{D''}^{1,1}+u^*G_{\tilde\theta}^{2,0}) = -G_{D_E''}\,u,
\end{equation}
where $G_{D_E''}:=D_E''\comp D_E''\in A^2(S,\End E)$ is the pseudo-curvature of $D_E''$.
\end{enumerate}
\end{lemma}

\begin{proof}
(i) Since the fiberwise Christoffel symbols of $\nabla^{\mathrm{fib}}$ vanish, the local formula \eqref{eq:linearized_higgs_local} simplifies, using \eqref{eq:nonlin_higgs_vb_local}, to
\[(\Theta_u)_{i,\beta}^\alpha = \bigl(\partial_\beta\tilde\theta_i^\alpha\bigr)\big|_{u(s)} = -\theta_{i,\beta}^\alpha(s).\]
Under $\iota_u$, which identifies $e_\beta|_s$ with $\partial_\beta|_{u(s)}$, this gives $(\Theta_u)_i = -\theta_{E,i}$ as endomorphisms of $E_s$.

(ii) We treat the three components separately. By Proposition \ref{prop:pseudo_curv_20}, Lemma \ref{lem:higgs_section_compat}, and (i), we have
\[
\iota_u^{-1}(u^* G_{\tilde{\theta}}^{2,0}) =-\iota_u^{-1}(\Theta_u\wedge (\tilde{\theta} u))=-\tfrac{1}{2}[\theta_E,\theta_E]u=-G_{D_E''}^{2,0} u.
\]
By Proposition \ref{prop:pseudo_curv_second_deriv}, \eqref{eq:vb_10_compat}, Lemma \ref{lem:vb_pullback_conn_curv} (i), and (i), we have
\begin{align*}
	\iota_u^{-1}(u^*G_{D''}^{1,1})&=\iota_u^{-1}(\nabla^{\hat{E},0,1}(\tilde{\theta}u)+\Theta_u\wedge\dbar u)=\dbar_E (\iota_u^{-1}(\tilde{\theta}u))-\theta_E\wedge \iota_u^{-1}(\dbar u)\\
	&=-\dbar_E(\theta_E u)-\theta_E\wedge \dbar_E u=-(\dbar_E \theta_E)u=-G_{D_E''}^{1,1}u.
\end{align*}
By Lemma \ref{lem:vb_pullback_conn_curv} (ii), we have $\iota_u^{-1}(u^* G_{\dbar}^{0,2})=-G_{D_E''}^{0,2}u$. Combining these, we obtain \eqref{eq:vb_pseudo_curv_compat}.
\end{proof}

\subsection{Principal bundles}\label{subsec:pb}
In this subsection, we specialize Sections \ref{subsec:cov_der}--\ref{subsec:second_order} to (local) sections of a principal bundle. Let $\pi_P:P\to S$ be a smooth principal $G$-bundle over a complex manifold $S$, where $G$ is a complex Lie group with Lie algebra $\g$. For $\xi\in\g$, let $\xi^*\in C^\infty(P,T_{P/S})$ be the vector field given by $\xi_p^*:=(\gamma_\xi)_*
\bigl(\frac{\partial}{\partial t}\big|_0\bigr)
\in T_{P/S,p}$, where
$\gamma_\xi(t):=p\cdot\exp(t\xi)$, $t\in\CC$. This provides a canonical $G$-equivariant trivialization of $T_{P/S}$ as
\begin{equation}\label{eq:pb_iota_triv}
\iota\colon P\times \g\xrightarrow{\;\cong\;}T_{P/S},\qquad (p,\xi)\longmapsto \xi^*_p,
\end{equation}
where $G$ acts on $P\times\g$ by $(p,\xi)\cdot g=(pg,\Ad(g^{-1})\xi)$. Taking the quotient by $G$ identifies $T_{P/S}/G$ with the adjoint bundle $\ad P:=P\times_G\g\to S$. Let
\begin{equation}\label{eq:pb_quotient_to_adP}
\mathrm{q}_P\colon T_{P/S}\longrightarrow \ad P,\qquad \xi^*_p\longmapsto [p,\xi]
\end{equation}
be the quotient map. Given any local smooth section $u_P:U\to P$, the restriction of $\mathrm q_P$ to $T_{P/S}|_{u_P(U)}$ yields an isomorphism of complex vector bundles
\begin{equation}\label{eq:pb_canonical_iota}
\iota_{u_P}\colon u_P^*T_{P/S}\xrightarrow{\;\cong\;}\ad P\big|_U,\qquad \xi^*_{u_P(s)}\longmapsto [u_P(s),\xi].
\end{equation}

For local computations, we work in the trivialization $P|_U\cong U\times G$ induced by $u_P$, in which $u_P(s)=(s,e_G)$ for all $s\in U$, together with exponential coordinates $z=z^ae_a$ on a neighborhood of $e_G\in G$ defined by $\exp:\g\to G$, for a basis $\{e_a\}$ of $\g$. Thus $p(s,z)=u_P(s)\cdot\exp(z)$, $u_P(s)=(s,0)$,  and $\partial_{z^a}|_{(s,0)}=e_a^*|_{u_P(s)}$, so that $\iota_{u_P}(\partial_{z^a}|_{u_P(s)})=[u_P(s),e_a]$.

The classical formula for the differential of $\exp$ gives, in this chart,
\begin{equation}\label{eq:pb_dexp}
\partial_{z^a}\big|_{p(s,z)}=\bigl(\Lambda(z)e_a\bigr)^*\big|_{p(s,z)},\qquad \Lambda(z):=\frac{1-e^{-\ad(z)}}{\ad(z)}=1-\tfrac12\ad(z)+O(|z|^2).
\end{equation}
For any principal connection $1$-form $\omega\in A^1(P,\g)$, this yields
\begin{equation}\label{eq:pb_omega_vertical}
\omega(\partial_{z^a})\big|_{(s,z)}=\Lambda(z)e_a=e_a-\tfrac12[z,e_a]+O(|z|^2).
\end{equation}
The $\Ad$-equivariance $R_g^*\omega=\Ad(g^{-1})\omega$ gives
\begin{equation}\label{eq:pb_omega_horizontal}
\omega(\partial_i)\big|_{(s,z)}=\Ad(\exp(-z))\bigl(A_i(s)\bigr)=A_i(s)-[z,A_i(s)]+O(|z|^2),
\end{equation}
where $A_i(s):=\omega(\partial_i)|_{(s,0)}=(u_P^*\omega)(\partial_i)|_s\in\g$.

Following \cite[\S 2.5]{LS26}, a holomorphic principal $G$-bundle structure on $P$ is equivalent to an integrable principal $\dbar$-operator $\dbar_P$ on $(\pi_P,T_{P/S})$. Choose a principal complex connection $\omega\in A^{1,0}(P,\g)$ on the resulting holomorphic principal bundle, and let $\nabla_P$ be the complexification of its horizontal splitting. Its $(1,0)$-part $\nabla_P^{1,0}$ is $G$-invariant and pure with respect to $\dbar_P$. We set $D_P:=\partial_P+\dbar_P$, and write $A:=u_P^*\omega=A^{1,0}+A^{0,1}$, with types taken on $U$. The $\g$-valued $1$-form $A$ corresponds, via the trivializing frame $\{[u_P,e_a]\}$, to an $\ad P|_U$-valued $1$-form which we denote by the same symbol $A$ when no confusion arises.

 A \emph{principal almost Higgs field} on $P$ is a $G$-invariant relatively holomorphic almost Higgs field $\theta_P\in A^{1,0}(S,(\pi_{P})_*T_{P/S})$. Under \eqref{eq:pb_iota_triv} it corresponds to a section
\begin{equation}\label{eq:pb_theta_equivariant}
\theta^{\ad}\in A^{1,0}(S,\ad P),\qquad \theta_P\bigl(\partial_i\bigr)\big|_{p(s,z)}=\bigl(\Ad(\exp(-z))\theta^{\ad}(\partial_i)|_s\bigr)^*\big|_{p(s,z)}.
\end{equation}
 We set $D_P'':=\dbar_P+\theta_P$. Applying Definitions \ref{def:cov_deriv_complex}, \ref{def:dbar_on_section}, \ref{def:partial_on_section}, and \ref{def:higgs_on_section} to the local section $u_P:U\to P$ produces $1$-forms valued in $u_P^*T_{P/S}$.

\begin{lemma}\label{lem:pb_first_order}
We have
\begin{equation}\label{eq:pb_first_order}
\iota_{u_P}\bigl(\partial_P u_P\bigr) = A^{1,0},\qquad
\iota_{u_P}\bigl(\dbar_P u_P\bigr) = A^{0,1},\qquad
\iota_{u_P}\bigl(\theta_P u_P\bigr) = \theta^{\ad}.
\end{equation}
In particular, $u_P$ is a Higgs section for $D_P''$ if and only if $A^{0,1}=\theta^{\ad}=0$.
\end{lemma}
\begin{proof}
At $u_P(s)=(s,0)$, we have $\nabla_P^{1,0}(\partial_i) = \partial_i + \Gamma_i^a \partial_{z^a}+\Gamma_i^{\bar a}\partial_{\bar z^a}$ where $\Gamma_i^a(s,0) = -A_i^a(s)$. By Lemma \ref{lem:cov_deriv_local}, $\partial u_P(\partial_i)=-\Gamma_i^a(s,0)\partial_{z^a}|_{u_P(s)}=A_i^a(s)\partial_{z^a}|_{u_P(s)}$, and applying $\iota_{u_P}$ yields $A_i^a\,[u_P,e_a]$. The case of $\dbar u_P$ is identical with $\bar i$ in place of $i$. For $\theta_Pu_P$, \eqref{eq:pb_theta_equivariant} at $z=0$ gives $\theta_P(\partial_i)|_{u_P(s)}=(\theta^{\ad}(\partial_i)|_s)^*=(\theta^{\ad})_i^a\partial_{z^a}|_{u_P(s)}$, and applying $\iota_{u_P}$ yields $(\theta^{\ad})_i^a\,[u_P,e_a]$.
\end{proof}

We now equip $T_{P/S}\to P$ with the fiberwise linear connection $\nabla^{\mathrm{fib}}$ characterized in our chart by $\nabla^{\mathrm{fib}}_{\partial_{z^a}}\partial_{z^b}=0$ in a neighborhood of $u_P(U)$, i.e., the trivial connection in the frame $\{\partial_{z^a}\}$. Clearly, $\nabla^{\mathrm{fib}}$ is fiberwise torsion-free and its $(0,1)$-part agrees with the canonical $\dbar^{\mathrm{fib}}$ on $T_{P/S}$. Together with $\nabla_P$, this determines a vertical connection $\nabla^{\mathrm{v}}$ on $T_{P/S}\to P$ in a neighborhood of $u_P(U)$ by Definition \ref{def:vert_conn}.

\begin{lemma}\label{lem:pb_pullback_conn_compat}
	The pullback connection $\nabla^{E^{\ad}}$ on $E^{\ad}:=u_P^*T_{P/S}$ from Definition \ref{def:pullback_conn} satisfies
	\begin{equation}\label{eq:pb_pullback_conn_compat}
	\iota_{u_P}\comp\nabla^{E^{\ad}}\comp\iota_{u_P}^{-1}=\D + \tfrac{1}{2}\ad_A.
	\end{equation}
\end{lemma}
\begin{proof}
We compute the $(1,0)$-part. Let $V=V^a(s)\,[u_P,e_a]\in C^\infty(U,\ad P|_U)$ and let $\widetilde V:=V^a(s)\partial_{z^a}$ be the constant-in-$z$ extension to a section of $T_{P/S}$ near $u_P(U)$, so that $\widetilde V|_{u_P(s)}=\iota_{u_P}^{-1}V|_s$. By Definition~\ref{def:pullback_conn},  $\nabla^{E^{\ad}}_{\partial_i}\iota_{u_P}^{-1}V=\nabla^{\mathrm{v}}_{\partial_i}\widetilde V|_{u_P(s)}$. The difference $\partial_i-H_i$ is vertical, and the fiberwise trivial connection annihilates the extension $\widetilde V$ in every vertical direction. Hence $\nabla^{\mathrm{v}}_{\partial_i}\widetilde V=\nabla^{\mathrm{v}}_{H_i}\widetilde V$. By \eqref{eq:vert_conn_hor_10}, $\nabla^{\mathrm{v}}_{H_i} \widetilde V = \partial_i V^a \partial_{z^a} - V^b \partial_{z^b} \Gamma_i^a \partial_{z^a}$. Since $\omega(H_i)=0$, $\omega(\partial_i)+\Gamma_i^a\,\omega(\partial_{z^a})=0$.
By \eqref{eq:pb_omega_vertical} and \eqref{eq:pb_omega_horizontal},
\[
\bigl(A_i-[z,A_i]\bigr)+\bigl(\Gamma_i^a(s,0)+z^b\partial_{z^b}\Gamma_i^a(s,0)\bigr)\bigl(e_a-\tfrac12[z,e_a]\bigr)+O(|z|^2)\;=\;0.
\]
The first-order equation in $z$ gives
\[
z^b\partial_{z^b}\Gamma_i^a(s,0)\,e_a=[z,A_i]-\tfrac12[z,A_i]=\tfrac12[z,A_i],
\]
using $\Gamma_i^a(s,0)e_a=-A_i$. Matching coefficients of $z^b$ yields $\partial_{z^b}\Gamma_i^a(s,0)\,e_a=\tfrac12[e_b,A_i]=-\tfrac12\ad_{A_i}(e_b)$, i.e., $\partial_{z^b}\Gamma_i^a(s,0)=-\tfrac{1}{2}\bigl(\ad_{A_i}(e_b)\bigr)^a$. Therefore, $\nabla^{\mathrm{v}}_{\partial_i}\widetilde V=\bigl(\partial_iV^a+\tfrac12[A_i,V]^a\bigr)\partial_{z^a}$, and applying $\iota_{u_P}$ yields $\bigl(\partial_iV+\tfrac12[A_i,V]\bigr)^a\,[u_P,e_a]=(\D+\frac12\ad_A)_{\partial_i}V$, as claimed. The $(0,1)$-part is similar.
\end{proof}

The principal curvature $F^P\in A^2(S,\ad P)$ satisfies $F^P = \D A + \frac{1}{2}[A,A]$. Using Lemmas \ref{lem:pb_first_order} and \ref{lem:pb_pullback_conn_compat}, on $U$ we have
\begin{align}
\iota_{u_P}\comp\nabla^{E^{\ad},1,0}\comp\iota_{u_P}^{-1}\bigl(A^{1,0}\bigr) &= F^{P,2,0},\label{eq:pb_curv_20}\\
\iota_{u_P}\comp\nabla^{E^{\ad},0,1}\comp\iota_{u_P}^{-1}\bigl(A^{1,0}\bigr) + \iota_{u_P}\comp\nabla^{E^{\ad},1,0}\comp\iota_{u_P}^{-1}\bigl(A^{0,1}\bigr) &= F^{P,1,1},\label{eq:pb_curv_11}\\
\iota_{u_P}\comp\nabla^{E^{\ad},0,1}\comp\iota_{u_P}^{-1}\bigl(A^{0,1}\bigr) &= F^{P,0,2}.\label{eq:pb_curv_02}
\end{align}
By \cite[Rmk.~2.33]{LS26}, $u_P^* F_{D_P} = -\iota_{u_P}^{-1}(F^P|_U)$. By Lemma \ref{lem:pb_first_order}, \eqref{eq:pb_curv_20}--\eqref{eq:pb_curv_02} verify Proposition \ref{prop:curv_second_deriv} in the principal bundle case.

\begin{lemma}\label{lem:pb_lin_higgs}
	The linearized almost Higgs field $\Theta_{u_P}\in A^{1,0}(U,\End(u_P^*T_{P/S}))$ of Definition \ref{def:linearized_higgs} satisfies
	\begin{equation}\label{eq:pb_Theta}
	\iota_{u_P}\comp\Theta_{u_P}\comp\iota_{u_P}^{-1}=\tfrac{1}{2}\ad_{\theta^{\ad}} \in A^{1,0}\bigl(U,\End(\ad P|_U)\bigr).
	\end{equation}
\end{lemma}
\begin{proof}
	In our chart, $\theta_P(\partial_i)|_{(s,z)}=\theta_i^a(s,z)\partial_{z^a}$. Combining \eqref{eq:pb_theta_equivariant} with \eqref{eq:pb_dexp},
\[\Lambda(z)\theta_i(s,z)=\E^{-\ad(z)}\theta^{\ad}_i(s),\]
where $\theta_i(s,z):=\theta_i^a(s,z)e_a\in\g$. To first order in $z$,
\begin{align*}
\theta_i(s,z) &= \Lambda(z)^{-1}e^{-\ad(z)}\theta^{\ad}_i(s)\\
&=\bigl(1+\tfrac12\ad(z)+O(|z|^2)\bigr)\bigl(\theta^{\ad}_i-[z,\theta^{\ad}_i]+O(|z|^2)\bigr)\\
&=\theta^{\ad}_i-\tfrac12[z,\theta^{\ad}_i]+O(|z|^2).
\end{align*}
Hence by \eqref{eq:linearized_higgs_local},
\[
(\Theta_{u_P})_{i,b}^a=\partial_{z^b}\theta_i^a\big|_{z=0}=-\tfrac12[e_b,\theta^{\ad}_i]^a=\tfrac12\bigl(\ad_{\theta^{\ad}_i}(e_b)\bigr)^a,
\]
which is the matrix of $\tfrac12\ad_{\theta^{\ad}(\partial_i)}$ in the trivializing frame $\{[u_P,e_a]\}$.
\end{proof}
Let $G^P$ be the principal pseudo-curvature of $(\dbar_P,\theta^{\ad})$ (\cite[Def.~5.3]{LS26}). By Lemmas \ref{lem:pb_pullback_conn_compat} and \ref{lem:pb_lin_higgs}, on $U$ we have
\begin{align}
\iota_{u_P}\comp\nabla^{E^{\ad},0,1}\comp \iota_{u_P}^{-1}(\theta^{\ad}) + (\iota_{u_P}\comp\Theta_{u_P}\comp\iota_{u_P}^{-1})\wedge A^{0,1} &= G^{P,1,1},\label{eq:pb_pseudo_11}\\
(\iota_{u_P}\comp\Theta_{u_P}\comp\iota_{u_P}^{-1})\wedge\theta^{\ad} &= G^{P,2,0}.\label{eq:pb_pseudo_20}
\end{align}
By \cite[Lem.~5.4]{LS26} (note that $\tau(\theta^{\ad})=-\theta_P$) and Lemma \ref{lem:pb_first_order}, \eqref{eq:pb_pseudo_11} and \eqref{eq:pb_pseudo_20} are exactly Propositions \ref{prop:pseudo_curv_second_deriv} and \ref{prop:pseudo_curv_20}.

\subsection{Linearizations along flat sections and Higgs sections}\label{subsec:flat_Higgs_sections}
The pullback construction of Section \ref{subsec:second_order} yields, for any smooth section $u:S\to X$ and any fiberwise linear connection $\nabla^{\mathrm{fib}}$, a pullback linear connection $\nabla^E$ on $E=u^*T_{X/S}$, and, under the hypotheses of Definition \ref{def:linearized_higgs}, an $\End(E)$-valued $(1,0)$-form $\Theta_u$. In general both $\nabla^E$ and $\Theta_u$ depend on $\nabla^{\mathrm{fib}}$, and $(E,\Theta_u)$ need not be a Higgs bundle even when $(f,\theta)$ is one (Remark~\ref{rmk:theta_u_not_holomorphic}). Similarly, $\nabla^E$ need not be flat even when the ambient connection on $f$ is flat. In this subsection we show that the linearizations along flat sections and Higgs sections inherit the integrability properties of their ambient data.

Let $f:X\to S$ be a smooth fiber bundle with a real connection $\nabla^\RR$. The construction of the vertical connection $\nabla^{\mathrm{v},\RR}$ on $T_{X/S}^\RR\to X$ requires a fiberwise linear connection $\nabla^{\mathrm{fib},\RR}$, and the induced pullback connection $\nabla^E$ on $E=u^*T_{X/S}^\RR$ is defined by the real analogue of \eqref{eq:pullback_conn_def}.

\begin{proposition}\label{prop:flat_pullback_flat}
Let $f:X\to S$ be a smooth fiber bundle with a real connection $\nabla^\RR$, and $u:S\to X$ a flat section. The pullback connection $\nabla^E$ on $E=u^*T_{X/S}^\RR$ is independent of the choice of the fiberwise linear connection $\nabla^{\mathrm{fib},\RR}$ used in Definition \ref{def:vert_conn}. Moreover, for all local vector fields $v,w$ on $S$, $F_{\nabla^\RR}(v,w)$  vanishes along $u$. Its fiberwise linearization along $u$ defines $\operatorname{Lin}_uF_{\nabla^\RR}\in A^2(S,\End E)$ by
\[
(\operatorname{Lin}_uF_{\nabla^\RR})(v,w)(\xi):=\D\bigl(F_{\nabla^\RR}(v,w)|_{X_s}\bigr)_{u(s)}(\xi),\qquad \xi\in E_s,
\]
where, for a vector field $Z$ vanishing at a point $p$, its linearization is defined by $\D Z_p(\xi):=[\tilde{\xi},Z]_p$ for any local vector field $\tilde{\xi}$ satisfying $\tilde{\xi}_p=\xi$. Then
\[
F_{\nabla^E}=-\operatorname{Lin}_uF_{\nabla^\RR}.
\]
In particular, if $\nabla^\RR$ is flat, then $\nabla^E$ is flat.
\end{proposition}
\begin{proof}
Let $\nabla^{\mathrm{fib},\RR}_1,\nabla^{\mathrm{fib},\RR}_2$ be two fiberwise linear connections, with corresponding vertical connections $\nabla^{\mathrm{v},\RR}_1,\nabla^{\mathrm{v},\RR}_2$ and pullback connections $\nabla^{E,1},\nabla^{E,2}$. By Definition \ref{def:vert_conn}, $\nabla^{\mathrm{v},\RR}_1$ and $\nabla^{\mathrm{v},\RR}_2$ agree in horizontal directions, so $B := \nabla^{\mathrm{v},\RR}_2 - \nabla^{\mathrm{v},\RR}_1 \in C^\infty\bigl(X,(T_{X/S}^{\RR})^*\otimes\End(T_{X/S}^\RR)\bigr)$. For $v\in T_sS^\RR$ and $V\in\Gamma(S,E)$ with extension $\widetilde V$,
\[
\bigl(\nabla^{E,2}_v - \nabla^{E,1}_v\bigr)V = B\bigl(\D u(v)^{\mathrm{v}}\bigr)\bigl(\widetilde V|_{u(s)}\bigr),
\]
where $\D u(v)^{\mathrm{v}}\in T_{X/S,u(s)}^\RR$ is the vertical component of $\D u(v)$ relative to the splitting $T_{u(s)}X^\RR = \im(\nabla^\RR)|_{u(s)}\oplus T_{X/S,u(s)}^\RR$. Flatness of $u$ yields $\D u(v)^{\mathrm{v}} = (\nabla^\RR u)(v) = 0$, so $\nabla^{E,1} = \nabla^{E,2}$.

Let $V$ be a local section of $E$ with a vertical extension $\widetilde V$. Since $H_v$ is projectable and $\widetilde V$ is vertical, $[H_v,\widetilde V]$ is vertical, and along $u$ one has $\nabla^E_vV=[H_v,\widetilde V]|_u$. Since $u$ is horizontal, $H_v,H_w$ are $u$-related to $v,w$ respectively. Therefore $[H_v,H_w]|_u=\D u([v,w])=H_{[v,w]}|_u$, and hence $F_{\nabla^\RR}(v,w)|_u=0$.

Using $[H_w,\widetilde V]$ and $[H_v,\widetilde V]$ as vertical extensions of $\nabla^E_wV$ and $\nabla^E_vV$ respectively, the Jacobi identity gives
\[
F_{\nabla^E}(v,w)V=\bigl([H_v,[H_w,\widetilde V]]-[H_w,[H_v,\widetilde V]]-[H_{[v,w]},\widetilde V]\bigr)\big|_u=[F_{\nabla^\RR}(v,w),\widetilde V]\big|_u.
\]
If a vector field $Z$ vanishes at a point $p$, then $[Z,\widetilde V]_p=-\D Z_p(\widetilde V_p)$. Applying this on the fiber $X_s$ to the vector field $Z=F_{\nabla^\RR}(v,w)|_{X_s}$ gives $F_{\nabla^E}(v,w)V=-(\operatorname{Lin}_uF_{\nabla^\RR})(v,w)V$.
\end{proof}

\begin{remark}\label{rmk:pullback_bott}
In the setting of Proposition \ref{prop:flat_pullback_flat}, let $U$ be a neighborhood of $u(S)$ on which $\nabla^\RR$ is flat, and put $\mathcal H:=\im(\nabla^\RR)|_U$.
The integrable distribution $\mathcal H$ has normal bundle $Q:=TU^\RR/\mathcal H$, with quotient map $q$, and Bott partial connection
\[
\nabla^{\mathrm{Bott}}_H(qW):=q[H,W],\qquad H\in\Gamma(\mathcal H),\quad W\in\Gamma(TU^\RR).
\]
Integrability makes this independent of the representative $W$. Its extensions to full connections are the basic connections of \cite[Def.~6.1]{bott72lectures}.
The splitting $TU^\RR=\mathcal H\oplus T_{X/S}^\RR|_U$ identifies $Q$ with $T_{X/S}^\RR|_U$.
Since $u$ is horizontal, restriction along $u$ yields
\[
\nabla^{\mathrm{Bott},u}_vV=[H_v,\widetilde V]|_u=\nabla^E_vV,
\]
where $\widetilde V$ is a vertical extension of $V$.
Thus the pullback connection coincides with the Bott connection along $u$ and is flat, in agreement with \cite[Lem.~6.3]{bott72lectures}.
Without flatness of the ambient connection, independence of the fiberwise linear connection still holds, but $\nabla^E$ need not be flat.
\end{remark}

\begin{proposition}\label{prop:higgs_section_linearized_higgs}
Suppose $(f,\theta)$ is a nonlinear Higgs bundle. If $u:S\to X$ is a Higgs section, then the linearized Higgs field $\Theta_u$ is independent of the choice of $\nabla^{\mathrm{fib}}$ in Definition \ref{def:linearized_higgs} and is a Higgs field on the holomorphic vector bundle $E = u^* T_{X/S}$, i.e., $(E, \Theta_u)$ is a Higgs bundle over $S$.
\end{proposition}
\begin{proof}
 By Remark \ref{rmk:linearized_higgs_chern}, $\Theta_u$ is independent of the choice of $\nabla^{\mathrm{fib}}$. Since $\theta u = 0$, we have $\theta_i^\gamma(s, u(s)) = 0$ in adapted local holomorphic coordinates. Then the local expression \eqref{eq:linearized_higgs_local} for $\Theta_u$ simplifies to
\[
(\Theta_u)_{i,\beta}^\alpha(s) = \partial_\beta\theta_i^\alpha(s, u(s)).
\]
Applying the $\dbar$-operator on $E$ and using the chain rule,
\[
\partial_{\bar{j}} \bigl( (\Theta_u)_{i,\beta}^\alpha(s) \bigr) = \frac{\partial^2 \theta_i^\alpha}{\partial \bar{s}^j \partial z^\beta}\bigg|_{u(s)} + \frac{\partial^2 \theta_i^\alpha}{\partial z^\gamma \partial z^\beta}\bigg|_{u(s)} \partial_{\bar{j}} u^\gamma(s) + \frac{\partial^2 \theta_i^\alpha}{\partial \bar{z}^\gamma \partial z^\beta}\bigg|_{u(s)} \partial_{\bar{j}} \bar{u}^\gamma(s).
\]
Because $\theta$ is a holomorphic Higgs field, $\partial_{\bar{j}} \theta_i^\alpha = 0$ and $\partial_{\bar{\gamma}} \theta_i^\alpha = 0$. Since $u$ is a holomorphic section, $\partial_{\bar{j}} u^\gamma = 0$. Therefore, $\partial_{\bar{j}} (\Theta_u)_{i,\beta}^\alpha = 0$, proving that $\Theta_u$ is holomorphic.

The integrability condition $[\theta, \theta] = 0$ implies that for all $i, j$,
\[
\theta_i^\gamma \partial_\gamma \theta_j^\alpha - \theta_j^\gamma \partial_\gamma \theta_i^\alpha = 0.
\]
 Differentiating this identity with respect to the fiber coordinate $z^\beta$ yields
\[
(\partial_\beta \theta_i^\gamma)(\partial_\gamma \theta_j^\alpha) + \theta_i^\gamma (\partial_\beta \partial_\gamma \theta_j^\alpha) - (\partial_\beta \theta_j^\gamma)(\partial_\gamma \theta_i^\alpha) - \theta_j^\gamma (\partial_\beta \partial_\gamma \theta_i^\alpha) = 0.
\]
Evaluating this equation at the point $z = u(s)$, the terms involving the second derivatives vanish because they are multiplied by $\theta_i^\gamma(s, u(s)) = 0$ and $\theta_j^\gamma(s, u(s)) = 0$. Then
\[
(\partial_\beta \theta_i^\gamma)\big|_{u(s)} (\partial_\gamma \theta_j^\alpha)\big|_{u(s)} - (\partial_\beta \theta_j^\gamma)\big|_{u(s)} (\partial_\gamma \theta_i^\alpha)\big|_{u(s)} = 0.
\]
Substituting $(\Theta_u)_{i,\beta}^\alpha = \partial_\beta \theta_i^\alpha|_{u(s)}$, this becomes
\[
(\Theta_u)_{i,\gamma}^\alpha (\Theta_u)_{j,\beta}^\gamma - (\Theta_u)_{j,\gamma}^\alpha (\Theta_u)_{i,\beta}^\gamma = 0,
\]
which is precisely $[\Theta_u(\partial_i), \Theta_u(\partial_j)] = 0$. Hence $(E, \Theta_u)$ is a Higgs bundle.
\end{proof}

\section{Nonlinear Bochner--Kodaira--Nakano identity}\label{sec:nonlin_BKN}
The aim of this section is to establish a Bochner--Kodaira--Nakano identity for a smooth section $u$ of a complex fiber bundle. As applications, we obtain a nonlinear generalization of a Bochner-type vanishing theorem for vector bundles, and a generalization of Yau's Schwarz lemma.

Throughout this section, $(S,\omega_S)$ denotes a connected Hermitian manifold of complex dimension $n\geq1$, and $(f:X\to S,T_{X/S})$ denotes a complex fiber bundle with connected fibers of complex dimension $m$, equipped with a $\dbar$-operator $\dbar$ and a fiberwise K\"ahler metric $\omega_{X/S}$. We denote by $J$ the almost complex structure on $X$ determined by $\dbar$, by $TX^\CC=TX\oplus\widebar{TX}$ the splitting induced by $J$, and by $h_{X/S}$ the fiberwise Hermitian metric on $T_{X/S}$ determined by $\omega_{X/S}$.  Let $\nabla^{1,0}$ be a symplectic connection on $f$ associated to $\dbar$ and $\omega_{X/S}$ (see \cite[\S4]{LS26}), with underlying real connection $\nabla^\RR$, whose complexification is $\nabla^\CC=\nabla^{1,0}+\widebar{\nabla^{1,0}}$. Denote the curvature of $\nabla^\RR$ by $F_{\nabla^\RR}$.

\subsection{The identity}

Let $u:S\to X$ be a smooth section of $f$. We begin by decomposing the differential $\D u: TS^{\RR}\to u^*TX^{\RR}$ into horizontal and vertical parts. Write $\mathcal H\subset TX^{\RR}$ for the horizontal subbundle determined by $\nabla^\RR$, so that $TX^{\RR}=T_{X/S}^{\RR}\oplus\mathcal H$, and let $H:f^*TS^{\RR}\xrightarrow{\;\cong\;}\mathcal H$ denote the horizontal lift. Since $f\comp u=\id_S$, for every $v\in T_sS^{\RR}$ we have a unique decomposition
\begin{equation}\label{eq:du_decomp}
\D u(v)=H_{u(s)}(v)+(\nabla^{\RR}u)(v)\in T_{u(s)}X^{\RR},
\end{equation}
with $(\nabla^{\RR}u)(v)\in T_{X/S,u(s)}^{\RR}$. This is precisely Definition \ref{def:cov_deriv_real}. Complexifying and taking types with respect to the complex structures on $S$ and on $X_s$, the projection of $\nabla^{1,0}u$ and $\nabla^{0,1}u$ onto $T_{X/S}$ yields the operators $\partial u\in A^{1,0}(S,u^*T_{X/S})$ and $\dbar u\in A^{0,1}(S,u^*T_{X/S})$ of Definitions \ref{def:partial_on_section} and~\ref{def:dbar_on_section}.

Let $\omega_X$ be a real $(1,1)$-form (with respect to $J$) on $X$ which restricts to $\omega_{X/S}$ on fibers and induces $\nabla^\RR$, i.e., $\mathcal H_{x}$ is the $\omega_X$-orthogonal complement to $T_{X/S,x}^{\RR}\subset T_{x}X^{\RR}$ for each $x\in X$. Such a $(1,1)$-form always exists, e.g.,
\begin{equation}\label{eq:omega_nabla}
	\omega_{\nabla^\RR}(V_1,V_2):=\omega_{X/S}(\mathrm{pr}^{\mathrm{v}}_{\nabla^\RR} V_1,\mathrm{pr}^{\mathrm{v}}_{\nabla^\RR} V_2)
\end{equation}
is such a $(1,1)$-form (\cite[Eq.~(4.2) and (4.4)]{LS26}). Then, for all $v,w\in T_sS^{\RR}$,
\begin{equation}\label{eq:pullback_orthog}
(u^*\omega_X)(v,w)=\omega_X\bigl(H_{u(s)}(v),H_{u(s)}(w)\bigr)+\omega_{X/S}\bigl((\nabla^{\RR}u)(v),(\nabla^{\RR}u)(w)\bigr),
\end{equation}
where $\nabla^\RR u$ is defined in Definition \ref{def:cov_deriv_real}.

We denote by $\omega_X^H\in C^{\infty}(X,f^*\!\wedge^{1,1}T^*S)$ the \emph{horizontal part} of $\omega_X$, defined fiberwise by
\[
\omega_X^H(x)\bigl(v,w\bigr):=\omega_X\bigl(H_x(v),H_x(w)\bigr),\qquad v,w\in T_{f(x)}S^{\CC}.
\]
This is a $(1,1)$-form on $S$ with coefficients in $C^{\infty}(X)$, and $u^*\omega_X^H\in A^{1,1}(S)$. Note that $\omega_{\nabla^\RR}^H=0$. Recall that there is a contraction operator
\[
\Lambda_{\omega_S}:A^2(S)\to C^{\infty}(S),\qquad n\,\eta\wedge\omega_S^{n-1}=(\Lambda_{\omega_S}\eta)\,\omega_S^n.
\]
Given a smooth section $u$, the pullback bundle $E:=u^*T_{X/S}$ inherits from $h_{X/S}$ a Hermitian metric $h_E=u^*h_{X/S}$, and the pointwise norms of $\partial u\in A^{1,0}(S,E)$ and $\dbar u\in A^{0,1}(S,E)$ are defined using $\omega_S$ on $S$ and $h_E$ on $E$. Our normalization is $h_S(v,w)=-2\I\omega_S(v,\bar w)$ for $v,w\in TS$, and norms on forms are induced by the dual metric on $T^*S$ and its exterior powers.

\begin{proposition}\label{prop:pointwise_identity}
Let $(f:X\to S,T_{X/S})$ be a complex fiber bundle with a $\dbar$-operator $\dbar$, equipped with a fiberwise K\"ahler metric $\omega_{X/S}$ and a symplectic connection $\nabla^{1,0}$ induced by a real $(1,1)$-form $\omega_X$ on $X$ \textup(with respect to the almost complex structure determined by $\dbar$\textup), and let $(S,\omega_S)$ be a Hermitian manifold. For every smooth section $u:S\to X$, we have
	\begin{equation}\label{eq:pointwise_id}
	|\partial u|^2-|\dbar u|^2=\Lambda_{\omega_S}\bigl(u^*\omega_X-u^*\omega_X^H\bigr).
	\end{equation}
\end{proposition}

\begin{proof}
Fix $s_0\in S$ and set $x_0:=u(s_0)$. It suffices to verify \eqref{eq:pointwise_id} at $s_0$.	Choose local holomorphic coordinates $(s^1,\ldots,s^n)$ on a neighborhood of $s_0$ with	$h_{i\bar j}(s_0)=\frac{1}{2}\delta_{ij}$, where $\omega_S=\I \,h_{i\bar j}\,\D s^i\wedge\D\bar{s}^j$, and adapted  local coordinates $(s^i,z^1,\ldots,z^m)$ on $X$ near $x_0$ (\cite[\S2]{LS26}; the $z^\alpha$ are fiberwise holomorphic) such that $g_{\alpha\bar\beta}(x_0)=\frac{1}{2}\delta_{\alpha\beta}$ and the coefficients $\Gamma_i^\alpha$ of	$\nabla^{1,0}$ and $\Gamma_{\bar i}^\alpha$ of $\dbar$ satisfy $\Gamma_i^\alpha(x_0)=\Gamma_{\bar i}^\alpha(x_0)=0$. The latter is possible: in any given adapted chart, replace $z^\alpha$ by $z^\alpha-\Gamma_i^\alpha(x_0)(s^i-s_0^i) -\Gamma_{\bar i}^\alpha(x_0)(\bar s^i-\bar s_0^i)$, an adapted coordinate change which cancels both sets of coefficients at $x_0$ by the
transformation laws of \cite[Eqs.~(2.8),\,(2.9)]{LS26} and leaves $g_{\alpha\bar\beta}(x_0)$
unchanged. Since $\nabla^{1,0}$ is pure (\cite[Lem.~4.6]{LS26}), $H_i=\partial_i+\Gamma_i^\alpha\partial_\alpha+\overline{\Gamma_{\bar i}^\beta}\,\partial_{\bar\beta}$ and $H_i|_{x_0}=\partial_i$. By the	$\omega_X$-orthogonality of horizontal and vertical tangent vectors at $x_0$, we then have	$g_{i\bar\alpha}(x_0)=0$, since
\[	0=\omega_X(H_i,\partial_{\bar\alpha})|_{x_0}=\omega_X(\partial_i,\partial_{\bar\alpha})|_{x_0}=\I\,g_{i\bar\alpha}(x_0).\]

Write $u$ locally as $s\mapsto(s,u^1(s),\ldots,u^m(s))$. By Lemma \ref{lem:cov_deriv_local} and Definitions \ref{def:dbar_on_section}, \ref{def:partial_on_section}, at $s_0$,
\[
\partial u\big|_{s_0}=(\partial_i u^\alpha)\big|_{s_0}\,\D s^i\otimes\partial_\alpha,\qquad \dbar u\big|_{s_0}=(\partial_{\bar i}u^\alpha)\big|_{s_0}\,\D\bar s^i\otimes\partial_\alpha,
\]
since $\Gamma_i^\alpha(x_0)=\Gamma_{\bar i}^\alpha(x_0)=0$. Hence
\begin{equation}\label{eq:norm_diff_at_s0}
(|\partial u|^2-|\dbar u|^2)(s_0)=\sum_{i,\alpha}\bigl(|\partial_i u^\alpha|^2-|\partial_{\bar i}u^\alpha|^2\bigr)\big|_{s_0}.
\end{equation}

Now we compute $u^*\omega_X$. At the point $x_0$ the coefficients of $\dbar$ and of $H_i$ vanish, so $T_{x_0}X=\operatorname{span}_\CC\{\partial_i,\partial_\alpha\}$. Hence the $(1,1)$-type of $\omega_X$ and the $\omega_X$-orthogonality of horizontal and vertical vectors give, at $x_0$,
\[\omega_X(\partial_\alpha,\partial_\beta)=\omega_X(\partial_i,\partial_\alpha)=\omega_X(\partial_i,\partial_{\bar\alpha})=0,\quad\omega_X(\partial_\alpha,\partial_{\bar\beta})=\I\,g_{\alpha\bar\beta}=\tfrac{\I}{2}\delta_{\alpha\beta},\quad \omega_X(\partial_i,\partial_{\bar j})=\I\, g_{i\bar j}.\]
Evaluating $u^*\omega_X(\partial_i,\partial_{\bar j})\big|_{s_0}=\omega_X\bigl(\D u(\partial_i),\D u(\partial_{\bar j})\bigr)\big|_{x_0}$ with $\D u(\partial_i)=\partial_i+\partial_iu^\alpha\partial_\alpha
+\overline{\partial_{\bar i}u^\beta}\,\partial_{\bar\beta}$ gives
\[
\eta_{i\bar j}(s_0)=g_{i\bar j}(x_0)+\tfrac{1}{2}\sum_\alpha\bigl(\partial_iu^\alpha\,\overline{\partial_ju^\alpha}-\partial_{\bar j}u^\alpha\,\overline{\partial_{\bar i}u^\alpha}\bigr)\big|_{s_0},
\]
where $u^*\omega_X(\partial_i,\partial_{\bar j})\big|_{s_0}=\I\,\eta_{i\bar j}(s_0)$. Since $h_{i\bar j}(s_0)=\frac{1}{2}\delta_{ij}$, $\Lambda_{\omega_S}(u^*\omega_X)(s_0)=2\sum_i\eta_{i\bar i}(s_0)$, giving
\begin{equation}\label{eq:Lambda_at_s0}
\Lambda_{\omega_S}(u^*\omega_X)(s_0)=2\sum_i g_{i\bar i}(x_0)+\sum_{i,\alpha}\bigl(|\partial_iu^\alpha|^2-|\partial_{\bar i}u^\alpha|^2\bigr)\big|_{s_0}.
\end{equation}

Finally, $\omega_X^H$ at $x_0$ is the $(1,1)$-form on $T_{s_0}S$ obtained by evaluating $\omega_X$ on horizontal lifts. In our coordinates at $s_0$, the horizontal lift of $\partial_i$ at $x_0$ is $\partial_i$ itself, so $\omega_X^H(x_0)=\I\,g_{i\bar j}(x_0)\,\D s^i\wedge\D\bar s^j$, and
\begin{equation}\label{eq:omegaH_at_s0}
\Lambda_{\omega_S}(u^*\omega_X^H)(s_0)=2\sum_i g_{i\bar i}(x_0).
\end{equation}
Subtracting \eqref{eq:omegaH_at_s0} from \eqref{eq:Lambda_at_s0} and comparing with \eqref{eq:norm_diff_at_s0} yields \eqref{eq:pointwise_id} at $s_0$. Since $s_0$ was arbitrary, the identity holds pointwise on $S$.
\end{proof}

\begin{remark}\label{rmk:pointwise_map_identity}
Proposition \ref{prop:pointwise_identity} is a relative analogue of the classical identity for a smooth map $\varphi:(M,\omega_M)\to(N,\omega_N)$ between K\"ahler manifolds (\cite{lichnerowicz69applications}, see also \cite[Th.~3.2]{hattori24smooth}):
\[
|(\partial\varphi)^{1,0}|^2-|(\dbar\varphi)^{1,0}|^2=\Lambda_{\omega_M}(\varphi^*\omega_N).
\]
Despite the formal similarity, the first-order operators in the two settings have different meanings. In the map case, $(\partial\varphi)^{1,0}\in A^{1,0}(M,\varphi^*TN)$ and $(\dbar\varphi)^{1,0}\in A^{0,1}(M,\varphi^*TN)$ are obtained by decomposing the full differential $\D\varphi$ with respect to the complex structures of both $M$ and $N$, and their norms use the Hermitian metrics on both manifolds. Here, $\partial u$ and $\dbar u$ are the vertical $T_{X/S}$-valued derivatives defined in Section \ref{subsec:dbar_sections}. Their norms use only the Hermitian metric $\omega_S$ on the base and the fiberwise Hermitian metric $h_{X/S}$ determined by $\omega_{X/S}$. Suppose that $(S,\omega_S)$ and $(Y,\omega_Y)$ are K\"ahler, $X=S\times Y$ is equipped with the product complex structure and the trivial connection, $\omega_X=\operatorname{pr}_Y^*\omega_Y$, and $u=(\id_S,\varphi)$ for a smooth map $\varphi:S\to Y$. Then $\omega_X^H=0$, $\partial u=(\partial\varphi)^{1,0}$, and $\dbar u=(\dbar\varphi)^{1,0}$ under the canonical identification $u^*T_{X/S}\cong\varphi^*TY$, so \eqref{eq:pointwise_id} reduces to the classical identity above.
\end{remark}

Integrating Proposition \ref{prop:pointwise_identity} over the Hermitian manifold $(S,\omega_S)$ (now assumed to be compact) yields the energy identity. For the right-hand side, define
\begin{equation}\label{eq:def_deg}
\deg_{\omega_X}(u):=\int_S u^*\omega_X\wedge \omega_S^{n-1},\qquad \deg_{\omega_X^H}(u):=\int_S u^*\omega_X^H\wedge\omega_S^{n-1}.
\end{equation}

\begin{corollary}\label{cor:energy_id}
In the setup of Proposition \ref{prop:pointwise_identity}, suppose moreover that $S$ is compact. Then
\begin{equation}\label{eq:energy_id}
\int_S|\partial u|^2\,\omega_S^n-\int_S|\dbar u|^2\,\omega_S^n=n(\deg_{\omega_X}(u)-\deg_{\omega_X^H}(u)).
\end{equation}
\end{corollary}

\begin{remark}\label{rmk:nonlin_BKN_recovers_linear}
Let $(E,h)\to(S,\omega_S)$ be a Hermitian holomorphic vector bundle over a compact Hermitian manifold, let $D_h=D_h'+D_h''$ be its Chern connection, with $D_h''=\dbar_E$, and let $F_h=D_h^2$. Put $\tau=[\Lambda_{\omega_S},(\partial\omega_S)\wedge]$ and $\bar{\tau}=[\Lambda_{\omega_S},(\bar\partial\omega_S)\wedge]$. With graded commutators and formal adjoints, the classical Bochner--Kodaira--Nakano identity (see \cite[Th.~VII.1.2]{demailly12CADG}), adapted to our normalization, is
\[
\Delta_h''=\Delta_h'+\tfrac12[\I F_h,\Lambda_{\omega_S}]+[D_h',\tau^\ast]-[D_h'',\bar{\tau}^{\,\ast}]
\]
on $E$-valued forms. Set $\mathscr T_{h,\omega_S}:=[D_h',\tau^\ast]-[D_h'',\bar{\tau}^{\,\ast}]$. Since $[\I F_h,\Lambda_{\omega_S}]u=-\I\Lambda_{\omega_S}F_h\,u$ for an $E$-valued $0$-form $u$, taking the $L^2$ inner product with $u$ gives
\begin{equation}\label{eq:linear_energy0}
\|D_h''u\|_{L^2}^2=\|D_h'u\|_{L^2}^2-\tfrac12\int_S\langle\I\Lambda_{\omega_S}F_h\,u,u\rangle_h\,\omega_S^n+(\mathscr T_{h,\omega_S}u,u)_{L^2},
\end{equation}
where all $L^2$ norms and inner products are computed using $h$ and $\omega_S^n$.

Now regard $E$ as the holomorphic fiber bundle $f:E\to S$, let $\mathbf z\in C^\infty(E,f^*E)$ be the tautological section, set $\phi(v)=|v|_h^2$, and take $\omega=\I\,\partial\bar\partial\phi$ as in \cite[Ex.~4.17]{LS26}. The canonical identification $\iota_u:E\xrightarrow{\cong}u^*T_{E/S}$ of \eqref{eq:pullback_ident_vb} identifies the differential operators in Corollary \ref{cor:energy_id} with $D_h'u$ and $D_h''u$, as follows from Section \ref{subsec:vb_diff_op_curv}, while the induced vertical Hermitian metric pulls back to $2h$. Thus $|\partial u|^2=2|D_h'u|_h^2$ and $|\dbar u|^2=2|D_h''u|_h^2$. Since $\omega$ induces the Ehresmann connection associated with $D_h$, the horizontal--vertical decomposition gives $\omega=\omega_{\nabla^\RR}+\omega^H$. The curvature computation in \cite[Ex.~4.17]{LS26} therefore yields
\begin{equation}\label{eq:vb_omegaH}
\omega^H=-\I\langle f^*F_h\,\mathbf z,\mathbf z\rangle_{f^*h},\qquad u^*\omega^H=-\I\langle F_hu,u\rangle_h.
\end{equation}
Consequently, Corollary \ref{cor:energy_id} specializes to
\begin{equation}\label{eq:linear_energy1}
\|D_h'u\|_{L^2}^2-\|D_h''u\|_{L^2}^2=\frac n2\int_Su^*\omega\wedge\omega_S^{n-1}+\frac12\int_S\langle\I\Lambda_{\omega_S}F_h\,u,u\rangle_h\,\omega_S^n.
\end{equation}

Let $\eta:=\Lambda_{\omega_S}(\partial\omega_S)\in A^{1,0}(S)$ and $\bar\eta:=\Lambda_{\omega_S}(\bar\partial\omega_S)\in A^{0,1}(S)$.
The pointwise Lefschetz decomposition gives $\partial\omega_S^{n-1}=\eta\wedge\omega_S^{n-1}$ and $\bar\partial\omega_S^{n-1}=\bar\eta\wedge\omega_S^{n-1}$. Indeed, for $n\geq 2$, write $\partial\omega_S=P+\frac{1}{n-1}\omega_S\wedge\eta$, where $\Lambda_{\omega_S}P=0$. Since the primitive $3$-form $P$ satisfies $P\wedge\omega_S^{n-2}=0$, multiplication by $(n-1)\omega_S^{n-2}$ gives the first equality, and the second is its complex conjugate. For $n=1$, both equalities are trivial. Let $\alpha:=u^*(\partial\phi)\in A^1(S)$. In a local holomorphic frame, writing $\phi=h_{\alpha\bar\beta}z^\alpha\bar z^\beta$ and $u=u^\alpha e_\alpha$, we have $\alpha=\bigl((\partial_i h_{\alpha\bar\beta})u^\alpha\bar u^\beta+h_{\alpha\bar\beta}(\partial_i u^\alpha)\bar u^\beta\bigr)\D s^i+h_{\alpha\bar\beta}(\partial_{\bar j}u^\alpha)\bar u^\beta\D\bar s^j=h(D_h'u,u)+h(D_h''u,u)$.

Since $\omega=\I\,\partial\bar\partial\phi=-\I\,\D(\partial\phi)$ and $S$ is compact, Stokes' theorem gives
\begin{align*}
n\int_Su^*\omega\wedge\omega_S^{n-1}&=-n\I\int_S\D\alpha\wedge\omega_S^{n-1}=-n\I\int_S\alpha\wedge\D\omega_S^{n-1}\\
&=-n\I\int_S\left(h(D_h'u,u)\wedge\bar\eta+h(D_h''u,u)\wedge\eta\right)\wedge\omega_S^{n-1}.
\end{align*}
On $E$-valued $0$-forms, $\tau u=\eta\,u$, $\bar\tau u=\bar\eta\,u$, and $\tau^*u=\bar\tau^*u=0$. Hence $\mathscr T_{h,\omega_S}u=\tau^*D_h'u-\bar\tau^*D_h''u$. Then we have
\begin{align*}
(\mathscr T_{h,\omega_S}u,u)_{L^2}
&=(D_h'u,\tau u)_{L^2}-(D_h''u,\bar\tau u)_{L^2}\\
&=\frac{n\I}{2}\int_Sh(D_h'u,u)\wedge\bar\eta\wedge\omega_S^{n-1}-\frac{n\I}{2}\int_S\eta\wedge h(D_h''u,u)\wedge\omega_S^{n-1}\\
&=\frac{n\I}{2}\int_S\left(h(D_h'u,u)\wedge\bar\eta+h(D_h''u,u)\wedge\eta\right)\wedge\omega_S^{n-1}=-\frac n2\int_Su^*\omega\wedge\omega_S^{n-1}.
\end{align*}
Therefore, \eqref{eq:linear_energy0} and \eqref{eq:linear_energy1} are equivalent.
\end{remark}

The following result is standard.
\begin{lemma}\label{lem:top_inv}
Let $(S,\omega_S)$ be a compact semi-K\"ahler manifold, and let $\omega_X$ be any closed real $(1,1)$-form on $X$ \textup(in particular, any relatively K\"ahler form in the sense of \cite[\S 4.2]{LS26}\textup). Then $\deg_{\omega_X}(u)$ depends only on the de Rham cohomology class $[u^*\omega_X]\in H^2(S,\RR)$ and the class $[\omega_S^{n-1}]\in H^{2n-2}(S,\RR)$. In particular, if $u_1$ and $u_2$ are homotopic, then $\deg_{\omega_X}(u_1)=\deg_{\omega_X}(u_2)$.
\end{lemma}
\begin{remark}\label{rmk:Hpart_not_top}
The quantity $\deg_{\omega_X^H}(u)$ is in general not topologically invariant.
\end{remark}

\begin{example}\label{ex:degree_vertical_tangent}
Assume that $(S,\omega_S)$ is compact semi-K\"ahler. Let $\omega_X$ be a closed real $(1,1)$-form on $X$. Suppose that in $H^2(X,\RR)$ one has
\begin{equation}\label{eq:omega_c1_vertical_comparison}
[\omega_X]=c\,c_1(T_{X/S})+f^*[\eta]
\end{equation}
for some $c\in\RR$ and some class $[\eta]\in H^2(S,\RR)$ represented by a closed real $(1,1)$-form $\eta$. Then every smooth section $u:S\to X$ satisfies
\begin{equation}\label{eq:degree_c1_vertical_comparison}
\deg_{\omega_X}(u)=c\,\deg(u^*T_{X/S})+\int_S\eta\wedge\omega_S^{n-1}.
\end{equation}
In particular, if $[\eta]$ pairs trivially with $[\omega_S^{n-1}]$, then $\deg_{\omega_X}(u)=c\,\deg(u^*T_{X/S})$.

A geometric source of \eqref{eq:omega_c1_vertical_comparison} is provided by Fano K\"ahler--Einstein fibrations.  Let $f:X\to S$ be a holomorphic submersion with compact fibers, and suppose that $T_{X/S}$ is equipped with a smooth fiberwise K\"ahler--Einstein metric $\omega_{X/S}$ such that $\frac{1}{2\pi}\operatorname{Ric}(\omega_s)=\lambda\,\omega_s$ ($\lambda>0$) on every fiber $X_s$.  Let $\rho_{X/S}$ denote the Chern--Weil representative of $c_1(T_{X/S})$ obtained from the induced Hermitian metric on $K_{X/S}^{-1}=\det T_{X/S}$.  Thus $\rho_{X/S}$ is a closed real $(1,1)$-form on $X$, and $\rho_{X/S}|_{X_s}=\frac{1}{2\pi}\operatorname{Ric}(\omega_s)=\lambda\,\omega_s$. Consequently $\omega_X:=\lambda^{-1}\rho_{X/S}$ is a closed real $(1,1)$-form whose restriction to each fiber is $\omega_s$, and $[\omega_X]=\lambda^{-1}c_1(T_{X/S})$. Hence, for every smooth section $u:S\to X$,
\begin{equation}\label{eq:degree_ke_vertical_tangent}
\deg_{\omega_X}(u) = \lambda^{-1}\deg(u^*T_{X/S}).
\end{equation}

Let $(X,H)\to (S,L)$ be a polarized holomorphic submersion such that every fiber $(X_s,H_s)$ admits a cscK metric in $c_1(H_s)$. Then \cite[Lem.~3.8]{dervan_sektnan2021optimal} yields a relatively cscK form $\omega_X\in c_1(H)$. If, in addition, $H\cong K_{X/S}^{-q}$, then $c_1(H)=q\,c_1(T_{X/S})$, and hence $\deg_{\omega_X}(u)=q\,\deg(u^*T_{X/S})$.

Fix a smooth polarized compact cscK manifold $(Y,H_Y,\omega_Y)$ with $\omega_Y\in c_1(H_Y)$. Assume that a reductive group $G$ acts linearly on $(Y,H_Y)$, that a compact real form $K$ of $G$ preserves $\omega_Y$ (then the action of $K$ on $(Y,\omega_Y)$ is Hamiltonian), and that there is a $G$-linearized isomorphism  $H_Y\cong K_Y^{-q}$ for a positive integer $q$. Let $P$ be a holomorphic principal $G$-bundle over a compact complex manifold $S$ which admits a reduction $P_K\subset P$ of the structure group to $K$. Consider the associated fiber bundle $f:X=P\times_G Y=P_K\times_K Y\to S$, with fiberwise K\"ahler metric $\omega_{X/S}$ (\cite[Lem.~4.21]{LS26}). Let $A$ be the Chern connection on $P$ induced by the reduction $P_K$, and let $\nabla_A^{\RR}$ be the real Ehresmann connection on $f$ induced by $A$. The chosen $G$-linearized isomorphism induces $H:=P\times_G H_Y\cong K_{X/S}^{-q}$, and hence
$c_1(H)=q\,c_1(T_{X/S})$. By \cite[Prop.~3.3]{mccarthy2022canonical}, there exists a closed $(1,1)$-form $\omega_X\in c_1(H)$ which restricts to $\omega_{X/S}$ on fibers. Then we have $\deg_{\omega_X}(u)=q\,\deg(u^*T_{X/S})$. By \cite[Lem.~4.26]{LS26}, the symplectic connection induced by $\omega_X$ is $\nabla_A^\RR$.
\end{example}

\subsection{Vanishing theorem}
Let $(S,\omega_S)$ be a connected Hermitian manifold of complex dimension $n$, with associated Hermitian metric $h_S$, and let $(f:X\to S,T_{X/S})$ be a complex fiber bundle with an integrable $\dbar$-operator $\dbar$. Denote $\Delta_{\omega_S}:=\I\Lambda_{\omega_S}\partial_S\dbar_S$.

\begin{lemma}\label{lem:potential_identity}
Let $\phi\in C^\infty(X,\RR)$ and $\omega_X:=\I\,\partial_X\dbar_X\phi$. Assume that $\omega_X$ restricts to a fiberwise K\"ahler metric $\omega_{X/S}$. Let $\mathcal H\subset TX^\RR$ be the horizontal distribution \textup(which is a symplectic connection\textup) determined by $\omega_X$, and let $H:f^*TS^{\RR}\to \mathcal H$ be its horizontal lift. Set $q_\phi:=\Lambda_{\omega_S}\omega_X^H$. Then every holomorphic section $u:S\to X$ satisfies
\[
\Delta_{\omega_S}(\phi\comp u)=|\partial u|^2+q_\phi\comp u.
\]
\end{lemma}
\begin{proof}
Since $u$ is holomorphic, $\dbar u=0$ and $u^*\omega_X=\I\,\partial_S\dbar_S(\phi\comp u)$. By Proposition \ref{prop:pointwise_identity},
\[
|\partial u|^2=\Lambda_{\omega_S}\bigl(u^*\omega_X-u^*\omega_X^{H}\bigr)=\Delta_{\omega_S}(\phi\comp u)-q_\phi\comp u,
\]
which proves the assertion.
\end{proof}
\begin{remark}
If $\dim_{\CC}X_s=m\geq1$ and $X_s$ is compact, then no exact $2$-form on $X$ can restrict to a Kähler form on $X_s$. Indeed, if $\omega_X=\D_X\eta$, then $\int_{X_s}\omega_s^m=\int_{X_s}\D_{X_s}(\eta|_{X_s}\wedge\omega_s^{m-1})=0$, a contradiction.
\end{remark}

\begin{theorem}\label{thm:potential_rigidity}
In the setting of Lemma \ref{lem:potential_identity}, assume that $S$ is compact and that $q_\phi\geq0$ on $X$. Then every holomorphic section is $H$-horizontal and has image in $Z_\phi:=q_\phi^{-1}(0)$. Conversely, every smooth $H$-horizontal section is holomorphic. In particular, if $q_\phi>0$ everywhere on $X$, then $f$ admits no holomorphic section.
\end{theorem}

\begin{proof}
For a holomorphic section $u$, Lemma \ref{lem:potential_identity} gives $\Delta_{\omega_S}(\phi\comp u)\geq0$. Since $S$ is compact and connected, the maximum principle makes $\phi\comp u$ constant. Hence $\partial u=0$ and $q_\phi\comp u=0$. Together with $\dbar u=0$, the first equality implies that $u$ is $H$-horizontal, while the second gives $u(S)\subset Z_\phi$. Conversely, let $u$ be $H$-horizontal. Then $\dbar u=0$ and $u$ is holomorphic.
\end{proof}

\begin{remark}\label{rmk:vector_section_vanishing}
Let $(E,h)\to(S,\omega_S)$ be a Hermitian holomorphic vector bundle over a compact connected Hermitian manifold, let $F_h$ be its Chern curvature, and put $\phi(v)=|v|_h^2$ on the total space of $E$. By \eqref{eq:vb_omegaH}, $q_\phi(v)=\langle-\I\Lambda_{\omega_S}F_h\,v,v\rangle_h$. Lemma \ref{lem:potential_identity} implies that for every holomorphic section $u$,
\[
\Delta_{\omega_S}|u|_h^2=2|D_h'u|^2+\langle-\I\Lambda_{\omega_S}F_h\,u,u\rangle_h.
\]
This is the Weitzenb\"ock formula of \cite[Prop.~3.1.8]{Kobayashi87}. If $A_h:=-\I\Lambda_{\omega_S}F_h$ is positive semidefinite everywhere, Theorem \ref{thm:potential_rigidity} shows that every holomorphic section is Chern-parallel and satisfies $A_hu=0$ pointwise. If $A_h(s_0)$ is positive definite for some $s_0\in S$, then $u(s_0)=0$, and the parallelism gives $u\equiv0$. Hence $H^0(S,E)=\{0\}$. This recovers the Bochner-type vanishing theorem of \cite[Th.~3.1.9]{Kobayashi87}.
\end{remark}

\begin{theorem}\label{thm:degree_obstruction}
Let $(S,\omega_S)$ be compact Hermitian, and let $\omega_X$ be a real $(1,1)$-form on $X$ which is K\"ahler on every fiber. Assume that $\omega_X$ induces a symplectic connection $H$ with horizontal distribution $\mathcal H$. Set $q_{\omega_X}:=\Lambda_{\omega_S}\omega_X^H$. Then every holomorphic section satisfies
\begin{equation}\label{eq:energy_id1}
\|\partial u\|_{L^2}^2+\int_S(q_{\omega_X}\comp u)\,\omega_S^n=n\,\deg_{\omega_X}(u).
\end{equation}

Assume that $q_{\omega_X}\geq0$. Then $\deg_{\omega_X}(u)\geq0$ for every holomorphic section, and equality holds if and only if $u$ is $\mathcal H$-horizontal and $u(S)\subset q_{\omega_X}^{-1}(0)$. If $q_{\omega_X}\geq\kappa>0$, there exists a holomorphic section only if $\deg_{\omega_X}(u)\geq\frac{\kappa}{n}\int_S\omega_S^n$.
\end{theorem}

\begin{proof}
Since $u$ is holomorphic, $\dbar u=0$. By definition, $n\,\deg_{\omega_X^H}(u)=\int_S(q_{\omega_X}\comp u)\,\omega_S^n$.
\eqref{eq:energy_id1} then follows from \eqref{eq:energy_id}. When $q_{\omega_X}\geq0$, both terms on the left-hand side of \eqref{eq:energy_id1} are nonnegative. Therefore, $\deg_{\omega_X}(u)\ge 0$ and equality holds if and only if $\partial u=0$ and $q_{\omega_X}\comp u=0$. The last statement also follows from \eqref{eq:energy_id1}.
\end{proof}
\begin{remark}
In the setting of the above theorem, assume that $(S,\omega_S)$ is semi-K\"ahler and that $\omega_X=\D_X\eta$ is exact. Then
\[
\deg_{\omega_X}(u)=\int_Su^*\omega_X\wedge\omega_S^{n-1}
=\int_S\D_S\bigl(u^*\eta\wedge\omega_S^{n-1}\bigr)=0.
\]
Let $(L,h)\to (S,\omega_S)$ be a Hermitian holomorphic line bundle over a compact K\"ahler manifold, where $h$ is chosen so that $c_1(L,h)$ is harmonic \cite[Prop.~2.2.23]{Kobayashi87}. Let $\omega:=\I\,\partial\dbar \phi$ be as in Remark \ref{rmk:nonlin_BKN_recovers_linear}, where $\phi(v)=|v|_h^2$. Then by \eqref{eq:vb_omegaH}, \eqref{eq:energy_id1} becomes $\|\partial u\|_{L^2}^2=2n \pi \|u\|_{L^2}^2 \deg L/\int_S \omega_S^n$, where $\deg L:=\int_S c_1(L,h)\wedge\omega_S^{n-1}$. Therefore, if $\deg L<0$, then $H^0(S,L)=\{0\}$; if $\deg L=0$, then every holomorphic section $u$ of $L$ has no zeros unless $u\equiv 0$. This recovers \cite[Th.~3.1.24]{Kobayashi87}.
\end{remark}

\subsection{Schwarz lemma}
Let $(f:X\to S,T_{X/S})$ be a complex fiber bundle with an integrable $\dbar$-operator $\dbar$. Equip $T_{X/S}$ with a smooth Hermitian metric $h_{X/S}$. Let $H:f^*TS\to TX$ be a smooth splitting of the holomorphic tangent sequence, i.e., a pure $(1,0)$-connection. Let $\partial$ be the induced almost connection. Define
\[
\mathcal A_H:=\dbar_{\Hom(f^*TS,TX)}H\in A^{0,1}(X,f^*T^*S\otimes T_{X/S}).
\]
It is the Dolbeault representative determined by $H$ of the extension class of $0\to T_{X/S}\to TX\to f^*TS\to 0$ (see \cite[\S3]{LS26}). For a holomorphic section $u:S\to X$, set $E:=u^*T_{X/S}$, $\mathcal V_u:=T^*S\otimes E$. Then $\alpha:=\partial u=\D u|_{TS}-H|_u\in C^\infty(S,\mathcal V_u)$. Equip $E$ and $\mathcal V_u$ with the Hermitian metrics and Chern connections induced by $(T_{X/S},h_{X/S})$ and $\omega_S$.

\begin{proposition}\label{prop:relative_BK_sections}
Every holomorphic section $u:S\to X$ satisfies $\dbar_{\mathcal V_u}\partial u=-u^*\mathcal A_H$. If $u^*\mathcal A_H=0$, then $\alpha$ is holomorphic and
\begin{equation}\label{eq:Lap_delusq}
\Delta_{\omega_S}|\alpha|^2=2\bigl|\nabla^{\mathcal V_u,1,0}\alpha\bigr|^2+\langle-\I\Lambda_{\omega_S}F_{\mathcal V_u}\alpha,\alpha\rangle.
\end{equation}
Suppose in addition that $S$ is compact. Set $\tau_S:=[\Lambda_{\omega_S},(\partial\omega_S)\wedge]$, $\bar\tau_S:=[\Lambda_{\omega_S},(\dbar\omega_S)\wedge]$, and $\mathscr T_{\mathcal V_u,\omega_S}:=[\nabla^{\mathcal V_u,1,0},\tau_S^*]-[\dbar_{\mathcal V_u},\bar\tau_S^*]$. Then
\begin{equation}\label{eq:hol_sec_Bochner}
\|u^*\mathcal A_H\|_{L^2}^2=\|\nabla^{\mathcal V_u,1,0}\alpha\|_{L^2}^2+\frac12\int_S\left\langle-\I\Lambda_{\omega_S}F_{\mathcal V_u}\alpha,\alpha\right\rangle\omega_S^n+(\mathscr T_{\mathcal V_u,\omega_S}\alpha,\alpha)_{L^2},
\end{equation}
where $F_{\mathcal V_u}=F_{T^*S}\otimes\id_E+\id_{T^*S}\otimes F_E$. If $\{e_i\}_{i=1}^n$ is a local $h_S$-unitary frame, with $H_i:=H(e_i)|_u$ and $\alpha_i:=\alpha(e_i)$, then
\begin{equation}\label{eq:FVu_expand}
\left\langle-\I\Lambda_{\omega_S}F_{\mathcal V_u}\alpha,\alpha\right\rangle=\bigl\langle(-\I\Lambda_{\omega_S}F_{T^*S}\otimes\id_E)\alpha,\alpha\bigr\rangle-2\sum_{i,j=1}^n\bigl\langle F_{T_{X/S}}(H_i+\alpha_i,\widebar{H_i+\alpha_i})\alpha_j,\alpha_j\bigr\rangle.
\end{equation}
\end{proposition}

\begin{proof}
Since $u$ is holomorphic, $\D u|_{TS}:TS\to u^*TX$ is a holomorphic bundle morphism. Hence $\dbar_{\mathcal V_u}\alpha=\dbar(\D u|_{TS}-u^*H)=-u^*(\dbar H)=-u^*\mathcal A_H$. If $u^*\mathcal A_H=0$, the Bochner formula for the holomorphic section $\alpha$ of $\mathcal V_u$ gives $\Delta_{\omega_S}|\alpha|^2=2|\nabla^{\mathcal V_u,1,0}\alpha|^2+\langle-\I\Lambda_{\omega_S}F_{\mathcal V_u}\alpha,\alpha\rangle$.
The Bochner--Kodaira--Nakano identity on $\mathcal V_u$-valued $0$-forms is
\[
\Delta_{\mathcal V_u}''=\Delta_{\mathcal V_u}'+\tfrac12[\I F_{\mathcal V_u},\Lambda_{\omega_S}]+[\nabla^{\mathcal V_u,1,0},\tau_S^*]-[\dbar_{\mathcal V_u},\bar\tau_S^*].
\]
For compact $S$, taking the $L^2$ inner product with $\alpha$ proves the integrated identity. The tensor-product curvature formula gives $F_{\mathcal V_u}=F_{T^*S}\otimes\id_E+\id_{T^*S}\otimes F_E$. Since $E=u^*T_{X/S}$, $F_E(v,\widebar w)=F_{T_{X/S}}(\D u(v),\overline{\D u(w)})$. Note that $\D u(e_i)=H_i+\alpha_i$, and in a unitary coframe $\omega_S=\frac{\I}{2}\sum_i e^i\wedge\widebar{e^i}$ and $\Lambda_{\omega_S}(e^i\wedge\widebar{e^j})=-2\I\delta_{ij}$, the last assertion follows.
\end{proof}

Under the norm-preserving alternation $A^{0,1}(S,T^*S\otimes E)\to A^{1,1}(S,E)$, $\eta\otimes(\beta\otimes v)\mapsto\eta\wedge\beta\otimes v$, the tensor $u^*\mathcal A_H$ corresponds to $u^*F_{D_H}^{1,1}$ when $H$ is relatively holomorphic, where $D_H=\partial+\dbar$.

\begin{corollary}\label{cor:extension_defect_gap}
Assume that $S$ is compact K\"ahler, that $-\I\Lambda_{\omega_S}F_{T^*S}\geq\mu\id_{T^*S}$ for some $\mu>0$, and that $(T_{X/S},h_{X/S})$ is Griffiths semi-negative, namely $\langle F_{T_{X/S}}(\xi,\widebar\xi)w,w\rangle\leq0$ for every $x\in X$, $\xi\in T_xX$, and $w\in T_{X/S,x}$. Then every holomorphic section satisfies
\[
\|\nabla^{\mathcal V_u,1,0}\partial u\|_{L^2}^2+(\mu/2)\|\partial u\|_{L^2}^2\leq\|u^*\mathcal A_H\|_{L^2}^2.
\]
In particular, if $H$ is holomorphic, then every holomorphic section is $H$-horizontal.
\end{corollary}

\begin{proof}
By the assumptions, $-\I\Lambda_{\omega_S}F_{\mathcal V_u}\geq\mu\id_{\mathcal V_u}$ along $u$, so the assertion follows from \eqref{eq:hol_sec_Bochner}.
\end{proof}

We now specialize to flat holomorphic bundles, and give a relative version of Yau's Schwarz lemma \cite[Th.~2]{yau78schwarz}. Let $(Y,\omega_Y)$ be a Hermitian manifold, let $\rho:\pi_1(S)\to\Aut(Y,\omega_Y)$ be a representation by holomorphic isometries, and let $f:X:=\widetilde S\times_\rho Y\to S$ be the associated flat holomorphic bundle with flat holomorphic connection $H$. Let $\operatorname{pr}^{\mathrm v}_H:TX^\RR\to T_{X/S}^\RR$ be the vertical projection of the underlying real flat connection. Let $\omega_{X/S}$ be the fiberwise Hermitian form induced by $\omega_Y$, $h_{X/S}$ be the corresponding fiberwise Hermitian metric, and $\omega_X$ be the real $(1,1)$-form determined by $\omega_{X/S}$ and $H$ as in \eqref{eq:omega_nabla} using $\operatorname{pr}^{\mathrm v}_H$. Then $\omega_X^H=0$.

Let $F_{TY}$ denote the Chern curvature of $(TY,h_Y)$. Since every element of $\rho(\pi_1(S))$ is a holomorphic isometry, $F_{TY}$ descends to an element in $C^\infty(X,\Lambda^{1,1}T_{X/S}^*\otimes\End T_{X/S})$, still denoted by $F_{TY}$. Using the flat horizontal splitting, extend this tensor by zero whenever one of its two form arguments is horizontal. This extension is
the Chern curvature $F_{T_{X/S}}\in A^{1,1}(X,\End T_{X/S})$ of $(T_{X/S},h_{X/S})$.

\begin{lemma}\label{lem:rel_chern_lu}
$\alpha$ is a holomorphic section of $\mathcal V_u$. If $\{e_i\}_{i=1}^n$ is a local unitary frame of $TS$ and $\alpha_i:=\alpha(e_i)$, then
\begin{equation}\label{eq:Lap_delusq_flat}
\Delta_{\omega_S}|\alpha|^2=2|\nabla^{\mathcal V_u,1,0}\alpha|^2+\langle(-\I\Lambda_{\omega_S}F_{T^*S}\otimes\id_E)\alpha,\alpha\rangle-2\sum_{i,j=1}^n\langle F_{TY}(\alpha_i,\widebar{\alpha_i})\alpha_j,\alpha_j\rangle.
\end{equation}
Moreover, $u^*\omega_X$ is semipositive, $\Lambda_{\omega_S}(u^*\omega_X)=|\alpha|^2$, and $\alpha\equiv0$ if and only if $u$ is flat.
\end{lemma}
\begin{proof}
The flat connection $H$ is holomorphic, so Proposition \ref{prop:relative_BK_sections} applies with $\mathcal A_H=0$. In a flat holomorphic trivialization $X|_U\cong U\times Y$, the section has the form $u=(\id_U,\varphi)$ for a holomorphic map $\varphi:U\to Y$, and $\alpha$ identifies with $\partial\varphi$. The Chern curvature $F_{T_{X/S}}$ of $(T_{X/S},h_{X/S})$ vanishes when one of its arguments is horizontal. Therefore, $F_E(v,\widebar w)=F_{TY}\bigl(\alpha(v),\widebar{\alpha(w)}\bigr)$,
and \eqref{eq:Lap_delusq_flat} follows from \eqref{eq:Lap_delusq} and \eqref{eq:FVu_expand}. In the same trivialization, $u^*\omega_X=\varphi^*\omega_Y$, which is semipositive and has trace $|\partial\varphi|^2=|\alpha|^2$. Finally, $\alpha=0$ if and only if $\D\varphi=0$, equivalently if and only if $u$ is flat.
\end{proof}

\begin{definition}\label{def:omori_yau_cond}
We say that $(S,\omega_S)$ satisfies the \emph{Omori--Yau condition} if every $C^2$ function $\psi$ bounded from below admits points $p_j\in S$ ($j\in\ZZ^+$) such that $\lim_{j\to\infty}\psi(p_j)=\inf_S\psi$, $|\partial\psi|(p_j)<j^{-1}$, and $\Delta_{\omega_S}\psi(p_j)>-j^{-1}$.
\end{definition}

\begin{lemma}\label{lem:quadratic_omori_yau}
Let $e\geq0$ be a $C^2$ function satisfying $\Delta_{\omega_S}e\geq-ae+be^2$ for constants $a\in\RR$ and $b>0$. Assume that either $S$ is compact or $(S,\omega_S)$ satisfies the Omori--Yau condition. Then
\[
e\le b^{-1}\max\{a,0\}.
\]
\end{lemma}
\begin{proof}
If $S$ is compact, evaluate the differential inequality at a maximum point of $e$. Suppose that the Omori--Yau condition holds with $p_j\in S$ for $\psi:=(e+1)^{-1/2}$. Then
\[
\Delta_{\omega_S}\psi=-\frac{\Delta_{\omega_S}e}{2(e+1)^{3/2}}+\frac{3|\partial e|^2}{2(e+1)^{5/2}}\leq\frac{ae-be^2}{2(e+1)^{3/2}}+6(e+1)^{1/2}|\partial\psi|^2.
\]
It follows that
\[
be(p_j)^2-ae(p_j)\leq2j^{-1}(e(p_j)+1)^{3/2}+12j^{-2}(e(p_j)+1)^2.
\]
If $\sup_Se=+\infty$, then $\psi(p_j)\to0$ and hence $e(p_j)\to+\infty$. Dividing by $e(p_j)^2$ and passing to the limit gives $b\leq0$, a contradiction. Thus $\sup_Se<+\infty$, and the choice of $p_j$ implies $e(p_j)\to\sup_Se$. Letting $j\to\infty$ in the preceding inequality gives $b(\sup_Se)^2-a\sup_Se\leq0$, which proves the assertion.
\end{proof}

\begin{theorem}\label{thm:relative_yau_schwarz}
Assume that either $S$ is compact or $(S,\omega_S)$ satisfies the Omori--Yau condition. Suppose that
\begin{equation}\label{eq:base_curv_bd}
-\I\Lambda_{\omega_S}F_{T^*S}\geq-\mu\id_{T^*S}
\end{equation}
for some $\mu\geq0$, and that
\begin{equation}\label{eq:fib_curv_bd1}
-2\langle F_{TY}(v,\bar v)w,w\rangle\geq\lambda|v|^2|w|^2
\end{equation}
for some $\lambda>0$ and all $y\in Y$, $v,w\in T_y Y$. Then every holomorphic section $u:S\to X$ satisfies
\[
|\partial u|^2\le \lambda^{-1}\mu,\qquad 0\le u^*\omega_X\le\lambda^{-1}\mu\,\omega_S.
\]
In particular, if $\mu=0$, then every holomorphic section is flat.
\end{theorem}

\begin{proof}
\eqref{eq:Lap_delusq_flat} gives $\Delta_{\omega_S}|\alpha|^2\geq-\mu |\alpha|^2+\lambda |\alpha|^4$. Lemma \ref{lem:quadratic_omori_yau} gives $|\alpha|^2\le \lambda^{-1}\mu$. Since $u^*\omega_X$ is semipositive and has trace $|\alpha|^2$ with respect to $\omega_S$ by Lemma \ref{lem:rel_chern_lu}, each of its eigenvalues is at most $|\alpha|^2$, which gives the second estimate.
\end{proof}

\begin{remark}\label{rmk:complete_kahler_omori_yau}
If $(S,\omega_S)$ is complete K\"ahler and $-\I\Lambda_{\omega_S}F_{T^*S}$ is bounded from below by a constant, then the Omori--Yau condition holds by \cite[Th.~1]{yau78schwarz}.

Take $X=M\times N\to M$ with the product connection and $u=(\id_M,\varphi)$ for a holomorphic map $\varphi:(M,\omega_M)\to(N,\omega_N)$. Then $\alpha=\partial\varphi$, $|\alpha|^2=\Lambda_{\omega_M}(\varphi^*\omega_N)$, and $u^*\omega_X=\varphi^*\omega_N$. Suppose that $M$ is complete K\"ahler and
\[
-\I\Lambda_{\omega_M}F_{T^*M}\geq K_1\id_{T^*M},\qquad 2\left\langle F_{TN}(v,\bar v)w,w\right\rangle\leq K_2|v|^2|w|^2
\]
for a constant $K_2<0$. If $K_1\geq0$, Theorem \ref{thm:relative_yau_schwarz} applies with $\mu=0$ and forces $\varphi$ to be constant. If $K_1<0$, it applies with $\mu=-K_1$ and $\lambda=-K_2$ and gives
\[
\varphi^*\omega_N\leq\frac{K_1}{K_2}\omega_M.
\]
This is Yau's Schwarz lemma \cite[Th.~2]{yau78schwarz}.
\end{remark}

\begin{theorem}\label{thm:relative_yau_royden}
Assume that either $S$ is compact or $(S,\omega_S)$ satisfies the Omori--Yau condition. Suppose that \eqref{eq:base_curv_bd} holds for some $\mu\geq0$, that $(Y,\omega_Y)$ is K\"ahler, and that
\begin{equation}\label{eq:curv_Y_cond}
-\left\langle F_{TY}(v,\bar v)v,v\right\rangle\geq\kappa|v|^4
\end{equation}
for some $\kappa>0$ and every $v\in TY$. Let $u:S\to X$ be holomorphic and set $d:=\max_{s\in S}\rank_\CC(\partial u)_s$. Then $d=0$ if and only if $u$ is flat. If $d>0$, then
\begin{equation}\label{eq:rel_yau_royden}
|\partial u|^2\leq\frac{\mu d}{\kappa(d+1)},\qquad 0\leq u^*\omega_X\leq\frac{\mu d}{\kappa(d+1)}\omega_S.
\end{equation}
If moreover $S$ is compact, then
\[
\|\partial u\|_{L^2}^2\leq\frac{\mu d}{\kappa(d+1)}\int_S\omega_S^n,\qquad 0\leq\deg_{\omega_X}(u)\leq\frac{\mu d}{\kappa n(d+1)}\int_S\omega_S^n.
\]
\end{theorem}
\begin{proof}
Since $u$ is holomorphic, $d=0$ if and only if $u$ is flat. Assume $d>0$ and let $r:=\rank_\CC\alpha$ at a point where $|\alpha|>0$. The algebraic estimate of \cite[Lem.,~p.~552]{royden80schwarz} in the present normalization gives
\[
-2\sum_{i,j=1}^n\left\langle F_{TY}(\alpha_i,\widebar{\alpha_i})\alpha_j,\alpha_j\right\rangle\geq\kappa\frac{r+1}{r}|\alpha|^4\geq\kappa\frac{d+1}{d}|\alpha|^4.
\]
By \eqref{eq:Lap_delusq_flat}, $\Delta_{\omega_S}|\alpha|^2\geq-\mu |\alpha|^2+\kappa\frac{d+1}{d}|\alpha|^4$ wherever $|\alpha|>0$. At a zero of $\alpha$, \eqref{eq:Lap_delusq_flat} gives $\Delta_{\omega_S}|\alpha|^2 \ge 0$, so the same differential inequality holds on all of $S$. Lemma \ref{lem:quadratic_omori_yau} gives the first estimate of \eqref{eq:rel_yau_royden}, and the second estimate follows from the first. Since $u$ is holomorphic and $\omega_X^H=0$, Corollary \ref{cor:energy_id} gives $\|\partial u\|_{L^2}^2=n\deg_{\omega_X}(u)$. The last statement follows by integrating \eqref{eq:rel_yau_royden}.
\end{proof}

\begin{corollary}\label{cor:flat_bun_fixed_points}
Assume that $S$ is compact, that $(Y,\omega_Y)$ is K\"ahler, that $-\I\Lambda_{\omega_S}F_{T^*S}\geq\mu\id_{T^*S}$ for some $\mu\geq 0$, and that \eqref{eq:curv_Y_cond} holds for some $\kappa\geq 0$ and every $v\in TY$. If $\mu+\kappa>0$, then every holomorphic section is flat. Consequently, $\Gamma_{\mathrm{hol}}(S,X)\cong Y^{\rho(\pi_1(S))}$. In particular, $X\to S$ admits a holomorphic section if and only if the monodromy has a common fixed point in $Y$.
\end{corollary}

\begin{proof}
Let $u:S\to X$ be holomorphic. If $\alpha:=\partial u$ is nonzero, then $d:=\max_S\rank_\CC\alpha>0$. Let $p$ be a maximum point of $|\alpha|$. The proof of Theorem \ref{thm:relative_yau_royden} gives
\[
0\ge \Delta_{\omega_S}|\alpha|^2(p)\ge \mu |\alpha|^2(p)+\kappa\frac{d+1}{d}|\alpha|^4(p),
\]
which is impossible. Thus $u$ is flat. A flat section lifts to a constant map $\widetilde S\to Y$, and equivariance says that its value is fixed by $\rho(\pi_1(S))$. Conversely, every common fixed point defines a flat holomorphic section.
\end{proof}

\begin{corollary}\label{cor:sectional_milnor_wood}
Assume that $S$ is a compact Riemann surface of genus $g\geq2$, equipped with its hyperbolic K\"ahler form $\omega_S$ of Gaussian curvature $-1$, and that $(Y,\omega_Y)$ is K\"ahler and satisfies \eqref{eq:curv_Y_cond} for some $\kappa>0$. For every holomorphic section $u:S\to X$,
\[
0\leq\tau_{\omega_Y}(u):=\frac{1}{2\pi}\int_Su^*\omega_X\leq\frac{g-1}{\kappa}.
\]
If instead $u$ is a smooth section whose corresponding $\rho$-equivariant map $\tilde u:\widetilde S\to Y$ is antiholomorphic, then $-\frac{g-1}{\kappa}\leq\tau_{\omega_Y}(u)\leq0$.
\end{corollary}

\begin{proof}
The Hermitian metric on $T^*S$ induced by $\omega_S$ satisfies $-\I\Lambda_{\omega_S}F_{T^*S}=-\id_{T^*S}$. Let first $u$ be holomorphic. If $u$ is flat, the assertion is immediate. Otherwise $\max_S\rank_\CC(\partial u)=1$, and Theorem \ref{thm:relative_yau_royden}, applied with $\mu=d=n=1$, gives
\[
\deg_{\omega_X}(u)\leq\frac{1}{2\kappa}\int_S\omega_S=\frac{2\pi(g-1)}{\kappa},
\]
where the last equality follows from the Gauss--Bonnet formula. The first assertion follows. When $\tilde u$ is antiholomorphic, we regard it as a holomorphic map to the conjugate K\"ahler manifold $(\overline Y,-\omega_Y)$. The same curvature hypothesis holds, and applying the first assertion to the holomorphic section $u:S\to\widebar{X}:=\widetilde S\times_\rho\overline Y$ gives the reversed inequalities.
\end{proof}

\begin{remark}\label{rmk:sectional_toledo}
If $Y$ is contractible, then $X\to S$ is a homotopy equivalence and any two smooth sections are homotopic. Since $\omega_Y$ is closed, so is $\omega_X$. Lemma \ref{lem:top_inv} shows that $\tau_{\omega_Y}(u)$ is independent of $u$, which is then denoted by $\tau_{\omega_Y}(\rho)$. Suppose in addition that $Y=G/K$ is a Hermitian symmetric space of noncompact type, where $G$ is a connected semisimple real Lie group of Hermitian type with finite center and no compact factors and $K<G$ is a maximal compact subgroup, that $\omega_Y$ is the sum of the invariant K\"ahler forms on the irreducible factors of $Y$, each normalized to have minimal holomorphic sectional curvature $-1$, and that $\rho$ takes values in $G$. Then $\tau_{\omega_Y}(\rho)$ is the Toledo invariant \cite[Prop.~4.2]{BiquardGarciaPradaRubio2017}. If $r=\rank(Y)$, the $1/r$-pinching of the holomorphic sectional curvature \cite[Th.,~p.~33]{OgiueTakagi1981} allows one to take $\kappa=1/(2r)$ in \eqref{eq:curv_Y_cond}. Corollary \ref{cor:sectional_milnor_wood} therefore gives the Milnor--Wood inequality (\cite[Th.~4.8]{BiquardGarciaPradaRubio2017})
\[
|\tau_{\omega_Y}(\rho)|\leq r(2g-2),
\]
whenever $\rho$ admits a holomorphic or antiholomorphic equivariant map.
\end{remark}

\section{Nonlinear Hodge correspondence for sections}\label{sec:kahler_id}
For the rest of the paper, we develop the nonlinear Hodge correspondence first for sections, then for sub-fibrations and morphisms.  The nonlinear Bochner--Kodaira--Nakano identity from the preceding section provides the key input, but does not by itself suffice for our purposes. Incorporating Higgs fields leads to two new BKN-type identities, from which we derive a correspondence between Higgs sections and flat sections, generalizing the classical vector-bundle result \cite[Lem.~1.2]{Si1}.
\subsection{The $D'$-$D''$ identity}
Let $f:X\to S$, $\dbar$, $\omega_{X/S}$, and $\nabla^{1,0}$ be as in the previous section, and let $\theta\in A^{1,0}(S,f_*T_{X/S})$ be a relatively holomorphic almost Higgs field. Assume that the fiberwise K\"ahler metric $\omega_{X/S}$ is $\theta$-adapted (\cite[Def.~5.5]{LS26}), so that $\bar{\theta}_{\omega_{X/S}}\in A^{0,1}(S,\mathfrak{k}_{X/S}^\CC)$ is defined (\cite[Def.~5.6]{LS26}), where $\mathfrak k_{X/S}$ is a smooth real vector bundle whose fiber
\[\mathfrak k_s\subset\mathfrak{aut}(X_s,\omega_s):=\{v\in H^0(X_s,TX_s)\,|\, v_\RR \text{ is Killing}\}\]
satisfies $\mathfrak k_s\cap\I\mathfrak k_s=\{0\}$, and $\mathfrak{k}_{X/S}^\CC:=\mathfrak{k}_{X/S}\oplus\I \mathfrak{k}_{X/S}$.

At the level of horizontal lifts, we define
\[
D'':=\dbar+\theta,\qquad D':=\partial+\bar{\theta}_{\omega_{X/S}},\qquad D:=D'+D''.
\]
On sections we use the induced action (due to Lemma \ref{lem:dbar_section_affine})
\[
D''u=\dbar u-\theta u,\qquad D'u=\partial u-\bar\theta_{\omega_{X/S}}u,\qquad
Du=D'u+D''u.
\]
A section $u$ is called a \emph{Higgs section} if $D''u=0$, and a \emph{flat section} if $Du=0$. By the type orthogonality, we have
\begin{equation}\label{eq:type_orthog}
|D'u|^2-|D''u|^2 = \bigl(|\partial u|^2-|\dbar u|^2\bigr)+\bigl(|\bar{\theta}_{\omega_{X/S}} u|^2-|\theta u|^2\bigr),
\end{equation}
where all norms are computed using $\omega_S$ on $S$ and $h_E=u^*h_{X/S}$ on $E$.

We now compute the remaining term $|\bar{\theta}_{\omega_{X/S}} u|^2 - |\theta u|^2$ in \eqref{eq:type_orthog}. Throughout, we abbreviate $\bar{\theta}:=\bar{\theta}_{\omega_{X/S}}$. For each $s\in S$ and $v\in T_sS$, write
\begin{equation}\label{eq:XY_decomp}
\theta(v) = X_v + \I Y_v \in \mathfrak{k}_s^\CC, \qquad X_v, Y_v \in \mathfrak{k}_s.
\end{equation}
Then
\begin{equation}\label{eq:thetabar_XY}
\bar{\theta}(\bar{v}) = -X_v + \I Y_v.
\end{equation}
Each element of $\mathfrak{k}_s^\CC\subset H^0(X_s, TX_s)$ is a holomorphic vector field on $X_s$.  In adapted holomorphic coordinates $(z^1,\ldots,z^m)$ on $X_s$ near $u(s)$, we may write
\[
\theta(v) = (X_v^\alpha + \I Y_v^\alpha)\,\partial_\alpha, \qquad \bar{\theta}(\bar{v}) = (-X_v^\alpha + \I Y_v^\alpha)\,\partial_\alpha,
\]
where $X_v^\alpha, Y_v^\alpha$ are holomorphic functions of $z$.

\begin{lemma}\label{lem:higgs_norm_identity}
Let $f:X\to S$ be a complex fiber bundle equipped with a $\theta$-adapted fiberwise K\"ahler metric $\omega_{X/S}$, and let $u:S\to X$ be a smooth section. For every $s\in S$ and $v\in T_sS$,
\begin{equation}\label{eq:thetabar_theta_symplectic}
|\bar{\theta}u(\bar{v})|^2_{h_E} - |\theta u(v)|^2_{h_E} = 4\,\omega_{X/S}\bigl(X_v^\RR,Y_v^\RR\bigr)\big|_{u(s)},
\end{equation}
where $X_v^\RR:=X_v+\widebar{X_v}$ and $Y_v^\RR:=Y_v+\widebar{Y_v}$ are the real vector fields corresponding to $X_v$ and $Y_v$.
\end{lemma}

\begin{proof}
Fix $s_0\in S$ and set $x_0 = u(s_0)$.  Choose holomorphic coordinates $(z^\alpha)$ on $X_{s_0}$ near $x_0$ with $g_{\alpha\bar{\beta}}(x_0) = \frac{1}{2}\delta_{\alpha\beta}$. At $x_0$,
\begin{align*}
	|\bar{\theta}u(\bar{v})|^2_{h_E} - |\theta u(v)|^2_{h_E} &= \sum_\alpha (|-X_v^\alpha + \I Y_v^\alpha|^2-|X_v^\alpha + \I Y_v^\alpha|^2)\\
	&=-4\sum_\alpha \Im\bigl(X_v^\alpha\,\overline{Y_v^\alpha}\bigr)\big|_{x_0} = -4\Im\langle X_v, Y_v\rangle_{h_{X/S}}\big|_{x_0}\\
	&=4\,\omega_{X/S}\bigl(X_v^\RR,Y_v^\RR\bigr)\big|_{x_0}.
\end{align*}
Both sides are invariantly defined, so the identity holds independently of the coordinates.
\end{proof}
\begin{assumption}\label{assum:comoment}
\begin{enumerate1}
\item For each $s\in S$, $\mathfrak{k}_s\subset \mathrm{Ham}^{1,0}(X_s)$, where $\mathrm{Ham}(X_s)$ denotes the Lie algebra of real Hamiltonian vector fields on $X_s$ and $\mathrm{Ham}^{1,0}(X_s)$ consists of $V^{1,0}=\frac{1}{2}(V-\I J_s V)$ for $V\in \mathrm{Ham}(X_s)$. Let $\mathrm{Ham}_{X/S}$ be the smooth bundle with fibers $\mathrm{Ham}(X_s)$.
\item There is a smooth \emph{fiberwise equivariant comoment map}
\[\mu^*: C^\infty(S,\mathrm{Ham}_{X/S}) \longrightarrow C^\infty(X,\RR),\]
which satisfies $\D_{X_s}\mu^*(\xi)=-\iota_\xi\,\omega_s$ for any $\xi \in \mathrm{Ham}(X_s)$ and is linear in $\xi$. The equivariance means \[\mu^*([\xi,\eta])=\omega_s(\xi,\eta),\]
for any $s\in S$ and $\xi,\eta\in \mathrm{Ham}(X_s)$. When $\mathfrak{k}_s\subset \mathrm{Ham}^{1,0}(X_s)$ is closed under the Lie bracket on $X_s$, one may replace $\mathrm{Ham}_{X/S}$ by $\mathfrak{k}_{X/S}^{\mathrm{R}} := \{ \xi^\RR \mid \xi \in \mathfrak{k}_{X/S} \}$ in the definition of $\mu^*$.
\item  $\mu^*$ is parallel with respect to $\nabla^\RR$, i.e., for every local vector field $v$ on $S$ and every local section $\xi$ of $\mathrm{Ham}_{X/S}$ (resp.\ of $\mathfrak{k}_{X/S}^{\mathrm{R}}$, if $\mu^*$ is defined on $\mathfrak{k}_{X/S}^{\mathrm{R}}$),
\begin{equation}\label{eq:comoment_parallel}
H_v\bigl(\mu^*(\xi)\bigr)=\mu^*\bigl([H_v,\xi]\bigr),
\end{equation}
where $H_v$ is the horizontal lift of $v$ with respect to $\nabla^\RR$. Note that $[H_v,\xi]$ is vertical, and that its
restriction to $X_s$ depends only on $v|_s$. When $\mu^*$ is defined on $\mathfrak{k}_{X/S}^{\mathrm{R}}$, we need to assume that
$\nabla^\RR$ preserves $\mathfrak{k}_{X/S}^{\mathrm{R}}$, i.e.,
\begin{equation}\label{eq:k_parallel}
[H_v,\xi]\in C^\infty\bigl(U,\mathfrak{k}_{X/S}^{\mathrm{R}}\bigr)
\qquad\text{for all }v\in C^\infty(U,TS^\RR),\ \xi\in C^\infty\bigl(U,\mathfrak{k}_{X/S}^{\mathrm{R}}\bigr),\ U\subset S\text{ open}.
\end{equation}
\item $\omega_X$ is closed. In particular, $\mu^*F_{\nabla^\RR}$ is well-defined by \cite[Eq.~(1.12)]{GLS96} (Lemma \ref{lem:minimal-coupling} below) if $\mu^*$ in (2) is defined on $\mathrm{Ham}_{X/S}$. If $\mu^*$ in (2) is defined on $\mathfrak{k}_{X/S}^{\mathrm{R}}$, we further assume that $F_{\nabla^\RR}$ takes values in $\mathfrak{k}_{X/S}^{\mathrm{R}}$ so that $\mu^*F_{\nabla^\RR}$ is well-defined.
\end{enumerate1}
\end{assumption}
\begin{remark}\label{rmk:moment_existence}
When $f$ is proper, Assumption \ref{assum:comoment}\,(2) and \eqref{eq:comoment_parallel} in (3) are automatic. For each $s$ one may define $\mu^*_s$ by sending a Hamiltonian vector field $\xi$ on $X_s$ to its unique mean-zero Hamiltonian potential $f_\xi$ with respect to $\omega_s$, i.e., $\int_{X_s} f_\xi\,\omega_s^m=0$. As $s$ varies, these assemble into a smooth map $\mu^*$. By Cartan's formula, $\omega_s(\xi,\eta)\omega_s^m$ is exact and $\D_{X_s}(\omega_s(\xi, \eta)) = -\iota_{[\xi, \eta]} \omega_s$, so the equivariance follows. Fix $s\in S$ and $v\in T_{s}S^\RR$. Pick a smooth curve $\gamma:(-\varepsilon,\varepsilon)\to S$ with $\gamma(0)=s$ and $\dot\gamma(0)=v$, and let $\Phi_t:X_{s}\to X_{\gamma(t)}$ denote the parallel transport diffeomorphism. Since $\nabla^\RR$ is a symplectic connection, $\Phi_t^*\omega_{\gamma(t)}=\omega_{s}$. Let $\xi\in C^\infty(S,\mathrm{Ham}_{X/S})$, and write $\xi_t:=\xi|_{X_{\gamma(t)}}$. By the definition of $\mu^*$, $\Phi_t^*\!\bigl(\mu^*(\xi)\big|_{X_{\gamma(t)}}\bigr)=\mu^*\!\bigl(\Phi_t^*\xi_t\bigr)$. Differentiating at $t=0$ yields \eqref{eq:comoment_parallel}.

All of the above assumptions hold for an associated bundle $f:X=P\times_G Y\to S$ as in \cite[Prop.~4.24]{LS26}, with the following normalization of the moment map.
Let $\tilde\tau_0:\mathfrak k\to\mathrm{Ham}(Y)$ be the real infinitesimal action of \cite[Eq.~(2.29)]{LS26} and set $\mathfrak n:=\ker\tilde\tau_0$.
Since $Y$ is connected, $\mu_Y(y)|_{\mathfrak n}=c$ for a fixed $c\in\mathfrak n^*$, which is $K$-invariant by the equivariance of $\mu_Y$.
Choose $\lambda\in(\mathfrak k^*)^K$ with $\lambda|_{\mathfrak n}=c$. Such a $\lambda$ exists by averaging any linear extension of $c$ over $K$.
Replacing $\mu_Y$ by $\mu_Y-\lambda$, we arrange that $\mu_Y(y)|_{\mathfrak n}=0$ for every $y\in Y$, and use this normalized moment map to construct $\omega_X$.
We take $\mathfrak{k}_{X/S}^{\mathrm R}=P_K\times_K\tilde\tau_0(\mathfrak k)$.
For a section $\xi\in C^\infty(S,\mathfrak{k}_{X/S}^{\mathrm R})$, choose a $K$-equivariant smooth lift $\hat\xi:P_K\to\mathfrak k$. Define
\begin{equation}\label{eq:assoc_comoment}
\mu^*(\xi)\bigl([p,y]\bigr):=\bigl\langle\mu_Y(y),\hat\xi(p)\bigr\rangle,
\qquad p\in P_{K,s},\;y\in Y.
\end{equation}
This is independent of the lift because $\mu_Y$ annihilates $\ker\tilde\tau_0$, and descends by $K$-equivariance. The resulting $\mu^*$ is a parallel fiberwise equivariant comoment map, and the construction of $\omega_X$ in \cite[Prop.~4.24]{LS26} gives $\zeta=0$ in \eqref{eq:min-coupling}.
\end{remark}
Let $\Phi$ be a $(1,1)$-form on $S$ with values in vertical $(1,0)$-vector fields. We denote by $[\Phi]_{\RR}$ the unique real $(1,1)$-form with values in vertical real vector fields whose $(1,0)$-component is $\Phi$. Explicitly,
\begin{equation}\label{eq:realification_11}
[\Phi]_{\RR}(\partial_i,\partial_{\bar j})=\Phi(\partial_i,\partial_{\bar j})-\overline{\Phi(\partial_j,\partial_{\bar i})},
\end{equation}
where we extended $[\Phi]_{\RR}$ by complex linearity. For real tangent vectors, we have
\begin{align*}
	[\Phi]_\RR(\partial_{x^i}, \partial_{y^j}) &= [\Phi]_\RR(\partial_i + \partial_{\bar i}, \I\partial_j - \I\partial_{\bar j})= -\I \bigl( [\Phi]_\RR(\partial_i, \partial_{\bar j}) + [\Phi]_\RR(\partial_j, \partial_{\bar i}) \bigr)\\
	&= -\I \bigl( \bigl( \Phi(\partial_i, \partial_{\bar j}) - \overline{\Phi(\partial_i, \partial_{\bar j})} \bigr) + \bigl( \Phi(\partial_j, \partial_{\bar i}) - \overline{\Phi(\partial_j, \partial_{\bar i})} \bigr) \bigr),\\
[\Phi]_\RR(\partial_{x^i}, \partial_{x^j}) &=[\Phi]_\RR(\partial_{y^i}, \partial_{y^j})= [\Phi]_\RR(\partial_i, \partial_{\bar j}) - [\Phi]_\RR(\partial_j, \partial_{\bar i})\\&= \bigl( \Phi(\partial_i, \partial_{\bar j}) + \overline{\Phi(\partial_i, \partial_{\bar j})} \bigr) - \bigl( \Phi(\partial_j, \partial_{\bar i}) + \overline{\Phi(\partial_j, \partial_{\bar i})} \bigr),
\end{align*}
where $s^i=x^i+\I y^i$.
\begin{corollary}
	Under Assumption \ref{assum:comoment} (1) and (2), every smooth section $u:S\to X$ satisfies
	\begin{equation}\label{eq:diff_theta_fin}
		|\bar{\theta} u|^2-|\theta u|^2=\Lambda_{\omega_S}u^*\mu^*\bigl[[\theta,\bar{\theta}]\bigr]_\RR.
	\end{equation}
\end{corollary}
\begin{proof}
	\eqref{eq:thetabar_theta_symplectic} becomes
 \begin{equation}\label{eq:thetabar_theta_symplectic_comoment}
 |\bar{\theta}u(\bar{v})|^2_{h_E} - |\theta u(v)|^2_{h_E} = 4\mu_s^*\bigl([X_v^\RR,Y_v^\RR]\bigr)\big|_{u(s)}.
 \end{equation}
 On the other hand, we have
 \begin{equation}\label{eq:theta_bracket}
	 [\theta(v),\bar{\theta}(\bar{v})]=[X_v+\I Y_v,-X_v+\I Y_v]=2\I[X_v,Y_v].
 \end{equation}
 $[X_v,Y_v]\in \mathrm{Ham}^{1,0}(X_s)$, and the corresponding real Hamiltonian vector field is $[X_v,Y_v]^\RR=[X_v^\RR,Y_v^\RR]$. Fix $s_0\in S$ and choose local holomorphic coordinates for which
$h_{i\bar j}(s_0)=\frac12\delta_{ij}$. By \eqref{eq:thetabar_theta_symplectic_comoment},
 \begin{align*}
	 (|\bar{\theta} u|^2-|\theta u|^2)(s_0)&=4\sum_{i=1}^n\mu_{s_0}^*([X_i^\RR,Y_i^\RR])\big|_{u(s_0)}\\
	 &=\frac{2}{\I}\sum_{i=1}^n\mu_{s_0}^*([\theta(\partial_i),\bar{\theta}(\partial_{\bar{i}})]-\overline{[\theta(\partial_i),\bar{\theta}(\partial_{\bar{i}})]})\big|_{u(s_0)}\\
	 &=\Lambda_{\omega_S}u^*\mu^*\bigl[[\theta,\bar{\theta}]\bigr]_\RR\big|_{s_0}.\qedhere
 \end{align*}
\end{proof}

\begin{lemma}[minimal coupling, {\cite[Eq.~(1.12)]{GLS96}}]\label{lem:minimal-coupling}
Suppose $\omega_X$ is closed (Assumption \ref{assum:comoment}\,(4)). For every $s\in S$ and every
$v_1,v_2\in T_sS^{\RR}$,
\begin{equation}\label{eq:min-coupling0}
\D_{X_s}\omega_X^H(v_1,v_2)=\iota_{F_{\nabla^\RR}(v_1,v_2)}\,\omega_s.
\end{equation}
\end{lemma}
\begin{corollary}
Under Assumption \ref{assum:comoment}\,(2) and (4), there is a unique real $2$-form $\zeta\in A^2(S,\RR)$ on $S$ such that
\begin{equation}\label{eq:min-coupling}
\omega_X^H=-\mu^*F_{\nabla^\RR}+f^*\zeta.
\end{equation}
If Assumption \ref{assum:comoment}\,(3) also holds, then $\zeta$ is
closed.
\end{corollary}
\begin{proof}
\eqref{eq:min-coupling} follows because $\mu^*(F_{\nabla^\RR}(v_1,v_2))$ and $-\omega_X^H(v_1,v_2)$ are both Hamiltonian potentials on $X_s$ for the same vector field, so their difference is locally constant on $X_s$, hence constant by connectedness. Smoothness of $\zeta$ is clear.

Suppose Assumption \ref{assum:comoment}\,(3) holds. Let $\Omega_{\nabla,\mu}\in A^2(X,\RR)$ be the form defined by
\[
\Omega_{\nabla,\mu}|_{T_{X/S}^\RR}=\omega_{X/S},\qquad
\Omega_{\nabla,\mu}(H_v,V)=0,\qquad
\Omega_{\nabla,\mu}(H_v,H_w)
=-\mu^*(F_{\nabla^\RR}(v,w)).
\]
Following the proof of \cite[Th.~1.4.1]{GLS96}, we show that $\Omega_{\nabla,\mu}$ is closed. The components of $\D\Omega_{\nabla,\mu}$ involving vertical vectors vanish by the symplecticity of $\nabla^\RR$, the identity
$\D_{X_s}\mu^*(\xi)=-\iota_\xi\omega_s$, and \eqref{eq:min-coupling0}. For $v_i\in T_s S$ ($i=1,2,3$), extend them to local vector fields with $[v_i,v_j]=0$. Assumption \ref{assum:comoment}\,(3) and the Bianchi identity for the Ehresmann connection give
\[
\D\Omega_{\nabla,\mu}(H_{v_1},H_{v_2},H_{v_3})=-\mu^*\Bigl(\sum_{\mathrm{cyc}}[H_{v_1},F_{\nabla^\RR}(v_2,v_3)]\Bigr)=0.
\]
Thus $\D\Omega_{\nabla,\mu}=0$. By
\eqref{eq:min-coupling}, $\omega_X-\Omega_{\nabla,\mu}=f^*\zeta$. Since $\omega_X$ and $\Omega_{\nabla,\mu}$ are closed,
$f^*(\D\zeta)=0$. Since $f$ is a surjective submersion, $\D\zeta=0$.
\end{proof}
\begin{remark}\label{rmk:zeta}
When the fibers are compact and one chooses $\mu^*$ to be mean-zero, we have
\begin{equation}\label{eq:zeta}
\zeta=V^{-1}f_*(\omega_X^H\wedge\omega_X^m)=((m+1)V)^{-1}f_*(\omega_X^{m+1}),
\end{equation}
as in \cite[p.~2147]{dervan_sektnan2021optimal}, where $f_*$ denotes the fiber integration and $V:=f_*(\omega_X^m)=\int_{X_s}\omega_s^m$, which is independent of $s$ since $\omega_X$ is closed and $S$ is connected.
\end{remark}
\begin{lemma}\label{lem:same_vert_conn}
Let $f:X\to S$ be a smooth fiber bundle with connected fibers, and let $\omega_0,\omega_1\in A^2(X,\RR)$ be closed real $2$-forms such that their restrictions to $T_{X/S}^{\RR}$ coincide and are nondegenerate. Suppose that they induce the same symplectic connection. Then there exists a unique closed real $2$-form $\eta\in A^2(S,\RR)$ such that $\omega_1-\omega_0=f^*\eta$. If, in addition, $f$ is holomorphic and $\omega_0,\omega_1$ are of type $(1,1)$, then $\eta$ is also of type $(1,1)$.

In particular, under Assumption \ref{assum:comoment}\,(2) and (4), $u^*\omega_X-\zeta$ is independent of the choice of $\omega_X$ which restricts to $\omega_{X/S}$ on fibers and induces $\nabla^{1,0}$.
\end{lemma}
\begin{proof}
This is essentially \cite[Th.~1.4.3]{GLS96}. Set $\delta:=\omega_1-\omega_0$. Let $V$ be a vertical tangent vector and write an arbitrary tangent vector $W$ as $W=W^H+W^{\mathrm v}$ with respect to the common horizontal distribution. Since the two forms have the same vertical restriction, $\delta(V,W^{\mathrm v})=0$. Moreover, $W^H$ is symplectically orthogonal to the vertical tangent bundle with respect to both $\omega_0$ and $\omega_1$. Thus $\delta(V,W^H)=0$. It follows that $\iota_V\delta=0$ for every vertical vector field $V$. Since $\delta$ is closed, Cartan's formula gives
\[
\mathcal L_V\delta=\iota_V\D\delta+\D(\iota_V\delta)=0.
\]
Thus $\delta$ is basic. Since the fibers of $f$ are connected, there is a unique form $\eta\in A^2(S,\RR)$ satisfying $\delta=f^*\eta$. $\D\eta=0$, since $\D\delta=0$ and $f$ is a surjective submersion. If $f$ is holomorphic and $\delta$ is of type $(1,1)$, then $0=\delta^{2,0}=f^*(\eta^{2,0})$, and $\eta^{2,0}=0$. Since $\eta$ is real, $\eta^{0,2}=0$. Hence $\eta$ is of type $(1,1)$.

In particular, choosing a different $\omega_X$ amounts to replacing $\omega_X$ by $\omega_X+f^*\eta$, and then $\zeta$ is replaced by $\zeta+\eta$, so $u^*(\omega_X+f^*\eta)-(\zeta+\eta)=u^*\omega_X-\zeta$.
\end{proof}

From now on, we assume in addition that the symplectic connection $\nabla^{1,0}$ is relatively holomorphic and $\dbar$ satisfies the lifting condition. Equivalently, $\nabla^{1,0}$ is a K\"ahler connection on $(f,\omega_{X/S})$ (\cite[Def.~4.5]{LS26}). In particular, the $(1,1)$-curvature $F_{D_\omega}^{1,1}$ of $D_\omega=\partial+\dbar$ is well-defined.

In adapted local coordinates $(s^i,z^\alpha)$, write $\dbar(\partial_{\bar j})=[\partial_{\bar j}+\Gamma_{\bar j}^\beta\partial_\beta]\bmod\widebar{T_{X/S}}$. Since $\nabla^{1,0}$ is pure, the horizontal lifts of $\partial_i,\partial_{\bar j}$ are
\[
H_i=\partial_i+\Gamma_i^\alpha\partial_\alpha+\overline{\Gamma_{\bar i}^\beta}\,\partial_{\bar\beta},\qquad H_{\bar j}=\partial_{\bar j}+\Gamma_{\bar j}^\beta\partial_\beta+\overline{\Gamma_j^\alpha}\,\partial_{\bar\alpha}.
\]
Using $\partial_{\bar\beta}\Gamma_i^\alpha=0$ and $\partial_{\bar\beta}\Gamma_{\bar j}^\gamma=0$, we have
\[
[H_i,H_{\bar j}]^{\mathrm{vert}}=\bigl(\partial_i\Gamma_{\bar j}^\alpha-\partial_{\bar j}\Gamma_i^\alpha
+\Gamma_i^\beta\partial_\beta\Gamma_{\bar j}^\alpha
-\Gamma_{\bar j}^\beta\partial_\beta\Gamma_i^\alpha\bigr)\partial_\alpha
-\overline{\bigl(\partial_j\Gamma_{\bar i}^\alpha-\partial_{\bar i}\Gamma_j^\alpha
+\Gamma_j^\beta\partial_\beta\Gamma_{\bar i}^\alpha
-\Gamma_{\bar i}^\beta\partial_\beta\Gamma_j^\alpha\bigr)}\,\partial_{\bar\alpha}.
\]
By \cite[Eq.~(2.19)]{LS26}, the first summand equals
$F_{D_\omega}^{1,1}(\partial_i,\partial_{\bar j})$ and the second equals
$-\overline{F_{D_\omega}^{1,1}(\partial_j,\partial_{\bar i})}$, so
\begin{equation}\label{eq:symp_curv}
F_{\nabla^\RR}(\partial_i,\partial_{\bar j})=F_{D_\omega}^{1,1}(\partial_i,\partial_{\bar j})-\overline{F_{D_\omega}^{1,1}(\partial_j,\partial_{\bar i})}=\bigl[F_{D_\omega}^{1,1}\bigr]_\RR(\partial_i,\partial_{\bar j}).
\end{equation}
Therefore,
\begin{equation}\label{eq:symp_curv_11_curv}
 F_{\nabla^\RR}^{1,1}=\bigl[F_{D_\omega}^{1,1}\bigr]_\RR.
\end{equation}

\begin{lemma}\label{lem:curv11_decomp}
The almost connection $\partial+\theta$ and the $\dbar$-operator $\dbar+\bar\theta$ satisfy the lifting condition, so that the $(1,1)$-curvature of $D$ is well-defined, and
\begin{equation}\label{eq:curv11_decomp}
F_{D}^{1,1}=F_{D_\omega}^{1,1}+G_{D'}^{1,1}+G_{D''}^{1,1}+[\theta,\bar\theta],
\end{equation}
where $G^{1,1}_{D''}(v,\bar w):=\operatorname{pr}_{T_{X/S}}\bigl([\theta(v),\nabla^{0,1}(\bar w)]-\theta([v,\bar w]^{1,0})\bigr)$ \textup(see \cite[Eq.~(5.1)]{LS26}\textup), and
\[G^{1,1}_{D'}(v,\bar w):=\operatorname{pr}_{T_{X/S}}\bigl([\nabla^{1,0}(v),\bar\theta(\bar w)]-\bar\theta([v,\bar w]^{0,1})\bigr),\quad [\theta,\bar\theta](v,\bar w):=[\theta(v),\bar\theta(\bar w)].\]
\end{lemma}
\begin{proof}
Since $\theta$ and $\bar\theta$ are relatively holomorphic, the lifting conditions for $\partial+\theta$ and $\dbar+\bar\theta$ are satisfied, and $F^{1,1}_{D}$ is well-defined. In adapted local coordinates $(s^i,z^\alpha)$, the coefficients of $D$ are $\widetilde\Gamma_i^\alpha=\Gamma_i^\alpha+\theta_i^\alpha$ and
$\widetilde\Gamma_{\bar j}^\alpha=\Gamma_{\bar j}^\alpha+\bar\theta_{\bar j}^\alpha$. By the local formula \cite[Eq.~(5.2)]{LS26} for $G_{D''}^{1,1}$ ($G_{D'}^{1,1}$ is analogous),
\begin{align}
G_{D''}^{1,1}(\partial_i,\partial_{\bar j})	&=\bigl(-\partial_{\bar j}\theta_i^\beta+\theta_i^\gamma\partial_\gamma\Gamma_{\bar j}^\beta-\Gamma_{\bar j}^\gamma\partial_\gamma\theta_i^\beta\bigr)\partial_\beta, \label{eq:GD''_local}\\
G_{D'}^{1,1}(\partial_i,\partial_{\bar j})&=\bigl(\partial_i\bar\theta_{\bar j}^\beta+\Gamma_i^\gamma\partial_\gamma\bar\theta_{\bar j}^\beta-\bar\theta_{\bar j}^\gamma\partial_\gamma\Gamma_i^\beta\bigr)\,\partial_\beta.\label{eq:GD'_local}
\end{align}
By \cite[Eq.~(2.19)]{LS26},
\begin{align*}
F_{D}^{1,1}(\partial_i,\partial_{\bar j})&=\bigl(\partial_i\widetilde\Gamma_{\bar j}^\alpha-\partial_{\bar j}\widetilde\Gamma_i^\alpha+\widetilde\Gamma_i^\beta\partial_\beta\widetilde\Gamma_{\bar j}^\alpha-\widetilde\Gamma_{\bar j}^\beta\partial_\beta\widetilde\Gamma_i^\alpha\bigr)\partial_\alpha\\
&=\bigl(\partial_i\Gamma_{\bar j}^\alpha-\partial_{\bar j}\Gamma_i^\alpha+\Gamma_i^\beta\partial_\beta\Gamma_{\bar j}^\alpha-\Gamma_{\bar j}^\beta\partial_\beta\Gamma_i^\alpha\bigr)\partial_\alpha+\bigl(-\partial_{\bar j}\theta_i^\alpha+\theta_i^\gamma\partial_\gamma\Gamma_{\bar j}^\alpha-\Gamma_{\bar j}^\gamma\partial_\gamma\theta_i^\alpha\bigr)\partial_\alpha\\
&\quad+\bigl(\partial_i\bar\theta_{\bar j}^\alpha+\Gamma_i^\beta\partial_\beta\bar\theta_{\bar j}^\alpha-\bar\theta_{\bar j}^\beta\partial_\beta\Gamma_i^\alpha\bigr)\partial_\alpha+\bigl(\theta_i^\beta\partial_\beta\bar\theta_{\bar j}^\alpha-\bar\theta_{\bar j}^\beta\partial_\beta\theta_i^\alpha\bigr)\partial_\alpha\\
&=F_{D_\omega}^{1,1}(\partial_i,\partial_{\bar j})+G_{D''}^{1,1}(\partial_i,\partial_{\bar j})+G_{D'}^{1,1}(\partial_i,\partial_{\bar j})+[\theta_i,\bar\theta_{\bar j}].\qedhere
\end{align*}
\end{proof}

\begin{lemma}\label{lem:pseudo_curv_conjugate}
Assume that $\nabla^\RR$ preserves $\mathfrak{k}_{X/S}^{\mathrm{R}}$ in the sense
of \eqref{eq:k_parallel}. Then $G_{D''}^{1,1}$ and $G_{D'}^{1,1}$ take values in $\mathfrak{k}_{X/S}^\CC$, and in any local holomorphic coordinates on $S$,
\begin{equation}\label{eq:G_conjugate}
G_{D'}^{1,1}(\partial_i,\partial_{\bar j})=\phi\bigl(G_{D''}^{1,1}(\partial_j,\partial_{\bar i})\bigr),
\end{equation}
where $\phi$ is the involution of $\mathfrak{k}_{X/S}^\CC$ with respect to the real form $\mathfrak{k}_{X/S}$. In particular, $G_{D''}^{1,1}=0$ if and only if $G_{D'}^{1,1}=0$.
\end{lemma}
\begin{proof}
Let $\rho:C^\infty(U,\mathfrak{k}_{X/S}\otimes_\RR\CC)\to C^\infty(f^{-1}(U),T_{X/S}^\CC)$ be the complex-linear extension of $c\mapsto c^\RR=c+\bar c$, i.e., $\rho(c+\I d)=c^\RR+\I d^\RR$ for $c,d\in C^\infty(U,\mathfrak{k}_{X/S})$, which is injective. We identify $\mathfrak{k}_{X/S}\otimes_\RR\CC\cong\mathfrak{k}_{X/S}^\CC$, under which $\phi(c+\I d)=c-\I d$. For $\xi\in C^\infty(U,\mathfrak{k}_{X/S}^\CC)$, $\mathrm{pr}^{1,0}\comp\rho=\mathrm{id}$ and $\mathrm{pr}^{0,1}\rho(\xi)=\overline{\phi(\xi)}$, where $\mathrm{pr}^{1,0},\mathrm{pr}^{0,1}$ are the projections of $T_{X/S}^\CC=T_{X/S}\oplus\widebar{T_{X/S}}$. Indeed $\mathrm{pr}^{1,0}(c^\RR)=c$ and $\mathrm{pr}^{0,1}(c^\RR+\I d^\RR)=\bar c+\I\bar d=\overline{c-\I d}$.

Writing $\theta(\partial_j)=\theta_j\in C^\infty(U,\mathfrak{k}_{X/S}^\CC)$, we have
\begin{equation}\label{eq:rho_theta}
\rho(\theta_j)=\theta_j+\overline{\phi(\theta_j)}=\theta_j-\overline{\bar\theta_{\bar j}}.
\end{equation}
By \eqref{eq:GD''_local} and \eqref{eq:GD'_local}, the brackets
\[ [H_{\bar i},\theta_j]=-G_{D''}^{1,1}(\partial_j,\partial_{\bar i}), \qquad [H_i,\bar\theta_{\bar j}]=G_{D'}^{1,1}(\partial_i,\partial_{\bar j}) \]
are vertical $(1,0)$-vector fields. Together with $\widebar{H_i}=H_{\bar i}$, \eqref{eq:rho_theta} yields the type decomposition
\begin{equation}\label{eq:bracket_decomp}
[H_{\bar i},\rho(\theta_j)]=[H_{\bar i},\theta_j]-\overline{[H_i,\bar\theta_{\bar j}]}=-G_{D''}^{1,1}(\partial_j,\partial_{\bar i})-\overline{G_{D'}^{1,1}(\partial_i,\partial_{\bar j})}.
\end{equation}
On the other hand, writing $H_{\bar i}=\tfrac12(H_{x^i}+\I H_{y^i})$ and $\theta_j=c_j+\I d_j$ with $c_j,d_j\in C^\infty(U,\mathfrak{k}_{X/S})$, $[H_{\bar i},\rho(\theta_j)]$ is a complex linear combination of $[H_{x^i},c_j^\RR]$, $[H_{y^i},c_j^\RR]$, $[H_{x^i},d_j^\RR]$, $[H_{y^i},d_j^\RR]$, each lying in $C^\infty(U,\mathfrak{k}_{X/S}^{\mathrm{R}})$ by \eqref{eq:k_parallel}.
Hence there is a unique $\xi_{\bar ij}\in C^\infty(U,\mathfrak{k}_{X/S}^\CC)$ with $[H_{\bar i},\rho(\theta_j)]=\rho(\xi_{\bar ij})$.
Then we have
\[ \xi_{\bar ij}=-G_{D''}^{1,1}(\partial_j,\partial_{\bar i}), \qquad \overline{\phi(\xi_{\bar ij})}=-\overline{G_{D'}^{1,1}(\partial_i,\partial_{\bar j})}. \]
The first identity shows that $G_{D''}^{1,1}$ is $\mathfrak{k}_{X/S}^\CC$-valued. Substituting it into the second yields \eqref{eq:G_conjugate}, so $G_{D'}^{1,1}$ is $\mathfrak{k}_{X/S}^\CC$-valued as well.
\end{proof}
\begin{definition}\label{def:harmonic}
Assume that $\nabla^{1,0}$ is a K\"ahler connection associated to $\dbar$ and $\omega_{X/S}$, with almost connection $\partial$. We call $(f,\omega_{X/S}, D'', D)$ a \emph{nonlinear harmonic bundle} if $D$ is flat and $(\dbar,\theta)$ is a nonlinear Higgs bundle, i.e., $F_{D}=0$ and $G_{D''}=0$.
\end{definition}
\begin{remark}\label{rmk:relation_to_old_harmonic}
Under the hypotheses of \cite[Prop.~4.12]{LS26} where the unitary atlas is also an atlas for $\mathfrak{k}_{X/S}$, Definition \ref{def:harmonic} is equivalent to harmonicity in the sense of the Simpson mechanism of \cite[\S5.2]{LS26}, when the K\"ahler connection in Definition \ref{def:harmonic} is the Chern connection associated to $\dbar$ and $\omega_{X/S}$. In this case, we have
\begin{equation}\label{eq:chern-shift}
\partial^{\mathrm{Ch}}_{\omega_{X/S},\,\dbar+\bar\theta}
=\partial^{\mathrm{Ch}}_{\omega_{X/S},\,\dbar}-\theta.
\end{equation}
Indeed, writing $\theta_i=p_i+\I q_i$ with $p_i,q_i\in\mathfrak{k}$, the conjugate $\bar\theta=-\phi(\theta_i)\,\D\bar s^i$ contributes $(v_\theta,w_\theta)$ with $v_\theta(\partial_{\bar i})=-p_i$, $w_\theta(\partial_{\bar i})=q_i$ in the notation of \cite[Eq.~(4.5)]{LS26}, whence the shift $v_\theta(\partial_{\bar i})-\I\,w_\theta(\partial_{\bar i})=-\theta_i$. Taking $\partial=\partial^{\mathrm{Ch}}_{\omega_{X/S},\dbar}$, \eqref{eq:chern-shift} identifies $D$ with the operator $\partial_{\omega_{X/S}}+\dbar_{\omega_{X/S}}$ of the Simpson mechanism, where $\partial_{\omega_{X/S}}=\partial^{\mathrm{Ch}}_{\omega_{X/S},\dbar+\bar\theta}+2\theta=\partial+\theta$ and $\dbar_{\omega_{X/S}}=\dbar+\bar\theta$.
\end{remark}

Combining \eqref{eq:pointwise_id}, \eqref{eq:type_orthog}, \eqref{eq:diff_theta_fin}, \eqref{eq:min-coupling}, \eqref{eq:symp_curv_11_curv}, \eqref{eq:curv11_decomp}, and \eqref{eq:G_conjugate}, we obtain the following.
\begin{theorem}\label{thm:D'_D''u_identity}
Under Assumption \ref{assum:comoment}\,(1), (2), and (4), every smooth section $u:S\to X$ satisfies
\begin{equation}\label{eq:D'_D''u_real}
|D'u|^2-|D''u|^2=\Lambda_{\omega_S}\bigl(u^*\omega_X+u^*\mu^*(F_{\nabla^\RR}+\bigl[[\theta,\bar\theta]\bigr]_\RR)\bigr)-\Lambda_{\omega_S}\zeta.
\end{equation}
If moreover $\nabla^{1,0}$ is a K\"ahler connection, then
\begin{equation}\label{eq:D'_D''u_identity}
|D'u|^2-|D''u|^2=\Lambda_{\omega_S}\bigl(u^*\omega_X+u^*\mu^*\bigl[F_D^{1,1}-G_{D'}^{1,1}-G_{D''}^{1,1}\bigr]_\RR\bigr)-\Lambda_{\omega_S}\zeta.
\end{equation}
In particular, if $(f,\omega_{X/S},D'',D)$ is a harmonic bundle and $\nabla^\RR$ preserves $\mathfrak{k}_{X/S}^{\mathrm{R}}$, then
\[	|D'u|^2-|D''u|^2=\Lambda_{\omega_S}(u^*\omega_X-\zeta).\]
Let $\deg_{\omega_X,\zeta}(u):=\int_S (u^*\omega_X-\zeta)\wedge \omega_S^{n-1}$. Assume moreover that $S$ is compact. Then
\begin{enumerate}
	\item any flat section $u$ satisfies $\deg_{\omega_X,\zeta}(u)=0$;
	\item any Higgs section $u$ satisfies $\deg_{\omega_X,\zeta}(u)\ge 0$ and is flat if and only if $\deg_{\omega_X,\zeta}(u)=0$.
\end{enumerate}
\end{theorem}
\begin{remark}\label{rmk:presv_cond}
When $\mu^*$ is defined only on $\mathfrak k_{X/S}^{\mathrm R}$, the preservation condition is already part of Assumption \ref{assum:comoment}\,(3). When $\mu^*$ is defined on all Hamiltonian vector fields, it is an additional hypothesis used to identify $G_{D'}^{1,1}$ with the involution of $G_{D''}^{1,1}$ in \eqref{eq:G_conjugate}.
\end{remark}

\begin{example}\label{ex:triv_bundle_sec}
Consider the trivial holomorphic bundle $f:S\times Y\to S$, where $(S,\omega_S)=(Y,\omega_Y)=(\CC P^1,\omega_{\mathrm{FS}})$. Let $\omega_X=\mathrm{pr}_2^*\,\omega_Y$, which is a relatively K\"ahler form inducing the product connection. Then $\zeta=0$. Set $\theta=0$. Then $D'=\partial$ and $D''=\dbar$. A Higgs section $u$ has the form $s\mapsto (s,u(s))$ (by abuse of notation), where $u:\CC P^1\to\CC P^1$ is a holomorphic map. Then \[\deg_{\omega_X,\zeta}(u)=\deg_{\omega_X}(u)=\int_{\CC P^1}u^*\omega_{\mathrm{FS}}=\deg(u)\cdot\int_{\CC P^1}\omega_{\mathrm{FS}}\ge 0, \] which vanishes precisely for constant maps, which correspond to the flat sections.
\end{example}

\subsection{The $D^c$-$D$ identity and the correspondence}\label{subsec:DcD_id}
For any $V\in T_s S^\RR$, we decompose it as $V = v + \bar{v}$ with $v \in T_s S$. We define
\begin{equation}\label{eq:def_KV}
K_V := \bigl(\bar{\theta}(\bar{v})-\theta(v)\bigr)^\RR.
\end{equation}
By \eqref{eq:XY_decomp}, $\theta(v) = X_v + \I Y_v \in \mathfrak{k}_s^\CC$ and $\bar{\theta}(\bar{v}) = -X_v + \I Y_v$. Then $ \bar{\theta}(\bar{v})-\theta(v) = -2X_v \in \mathfrak{k}_s$. If $\mathfrak{k}_s \subset \mathrm{Ham}^{1,0}(X_s)$, then $K_V = -2X_v^\RR$ is a real Hamiltonian Killing vector field on the fibers. Let $K \in A^1(S, \mathrm{Ham}_{X/S})$ denote the corresponding $1$-form.

We introduce two modified real connections $\nabla_+ := \nabla^\RR + K$ and $\nabla_- := \nabla^\RR - K$. $\nabla_+$ (resp.\ $\nabla_-$) is a symplectic connection associated to $\omega_{X/S}$ and $\dbar_+:=\dbar+\bar{\theta}$ (resp.\ $\dbar_-:=\dbar-\bar{\theta}$). Let $D^c:= D''-D'$. Applying Proposition \ref{prop:pointwise_identity} to $\omega_{\nabla_+}$ and $\omega_{\nabla_-}$ (defined by \eqref{eq:omega_nabla}), we obtain the following (note that by Lemma \ref{lem:dbar_section_affine}, adding $\bar{\theta}$ to the $\dbar$-operator (as a splitting) subtracts $\bar{\theta} u$ from the covariant derivative on sections).
\begin{corollary}\label{cor:Du_Dc_diff}
For any smooth section $u: S \to X$, we have
\begin{equation}\label{eq:Du_Dc_diff}
|D^c u|^2 - |D u|^2 = \Lambda_{\omega_S} u^* \omega_{\nabla_+} - \Lambda_{\omega_S} u^* \omega_{\nabla_-}.
\end{equation}
\end{corollary}

\begin{theorem}\label{thm:Dc_Du_identity}
Suppose Assumption \ref{assum:comoment} (1)--(3) hold and $\nabla^{1,0}$ is a K\"ahler connection. Let $\alpha\in A^1(X,\RR)$ be the real $1$-form defined by $\alpha(V) := \mu_{f(x)}^*(K_{\D f_x(V)})(x)$ for $V \in T_xX^\RR$. Then every smooth section $u:S\to X$ satisfies

\begin{equation}\label{eq:Dc_Du_formula}
|D^c u|^2 - |D u|^2 = \Lambda_{\omega_S} \bigl( -2\D (u^*\alpha) + 2 u^*\mu^*\bigl(\bigl[G_{D'}^{1,1}-G_{D''}^{1,1}\bigr]_\RR\bigr) \bigr).
\end{equation}
In particular, if $(f, \omega_{X/S}, D'', D)$ is a harmonic bundle \textup(so that $G_{D''}^{1,1}=0$\textup) over a compact semi-K\"ahler manifold $S$ and $\nabla^\RR$ preserves $\mathfrak{k}_{X/S}^{\mathrm{R}}$, then any flat section $u$ is necessarily a Higgs section.
\end{theorem}

\begin{proof}
Fix $x\in X$, let $s:=f(x)$, and let $V_1,V_2\in T_xX^\RR$.  Set $w_a:=\D f_x(V_a)\in T_sS^\RR$. Choose local real vector fields $W_a$ on a neighborhood of $s$ such that $W_a(s)=w_a$, and let $H_a:=H_{W_a}$ be their horizontal lifts. Write $V_a=H_a(x)+v_a$, $v_a\in T_{X/S,x}^\RR$. Choose local vertical vector fields $\widetilde v_a$ near $x$ such that $\widetilde v_a(x)=v_a$, and set $\widetilde V_a:=H_a+\widetilde v_a$. Then $\widetilde V_a(x)=V_a$, each $\widetilde V_a$ is $f$-projectable, and $f_*\widetilde V_a=W_a$. Consequently, $f_*[\widetilde V_1,\widetilde V_2]=[W_1,W_2]$. Since the expressions below are tensorial at $x$, these extensions may be used in the computation.

Let $\mathrm{pr}_{\pm}^{\mathrm v}=\mathrm{pr}_{\nabla^\RR}^{\mathrm v}\mp f^*K$ be the vertical projections associated to $\nabla_\pm$. Then
\begin{align*}
(\omega_{\nabla_+}-\omega_{\nabla_-})(V_1,V_2)&=\omega_{X/S}(\mathrm{pr}_{+}^{\mathrm{v}} V_1, \mathrm{pr}_{+}^{\mathrm{v}} V_2)- \omega_{X/S}(\mathrm{pr}_{-}^{\mathrm{v}} V_1, \mathrm{pr}_{-}^{\mathrm{v}} V_2)
\\&=2\omega_{X/S}(v_2,K_{w_1})-2\omega_{X/S}(v_1,K_{w_2})\\
&=2v_2\bigl(\mu^*(K_{w_1})\bigr)-2v_1\bigl(\mu^*(K_{w_2})\bigr).
\end{align*}
On the other hand, by Assumption \ref{assum:comoment}\,(3),
\begin{align*}
(\D\alpha)_x(V_1,V_2)&= \widetilde V_1\bigl(\mu^*(K_{W_2})\bigr)\big|_x- \widetilde V_2\bigl(\mu^*(K_{W_1})\bigr)\big|_x-\mu^*(K_{[W_1,W_2]})(x)\\
&=v_1\bigl(\mu^*(K_{w_2})\bigr)-v_2\bigl(\mu^*(K_{w_1})\bigr)\\
&\quad+\mu^*\bigl([H_1,K_{W_2}]-[H_2,K_{W_1}]-K_{[W_1,W_2]}\bigr)(x)\\
&= \tfrac{1}{2}(\omega_{\nabla_-} - \omega_{\nabla_+})(V_1, V_2) + \mu^*\bigl((\D_{\mathfrak h} K)(W_1, W_2)\bigr)(x),
\end{align*}
where $(\D_{\mathfrak h} K)(W_1, W_2):=[H_1,K_{W_2}]-[H_2,K_{W_1}]-K_{[W_1,W_2]}$. Therefore, we obtain
\begin{equation}\label{eq:omega_diff_final}
u^*(\omega_{\nabla_+} - \omega_{\nabla_-}) = -2\D (u^*\alpha) + 2 u^* \mu^*(\D_{\mathfrak h} K).
\end{equation}

Next, we evaluate the $(1,1)$-part of $\D_{\mathfrak h} K$. Extend $K$ to $TS^\CC$ by $\CC$-linearity. For $v \in T_s S$, writing $v = \frac{1}{2}(V - \I J_S V)$ with real $V$, one computes
\[
K_v = \tfrac{1}{2}(K_V - \I K_{J_S V}) = -\theta(v) + \overline{\bar{\theta}(\bar{v})},
\]
where the bar over $\bar{\theta}$ denotes complex conjugation in the complexified bundle $TX^\CC$. Similarly, for $\bar{w} \in \widebar{T_s S}$, $K_{\bar{w}} = \bar{\theta}(\bar{w}) - \widebar{\theta(w)}$.  We now compute $(\D_{\mathfrak h} K)^{1,1}$. Since $K_{\partial_i}=-\theta_i+\overline{\bar\theta_{\bar i}}$ and $K_{\partial_{\bar j}}=\bar\theta_{\bar j}-\overline{\theta_j}$,
\[
[H_i,K_{\partial_{\bar j}}]=G_{D'}^{1,1}(\partial_i,\partial_{\bar j})+\overline{G_{D''}^{1,1}(\partial_j,\partial_{\bar i})},
\qquad [H_{\bar j},K_{\partial_i}]=G_{D''}^{1,1}(\partial_i,\partial_{\bar j})+\overline{G_{D'}^{1,1}(\partial_j,\partial_{\bar i})}.
\]
$K_{[\partial_i,\partial_{\bar j}]}=0$ since $[\partial_i,\partial_{\bar j}]=0$. Hence
\[\bigl(\D_{\mathfrak h} K\bigr)^{1,1}(\partial_i,\partial_{\bar j})
=[H_i,K_{\partial_{\bar j}}]-[H_{\bar j},K_{\partial_i}]
=\bigl[G_{D'}^{1,1}-G_{D''}^{1,1}\bigr]_\RR(\partial_i,\partial_{\bar j}).\]
Therefore, \eqref{eq:Dc_Du_formula} follows.

Finally, suppose the bundle is harmonic, so $G_{D''}^{1,1} = 0$. This implies $G_{D'}^{1,1} = 0$ by \eqref{eq:G_conjugate}. The identity \eqref{eq:Dc_Du_formula} reduces to $|D^c u|^2 - |D u|^2 = -2\Lambda_{\omega_S} \D(u^*\alpha)$. If $u$ is a flat section, then $|D u|^2 = 0$. Integrating over $S$, using $\D(\omega_S^{n-1})=0$ and Stokes' theorem, we obtain $\int_S |D^c u|^2\omega_S^n = 0$, and $D^c u=0$. Since $Du=D^c u=0$, we have $D''u=0$.
\end{proof}

Theorem \ref{thm:Dc_Du_identity} fails without the condition $\mathfrak{k}_s\subset \mathrm{Ham}^{1,0}(X_s)$, as the following example shows.
\begin{example}\label{ex:flat_not_higgs}
Let $\Lambda:=\ZZ+\I\ZZ\subset\CC$ and set $S:=\CC/\Lambda$, $Y:=\CC/\Lambda$, both equipped with the flat K\"ahler form $\frac{\I}{2}\D s\wedge\D\bar{s}$ (resp.\ $\frac{\I}{2}\D z\wedge\D\bar{z}$). Consider the trivial holomorphic fibration
$f:X:=S\times Y\to S$ with the relatively K\"ahler form
$\omega_X:=\mathrm{pr}_2^*\,\omega_Y$.

The translations on $Y$ act by isometries. Choose the real form $\mathfrak{k}_{X/S} = \RR\, \partial_z\not\subset \mathrm{Ham}^{1,0}_{X/S}$. Its complexification is $\mathfrak{k}_{X/S}^\CC = \CC\,\partial_z$, consisting of vertical holomorphic vector fields. The involution $\phi$ fixes the real form, so $\phi(\partial_z) = \partial_z$. Define the Higgs field
\[
\theta:=-\D s\otimes\partial_z\;\in\;A^{1,0}(S,\mathfrak{k}_{X/S}^\CC).
\]
This is relatively holomorphic with $[\theta,\theta]=0$. Moreover, $\bar{\theta}=\D\bar{s}\otimes \partial_z$, $[\theta,\bar{\theta}]=0$, and $F_{D_\omega}^{1,1}=F_D^{1,1}=0$. Consider the smooth section given by the map $u:S\to Y$, $u(s)=\bar{s}-s$. By abuse of notation, denote this section by $u$. Shifting $s \mapsto s+1$ leaves $u$ unchanged, and $s \mapsto s+\I$ shifts $u \mapsto u-2\I$, which is the same point in $Y$. Thus $u$ is well-defined. Then we have
\[D'u = \partial u-\bar{\theta}u=-\D s\otimes\partial_z-\D\bar{s}\otimes\partial_z,\qquad D''u = \dbar u-\theta u=\D\bar{s}\otimes\partial_z+\D s\otimes\partial_z.\]
We see that $D'u+D''u=0$, so $Du=0$ and $u$ is a flat section. On the other hand, $D''u\neq 0$, so $u$ is not a Higgs section.
\end{example}

Combining Theorems \ref{thm:D'_D''u_identity} and \ref{thm:Dc_Du_identity}, we obtain the following.

\begin{corollary}\label{cor:equiv_higgs_flat}
Suppose $(S,\omega_S)$ is compact semi-K\"ahler, and Assumption \ref{assum:comoment} (1)--(4) hold. If $(f, \omega_{X/S}, D'', D)$ is a harmonic bundle and $\nabla^\RR$ preserves $\mathfrak{k}_{X/S}^{\mathrm{R}}$, then a smooth section $u$ is a flat section if and only if it is a Higgs section and $\deg_{\omega_X,\zeta}(u)=0$.
\end{corollary}

The following two examples illustrate that a harmonic bundle satisfying Assumption \ref{assum:comoment}\,(1)--(4) may fail to have a reductive monodromy representation, and hence need not be a harmonic bundle reductive of K\"ahler type in the sense of \cite[Def.~1.10]{LS26}.

\begin{example}\label{ex:compact_fiber_nonreductive_monodromy}
Let $S=\CC/\Lambda_S$ be an elliptic curve and let $Y=\CC/\Lambda_Y$ be another elliptic curve with the flat K\"ahler form $\omega_Y=\tfrac{\I}{2}\,\D z\wedge\D\bar z$. Fix $a\in Y$ of infinite order and let
\[
\rho:\pi_1(S)\cong\ZZ^2\longrightarrow \Aut_0(Y)=Y,\qquad
\gamma_1\mapsto t_a:\ y\mapsto y+a,\quad \gamma_2\mapsto \id,
\]
with associated flat holomorphic bundle $f:X:=\widetilde S\times_\rho Y\to S$.  As translations preserve $\omega_Y$, the form $\mathrm{pr}_Y^*\omega_Y$ on $\widetilde S\times Y$ is $\pi_1(S)$-invariant and descends to a closed real $(1,1)$-form $\omega_X\in A^{1,1}(X,\RR)$ restricting to $\omega_Y$ on each fiber. The symplectic connection induced by $\omega_X$ is the flat connection above, and parallel transport acts by holomorphic isometries, so $\nabla^{1,0}$ is a K\"ahler connection. Take $\theta=0$. Then $(f,\omega_{X/S},D'',D)$ is a nonlinear harmonic bundle.

We verify Assumption \ref{assum:comoment}. Since $\theta=0$ we may take $\mathfrak{k}_{X/S}=0$, so (1) is vacuous and the comoment map is taken on $\mathrm{Ham}_{X/S}$. As $f$ is proper, (2) and the parallelism in (3) hold by Remark \ref{rmk:moment_existence} via the mean-zero Hamiltonian potentials, and no $\mathfrak{k}_{X/S}^{\mathrm R}$-preservation is needed. For (4), $\omega_X$ is closed, $\omega_X^H=0$, and $F_{\nabla^\RR}=0$, so $\zeta=0$ and
$\mu^*F_{\nabla^\RR}=0$ is well-defined. Thus Assumption \ref{assum:comoment}\,(1)--(4) hold with $\zeta=0$.

The image $\rho(\pi_1(S))=\langle t_a\rangle$ is infinite, and every proper algebraic subgroup of the elliptic curve $Y$ is finite, so $\overline{\rho(\pi_1(S))}^{\,\mathrm{Zar}}=Y$, which is not a complex reductive algebraic group.
\end{example}

\begin{example}\label{ex:affine_line_nonreductive_monodromy}
Let $S=\CC/\Lambda_S$ be an elliptic curve and $Y=\CC$ with the flat K\"ahler form $\omega_Y=\tfrac{\I}{2}\,\D z\wedge\D\bar z$. The holomorphic isometries of $(\CC,\omega_Y)$ are the maps $z\mapsto\lambda z+b$ with $|\lambda|=1$. Fix $b\in\CC\setminus\{0\}$ and let
\[
\rho:\pi_1(S)\cong\ZZ^2\longrightarrow \Aut_0(Y),\qquad \gamma_1\mapsto t_b:\ z\mapsto z+b,\quad \gamma_2\mapsto \id,
\]
with associated flat holomorphic bundle $f:X:=\widetilde S\times_\rho Y\to S$, an affine line bundle over $S$. Exactly as in Example \ref{ex:compact_fiber_nonreductive_monodromy}, $\mathrm{pr}_Y^*\omega_Y$ descends to a closed real $(1,1)$-form $\omega_X\in A^{1,1}(X,\RR)$ restricting to $\omega_Y$ on each fiber, the induced symplectic connection is the flat connection $\nabla^{1,0}$ whose parallel transport acts by holomorphic isometries, and with $\theta=0$ the datum $(f,\omega_{X/S},D'',D)$ is a harmonic bundle.

Since the fiber is noncompact, the mean-zero normalization of Remark \ref{rmk:moment_existence} is unavailable. Instead we take $\mathfrak{k}_{X/S}=0$ and use the comoment map on $\mathfrak{k}_{X/S}^{\mathrm R}=0$, which is the zero map. Then Assumption \ref{assum:comoment}\,(1)--(3) hold trivially. The connection is flat, so $F_{\nabla^\RR}=0$ takes values in $\mathfrak{k}_{X/S}^{\mathrm R}=0$ and $\mu^*F_{\nabla^\RR}=0$. Together with $\omega_X^H=0$ this gives $\zeta=0$ in \eqref{eq:min-coupling}. Thus Assumption \ref{assum:comoment}\,(1)--(4) holds with $\zeta=0$.

The image $\rho(\pi_1(S))=\langle t_b\rangle\cong\ZZ$ is an infinite subgroup of the translation group $\CC\cong\G_a\le\Aut_0(Y)$, and over $\CC$ the group $\G_a$ has no nontrivial proper algebraic subgroup, so $\overline{\rho(\pi_1(S))}^{\,\mathrm{Zar}}=\G_a$, which is not reductive.
\end{example}

\subsection{Harmonic vector bundles}\label{subsec:harm_vec_bun}
Corollary \ref{cor:equiv_higgs_flat} recovers \cite[Lem.~1.2]{Si1}: for a harmonic vector bundle over a compact K\"ahler manifold, a Higgs section is equivalent to a flat section. Note that in this case, the assumptions of Corollary \ref{cor:equiv_higgs_flat} are automatically satisfied with $\zeta=0$ \cite[Ex.~4.17]{LS26}. Any section $u$ of a vector bundle is homotopic to the zero section $u_0$, and $u_0^*\,\omega=0$, where $\omega$ is the relatively K\"ahler form in \cite[Ex.~4.17]{LS26}.

Note that for a harmonic vector bundle, the Higgs subbundles and flat subbundles need not be in one-to-one correspondence.
\begin{example}\label{ex:higgs_not_flat_subfib}
Let $S$ be a compact Riemann surface, and let $E = L_1\oplus L_2$ be a holomorphic rank-$2$ vector bundle with $\deg E=0$ and $L_1,L_2$ being holomorphic line bundles, equipped with a Higgs field
\[\theta_E = \begin{pmatrix} 0 & \phi \\ 0 & 0\end{pmatrix}\in A^{1,0}(S,\End E),\qquad \phi\in H^0(S,\Omega_S^1\otimes\mathrm{Hom}(L_2,L_1)),\;\phi\ne 0.\]
If $\deg L_1<\deg L_2$, then $(E,\theta_E)$ is stable and there exists a unique harmonic metric $h$ up to scaling, which must be of the form $h_1\oplus h_2$ satisfying
\[F_{h_1}+\phi\wedge\phi^{*_h}=0,\qquad F_{h_2}+\phi^{*_h}\wedge\phi=0,\]
where $\phi^{*_h}\in A^{0,1}(S,\mathrm{Hom}(L_1,L_2))$. The subbundle $L_1\subset E$ is a Higgs subbundle. However, $\theta_E^{*_h}$ does not preserve $L_1$. Consequently, $L_1$ is not a flat subbundle.
\end{example}
\begin{lemma}[{\cite[Lem.~3.2]{simpson1988construct}}]\label{lem:higgs_flat_subbundle_equiv1}
Let $(E,\theta_E,h)$ be a harmonic vector bundle over a compact connected K\"ahler manifold $(S,\omega_S)$ of complex dimension $n$, with flat connection $D_E = \partial_h+\dbar_E+\theta_E+\theta_E^{*_h}$. Let $F\subset E$ be a Higgs subbundle of $(E,\dbar_E,\theta_E)$, i.e., a holomorphic subbundle satisfying $\theta_E(F)\subset \Omega_S^1\otimes F$. Then $F$ is a flat subbundle of $(E,D_E)$, i.e., $D_E(A^0(S,F))\subset A^1(S,F)$, if and only if $\deg(F):=\int_S c_1(F)\wedge \omega_S^{n-1}=0$.
\end{lemma}
\begin{lemma}[{\cite[Prop.~3.2]{corlette88}}]\label{lem:higgs_flat_subbundle_equiv2}
	In the setup above, if $F$ is a flat subbundle of $(E,D_E)$, then it is a Higgs subbundle of $(E,\dbar_E,\theta_E)$.
\end{lemma}

Let $(E,\dbar_E,\theta_E,h)$ be a Hermitian Higgs vector bundle of rank $m\ge 2$ over a compact connected K\"ahler manifold $(S,\omega_S)$ of complex dimension $n$, and put $D_E:=\partial_h+\dbar_E+\theta_E+\theta_E^{*_h}$. Here $\partial_h+\dbar_E$ is the Chern connection of $(E,\dbar_E,h)$. We shall derive Lemmas \ref{lem:higgs_flat_subbundle_equiv1} and \ref{lem:higgs_flat_subbundle_equiv2} from Corollary \ref{cor:equiv_higgs_flat} by passing to the projectivization.

Let $P:=\mathrm{Fr}(E)$ denote the holomorphic frame bundle of $E$, a principal $G:=\GL(m,\CC)$-bundle on $S$, equipped with the Chern principal connection induced by $h$. The pair $(P,\theta_E)$ is a $G$-Higgs bundle, with reduction $P_K\subset P$ to $K:=\U(m)$ corresponding to $h$. Take $Y:=\CC P^{m-1}$ with the Fubini--Study form $\omega_{\mathrm{FS}}$ induced by the standard Hermitian inner product on $\CC^m$, normalized so that $\int_{\CC P^{m-1}}\omega_{\mathrm{FS}}^{m-1}=1$, equivalently $[\omega_{\mathrm{FS}}]=c_1(\mathcal{O}_{\CC P^{m-1}}(1))$. Consider the associated fiber bundle $f:X:=\PP(E)=P\times_G Y\to S$. Its tautological line bundle $\mathcal{O}_{\PP(E)}(-1)\subset f^*E$ inherits a Hermitian metric from $h$, whose dual is a Hermitian metric $h_\PP$ on $\mathcal{O}_{\PP(E)}(1)$. The curvature of the Chern connection of $(\mathcal{O}_{\PP(E)}(1),h_\PP)$ multiplied by $\frac{\I}{2\pi}$ is a relatively K\"ahler form $\omega_X$ on $X$. In particular,
\begin{equation}\label{eq:cohom_omegaX}
[\omega_X]=c_1(\mathcal{O}_{\PP(E)}(1))\in H^2(X,\RR),
\end{equation}
$\omega_X$ restricts to the fiberwise Fubini--Study metric $\omega_{X/S}$ induced by $h$ on each fiber, and $\omega_X$ induces the Chern $(1,0)$-connection $\nabla^{1,0}$ on $f$. The Higgs field on $f$ is $\theta:=\tau(\theta_E)\in H^0(S,\Omega_S^1\otimes f_*^{\mathrm{hol}}T_{X/S})$, where $\tau:\End E=\ad P\to f_*^{\mathrm{hol}}T_{X/S}$ (\cite[Eq.~(2.30)]{LS26}) has kernel $\mathcal{O}_S\cdot\id_E$. By \cite[Lem.~5.8]{LS26}, $\bar\theta_{\omega_{X/S}}=\tau(\theta_E^{*_h})$. $(f,\omega_X,\theta)$ satisfies Assumption \ref{assum:comoment}\,(1)--(4), with $\zeta=-\frac1m c_1(E,h)$ in \eqref{eq:min-coupling}. Indeed, for $A\in\mathfrak u(m)$, let $\xi_A=\tilde{\tau}_0(A)$ (\cite[Eq.~(2.29)]{LS26}) be the real vector field generated
on $\CC P^{m-1}$, and define
\[\mu^*(\xi_A)([v]):=\frac{\I}{2\pi}\Bigl(\frac{\langle Av,v\rangle_h}{\langle v,v\rangle_h}-\frac1m\operatorname{tr}A\Bigr),
\]
which is the mean-zero equivariant comoment map and annihilates the scalar kernel. If $F_h$ is the Chern curvature of $(E,h)$, then
\[\omega_X^H=-\frac{\I}{2\pi}\frac{\langle F_hv,v\rangle_h}{\langle v,v\rangle_h},\qquad \mu^*F_{\nabla^\RR}=\frac{\I}{2\pi}\Bigl(\frac{\langle F_hv,v\rangle_h}{\langle v,v\rangle_h}-\frac1m\operatorname{tr}F_h\Bigr).\]
Consequently, $\omega_X^H=-\mu^*F_{\nabla^\RR}+f^*\zeta$, with $\zeta=-\frac1m c_1(E,h)$.

The assignment $L\mapsto u_L$ defined by $u_L(s):=[L_s]\in\PP(E_s)$ is a bijection between smooth complex line subbundles $L\subset E$ and smooth sections $u:S\to X=\PP(E)$, characterized by
\begin{equation}\label{eq:tautological}
u_L^*\mathcal{O}_{\PP(E)}(-1)=L.
\end{equation}
Let $L\subset E$ be a smooth complex line subbundle. Then $L$ is a Higgs (resp.\ flat) subbundle of $(E,\dbar_E,\theta_E)$ if and only if $u_L$ is a Higgs (resp.\ flat) section of $(f,\theta, \nabla^{1,0})$.
\begin{lemma}\label{lem:deg_line_subbundle}
For every smooth complex line subbundle $L\subset E$,
\begin{equation}\label{eq:deg_line_subbundle}
\deg_{\omega_X}(u_L)=-\deg(L),\qquad \deg(L):=\int_S c_1(L)\wedge\omega_S^{n-1}.
\end{equation}
\end{lemma}
\begin{proof}
By \eqref{eq:tautological}, $u_L^*\mathcal{O}_{\PP(E)}(1)=L^{-1}$, hence
\[u_L^*c_1(\mathcal{O}_{\PP(E)}(1))=-c_1(L)\in H^2(S,\RR).\]
Combining this with \eqref{eq:cohom_omegaX}, we obtain $[u_L^*\omega_X]=-c_1(L)$ in $H^2(S,\RR)$. Therefore,
\[\deg_{\omega_X}(u_L)=\int_S u_L^*\omega_X\wedge\omega_S^{n-1}=-\int_S c_1(L)\wedge\omega_S^{n-1}=-\deg(L).\qedhere\]
\end{proof}

To treat subbundles of arbitrary rank, we reduce to the rank-one case by passing to the exterior power. $\Lambda^k E$ carries the induced Hermitian Higgs bundle structure $(\Lambda^k E,\dbar_E^{(k)},\theta_E^{(k)},h^{(k)})$, where $h^{(k)}:=\Lambda^k h$, $\dbar_E^{(k)}$ is induced by $\dbar_E$, and
\begin{equation}\label{eq:exterior_higgs}
\theta_E^{(k)}(a_1\wedge\cdots\wedge a_k):=\sum_{i=1}^{k}a_1\wedge\cdots\wedge\theta_E(a_i)\wedge\cdots\wedge a_k.
\end{equation}
Suppose $h$ is harmonic. Then $h^{(k)}$ is harmonic. The flat connection of $(\Lambda^k E,\theta_E^{(k)},h^{(k)})$ is $D_E^{(k)}:=\Lambda^k D_E$. Let $F\subset E$ be a smooth complex subbundle of rank $k$, and set $\det F:=\Lambda^k F\subset \Lambda^k E$, a smooth complex line subbundle. $F$ is a Higgs (resp.\ $D_E$-flat) subbundle of $E$ if and only if $\det F$ is a Higgs (resp.\ $D_E^{(k)}$-flat) line subbundle of $\Lambda^k E$. Applying the projectivization construction to $(\Lambda^k E,\theta_E^{(k)},h^{(k)})$ yields a nonlinear harmonic bundle $(f^{(k)},\theta^{(k)},\omega_{X^{(k)}/S})$ on $f^{(k)}:X^{(k)}:=\PP(\Lambda^k E)\to S$ satisfying Assumption \ref{assum:comoment}\,(1)--(4) with $\zeta=0$. Since $c_1(F)=c_1(\det F)$, Corollary \ref{cor:equiv_higgs_flat} and Lemma \ref{lem:deg_line_subbundle} imply Lemmas \ref{lem:higgs_flat_subbundle_equiv1} and \ref{lem:higgs_flat_subbundle_equiv2}.

\subsection{Linearizations along Higgs--flat sections}\label{subsec:harmonic_metric_hE}
In this subsection, we show that the linearization of a nonlinear harmonic bundle along a Higgs--flat section is again harmonic. Let $(f:X\to S, \omega_{X/S},D'',D)$ be a harmonic bundle ($S$ may be noncompact). Let $u:S\to X$ be a smooth section that is simultaneously Higgs and flat. Set $E:=u^*T_{X/S}$, with the pullback holomorphic structure $\dbar_E=\nabla^{E,0,1}$
(Lemma~\ref{lem:pullback_dbar_holomorphic}) and the pullback Hermitian metric $h_E:=u^*h_{X/S}$. Let $\Theta_u\in A^{1,0}(S,\End E)$ be the linearized almost Higgs field with respect to the fiberwise Chern connection $\nabla^{\mathrm{fib}}$ (Definition \ref{def:linearized_higgs}). Similarly, the linearization of $\bar\theta$ along $u$,
\begin{equation}\label{eq:def_thetabar_lin}
\widebar\Theta_u(\bar v):=\nabla^{\mathrm{fib}}|_{X_s}\bigl(\bar\theta(\bar v)|_{X_s}\bigr)\big|_{u(s)}\in\End(E_s)\qquad(\bar{v}\in \widebar{T_sS}),
\end{equation}
is a well-defined element of $A^{0,1}(S,\End E)$.

\begin{lemma}\label{lem:isotropy_skew}
Let $(Y,J,g)$ be a K\"ahler manifold with Levi--Civita connection $\nabla$ (its restriction to $TY$ being the Chern connection), and let $h$ be the Hermitian metric on $TY$. Let $W\in H^0(Y,TY)$ such that $W^\RR:=W+\overline W$ is Killing, i.e., $\mathcal{L}_{W^\RR}g=0$. If $W(p)=0$, then $L_p(W):=\nabla W\big|_p\in\End(T_pY)$ is locally given by $(L_pW)^\alpha_\beta=\partial_\beta W^\alpha(p)$, and $\bigl(L_p(W)\bigr)^{*_{h}}=-L_p(W)$.
\end{lemma}

\begin{proof}
Set $\xi:=W^\RR$ and let $A:=\nabla\xi|_p\in\End(T_pY^\CC)$, so that $\nabla_Z\xi|_p=A(Z)$ for $Z\in T_pY^\CC$. Locally, $\nabla_{\partial_\beta}W=(\partial_\beta W^\alpha+\Gamma^\alpha_{\beta\gamma}W^\gamma)\partial_\alpha$, where $\Gamma^\alpha_{\beta\gamma}$ are the Christoffel symbols of the Chern connection. As $W^\gamma(p)=0$, this gives $(L_pW)^\alpha_\beta=\partial_\beta W^\alpha(p)$.

For all vector fields $V,Z$,
\[
0=(\mathcal{L}_\xi g)(V,Z)=\xi\,g(V,Z)-g\bigl([\xi,V],Z\bigr)-g\bigl(V,[\xi,Z]\bigr).
\]
Since $\nabla$ is a metric connection, $\xi\,g(V,Z)=g(\nabla_\xi V,Z)+g(V,\nabla_\xi Z)$. Together with the torsion-free identity $[\xi,V]=\nabla_\xi V-\nabla_V\xi$ (likewise for $Z$), we have
\[
(\mathcal{L}_\xi g)(V,Z)=g(\nabla_V\xi,Z)+g(V,\nabla_Z\xi).
\]
Hence $g(\nabla_V\xi,Z)+g(V,\nabla_Z\xi)=0$ at every point. Evaluating at $p$ and extending complex-bilinearly,
\begin{equation}\label{eq:A_skew}
g^\CC(AU,V)+g^\CC(U,AV)=0\qquad\text{for all }U,V\in T_pY^\CC.
\end{equation}

Since $W$ is holomorphic, $\mathcal{L}_\xi J=0$. For all $V$,
\[
0=(\mathcal{L}_\xi J)V=[\xi,JV]-J[\xi,V].
\]
Expanding with the torsion-free identity and $\nabla J=0$,
\[
[\xi,JV]=\nabla_\xi(JV)-\nabla_{JV}\xi=J\nabla_\xi V-\nabla_{JV}\xi,\qquad J[\xi,V]=J\nabla_\xi V-J\nabla_V\xi.
\]
So $\nabla_{JV}\xi=J\nabla_V\xi$ for all $V$. At $p$ this becomes $A J_p=J_p A$. Hence $A$ preserves $T_p Y$ and $\widebar{T_p Y}$. For $\partial_\beta\in T_p Y$,
\[
A(\partial_\beta)=\nabla_{\partial_\beta}\xi\big|_p=\nabla_{\partial_\beta}W\big|_p+\nabla_{\partial_\beta}\widebar W\big|_p.
\]
The first term equals $\partial_\beta W^\alpha(p)\,\partial_\alpha=L_p(W)(\partial_\beta)$ by the local formula above. For the second, writing $\widebar W=\widebar{W^\gamma}\,\partial_{\bar\gamma}$, since $\partial_\beta\widebar{W^\gamma}=0$, and $\widebar{W^\gamma}(p)=0$,
\[
\nabla_{\partial_\beta}\widebar W\big|_p=\bigl(\partial_\beta\overline{W^\gamma}\bigr)\partial_{\bar\gamma}\big|_p+\overline{W^\gamma}(p)\,\nabla_{\partial_\beta}\partial_{\bar\gamma}=0.
\]
Thus $A|_{T_p Y}=L_p(W)$.

Because $\xi$ is real, $A=\nabla\xi|_p$ is a real endomorphism, and $\widebar{A Z}=A\widebar{Z}$. Let $\xi',\eta'\in T_pY$. Then $L_p(W)\xi'=A\xi'$, $L_p(W)\eta'=A\eta'$, and $\widebar{A\eta'}=A\widebar{\eta'}$. Using $h(\zeta,\mu)=2g^\CC(\zeta,\bar\mu)$ and \eqref{eq:A_skew} with $U=\xi'$, $V=\overline{\eta'}$,
\[
h\bigl(L_p(W)\xi',\eta'\bigr)+h\bigl(\xi',L_p(W)\eta'\bigr)=2\bigl(g^\CC\bigl(A\xi',\overline{\eta'}\bigr)+g^\CC\bigl(\xi',A\overline{\eta'}\bigr)\bigr)=0.
\]
Therefore $\bigl(L_p(W)\bigr)^{*_{h}}=-L_p(W)$.
\end{proof}

\begin{lemma}\label{lem:conj_adjoint}
$\widebar\Theta_u=\Theta_u^{*_{h_E}}$, i.e., for all $s\in S$ and $v\in T_sS$, $\widebar\Theta_u(\bar v)=\bigl(\Theta_u(v)\bigr)^{*_{h_E}}$ in $\End(E_s)$.
\end{lemma}
\begin{proof}
As in \eqref{eq:XY_decomp} and \eqref{eq:thetabar_XY}, for $s\in S$ and $v\in T_sS$ we write
\[
\theta(v)=X_v+\I Y_v,\qquad \bar\theta(\bar v)=-X_v+\I Y_v,\qquad X_v,Y_v\in\mathfrak{k}_s.
\]
Since $D''u=Du=0$, we have $\theta(v)|_{u(s)}=0$ and $\bar\theta(\bar v)|_{u(s)}=0$, i.e., in $T_{u(s)}X_s$,
\[
X_v|_{u(s)}+\I Y_v|_{u(s)}=0,\qquad -X_v|_{u(s)}+\I Y_v|_{u(s)}=0.
\]
Then $X_v|_{u(s)}=Y_v|_{u(s)}=0$. By Lemma \ref{lem:isotropy_skew}, the endomorphisms $L(X_v):=L_{u(s)}(X_v)$ and $L(Y_v):=L_{u(s)}(Y_v)$ are skew-Hermitian for $h_E|_s=h_{X/S}|_{u(s)}$. By definition,
\[
\Theta_u(v)=L(X_v)+\I L(Y_v),\qquad \widebar\Theta_u(\bar v)=-L(X_v)+\I L(Y_v).
\]
Then we have
\[
\bigl(\Theta_u(v)\bigr)^{*_{h_E}}=L(X_v)^{*_{h_E}}+\overline{\I}\,L(Y_v)^{*_{h_E}}=-L(X_v)+\I L(Y_v)=\widebar\Theta_u(\bar v).\qedhere
\]
\end{proof}

\begin{theorem}\label{thm:hE_harmonic}
Suppose that $(f,\omega_{X/S},D'',D)$ is a nonlinear harmonic bundle and that $u:S\to X$ is simultaneously Higgs and flat. Then $(E,h_E,D_E'',D_E)$ is a harmonic bundle, where \[D_E''=\dbar_E+\Theta_u,\quad D_E= \partial_{h_E}+\dbar_E+\Theta_u+\Theta_u^{*_{h_E}},\] with $\partial_{h_E}+\dbar_E$ being the Chern connection associated to $\dbar_E$ and $h_E$.
\end{theorem}
\begin{proof}
By Corollary \ref{cor:pullback_chern}\,(ii), $\nabla^E$ is the Chern connection of $(E,\dbar_E,h_E)$. By Proposition \ref{prop:higgs_section_linearized_higgs}, $(E,\dbar_E,\Theta_u)$ is a Higgs bundle. By Lemma \ref{lem:conj_adjoint}, $\widebar{\Theta}_u=\Theta_u^{*_{h_E}}$. Therefore, the pullback connection $\nabla_D^E$ constructed from $D$ and the fiberwise Chern connection $\nabla^{\mathrm{fib}}$ is $\nabla_D^E=\partial_{h_E}+\dbar_E-\Theta_u-\Theta_u^{*_{h_E}}$, which is flat by Proposition \ref{prop:flat_pullback_flat}, so $D_E$ is also flat.
\end{proof}

\section{Higgs sub-fibrations and flat sub-fibrations}\label{sec:subfibrations}
In this section, we extend the correspondence between Higgs sections and flat sections of Section \ref{sec:kahler_id} to higher-dimensional sub-objects of $f:X\to S$ whose fibers are immersed complex submanifolds of fixed dimension $k\ge0$.

Throughout this section, $(S,\omega_S)$ is a connected Hermitian manifold of complex dimension $n\ge 1$, and $(f:X\to S,T_{X/S})$ is a complex fiber bundle with connected fibers of complex dimension $m$, equipped with a $\dbar$-operator $\dbar$, a $\theta$-adapted fiberwise K\"ahler metric $\omega_{X/S}$, a symplectic connection $\nabla^{1,0}$ induced by a real $(1,1)$-form $\omega_X$ on $X$, and a relatively holomorphic almost Higgs field $\theta\in A^{1,0}(S,\mathfrak{k}_{X/S}^{\CC})$. We write $h_{X/S}$ for the Hermitian metric associated to $\omega_{X/S}$ and abbreviate $\bar\theta:=\bar\theta_{\omega_{X/S}}\in A^{0,1}(S,\mathfrak{k}_{X/S}^{\CC})$.

\subsection{Sub-fibrations}\label{subsec:subfibdatum}
\begin{definition}\label{def:subfib_datum}
A \emph{sub-fibration} of relative dimension $0\le k\le m$ of $f:X\to S$ is a pair $(\Sigma,u)$ consisting of a complex manifold $\Sigma$ of complex dimension $n+k$ and a smooth map $u:\Sigma\to X$, such that
\begin{enumerate}[label=(\alph*)]
\item the composition $\pi:=f\comp u:\Sigma\to S$ is a surjective holomorphic map whose fibers have complex dimension $k$;
\item if $k=0$, $\pi$ is \emph{locally semi-finite} (\cite[\S 2]{AJS04}): every $s\in S$ admits a neighborhood $U$ such that for every connected component $U'$ of $\pi^{-1}(U)$ the restriction $\pi|_{U'}:U'\to U$ is finite;
\item if $k\ge 1$, $\pi$ is a holomorphic submersion, and for every $s\in S$ the restriction $u_s:=u|_{\pi^{-1}(s)}:\pi^{-1}(s)\to X_s$ is a holomorphic immersion.
\end{enumerate}
\end{definition}
If $(\Sigma,u)$ is a sub-fibration of relative dimension $k\ge 1$, then $u:\Sigma\to X$ is an immersion. Indeed, if $\D u(v)=0$, then $\D\pi(v)=\D f(\D u(v))=0$, so $v\in T_{\Sigma/S}$, and the fiberwise immersion hypothesis gives $v=0$.

\begin{remark}\label{rmk:u_smooth}
We discuss the structure of $\pi$ in several cases.
\begin{enumerate}[label=(\roman*)]
\item $\Sigma$ compact, $k=0$. In this case $\pi$ is a (possibly ramified) covering of finite degree. Local semi-finiteness in (b) is automatic. The ramification locus $R\subset\Sigma$ is a proper closed analytic subset, and $\pi:\Sigma\setminus \pi^{-1}(\pi(R))\to S\setminus\pi(R)$ is an unramified finite covering.
\item  $\Sigma$ noncompact, $k=0$. The fiber may contain infinitely many points. (b) is satisfied when $\pi$ is a (not necessarily finitely sheeted) ramified covering (\cite[\S 2]{AJS04}), i.e., every $s\in S$ admits a neighborhood $U$ such that for every connected component $U'$ of $\pi^{-1}(U)$ the restriction $\pi|_{U'}:U'\to U$ is an analytic covering.
\item $\Sigma$ compact, $k\ge 1$. By Ehresmann's theorem,
$\pi$ is a smooth fiber bundle with compact fibers
of complex dimension $k$.
\item $\Sigma$ noncompact, $k\ge 1$. $\pi$ is a holomorphic
submersion of relative dimension $k$ with possibly noncompact
fibers. The prototypical example is a holomorphic vector bundle $\pi:\Sigma\to S$, which has fibers $\CC^{k}$. In this case, by Calabi's theorem \cite[Th.~3.2]{calabi79}, $\Sigma$ admits a complete K\"ahler metric when $(S,\omega_S)$ is complete K\"ahler and the vector bundle admits a Hermitian metric $h$ whose Chern curvature $F_h$ satisfies $\langle F_h(\xi,\bar\xi)v,v\rangle_h \le c|\xi|_{\omega_S}^2|v|_h^2$ for some $c\ge0$, all $s\in S$, $\xi\in T_s S$, and $v\in\Sigma_s$.
\end{enumerate}
\end{remark}

We next address the existence of a K\"ahler form on $\Sigma$.
\begin{lemma}\label{lem:Sigma_kahler_k0}
If $k=0$ and $(S,\omega_S)$ is complete K\"ahler, then $\Sigma$ admits a complete K\"ahler metric.
\end{lemma}
\begin{proof}
This follows directly from \cite[Cor.~2]{AJS04}. In the unramified case one may take the complete pullback metric $\pi^*\omega_S$.
\end{proof}

\begin{lemma}\label{lem:Sigma_kahler_explicit}
Assume $\Sigma$ is compact, $k\ge 1$, and $u:\Sigma\to X$ is pseudo-holomorphic, i.e., $\D u\comp J_\Sigma=J\comp \D u$, where $J_\Sigma$ is the complex structure of $\Sigma$ and $J$ is the almost complex structure of $X$ determined by $\dbar$. Then there exists $\lambda_0>0$ such that for every $\lambda\ge\lambda_0$ the real $(1,1)$-form
\begin{equation}\label{eq:Sigma_kahler_explicit}
\omega_{\Sigma,\lambda}:=u^*\omega_X+\lambda\,\pi^*\omega_S
\end{equation}
is a Hermitian metric on $\Sigma$. If $\omega_S$ is K\"ahler and $\omega_X$ is closed, then $\omega_{\Sigma,\lambda}$ is K\"ahler.
\end{lemma}
\begin{proof}
Both $u^{*}\omega_{X}$ and $\pi^{*}\omega_{S}$ are real
$(1,1)$-forms on $\Sigma$ by pseudo-holomorphicity of $u$ and $\pi$, hence so is $\omega_{\Sigma,\lambda}$. It remains to prove that $\omega_{\Sigma,\lambda}$ is positive.

Equip $\Sigma$ with an auxiliary smooth Hermitian metric and let $\mathbb S\subset T\Sigma$ denote its unit sphere bundle, which is compact because $\Sigma$ is compact. Define continuous functions $a,b:\mathbb S\to\RR$ by
\[a(\sigma,v):=-\I\,(u^*\omega_X)(v,\bar v),\qquad b(\sigma,v):=-\I\,(\pi^*\omega_S)(v,\bar v)=-\I\,\omega_S\bigl(\D\pi(v), \overline{\D\pi(v)}\bigr).\]
Then $b\ge 0$, with $b(\sigma,v)=0$ if and only if $(\sigma,v)\in K:=\{(\sigma,v)\in\mathbb S:v\in T_{\Sigma/S,\sigma}\}$, the unit sphere bundle of $T_{\Sigma/S}\subset T\Sigma$. The set $K$ is compact. For $(\sigma,v)\in K$, the fiberwise immersion hypothesis (Definition \ref{def:subfib_datum}\,(c)) gives $\D u(v)\in T_{X/S,u(
\sigma)}\setminus\{0\}$, whence $a(\sigma,v)=-\I\omega_{X/S}(\D u
(v),\overline{\D u(v)})>0$. By compactness of $K$, $m_{0}:=\min_{K}a
>0$. By continuity of $a$, the open set $U:=\{(\sigma,v)\in\mathbb S: a(\sigma,v)>m_0/2\}$ contains $K$. It has compact complement $\mathbb{S}\setminus U$ on which $b>0$. If $\mathbb S\setminus U=\varnothing$, then $a>m_0/2$ on $\mathbb S$, so the conclusion holds with $\lambda_0=1$. Now assume that $\mathbb S\setminus U\neq\varnothing$. Let $b_{0}:=\min_{\mathbb{S}\setminus U}b>0$, $M:=\max(0,-\inf_{\mathbb{S}}a)<\infty$, and $\lambda_{0}:=\max(1,\,2(M+1)/b_{0})$. Suppose $\lambda\ge\lambda_{0}$. On $U$, $a+\lambda b\ge m_{0}/2$. On $\mathbb{S}\setminus U$, $a+\lambda b\ge -M+\lambda b_{0}\ge -M+2(M+1)>0$. Hence $a+\lambda b>0$ on $\mathbb{S}$, i.e., $\omega_{\Sigma,\lambda}$ is positive.

If $\omega_S$ is K\"ahler and $\omega_X$ is closed, then $u^{*}\omega_{X}$ and $\pi^{*}\omega_{S}$ are closed, and $\omega_{\Sigma,\lambda}$ is K\"ahler.
\end{proof}

\begin{convention}
For each sub-fibration $(\Sigma,u)$ in the sequel, we fix a Hermitian form $\omega_\Sigma$ on $\Sigma$.
\end{convention}

We consider the Cartesian square
\begin{equation}\label{eq:cartesian_subfib}
\begin{tikzcd}
\pi^*X \arrow[r, "\tilde\pi"] \arrow[d, "\pi^*\!f"'] & X \arrow[d, "f"]\\
\Sigma \arrow[r, "\pi"'] & S
\end{tikzcd}
\qquad\pi^*X:=\Sigma\times_S X=\{(\sigma,x)\in\Sigma\times X:\pi(\sigma)=f(x)\}.
\end{equation}

\begin{lemma}\label{lem:cartesian_smooth}
\begin{enumerate}[label=(\roman*)]
\item $\pi^*X$ is an almost complex submanifold of $\Sigma\times X$ of complex dimension $n+k+m$.
\item $(\pi^*\!f:\pi^*X\to\Sigma,T_{\pi^*X/\Sigma})$ is a complex fiber bundle with $T_{\pi^*X/\Sigma}\cong \tilde\pi^*T_{X/S}$ and the $\dbar$-operator $\pi^*\dbar$, which induces the almost complex structure on $\pi^*X$.
\item $\tilde\pi:\pi^*X\to X$ is a surjective pseudo-holomorphic map. For every $\sigma\in\Sigma$, the restriction $\tilde\pi|_{(\pi^*X)_\sigma}:(\pi^*X)_\sigma\xrightarrow{\;\cong\;} X_{\pi(\sigma)}$ is a biholomorphism, inducing a canonical isomorphism of complex vector bundles
\begin{equation}\label{eq:vert_iso_subfib}
T_{\pi^*X/\Sigma}\xrightarrow{\;\cong\;}\tilde\pi^*T_{X/S}.
\end{equation}
\end{enumerate}
\end{lemma}

\begin{proof}
(i) Write $\pi^*X=(\pi,f)^{-1}(\Delta_S)$, where $(\pi,f):\Sigma\times X\to S\times S$ and $\Delta_S\subset S\times S$ is the diagonal. The differential $\D(\pi,f)$ at $(\sigma,x)$ with $\pi(\sigma)=f(x)=s$ has image containing $\{0\}\oplus\D f(T_xX^\RR) =\{0\}\oplus T_sS^\RR$ in $T_sS^\RR\oplus T_sS^\RR$, while $T_{(s,s)}\Delta_S^\RR= \{(\xi,\xi):\xi\in T_sS^\RR\}$. Their sum is $T_{(s,s)}(S\times S)^\RR$, so $(\pi,f)\pitchfork\Delta_S$ at every point of $\pi^*X$, making $\pi^{*}X$ a smooth submanifold of real codimension $2n$. For $(\sigma,x)\in\pi^*X$, we have $T_{(\sigma,x)}(\pi^*X)^\RR=\{(v,w)\in T_\sigma\Sigma^\RR\oplus T_xX^\RR:\D\pi(v)=\D f(w)\}$. Since $\pi$ is holomorphic and $f$ is pseudo-holomorphic by the construction of $J$ from $\dbar$, this tangent space is invariant under $(v,w)\mapsto (J_\Sigma v,J_X w)$. Therefore $\pi^*X$ is an almost complex submanifold of $\Sigma\times X$.

(ii) Clearly $\pi^* f$ is a smooth fiber bundle. The vertical tangent space is $T_{\pi^*X/\Sigma,(\sigma,x)}^\RR=\{0\}\oplus T_{X/S,x}^\RR$, which gives the canonical isomorphism $T_{\pi^*X/\Sigma}^\RR\cong\tilde\pi^*T_{X/S}^\RR$. The fiberwise complex structure is transported through this isomorphism. For
$\bar v\in \widebar{T_\sigma\Sigma}$, choose $\bar w\in\widebar{T_xX}$ with $\D f(\bar w)=\D\pi(\bar v)$, and set $(\pi^*\dbar)(\bar v):=[(\bar v,\bar w)]$ in $T(\pi^*X)^\CC/\widebar{T_{\pi^*X/\Sigma}}$. This is independent of the choice of $\bar w$, since two choices differ by an element of $\widebar{T_{X/S}}$, hence by an element of $\widebar{T_{\pi^*X/\Sigma}}$ under the above vertical identification. The almost complex structure induced by $\pi^*\dbar$ is the restriction of $J_\Sigma\oplus J$ to $\pi^*X$.

(iii) Surjectivity follows from surjectivity of $\pi$. $\tilde\pi$ is pseudo-holomorphic since it is the restriction of the projection $(\Sigma\times X,J_\Sigma\oplus J)\to (X,J)$.
For fixed $\sigma\in\Sigma$, $(\pi^*X)_\sigma=\{\sigma\}\times X_{\pi(\sigma)}$, and $\tilde\pi|_{(\pi^*X)_\sigma}(\sigma,x)=x$ is a biholomorphism onto $X_{\pi(\sigma)}$.
\end{proof}
\begin{remark}\label{rmk:cartesian_integrable}
If $\dbar$ is integrable, then $X$ is a complex manifold and $f$ is a holomorphic fibration. In this case, $\pi^*X$ is a complex manifold, $\pi^*\!f$ is a holomorphic fibration, $\tilde\pi$ is a holomorphic map, and \eqref{eq:vert_iso_subfib} is an isomorphism of holomorphic vector bundles.
\end{remark}

\begin{remark}\label{rmk:tilde_pi_ramified}
For $k=0$ with $\pi$ ramified, $\tilde\pi:\pi^*X\to X$ is not a local diffeomorphism at points of $(\pi^*\!f)^{-1}(R)$, since its differential drops rank there. The fiberwise biholomorphism \eqref{eq:vert_iso_subfib} nevertheless persists over all of $\pi^*X$, because fiberwise the map is the identity onto
$X_{\pi(\sigma)}$.
\end{remark}

The map $u:\Sigma\to X$ determines a canonical smooth section of $\pi^*\!f$:
\begin{equation}\label{eq:section_tilde_u}
\tilde u:\Sigma\to\pi^*X,\qquad \tilde u(\sigma):=(\sigma,u(\sigma)), \qquad u=\tilde\pi\comp\tilde u.
\end{equation}
$\tilde u$ is pseudo-holomorphic if and only if $u$ is. Pulling back \eqref{eq:vert_iso_subfib} along $\tilde u$ yields the canonical smooth bundle isomorphism
\begin{equation}\label{eq:pullback_ident}
E:=u^*T_{X/S}\xrightarrow{\;\cong\;}\tilde u^*T_{\pi^*X/\Sigma},
\end{equation}
which we use throughout to transport tensorial objects between $E$ and $\tilde u^*T_{\pi^*X/\Sigma}$ without further comment.

\begin{lemma}\label{lem:pullback_basic_subfib}
The pulled-back data
\[(\pi^*\!f,\;\pi^*\dbar,\;\pi^*\theta,\;\pi^*\omega_{X/S},\;\tilde\pi^*\omega_X,\;\pi^*\nabla^{1,0})\]
endow $\pi^*\!f:\pi^*X\to\Sigma$ with a $\dbar$-operator, a relatively holomorphic almost Higgs field, a $\pi^*\theta$-adapted fiberwise K\"ahler metric, a real $(1,1)$-form on $\pi^* X$ (with respect to the almost complex structure on $\pi^* X$ induced by $\pi^*\dbar$) which induces a symplectic connection on $\pi^* f$. Moreover, the horizontal part of $\tilde\pi^*\omega_X$ with respect to $\pi^*\nabla^{1,0}$ satisfies, for every $(\sigma,x)\in\pi^*X$ and $v_1,v_2\in T_\sigma\Sigma^\CC$,
\begin{equation}\label{eq:pull_horizontal}
(\tilde\pi^*\omega_X)^H(\sigma,x)(v_1,v_2)=\omega_X^H(x)\bigl(\D\pi(v_1),\D\pi(v_2)\bigr).
\end{equation}
\end{lemma}

\begin{proof}
The pullback of fiberwise structures along the biholomorphism $\tilde\pi:(\pi^{*}X)_{\sigma}\to X_{\pi(\sigma)}$ preserves all fiberwise tensors and the K\"ahler property, hence $(\pi^*\theta)_{(\sigma,x)}(v):=\theta_x(\D\pi(v))$ ($v\in T_\sigma\Sigma$) under \eqref{eq:vert_iso_subfib} defines a relatively holomorphic almost Higgs field on $\pi^*f$ and $\pi^{*}\omega_{X/S}$ is a $\pi^{*}\theta$-adapted fiberwise K\"ahler metric.

$w\in T_{(\sigma,x)}(\pi^*X)^\RR$ is horizontal for $\pi^*\nabla^\RR$ if and only if $\D\tilde\pi(w)\in T_xX^\RR$ is horizontal for $\nabla^\RR$ at $x$. In the explicit description $\pi^*X\subset\Sigma\times X$, the projections give $\D\tilde\pi(v',w')=w'$ and $\D(\pi^*\!f)(v',w')=v'$ for $(v',w')\in T_{(\sigma,x)}(\pi^*X)$. The horizontal lift of $v\in T_\sigma\Sigma^\RR$ at $(\sigma,x)$ is therefore the unique $(v',w')\in T_{(\sigma,x)}(\pi^*X)$ with $v'=v$ and $w'$ horizontal. The constraint $\D\pi(v')=\D f(w')$ defining $T_{(\sigma,x)}(\pi^*X)$ becomes $\D f(w')=\D\pi(v)$. Combined with horizontality, this forces $w'=H_x(\D\pi(v))$. Thus
\begin{equation}\label{eq:horizontal_lift_pullback}
\widetilde{H}_{(\sigma,x)}(v)=\bigl(v,\,H_x(\D\pi(v))\bigr),\qquad \D\tilde\pi\bigl(\widetilde{H}_{(\sigma,x)}(v)\bigr)=H_x(\D\pi(v)).
\end{equation}
Complexifying and using the definition of $(\,\cdot\,)^H$,
\[(\tilde\pi^*\omega_X)^H(\sigma,x)(v_1,v_2)=(\tilde\pi^*\omega_X)\bigl(\widetilde{H}_{(\sigma,x)}(v_1),\widetilde{H}_{(\sigma,x)}(v_2)\bigr)=\omega_X\bigl(H_x(\D\pi(v_1)),H_x(\D\pi(v_2))\bigr),\]
which equals $\omega_X^H(x)(\D\pi(v_1),\D\pi(v_2))$ by definition.
\end{proof}

\begin{definition}\label{def:F_N}
The \emph{relative tangent subbundle} of $(\Sigma,u)$ is
\[F:=\D u\bigl(T_{\Sigma/S}\bigr)\subset E=u^*T_{X/S},\]
where $T_{\Sigma/S}$ denotes the relative tangent bundle of $\pi$, defined as $\ker(\D\pi:T\Sigma\to\pi^*TS)$ for $k\ge 1$, and as the zero bundle when $k=0$. The \emph{relative normal bundle} is the smooth complex quotient $N:=E/F$, of rank $m-k$. Denote by $\mathrm{pr}_N:E\to N$ the canonical projection.
\end{definition}

\begin{lemma}\label{lem:F_subbundle}
For $k\ge 1$, $F$ is a smooth complex subbundle of $E$ of rank $k$. If, in addition, $u$ is pseudo-holomorphic, then $F$ is invariant under $\dbar_E:=u^*\dbar_{T_{X/S}}$, where $\dbar_{T_{X/S}}$ is the canonical $\dbar$-operator on $T_{X/S}\to X$ induced by the $\dbar$-operator on $f$ (Definition \ref{def:canonical_dbar_TXS}). When the $\dbar$-operator on $f$ is integrable and $u$ is holomorphic, $E$ is a holomorphic bundle and $F$ is a holomorphic subbundle of $E$.
\end{lemma}

\begin{proof}
For $k\ge 1$, $\pi$ is a holomorphic submersion, so $T_{\Sigma/S}\subset T\Sigma$ is a holomorphic subbundle of rank $k$. The fiberwise immersion hypothesis ensures that $\D^{\mathrm v}u:=\D u|_{T_{\Sigma/S}}:T_{\Sigma/S}\to E$ is a smooth complex-linear bundle monomorphism of constant rank $k$, so $F=\im(\D^{\mathrm v}u)$ is a smooth complex subbundle of $E$ of rank $k$.

Now assume that $u$ is pseudo-holomorphic. Work locally on $\Sigma$ with holomorphic coordinates $(s^1,\ldots,s^n,t^1,\ldots,t^k)$ adapted to $\pi$, so that $\pi(s,t)=s$. Choose adapted local coordinates $(s^i,z^\alpha)$ on $X$. Write $u(s,t)=(s,u^\alpha(s,t))$. Since $u_s:\Sigma_s\to X_s$ is holomorphic, the vectors $V_a:=\frac{\partial u^\alpha}{\partial t^a}\,\partial_\alpha\big|_{u(s,t)}$ ($a=1,\ldots,k$) form a local smooth frame of $F$. Locally, we write $\dbar(\partial_{\bar i})=\bigl[\partial_{\bar i}+\Gamma_{\bar i}^{\alpha}\partial_\alpha\bigr] \mod \widebar{T_{X/S}}$. For a $(0,1)$-vector field $\partial_{\bar{\mathsf{a}}}$ on $\Sigma$, set $\Gamma_{\bar{\mathsf{a}}}^{\alpha}:=\frac{\partial\bar\pi^{ i}}{\partial \bar\sigma^{\mathsf{a}}}\Gamma_{\bar i}^{\alpha}$. The pseudo-holomorphicity of $u$ is locally equivalent to
\begin{equation}\label{eq:pseudo_hol_local}
\partial_{\bar{\mathsf{a}}}u^\alpha=\Gamma_{\bar{\mathsf{a}}}^{\alpha}(u(s,t)).
\end{equation}
Let $e_\alpha:=\partial_\alpha\comp u$ be the pulled-back local frame of $E=u^*T_{X/S}$, so $V_a=(\partial_a u^\alpha)\,e_\alpha$. By the Leibniz rule, $\dbar_E V_a=(\dbar_\Sigma\,\partial_a u^\alpha)\otimes e_\alpha+(\partial_a u^\alpha)\,\dbar_E e_\alpha$. By \eqref{eq:induced_dbar},
\begin{equation}\label{eq:induced_dbar_frame}
\dbar_{T_{X/S},\partial_{\bar\beta}}\partial_\alpha=0,\qquad \dbar_{T_{X/S},\widebar H_i}\partial_\alpha=-(\partial_\alpha\Gamma_{\bar i}^\gamma)\,\partial_\gamma,\qquad \widebar H_i:=\partial_{\bar i}+\Gamma_{\bar i}^\beta\partial_\beta.
\end{equation}
By \eqref{eq:pseudo_hol_local} and $\partial_{\bar{\mathsf a}}\pi^i=0$, we have
\begin{equation}\label{eq:du_subfib_decomp}
\D u(\partial_{\bar{\mathsf a}})=\tfrac{\partial\bar\pi^i}{\partial\bar\sigma^{\mathsf a}}\,\widebar H_i+(\partial_{\bar{\mathsf a}}\bar u^\beta)\,\partial_{\bar\beta},
\end{equation}
because its $T_{X/S}$-component $(\partial_{\bar{\mathsf a}}u^\alpha)\partial_\alpha=\tfrac{\partial\bar\pi^i}{\partial\bar\sigma^{\mathsf a}}\Gamma_{\bar i}^\alpha\partial_\alpha$ equals that of $\tfrac{\partial\bar\pi^i}{\partial\bar\sigma^{\mathsf a}}\widebar H_i$. By \eqref{eq:induced_dbar_frame} and \eqref{eq:du_subfib_decomp},
\[
(\dbar_E e_\alpha)(\partial_{\bar{\mathsf a}})=\tfrac{\partial\bar\pi^i}{\partial\bar\sigma^{\mathsf a}}\bigl(\dbar_{T_{X/S},\widebar H_i}\partial_\alpha\bigr)\big|_{u}=-\tfrac{\partial\bar\pi^i}{\partial\bar\sigma^{\mathsf a}}(\partial_\alpha\Gamma_{\bar i}^\gamma)\big|_{u}\,e_\gamma=-(\partial_\alpha\Gamma_{\bar{\mathsf a}}^\gamma)\big|_{u}\,e_\gamma,
\]
using $\Gamma_{\bar{\mathsf a}}^\gamma=\tfrac{\partial\bar\pi^i}{\partial\bar\sigma^{\mathsf a}}\Gamma_{\bar i}^\gamma$ ($\tfrac{\partial\bar\pi^i}{\partial\bar\sigma^{\mathsf a}}$ being constant along the fibers). Substituting into the Leibniz identity,
\[
(\dbar_E V_a)_{\bar{\mathsf{a}}}^{\alpha}=\partial_{\bar{\mathsf{a}}}\partial_a u^\alpha-(\partial_a u^\gamma)\,\partial_\gamma\Gamma_{\bar{\mathsf{a}}}^{\alpha}\big|_{u}.
\]
Differentiating \eqref{eq:pseudo_hol_local} with respect to $t^a$ gives $\partial_{\bar{\mathsf{a}}}\partial_a u^\alpha=(\partial_a u^\gamma)\partial_\gamma\Gamma_{\bar{\mathsf{a}}}^{\alpha}\big|_{u(s,t)}$. Hence $\dbar_E V_a=0$ for every $a$, and $F$ is a $\dbar_E$-invariant subbundle.

If $\dbar$ is integrable, then $f$ is a holomorphic fibration, and $\dbar_{T_{X/S}}$ is the usual Dolbeault operator on the holomorphic vector bundle $T_{X/S}\to X$ by Lemma \ref{lem:canonical_dbar_TXS_valid}. If $u$ is holomorphic, the above $\dbar_E$ is the Dolbeault operator of $E=u^*T_{X/S}$. The $\dbar_E$-invariance of $F$ means that $F$ is a holomorphic subbundle.
\end{proof}

The Hermitian metric $h_E:=u^*h_{X/S}$ on $E$ restricts to a smooth Hermitian metric $h_F$ on $F$. Let $F^\perp\subset E$ be the $h_E$-orthogonal complement. The projection $\mathrm{pr}_{N}|_{F^{\perp}}:F^{\perp}\xrightarrow{\;\cong\;}N$ is a smooth bundle isomorphism, and transporting $h_{E}|_{F^{\perp}}$ through this isomorphism gives a smooth Hermitian metric $h_{N}$ on $N$. We write $\mathrm{pr}_F:E\to F$ for the $h_E$-orthogonal projection. For $\alpha\in A^{\ph}(\Sigma,E)$, the components valued in $F$ and $N$ are denoted by $\alpha_{F}$ and $\alpha_{N}$ respectively.

\subsection{Higgs sub-fibrations and flat sub-fibrations}\label{subsec:defs}
For the smooth section $\tilde u:\Sigma\to\pi^*X$ of $\pi^*\!f$, Definitions \ref{def:cov_deriv_complex}, \ref{def:dbar_on_section}, \ref{def:partial_on_section}, and \ref{def:higgs_on_section} of Section \ref{sec:diff_op} produce $\tilde u^*T_{\pi^*X/\Sigma}$-valued forms on $\Sigma$. Transporting them via the canonical isomorphism \eqref{eq:pullback_ident}, we obtain $E$-valued forms
\begin{equation}\label{eq:op_subfib}
\partial u\in A^{1,0}(\Sigma,E),\quad \dbar u\in A^{0,1}(\Sigma,E),\quad \theta u\in A^{1,0}(\Sigma,E),\quad \bar\theta u\in A^{0,1}(\Sigma,E),
\end{equation}
and we set
\[ D''u:=\dbar u-\theta u,\qquad D'u:=\partial u-\bar\theta u,\qquad Du:=D'u+D''u. \]
When $\Sigma=S$ and $\pi=\id_S$, these reduce to the operators of Section \ref{sec:kahler_id}.

\begin{definition}\label{def:higgs_flat_subfib}
A sub-fibration $(\Sigma,u)$ for $f$ is called a
\begin{enumerate}[label=(\roman*)]
\item \emph{Higgs sub-fibration} if
\begin{equation}\label{eq:higgs_subfib}
(D''u)_N=0\in A^1(\Sigma,N), \quad\text{i.e.\ }\quad D''u\in A^1(\Sigma,F);
\end{equation}
\item \emph{flat sub-fibration} if
\begin{equation}\label{eq:flat_subfib}
(Du)_N=0\in A^1(\Sigma,N), \quad\text{i.e.\ }\quad Du\in A^1(\Sigma,F).
\end{equation}
\end{enumerate}
\end{definition}

\begin{remark}\label{rmk:holo_not_forced}
\eqref{eq:higgs_subfib} does not imply the pseudo-holomorphicity of $u$ when $k\ge 1$. For example, take $S=\CC$, $X=\CC^3$ with coordinates $(s,z_1,z_2)$ and $f(s,z_1,z_2)=s$, let $\Sigma=\CC^2$ with coordinates $(s,\sigma)$, and set $u(s,\sigma)=(s,\sigma+\bar s,0)$. With the trivial connection and $\theta=0$, the relative tangent bundle is generated by $\partial_{z_1}$, while $D''u=\D\bar s\otimes\partial_{z_1}$ is tangent to the sub-fibration. Hence $(D''u)_N=0$, although $u$ is
not pseudo-holomorphic. For $k=0$, $F=0$ and $\dbar u\in A^1(\Sigma,F)$ forces $\dbar u=0$, hence $u$ is pseudo-holomorphic by Remark \ref{rmk:dbar_hol_section}.
\end{remark}

\begin{example}\label{ex:higgs_subbundle}
Let $f: E \to S$ be a holomorphic vector bundle equipped with a holomorphic Higgs field $\theta_E \in A^{1,0}(S, \End E)$. By \eqref{eq:nonlin_higgs_vb}, it induces a nonlinear Higgs field $\tilde{\theta} \in A^{1,0}(S, f_*T_{E/S})$.

Let $\mathcal F \subset E$ be a smooth complex vector subbundle of rank $k \ge 1$ admitting a holomorphic vector bundle structure (so Definition \ref{def:subfib_datum}\,(a) is satisfied), and let $u: \mathcal F \hookrightarrow E$ denote the inclusion. Then $\pi := f|_{\mathcal F}: \mathcal F \to S$ makes $(\mathcal F, u)$ a sub-fibration of relative dimension $k$. Under the canonical identification \eqref{eq:pullback_ident_vb}, $u^*T_{E/S}$ is identified with $\pi^*E$, the relative tangent subbundle of Definition \ref{def:F_N} is identified with $\pi^*\mathcal F\subset \pi^*E$, and the relative normal bundle is identified with $\pi^*(E/\mathcal F)$.

We claim that $(\mathcal F,u)$ is a Higgs sub-fibration if and only if $\mathcal F$ is a classical holomorphic Higgs subbundle of $(E,\theta_E)$, that is, $\mathcal F$ is preserved by $\dbar_E$ and $\theta_E(\mathcal F)\subset \Omega_S^1\otimes \mathcal F$.

To verify this, fix $s_0\in S$ and choose a local holomorphic frame $\{e_\alpha\}_{\alpha=1}^m$ of $E$ near $s_0$, and a local frame $\{\tilde e_\alpha\}_{\alpha=1}^k$ of $\mathcal F$ that is holomorphic with respect to the given holomorphic structure on $\mathcal F$. Write $\tilde e_\alpha = c_\alpha^\beta(s,\bar s)\,e_\beta$, $c_\alpha^\beta\in C^\infty$, with the matrix $(c_\alpha^\beta)$ of rank $k$ at every point. Let $(y^1,\ldots,y^k)$ be the linear coordinates on the fibers of $\mathcal F$ dual to $\{\tilde e_\alpha\}$. Combined with holomorphic coordinates $(s^i)$ on $S$, these give holomorphic coordinates $(s^i,y^\alpha)$ on the complex manifold $\Sigma=\mathcal F$. In these coordinates, the inclusion $u$ has components $u^\beta(s,y)=c_\alpha^\beta(s,\bar s)\,y^\alpha$. Since $\{e_\beta\}$ is a holomorphic frame of $E$, by \eqref{eq:dbar_section_local},
\[
\dbar u=(\partial_{\bar s^i}u^\beta)\,\D\bar s^i\otimes\partial_{z^\beta}\big|_{u(\sigma)}=(\partial_{\bar s^i}c_\alpha^\beta)\,y^\alpha\,\D\bar s^i\otimes\partial_{z^\beta}\big|_{u(\sigma)}.
\]
Under the canonical isomorphism $\partial_{z^\beta}|_{u(\sigma)}\leftrightarrow e_\beta(s)$ of Lemma \ref{lem:canonical_identification},
\[
\dbar u \leftrightarrow (\partial_{\bar s^i}c_\alpha^\beta)\,y^\alpha\,\D\bar s^i\otimes e_\beta\in A^{0,1}(\Sigma,\pi^*E).
\]
The condition $(\dbar u)_N=0$ requires this section to lie in $\pi^*\mathcal F\subset\pi^*E$ at every $\sigma\in \mathcal F$. Since $\mathcal F_s=\mathrm{span}_{\CC}\{c_\alpha^\gamma(s,\bar s)\,e_\gamma(s)\,|\,\alpha=1,\ldots,k\}$, varying $y$ in $\CC^k$ shows this is equivalent to
\[
(\partial_{\bar s^i}c_\alpha^\beta)\,e_\beta\in \mathcal F_s,\qquad\text{for all }s,i,\alpha,
\]
which is precisely the condition that the smooth subbundle $\mathcal F\subset E$ is $\dbar_E$-invariant, i.e., a holomorphic subbundle of $E$. For the Higgs part, by \eqref{eq:nonlin_higgs_vb_local},
\[\tilde\theta u=-\theta_{i,\beta}^\alpha\,u^\beta\,\D s^i\otimes\partial_{z^\alpha}\big|_{u(\sigma)}=-\theta_{i,\beta}^\alpha\,c_\gamma^\beta\,y^\gamma\,\D s^i\otimes\partial_{z^\alpha}\big|_{u(\sigma)},\]
which under the canonical isomorphism corresponds to $-\theta_{i,\beta}^\alpha\,c_\gamma^\beta\,y^\gamma\,\D s^i\otimes e_\alpha\in A^{1,0}(\Sigma,\pi^*E)$. Varying $y$ shows that $(\tilde\theta u)_N=0$ is equivalent to $\theta_E(\partial_i)(\mathcal F_s)\subset \mathcal F_s$ for all $s,i$, i.e., $\theta_E(\mathcal F)\subset\Omega_S^1\otimes \mathcal F$. Note that both conditions $(\dbar u)_N=0$ and $(\tilde\theta u)_N=0$ are intrinsic to the smooth subbundle $\mathcal F\subset E$ and are independent of the auxiliary holomorphic vector bundle structure chosen on $\mathcal F$.

A similar argument shows that $(\mathcal F,u)$ is a flat sub-fibration of $(E,D_E)$ if and only if $\mathcal F$ is a $D_E$-invariant smooth subbundle of $E$, i.e., a flat subbundle of $(E,D_E)$ in the classical sense.
\end{example}

\subsection{Associated sub-fibrations}\label{subsec:assoc_subfib}
Let $Y$ be a connected complex manifold and $G\subset \Aut_0(Y)$ a connected complex Lie group with Lie algebra $\mathfrak{g}$. As in \cite[Eq.~(2.29)]{LS26}, the infinitesimal action gives the injective Lie algebra homomorphism $\tau_0:\mathfrak{g}\to H^0(Y,TY)$. Let $\pi_P:P\to S$ be a holomorphic principal $G$-bundle. The associated fiber bundle $f:X:=P\times_G Y\to S$ is a holomorphic fiber bundle with typical fiber $Y$, and $\tau_0$ induces an injective morphism of sheaves $\tau:\ad P\to f_*^{\mathrm{hol}}T_{X/S}$.

Let $\theta_P\in H^0(S,\Omega_S^1\otimes\ad P)$ be a Higgs field on $P$, which satisfies $[\theta_P,\theta_P]=0$. The induced Higgs field on $f$ is $\theta:=\tau(\theta_P)$.

\begin{definition}\label{def:higgs_subbundle_princ}
A \emph{Higgs subbundle} of the principal Higgs bundle $(P,\theta_P)$ is a holomorphic reduction $P'\subset P$ of the structure group to a closed complex Lie subgroup $G'\subset G$ such that
\begin{equation}\label{eq:higgs_subbundle_def}
\theta_P\in H^0(S,\Omega_S^1\otimes\ad P')\subset H^0(S,\Omega_S^1\otimes\ad P).
\end{equation}
\end{definition}

\begin{lemma}\label{lem:g_prime_tangent_princ}
Let $Y'\subset Y$ be a closed complex submanifold whose stabilizer $G':=\Stab_G(Y')=\{g\in G:g\cdot Y'=Y'\}$ is a closed complex Lie subgroup of $G$ with Lie algebra $\mathfrak{g}'$. Then
\begin{equation}\label{eq:g_prime_image_princ}
\tau_0(\mathfrak{g}')=\bigl\{V\in\tau_0(\mathfrak{g}):V|_{Y'}\subset TY'\bigr\}.
\end{equation}
\end{lemma}

\begin{proof}
For $\xi\in\mathfrak{g}'$, the one-parameter subgroup $\{\exp(t\xi)\}_{t\in\RR}$ lies in $G'$, hence preserves $Y'$. For every $y\in Y'$, the curve $t\mapsto\exp(t\xi)\cdot y$ remains in $Y'$, so the $(1,0)$-part of its tangent at $t=0$ lies in $T_y{Y'}$. Conversely, suppose $\xi\in\mathfrak{g}$ satisfies $\tau_0(\xi)|_y\in T_yY'$ for every $y\in Y'$. Then the restriction $\tau_0(\xi)|_{Y'}$ is a holomorphic vector field on $Y'$. For $y\in Y'$, the curve $\gamma(t):=\exp(-t\xi)\cdot y$ on $Y$ is an integral curve of $\tau_0(\xi)^\RR$ starting at $y\in Y'$. For $y\in Y'$, set $I_y:=\{t\in\RR\,|\,\exp(-t\xi)\cdot y\in Y'\}$, which is nonempty. $I_y$ is open by the local existence and uniqueness of integral curves, and is closed since $Y'$ is closed in $Y$, so $I_y=\RR$. Hence $\exp(t\xi)\cdot Y'\subset Y'$ for $t\in \RR$. Applying the same argument to $-\xi$ yields $\exp(t\xi)\cdot Y'=Y'$, and thus $\exp(t\xi)\in G'$ and $\xi\in\mathfrak{g}'$.
\end{proof}

\begin{proposition}\label{prop:higgs_subbundle_subfib}
Let $P'\subset P$ be a Higgs subbundle with structure group $G'\subset G$, and let $Y'\subset Y$ be a connected closed $k$-dimensional complex submanifold with $\Stab_G(Y')=G'$. Set $\Sigma:=P'\times_{G'}Y'$, and let $u:\Sigma\to X=P\times_G Y$ be the holomorphic map induced by the inclusions $P'\hookrightarrow P$ and $Y'\hookrightarrow Y$. Then $(\Sigma,u)$ is a Higgs sub-fibration of $(f,\theta)$ of relative dimension $k$.
\end{proposition}

\begin{proof}
Since $P'$ is a holomorphic principal $G'$-bundle, $G'$ is a complex Lie group, and $Y'$ is $G'$-invariant, the associated bundle $\Sigma=P'\times_{G'}Y'\to S$ is a holomorphic fiber bundle with typical fiber $Y'$, hence an $(n+k)$-dimensional complex manifold. The map $\pi:=f\comp u:\Sigma\to S$ is a holomorphic submersion when $k\ge 1$; when $k=0$, $Y'$ is a single point and $\pi:\Sigma=P'/G'\xrightarrow{\;\cong\;}S$ is a biholomorphism. A holomorphic local section $\varsigma$ of $P'$ over $U\subset S$ yields trivializations
\[
\Sigma|_U\xrightarrow{\;\cong\;}U\times Y',\qquad X|_U\xrightarrow{\;\cong\;}U\times Y,
\]
in which $u$ becomes $\id_U\times\iota_{Y'}$ for $\iota_{Y'}:Y'\hookrightarrow Y$ the inclusion. Hence each $u_s:\pi^{-1}(s)\to X_s$ is a holomorphic immersion, and Definition \ref{def:subfib_datum} is satisfied.

Since $u$ is holomorphic, the lift $\tilde u=(\id_\Sigma,u):\Sigma\to\pi^*X$ is also holomorphic. It remains to show that $\theta u\in A^{1,0}(\Sigma,F)$, where $F=\D u(T_{\Sigma/S})\subset E$. Fix $\sigma\in\Sigma$, set $s:=\pi(\sigma)$ and $x:=u(\sigma)\in X_s$. In the trivializations above induced by a local holomorphic section of $P'$ over $U\ni s$, write $\theta_P|_U=a_i(s)\,\D s^i$ with $a_i:U\to\mathfrak{g}$ holomorphic.
The hypothesis $\theta_P\in H^0(U,\Omega_U^1\otimes\ad P')$ means $a_i(s)\in\mathfrak{g}'$ for every $s\in U$. In the corresponding trivialization of $X|_U\cong U\times Y$, we have
\[
\theta(\partial_{s^i})|_x=\tau_0(a_i(s))|_{y'}\in T_{y'}Y',\qquad x=(s,y')\in U\times Y',
\]
by Lemma \ref{lem:g_prime_tangent_princ}. In the trivialization of $\Sigma$, $T_{\Sigma/S,\sigma}=T_{y'}Y'$, and $\D u|_{T_{\Sigma/S}}$ is the canonical inclusion $T_{y'}Y'\hookrightarrow T_{y'}Y=T_{X/S,x}$. Hence $ F_\sigma=T_{y'}Y'$, and $(\theta u)(\partial_{s^i})|_\sigma\in F_\sigma$ for every $i$. Therefore, $\theta u\in A^{1,0}(\Sigma,F)$, and $(\theta u)_N=0$.
\end{proof}
Suppose that $P$ is the holomorphic flat principal $G$-bundle associated with a representation $\rho:\pi_1(S)\to G$, and let $\nabla_P$ be its canonical flat connection.

\begin{definition}\label{def:flat_subbundle_princ}
A \emph{flat subbundle} of $(P,\nabla_P)$ is a reduction of structure group $P'\subset P$ to a closed complex Lie subgroup $G'\subset G$ such that the connection $\nabla_P$ restricts to a connection on $P'$. Equivalently, the connection $1$-form of $\nabla_P$ in any local section of $P'$ takes values in $\mathfrak{g}'\subset\mathfrak{g}$.
\end{definition}

Flat subbundles of $(P,\nabla_P)$ arise from representations $\rho':\pi_1(S)\to G'$ whose composition with $G'\hookrightarrow G$ is conjugate to $\rho$. Similarly to Proposition \ref{prop:higgs_subbundle_subfib}, we have the following.

\begin{proposition}\label{prop:flat_subbundle_subfib}
Let $P'\subset P$ be a flat subbundle with structure group $G'\subset G$. Let $Y'$, $Y$, $\Sigma$, and $u$ be as in Proposition \ref{prop:higgs_subbundle_subfib}. Then $(\Sigma,u)$ is a flat sub-fibration of $(f,\nabla)$ of relative dimension $k$, where $\nabla$ is induced by $\nabla_P$.
\end{proposition}
\begin{proof}
Choose a local section of $P'$ over an open set $U\subset S$. In the resulting trivializations $P|_U\cong U\times G$, $X|_U\cong U\times Y$, the connection form of $\nabla_P$ takes values in $\mathfrak g'$ because $P'$ is a flat reduction. Hence the horizontal lift of a vector field $v$ on $U$ is of the form $v+\tau_0(A(v))$, where $A(v)\in\mathfrak g'$. By Lemma \ref{lem:g_prime_tangent_princ}, $\tau_0(A(v))$ is tangent to $Y'$. Thus the horizontal lift is tangent along $\Sigma|_U\cong U\times Y'\subset U\times Y\cong X|_U$. $\D u$ maps $T_{\Sigma/S}$ into the relative tangent bundle $F$. Consequently $(\nabla u)_N=0$, and $(\Sigma,u)$ is a flat sub-fibration.
\end{proof}

\subsection{The correspondence between sub-fibrations}\label{subsec:corresp_subfib}
In this subsection, we establish the correspondence between Higgs sub-fibrations and flat sub-fibrations.
\begin{definition}\label{def:omega_F}
For a sub-fibration $(\Sigma,u)$, define the smooth real $(1,1)$-forms $\omega_{F,\partial},\omega_{F,\dbar}$ on $\Sigma$ by
\begin{align}
\omega_{F,\partial}(\xi,\bar\eta)&:=\omega_{X/S}\big|_{u(\sigma)}\bigl((\partial u)_F(\xi),\,\overline{(\partial u)_F(\eta)}\bigr),\label{eq:omega_F_partial_def}\\
\omega_{F,\dbar}(\xi,\bar\eta)&:=\omega_{X/S}\big|_{u(\sigma)}\bigl((\dbar u)_F(\bar\eta),\,\overline{(\dbar u)_F(\bar\xi)}\bigr),\label{eq:omega_F_dbar_def}
\end{align}
where $\xi,\eta\in T_\sigma\Sigma$. For $k=0$, $F=0$ and $\omega_{F,\partial}=\omega_{F,\dbar}=0$.
\end{definition}

\begin{lemma}\label{lem:omega_F_pos}
$-\I\,\omega_{F,\partial}(\xi,\bar\xi)\ge 0$ and $-\I\,\omega_{F,\dbar}(\xi,\bar\xi)\ge 0$ for every $\xi\in T\Sigma$, and
\begin{equation}\label{eq:Lambda_omega_F_pair}
\Lambda_{\omega_\Sigma}\omega_{F,\partial}=|(\partial u)_F|^2_{h_F},\qquad\Lambda_{\omega_\Sigma}\omega_{F,\dbar}=|(\dbar u)_F|^2_{h_F}.
\end{equation}
\end{lemma}
\begin{proof}
Positivity of $\omega_{X/S}|_{u(\sigma)}$ on $T_{u(\sigma)}X_{\pi(\sigma)}$ gives the semi-positivity of $-\I\omega_{F,\partial}$ (resp.\ $-\I\omega_{F,\dbar}$), with strict positivity exactly where $(\partial u)_F(\xi)\ne 0$ (resp.\ $(\dbar u)_F(\bar\xi)\ne 0$). In a local unitary frame of $T\Sigma$ with respect to $\omega_\Sigma$ at a point, both sides of each equation in \eqref{eq:Lambda_omega_F_pair} equal $\sum_i |(\partial u)_F(\partial_{\sigma^i})|_{h_F}^2$ (resp.\ $\sum_i|(\dbar u)_F(\partial_{\bar\sigma^i})|_{h_F}^2$).
\end{proof}

\begin{definition}\label{def:deg_subfib}
Let $(\Sigma,u)$ be a compact sub-fibration of relative dimension $k$. For fixed Hermitian metric $\omega_\Sigma$ and the
standing data $(\omega_X,\omega_{X/S},\theta)$, define
\begin{align}
\deg_{\omega_X}(\Sigma,u)&:=\int_\Sigma u^*\omega_X\wedge \omega_\Sigma^{n+k-1},\label{eq:deg_omegaX_subfib}\\
\deg_{F,\partial}(\Sigma,u)&:=\int_\Sigma\omega_{F,\partial}\wedge\omega_\Sigma^{n+k-1},\label{eq:deg_F_partial_subfib}\\
\deg_{F,\dbar}(\Sigma,u)&:=\int_\Sigma\omega_{F,\dbar}\wedge\omega_\Sigma^{n+k-1},\label{eq:deg_F_dbar_subfib}\\
\deg_{\theta,F}(\Sigma,u)&:=\frac{1}{n+k}\int_\Sigma\bigl(|(\theta u)_F|^2-|(\bar\theta u)_F|^2\bigr)\,\omega_\Sigma^{n+k}.\label{eq:deg_F_theta_subfib}
\end{align}
The ($\zeta$-) \emph{combined degree} of the sub-fibration is
\begin{align}
\deg_\zeta(\Sigma,u)&:=\deg(\Sigma,u)-\int_\Sigma \pi^*\zeta\wedge \omega_\Sigma^{n+k-1},\\
\deg(\Sigma,u)&:=\deg_{\omega_X}(\Sigma,u)-\deg_{F,\partial}(\Sigma,u)+\deg_{F,\dbar}(\Sigma,u)+\deg_{\theta,F}(\Sigma,u),\label{eq:combined_deg_subfib}
\end{align}
where $\zeta\in A^2(S,\RR)$. In the applications below, when Assumption \ref{assum:comoment}\,(2) and (4) hold, we let $\zeta$ be the form in \eqref{eq:min-coupling}.
\end{definition}
Throughout this subsection, let $(\Sigma,u)$ be a sub-fibration of relative dimension $k$. For $s\in S$, let $\Sigma_s:=\pi^{-1}(s)$, $u_s:=u|_{\Sigma_s}$, $\omega_s:=\omega_{X/S}|_{X_s}$, $\omega_{\Sigma/S,s}:=u_s^*\omega_s$, and let $g_s$ be the Riemannian metric associated with $\omega_s$. When the dependence on the Hermitian metric is relevant, we write the degrees in Definition \ref{def:deg_subfib} as $\deg_{\omega_X}(\Sigma,u;\omega_\Sigma)$ etc.

\begin{lemma}\label{lem:deg_theta_F_split}
Assume Assumption \ref{assum:comoment}\,(1) holds. Let $(\Sigma,u)$ be a compact sub-fibration of relative dimension $k\geq1$. Choose a pure $(1,0)$-connection $\nabla_\pi^{1,0}$ on $\pi$. The corresponding real connection $\nabla_\pi^\RR$ induces a smooth $J_\Sigma$-invariant splitting $T\Sigma^\RR=T_{\Sigma/S}^\RR\oplus H_\Sigma^\RR$. Let $\omega_{\nabla_\pi^\RR}$ be the real $(1,1)$-form on $\Sigma$ determined by $\nabla_\pi^\RR$ and $\omega_{\Sigma/S}$ as in \eqref{eq:omega_nabla}. For any Hermitian form $\omega_S$ on $S$ and any $\lambda>0$, the form $\omega_{\Sigma,\lambda}:=\omega_{\nabla_\pi^\RR}+\lambda\pi^*\omega_S$ is Hermitian and satisfies $\deg_{\theta,F}(\Sigma,u;\omega_{\Sigma,\lambda})=0$. For $k=0$, $\deg_{\theta,F}(\Sigma,u)=0$ automatically. If $\pi$ is unramified, we set $\omega_{\Sigma,\lambda}:=\lambda\pi^*\omega_S$, which is Hermitian.
\end{lemma}
\begin{proof}
The positivity of $\omega_{\Sigma,\lambda}$ is immediate from its definition. Fix $s\in S$ and $v\in T_s S$, and write $\theta(v)=X_v+\I Y_v$ and $\bar\theta(\bar v)=-X_v+\I Y_v$, where $X_v,Y_v\in\mathfrak k_s$. Let $x_v,y_v$ be the vector fields on $\Sigma_s$ whose images under $\D u_s$ are the $g_s$-orthogonal tangential projections of $X_v^\RR,Y_v^\RR$ along $u_s$. The calculation in the proof of Lemma \ref{lem:higgs_norm_identity} after orthogonal projection to $F$ gives
\[
|(\bar\theta u)_F(\bar v)|_{h_F}^2-|(\theta u)_F(v)|_{h_F}^2=4(u_s^*\omega_s)(x_v,y_v).
\]
Choose Hamiltonian potentials $h_X,h_Y$ for $X_v^\RR,Y_v^\RR$. Since the normal components of $X_v^\RR,Y_v^\RR$ are symplectically orthogonal to $\D u_s(T\Sigma_s^\RR)$, we have
\begin{equation}\label{eq:ortho_ham}
\D(h_X\comp u_s)(Z)=-(u_s^*\omega_s)(x_v,Z)
\end{equation}
for every $Z\in T\Sigma_s^\RR$, and similarly for $y_v$. Thus $x_v$ and $y_v$ are Hamiltonian vector fields on the compact symplectic manifold $(\Sigma_s,u_s^*\omega_s)$. Moreover, $(u_s^*\omega_s)(x_v,y_v)=-\D(h_X\comp u_s)(y_v)$, and $y_v$ preserves $(u_s^*\omega_s)^k$. Stokes' theorem therefore gives $\int_{\Sigma_s}(u_s^*\omega_s)(x_v,y_v)(u_s^*\omega_s)^k=0$. Choose an $\omega_S$-unitary frame $v_1,\ldots,v_n$ of $T_s S$. Since $\theta u$ and $\bar\theta u$ factor through $\D\pi$ and $(\omega_{\Sigma,\lambda})^{n+k}=\binom{n+k}{k}\lambda^n\omega_{\nabla_\pi^\RR}^k\wedge\pi^*\omega_S^n$, fiber integration and the preceding identity for each $v_j$ give $\deg_{\theta,F}(\Sigma,u;\omega_{\Sigma,\lambda})=0$. The assertion for $k=0$ follows from $F=0$.
\end{proof}
\begin{remark}\label{rmk:deg_theta_F_not_general}
When $\omega_\Sigma$ is the split metric in Lemma \ref{lem:deg_theta_F_split}, the combined degree $\deg(\Sigma,u;\omega_\Sigma)$ is independent of the almost Higgs field $\theta$. In general, $\deg_{\theta,F}(\Sigma,u;\omega_\Sigma)$ need not vanish.
\end{remark}

\begin{lemma}\label{lem:pullback_comoment_subfib}
For each $i\in \{1,2,3,4\}$, if Assumption \ref{assum:comoment}\,(i) holds for $f$, then the corresponding pulled-back assumption holds for $\pi^*\!f$, with fiberwise comoment map $\mu^*_{\pi^*\!f}:=\tilde\pi^*\mu^*$. More precisely, for every $\sigma\in \Sigma$, $s:=\pi(\sigma)$, and a Hamiltonian vector field $\tilde{\xi}$ on $(\pi^*X)_\sigma\cong X_s$ in the domain of $\mu^*$, \[\mu^*_{\pi^*f,\sigma}(\tilde\xi):=\tilde\pi_\sigma^*(\mu_s^*((\tilde\pi_\sigma)_*\tilde\xi)).\]

If (2) and (4) hold and $\zeta$ is the $2$-form of \eqref{eq:min-coupling} associated to $\omega_X$, then the $2$-form $\zeta_{\pi^*\!f}$ associated to $\tilde\pi^*\omega_X$ on $\pi^*\!f$ equals $\pi^*\zeta$.
\end{lemma}
\begin{proof}
By Lemma \ref{lem:cartesian_smooth}\,(iii), for each $\sigma\in\Sigma$ the map $\tilde\pi_\sigma:=\tilde\pi|_{(\pi^*X)_\sigma}$ is a biholomorphism. By Lemma \ref{lem:pullback_basic_subfib}, $\tilde\pi_\sigma$ is a K\"ahler isometry $((\pi^*X)_\sigma,(\pi^*\omega_{X/S})_\sigma)\to(X_s,\omega_s)$. Under \eqref{eq:vert_iso_subfib} it identifies, for every $\sigma$, the spaces of holomorphic, Hamiltonian, and Killing vertical vector fields on the two fibers. We verify the four parts separately.

(1) $\tilde\pi_\sigma$ identifies $\mathrm{Ham}^{1,0}((\pi^*X)_\sigma)$ with $\mathrm{Ham}^{1,0}(X_{s})$ and identifies $\mathfrak{k}_{\pi^*\!f,\sigma}$ with $\mathfrak{k}_{s}$. Hence $\mathfrak{k}_{s}\subset\mathrm{Ham}^{1,0}(X_{s})$ for all $\sigma$ if and only if $\mathfrak{k}_{\pi^*\!f,\sigma}\subset\mathrm{Ham}^{1,0}((\pi^*X)_\sigma)$ for all $\sigma$.

(2) Since $\tilde\pi_\sigma^*\omega_{s}=(\pi^*\omega_{X/S})_\sigma$, pulling back $\D_{X_{s}}\mu^*(\eta)=-\iota_\eta\omega_{s}$ for $\eta:=(\tilde\pi_\sigma)_*\tilde\xi$ along $\tilde\pi_\sigma$ gives \[\D_{(\pi^*X)_\sigma}\mu^*_{\pi^*\!f}(\tilde\xi)=-\iota_{\tilde\xi}(\pi^*\omega_{X/S})_\sigma,\] and the equivariance $\mu^*_{\pi^*\!f}([\tilde\xi,\tilde\eta])=(\pi^*\omega_{X/S})_\sigma(\tilde\xi,\tilde\eta)$ follows from that of $\mu^*$, since $\tilde\pi_\sigma$ is a bracket-preserving symplectomorphism.

(3) Let $\mathfrak h_{X/S}$ denote the domain of the comoment map, either $\mathrm{Ham}_{X/S}$ or $\mathfrak k_{X/S}^{\mathrm R}$. In Assumption \ref{assum:comoment}\,(3), we defined a connection on $\mathfrak h_{X/S}$ by $\nabla^{\mathcal H}_w\xi:=[H_w,\xi]$, for a local vector field $w$ on $S$ and a local smooth section $\xi$ of $\mathfrak h_{X/S}$. If $\mathfrak h_{X/S}=\mathfrak k_{X/S}^{\mathrm R}$, this is well-defined precisely when $\nabla^{\RR}$ preserves $\mathfrak k_{X/S}^{\mathrm R}$. The fiberwise biholomorphisms induced by $\tilde\pi$ identify the domain of $\mu^*_{\pi^*\!f}$ with $\pi^*\mathfrak h_{X/S}$. Let $\widetilde H$ and $H$ be the horizontal lifts for $\pi^*\nabla^\RR$ and $\nabla^\RR$. By \eqref{eq:horizontal_lift_pullback}, $\D\tilde\pi\bigl(\widetilde H_{(\sigma,x)}(v)\bigr)=H_x(\D\pi(v))$ for $v\in T_\sigma\Sigma^\RR$. Let $V\subset\Sigma$ be open, let $v$ be a smooth vector field on $V$, and let $\tilde\xi\in C^\infty\bigl(V,\pi^*\mathfrak h_{X/S}\bigr)$. Then
\begin{equation}\label{eq:pullback_connection_hamiltonian}
[\widetilde H_v,\tilde\xi]=(\pi^*\nabla^{\mathcal H})_v\tilde\xi.
\end{equation}
Indeed, this identity is immediate for pullback sections $\tilde\xi=\pi^*\xi$, and the general case follows from the Leibniz rule in a local trivialization of $\mathfrak h_{X/S}$. It follows that $[\widetilde H_v,\tilde\xi]$ is again a local section of $\pi^*\mathfrak h_{X/S}$. Moreover, since $\mu^*$ is parallel,
\[
\widetilde H_v\bigl(\mu^*_{\pi^*\!f}(\tilde\xi)\bigr)=\mu^*_{\pi^*\!f}\bigl((\pi^*\nabla^{\mathcal H})_v\tilde\xi\bigr)=\mu^*_{\pi^*\!f}\bigl([\widetilde H_v,\tilde\xi]\bigr).
\]
Thus $\mu^*_{\pi^*\!f}$ is parallel.

(4) $\tilde\pi^*\omega_X$ is closed because $\omega_X$ is. By \eqref{eq:horizontal_lift_pullback}, for $v_1,v_2\in T_\sigma\Sigma^\RR$,
\begin{equation}\label{eq:pullback_real_curv}
(\tilde\pi_\sigma)_*\left(F_{\pi^*\nabla^\RR}(v_1,v_2)\right)=F_{\nabla^\RR}\left(\D\pi_\sigma(v_1),\D\pi_\sigma(v_2)\right).
\end{equation}
Hence if $F_{\nabla^\RR}$ is $\mathfrak{k}^{\mathrm R}_{X/S}$-valued then $F_{\pi^*\nabla^\RR}$ is $\mathfrak{k}^{\mathrm R}_{\pi^*X/\Sigma}$-valued, and $\mu^*_{\pi^*\!f}F_{\pi^*\nabla^\RR}$ is well-defined.
\smallskip

Apply $\tilde\pi^*$ to \eqref{eq:min-coupling} for $f$. Using the inheritance of the comoment map and of the curvature just established, and $\tilde\pi^*(f^*\zeta)=(\pi^*\!f)^*\pi^*\zeta$ from the commutativity of \eqref{eq:cartesian_subfib},
\[
\tilde\pi^*\omega_X^H=-\mu^*_{\pi^*\!f}F_{\pi^*\nabla^\RR}+(\pi^*\!f)^*\pi^*\zeta.
\]
By \eqref{eq:pull_horizontal}, together with \eqref{eq:min-coupling} applied to $\pi^*\!f$, we obtain $(\pi^*\!f)^*(\zeta_{\pi^*\!f}-\pi^*\zeta)=0$. As $\pi^*\!f$ is a surjective submersion, $(\pi^*\!f)^*$ is injective, so $\zeta_{\pi^*\!f}=\pi^*\zeta$.
\end{proof}

\begin{lemma}\label{lem:pullback_curv_subfib}
Assume, in addition to Lemma \ref{lem:pullback_basic_subfib}, that $\nabla^{1,0}$ on $f$ is a K\"ahler connection. Then $\pi^* \nabla^{1,0}$ is a K\"ahler connection with respect to $\pi^*\dbar$ and $\pi^*\omega_{X/S}$, so that the curvature and pseudo-curvature tensors are well-defined, and
\begin{equation}\label{eq:curv_pullback}
F_{\pi^*D}^{1,1}=\pi^*F_D^{1,1},\qquad	G_{\pi^*D''}^{1,1}=\pi^*G_{D''}^{1,1},\qquad	G_{\pi^*D'}^{1,1}=\pi^*G_{D'}^{1,1},
\end{equation}
where the notation $\pi^*$ means that the base-form arguments are pulled back by $\D\pi$ and the vertical vector-field value is pulled back by $\tilde\pi$ using \eqref{eq:vert_iso_subfib}. Moreover, $F_{\pi^*D}=\pi^*F_D$ and $G_{\pi^*D''}=\pi^*G_{D''}$.
In particular, $(f,\omega_{X/S},D'',D)$ is a harmonic bundle if and only if $(\pi^*\!f,\pi^*\omega_{X/S},\pi^*D'',\pi^*D)$ is.
\end{lemma}
\begin{proof}
By Lemma \ref{lem:pullback_basic_subfib}, $\pi^*\nabla^{1,0}$ is a symplectic connection associated to $\pi^*\dbar$ and $\pi^*\omega_{X/S}$. Fix $\sigma_0\in\Sigma$, put $s_0:=\pi(\sigma_0)$, and choose holomorphic coordinates $(\sigma^a)_{a=1}^{n+k}$ near $\sigma_0$ and $(s^i)_{i=1}^n$ near $s_0$ in which $\pi$ is given by holomorphic functions $s^i=\pi^i(\sigma)$, together with adapted coordinates $(s^i,z^\alpha)$ on $X$. Since $\pi^*X=\{(\sigma,x):\pi(\sigma)=f(x)\}$, the functions $(\sigma^a,z^\alpha)$ are adapted coordinates on $\pi^*X$, with $\pi^*\!f(\sigma,z)=\sigma$ and $\tilde\pi(\sigma,z)=(\pi(\sigma),z)$.

Write $\pi_a^i:=\partial\pi^i/\partial\sigma^a$ (holomorphic in $\sigma$). By \eqref{eq:horizontal_lift_pullback} and the description of $\pi^*\dbar$ in Lemma \ref{lem:cartesian_smooth}\,(ii), the coefficients of $\pi^*\nabla^{1,0},\pi^*\dbar,\pi^*\theta,\pi^*\bar\theta$ are
\begin{align}\label{eq:coeff_pullback_local}
\hat\Gamma_a^\alpha&=\pi_a^i\,(\Gamma_i^\alpha\comp\tilde\pi),\qquad \hat\Gamma_a^{\bar\beta}=\pi_a^i\,(\Gamma_i^{\bar\beta}\comp\tilde\pi),\quad \hat\Gamma_{\bar a}^\alpha=\overline{\pi_a^i}\,(\Gamma_{\bar i}^\alpha\comp\tilde\pi),\qquad \\ \hat\theta_a^\alpha&=\pi_a^i\,(\theta_i^\alpha\comp\tilde\pi),\qquad \hat{\bar\theta}_{\bar a}^\alpha=\overline{\pi_a^i}\,(\bar\theta_{\bar i}^\alpha\comp\tilde\pi).\label{eq:coeff_theta_pullback_local}
\end{align}
Since $\nabla^{1,0}$ is a K\"ahler connection, $\partial_{\bar\beta}\Gamma_i^\alpha=0$ and $\partial_{\bar\beta}\Gamma_{\bar i}^\alpha=0$. As $\pi_a^i$ depends only on $\sigma$, \eqref{eq:coeff_pullback_local} gives $\partial_{\bar\beta}\hat\Gamma_a^\alpha=0$ and $\partial_{\bar\beta}\hat\Gamma_{\bar a}^\alpha=0$. Hence $\pi^*\nabla^{1,0}$ is relatively holomorphic and $\pi^*\dbar$ satisfies the lifting condition, so $\pi^*\nabla^{1,0}$ is a K\"ahler connection. Together with the relative holomorphicity of $\pi^*\theta$, this implies that the curvature and pseudo-curvature tensors of $\pi^*D,\pi^*D'',\pi^*D'$ are well-defined.

Because $\pi$ is holomorphic, $\partial_{\sigma^a}\overline{\pi_b^j}=0=\partial_{\bar\sigma^b}\pi_a^i$, and for every smooth $g$ on $X$,
\[
\partial_{\sigma^a}(g\comp\tilde\pi)=\pi_a^i\,(\partial_{s^i}g)\comp\tilde\pi,\qquad \partial_{\bar\sigma^b}(g\comp\tilde\pi)=\overline{\pi_b^j}\,(\partial_{\bar s^j}g)\comp\tilde\pi,\qquad \partial_{z^\gamma}(g\comp\tilde\pi)=(\partial_{z^\gamma}g)\comp\tilde\pi.
\]
The $(1,1)$-pseudo-curvatures of $D''$ and $D'$ are the expressions \eqref{eq:GD''_local} and \eqref{eq:GD'_local} (and similarly for the $(1,1)$-curvature of $D$) in the coefficients $\Gamma_i^\alpha$, $\theta_i^\alpha$, $\Gamma_{\bar j}^\alpha$, $\bar\theta_{\bar j}^\alpha$ and their first $s$- and $z$-derivatives. Substituting \eqref{eq:coeff_pullback_local} and \eqref{eq:coeff_theta_pullback_local} and applying the chain-rule identities above, each base index $i$ (resp.\ $\bar j$) contributes a factor $\pi_a^i$ (resp.\ $\overline{\pi_b^j}$), while the fiber derivatives are reproduced verbatim along $\tilde\pi$. Therefore
\[
F_{\pi^*D}^{1,1}(\partial_a,\partial_{\bar b})=\pi_a^i\,\overline{\pi_b^j}\,\bigl(F_D^{1,1}(\partial_i,\partial_{\bar j})\comp\tilde\pi\bigr),
\]
and likewise for $G_{D''}^{1,1}$ and $G_{D'}^{1,1}$. In the notation of the statement these are precisely \eqref{eq:curv_pullback}.

Applying the same chain-rule calculation to the coefficients \eqref{eq:coeff_pullback_local} and \eqref{eq:coeff_theta_pullback_local}, we obtain $F_{\pi^*D}=\pi^*F_D$ and $G_{\pi^*D''}=\pi^*G_{D''}$.  The terms containing $\partial_a\pi_b^i$ cancel after antisymmetrization because $\partial_a\pi_b^i=\partial_b\pi_a^i$. For $k\ge1$, $\pi$ is a surjective submersion, so the pullback is injective on these curvature tensors. For $k=0$, the same conclusion holds away from the ramification locus, the complement of which is dense. Continuity then gives the assertion everywhere. Thus $(f,\omega_{X/S},D'',D)$ is a harmonic bundle if and only if $(\pi^*\!f,\pi^*\omega_{X/S},\pi^*D'',\pi^*D)$ is.
\end{proof}

\begin{theorem}\label{thm:D'_D''u_identity_subfib}
Suppose Assumption \ref{assum:comoment}\,(1), (2), (4) hold. Let $(\Sigma,u)$ be a sub-fibration of relative dimension $k$. Then
\begin{multline}\label{eq:D'_D''u_identity_subfib_real}
|(D'u)_N|^2-|(D''u)_N|^2
=\Lambda_{\omega_\Sigma}\bigl(u^*\omega_X-\omega_{F,\partial}+\omega_{F,\dbar}+u^*\mu^*\bigl(F_{\nabla^\RR}+\bigl[[\theta,\bar\theta]\bigr]_\RR\bigr)-\pi^*\zeta\bigr)
\\+|(\theta u)_F|^2-|(\bar\theta u)_F|^2.
\end{multline}
If moreover $\nabla^{1,0}$ is a K\"ahler connection, then
\begin{multline}\label{eq:D'_D''u_identity_subfib}
|(D'u)_N|^2-|(D''u)_N|^2
=\Lambda_{\omega_\Sigma}\bigl(u^*\omega_X-\omega_{F,\partial}+\omega_{F,\dbar}+u^*\mu^*\bigl[F_D^{1,1}-G_{D'}^{1,1}-G_{D''}^{1,1}\bigr]_\RR-\pi^*\zeta\bigr)
\\+|(\theta u)_F|^2-|(\bar\theta u)_F|^2.
\end{multline}
Here, for a real $2$-form $\Phi$ valued in real vertical vector fields, \[(u^*\mu^*\Phi)_\sigma(v,w):=\mu_{\pi(\sigma)}^*(
\Phi_{\pi(\sigma)}(\D\pi_\sigma(v),\D\pi_\sigma(w)))(u(\sigma)), \quad v,w\in T_\sigma\Sigma^\RR.\]
In particular, if $(f,\omega_{X/S},D'',D)$ is a harmonic bundle and $\nabla^\RR$ preserves $\mathfrak{k}_{X/S}^{\mathrm{R}}$, then
\begin{equation}\label{eq:D'_D''u_identity_subfib_harmonic}
|(D'u)_N|^2-|(D''u)_N|^2=\Lambda_{\omega_\Sigma}\bigl(u^*\omega_X-\omega_{F,\partial}+\omega_{F,\dbar}-\pi^*\zeta\bigr)+|(\theta u)_F|^2-|(\bar\theta u)_F|^2,
\end{equation}
and if moreover $\Sigma$ is compact, then
\begin{enumerate}
\item any flat sub-fibration $(\Sigma,u)$ satisfies $\deg_\zeta(\Sigma,u)=0$;
\item any Higgs sub-fibration satisfies $\deg_\zeta(\Sigma,u)\ge 0$ and is flat if and only if $\deg_\zeta(\Sigma,u)=0$.
\end{enumerate}
\end{theorem}

\begin{proof}
By Lemma \ref{lem:pullback_comoment_subfib}, the pulled-back data
$(\pi^*\!f,\pi^*\theta,\pi^*\omega_{X/S},\tilde\pi^*\omega_X,\pi^*\nabla^{1,0})$
on $\pi^*X\to\Sigma$ satisfy Assumption \ref{assum:comoment}\,(1), (2), (4) with comoment map $\tilde\pi^*\mu^*$ and the corresponding $2$-form $\pi^*\zeta$. Applying \eqref{eq:D'_D''u_real} to $\tilde u:\Sigma\to\pi^*X$ gives
\[|D'\tilde u|^2-|D''\tilde u|^2=\Lambda_{\omega_\Sigma}\left(u^*\omega_X+u^*\mu^*\bigl(F_{\nabla^\RR}+[[\theta,\bar\theta]]_\RR\bigr)-\pi^*\zeta\right).\]
Here we used \eqref{eq:pullback_real_curv} and $\bigl[[\pi^*\theta,\pi^*\bar\theta]\bigr]_\RR=\pi^*\bigl[[\theta,\bar\theta]\bigr]_\RR$. On the other hand, Lemma \ref{lem:omega_F_pos} and type orthogonality give
\[|(D'u)_F|^2-|(D''u)_F|^2=\Lambda_{\omega_\Sigma}(\omega_{F,\partial}-\omega_{F,\dbar})+|(\bar\theta u)_F|^2-|(\theta u)_F|^2.\]
Subtracting proves \eqref{eq:D'_D''u_identity_subfib_real}. If moreover $\nabla^{1,0}$ is a K\"ahler connection, Lemma \ref{lem:pullback_curv_subfib} applies. Since $\pi$ is holomorphic, $[\pi^*\Phi]_\RR=\pi^*[\Phi]_\RR$ for any $\Phi$ as in \eqref{eq:realification_11}. Applying \eqref{eq:D'_D''u_identity} to $\tilde{u}$ proves \eqref{eq:D'_D''u_identity_subfib}.

Under the harmonicity and preservation hypotheses, $F_D=G_{D''}=0$ and $G_{D'}^{1,1}=0$ by Lemma \ref{lem:pseudo_curv_conjugate}. Hence \eqref{eq:D'_D''u_identity_subfib_harmonic} follows. Integrating \eqref{eq:D'_D''u_identity_subfib_harmonic} over $\Sigma$ with respect to $\omega_\Sigma^{n+k}$, the last statement follows.
\end{proof}

\begin{example}\label{ex:ramified_multisection_general}
Let $(S,\omega_S)$ be a compact Riemann surface ($n=1$), $(Y,\omega_Y)$ a compact K\"ahler manifold, and consider the trivial holomorphic bundle $f:X:=S\times Y\to S$ with $\omega_X:=\mathrm{pr}_2^*\omega_Y$, the trivial $(1,0)$-connection induced by $\omega_X$, and zero Higgs field $\theta=0$. We choose $\mathfrak k_{X/S}=0$ and the zero comoment map. Then $D'=\partial$, $D''=\dbar$, and $(f,\omega_{X/S},D'',D)$ is a nonlinear harmonic bundle satisfying Assumption~\ref{assum:comoment}\,(1)--(4) with $\zeta=0$.

Let $\Sigma$ be a compact connected Riemann surface and $\pi:\Sigma\to S$ a non-constant holomorphic map. By the Riemann--Hurwitz formula, $\pi$ is unramified if and only if $\chi(\Sigma)=\deg(\pi)\cdot\chi(S)$. A sub-fibration of relative dimension $k=0$ (also called a \emph{multi-section}) $(\Sigma,u)$ is equivalent to a smooth map $u_1:\Sigma\to Y$, with $u:\Sigma \to X$ given by $ u(\sigma):=(\pi(\sigma),u_1(\sigma))$. Equip $\Sigma$ with any K\"ahler form $\omega_\Sigma$ (which exists by Lemma \ref{lem:Sigma_kahler_k0}). By Definition \ref{def:higgs_flat_subfib}, $(\Sigma,u)$ is a Higgs sub-fibration if and only if $u_1$ is holomorphic and is a flat sub-fibration if and only if $u_1$ is constant.

The combined degree of Definition \ref{def:deg_subfib} reduces to
\[
\deg(\Sigma,u)=\deg_{\omega_X}(\Sigma,u)=\int_\Sigma u^*\omega_X=\int_\Sigma u_1^*\omega_Y.
\]
For a Higgs sub-fibration, the form $u_1^*\omega_Y$ is a non-negative $(1,1)$-form on $\Sigma$, and $\int_\Sigma u_1^*\omega_Y=0$ forces $u_1^*\omega_Y= 0$, i.e., $\D u_1=0$ and $u_1$ is constant. By Theorem \ref{thm:D'_D''u_identity_subfib}\,(2), $(\Sigma,u)$ is a flat sub-fibration if and only if $u_1$ is constant, in agreement with the direct characterizations above.
\end{example}

\begin{example}\label{ex:higgs_subfib_conic}
Let $S=\CC/(\ZZ+\tau\ZZ)$ be an elliptic curve with flat K\"ahler form $\omega_S$ normalized so that $\int_S\omega_S=1$, and let $\D s\in H^0(S,\Omega_S^1)$ be the standard holomorphic $1$-form. Consider the trivial $\CC P^2$-fibration $f:X:=S\times \CC P^2\to S$, equipped with the fiberwise Fubini--Study metric $\omega_{\mathrm{FS}}$ (normalized so that $\int_\ell \omega_{\mathrm{FS}}=\pi$ for any projective line $\ell\subset\CC P^2$), the relatively K\"ahler form $\omega_X:=\mathrm{pr}_2^*\omega_{\mathrm{FS}}$, and the trivial $(1,0)$-connection $\nabla^{1,0}$ induced by $\omega_X$. This setup arises from a rank-$3$ Higgs vector bundle: take the trivial bundle $\widetilde E=S\times\CC^3\to S$ with Higgs field $\theta_{\widetilde E}=\mathrm{diag}(0,-1,-2)\,\D s\in A^{1,0}(S,\End\widetilde E)$, then $f=\mathbb{P}(\widetilde E)\to S$ is the associated $\CC P^2$-fibration, and the induced fiberwise infinitesimal action gives the relatively holomorphic almost Higgs field
\[
\theta(\partial_s)=z^1\partial_{z^1}+2z^2\partial_{z^2}\in \mathfrak{k}_s^\CC\subset H^0(\CC P^2,T\CC P^2),
\]
in the affine chart $\{[1:z^1:z^2]\}$. We have $[\theta,\bar\theta]=0$, $\dbar\theta=0$, and $F_{D_\omega}^{1,1}=0$ by triviality. Hence $(f,\omega_{X/S},D'',D)$ is a harmonic bundle. All requirements of Assumption \ref{assum:comoment}\,(1)--(4) hold with minimal-coupling form $\zeta=0$ (Remark \ref{rmk:moment_existence}). Consider the smooth conic $C:=\bigl\{[w_0:w_1:w_2]\in\CC P^2:w_0w_2=w_1^2\bigr\}$, which in the affine chart $w_0=1$ is the parabola $z^2=(z^1)^2$.

Set $\Sigma:=S\times C$, a compact connected K\"ahler manifold of dimension $2$, equipped with the product K\"ahler form $\omega_\Sigma:=\mathrm{pr}_1^*\omega_S+\mathrm{pr}_2^*(\omega_{\mathrm{FS}}|_C)$. The inclusion $u:\Sigma\hookrightarrow X$ is holomorphic, $\pi=f\comp u=\mathrm{pr}_1:\Sigma\to S$ is a holomorphic submersion of relative dimension $k=1$, and each fiber $\pi^{-1}(s)=\{s\}\times C$ is a holomorphic embedding into $X_s=\CC P^2$. Thus $(\Sigma,u)$ is a sub-fibration. The $\CC^*$-action generated by $\theta(\partial_s)$ sends $(z^1,z^2)\mapsto (e^t z^1,e^{2t}z^2)$, and preserves $C$ since $e^{2t}z^2=(e^tz^1)^2$. Hence $\theta(\partial_s)|_p$ is tangent to $C$ at every $p\in C$, and $(\Sigma,u)$ is a Higgs sub-fibration.

Since $\omega_X=\mathrm{pr}_2^*\omega_{\mathrm{FS}}$, $u^*\omega_X=\mathrm{pr}_2^*(\omega_{\mathrm{FS}}|_C)$, and $(\mathrm{pr}_2^*(\omega_{\mathrm{FS}}|_C))\wedge(\mathrm{pr}_2^*(\omega_{\mathrm{FS}}|_C))=0$ for dimensional reasons,
\[
\deg_{\omega_X}(\Sigma,u)=\int_\Sigma \mathrm{pr}_2^*(\omega_{\mathrm{FS}}|_C)\wedge\mathrm{pr}_1^*\omega_S=\Bigl(\int_C\omega_{\mathrm{FS}}\Bigr)\Bigl(\int_S\omega_S\Bigr)=2\pi.
\]
Since $u$ is holomorphic, $\dbar u=0$, so $\omega_{F,\dbar}=0$ and $\deg_{F,\dbar}(\Sigma,u)=0$. At $\sigma=(s,y)\in\Sigma$, $(\partial u)(\partial_s)$ has zero vertical component, while $(\partial u)$ applied to a vertical tangent vector along $C$ equals itself as a tangent vector in $T_y \CC P^2$, hence lies in $F|_\sigma=T_y C$. By Definition \ref{def:omega_F},
\[
\omega_{F,\partial}=\mathrm{pr}_2^*(\omega_{\mathrm{FS}}|_C)=u^*\omega_X,
\]
so $\deg_{F,\partial}(\Sigma,u)=2\pi$. Finally, $\theta(\partial_s)=\bar{\theta}(\partial_{\bar{s}})$, which implies $|(\theta u)_F|=|(\bar\theta u)_F|$. Therefore, $\deg(\Sigma,u)=2\pi-2\pi=0$. Theorem \ref{thm:D'_D''u_identity_subfib}\,(2) implies that $(\Sigma,u)$ is also a flat sub-fibration. A direct computation also gives $(Du)_N=0$.
\end{example}

We next introduce a second degree for compact sub-fibrations $(\Sigma,u)$ of relative dimension $k\geq0$, assuming that $\pi$ is unramified when $k=0$.
When $k=0$, integration along $\pi$ is interpreted as summation over its finite fibers. Under the harmonic hypotheses, it also measures the obstruction to the flatness of a Higgs sub-fibration. For $s\in S$, let $\nu_s:=(u_s^*\omega_s)^k$, and define $N_s^\RR:=u_s^*TX_s^\RR/\D u_s(T\Sigma_s^\RR)$. Since $u_s$ is a holomorphic immersion, $\D u_s(T\Sigma_s^\RR)$ is invariant under the fiberwise complex structure. Its $g_s$-orthogonal complement is therefore also complex, identifies with $N_s^\RR$, and induces a complex structure $J_N$ on $N_s^\RR$. If $A,B\in C^\infty(\Sigma_s,N_s^\RR)$ and $A^\perp,B^\perp$ denote their unique $g_s$-orthogonal representatives, set $\omega_{N,s,\sigma}(A,B):=\omega_s(A^\perp,B^\perp)$ and
\[
\Omega_{s,u_s}(A,B):=\int_{\Sigma_s}\omega_N(A,B)\nu_s,
\qquad
\widehat g_{s,u_s}(A,B):=\int_{\Sigma_s}g_s(A^\perp,B^\perp)\nu_s.
\]
Then $\widehat g_{s,u_s}(A,B)=\Omega_{s,u_s}(A,J_NB)$. Equivalently, $\Omega_{s,u_s}(A,B)=\frac1{k+1}\int_{\Sigma_s}u_s^*\bigl(\iota_{\widetilde B}\iota_{\widetilde A}\omega_s^{k+1}\bigr)$ for arbitrary representatives $\widetilde A,\widetilde B$. Indeed, changing one representative by $\D u_s(Z)$ makes the integrand vanish since $2k+1$ of the arguments belong to the $2k$-dimensional space $\D u_s(T\Sigma_s^\RR)$. For the orthogonal representatives, the formula follows from $u_s^*\bigl(\iota_{B^\perp}\iota_{A^\perp}\omega_s^{k+1}\bigr)=(k+1)\omega_s(A^\perp,B^\perp)(u_s^*\omega_s)^k$.

For a real vector field $V$ on $S$, choose a local $\pi$-projectable lift $\widetilde V$ and define
\[
A_V(\sigma):=\mathrm{pr}_{N^\RR}\bigl(\D u(\widetilde V)-H_V|_{u(\sigma)}\bigr),\qquad \sigma\in\Sigma_s.
\]
Changing $\widetilde V$ by a vertical tangent vector changes the expression inside the projection by an element of $\D u_s(T\Sigma_s)$, so $A_V$ is well-defined. After complexification, $(A_v)^{1,0}=(\partial u)_N(v)$ and $(A_{\bar v})^{1,0}=(\dbar u)_N(\bar v)$ for $v\in T_s S$. The normal forms $(\partial u)_N$, $(\dbar u)_N$, $(\theta u)_N$, $(\bar\theta u)_N$ annihilate $T_{\Sigma/S}^\CC$, since the Higgs terms vanish on vertical vectors, while for $Z\in T_{\Sigma/S}$ and $\widebar{Z}\in \widebar{T_{\Sigma/S}}$ one has $(\partial u)(Z)=\D u_s(Z)\in F$ and $(\dbar u)(\widebar{Z})=0$ by fiberwise holomorphicity. We henceforth regard these forms canonically as sections of $\pi^*T^*S^\CC\otimes N$.

Whenever Assumption \ref{assum:comoment}\,(2) holds, let $\mathfrak h_s$ denote the domain of $\mu_s^*$, either $\mathrm{Ham}(X_s)$ or $\mathfrak{k}_s^{\mathrm R}$, and let $\mathfrak h_{X/S}$ denote the corresponding bundle over $S$. For $\xi\in\mathfrak h_s$, define
\[
\widehat\xi_{u_s}:=(\xi\comp u_s)_{N^\RR}\in  C^\infty(\Sigma_s,N_s^\RR),
\qquad
\widehat\mu_{s,u_s}^*(\xi):=\int_{\Sigma_s}\bigl(\mu_s^*(\xi)\comp u_s\bigr)\nu_s.
\]
For $\eta\in A^q(S,\mathfrak h_{X/S})$, write $(\widehat\mu_u^*\eta)(V_1,\ldots,V_q):=\widehat\mu_{s,u_s}^*\bigl(\eta(V_1,\ldots,V_q)\bigr)$.

\begin{definition}\label{def:transgressed_degree_subfib}
Let $\omega_X$ be a real $2$-form on $X$. Let $(\Sigma,u)$ be a sub-fibration of relative dimension $k\geq0$ such that $\pi:=f\comp u$ is proper and is unramified when $k=0$, and let $\zeta\in A^2(S,\RR)$. Define
\[
\hat\omega_u:=\pi_*(u^*\omega_X^{k+1})/(k+1),\qquad \mathsf V_u:=\pi_*(u^*\omega_X^k),\qquad \hat\omega_{u,\zeta}:=\hat\omega_u-\mathsf V_u\zeta,
\]
where $\pi_*$ means the fiber integration. If $(S,\omega_S)$ is compact Hermitian, define the \emph{transgressed degree} (associated to $\omega_X$ and $\zeta$) by
\begin{equation}\label{eq:transgressed_degree_subfib}
\deg_{\omega_X,\zeta}^{\mathrm{tr}}(\Sigma,u):=\int_S\hat\omega_{u,\zeta}\wedge\omega_S^{n-1}.
\end{equation}
\end{definition}
For every $\zeta\in A^2(S,\RR)$, $\hat\omega_{u,\zeta}=\pi_*\bigl(u^*(\omega_X-f^*\zeta)^{k+1}\bigr)/(k+1)$. If $\omega_X$ is closed, then $\hat\omega_u$ is closed and $\mathsf V_u$ is constant since $S$ is connected. The notion of transgressed degree generalizes that of the difference between slopes in the linear setting, see \eqref{eq:proj_subbundle_transgressed_degree} below.

\begin{lemma}\label{lem:transgressed_degree_cohomology}
Let $(S,\omega_S)$ be connected compact semi-K\"ahler of complex dimension $n\geq1$, and let $(\Sigma,u)$ be a sub-fibration of $f:X\to S$ of relative dimension $k\geq1$ such that $\pi:=f\comp u$ is proper. For $i=0,1$, let $\omega_i\in A^2(X,\RR)$ and $\zeta_i\in A^2(S,\RR)$, and suppose that $\Omega_i:=\omega_i-f^*\zeta_i$ is closed. If $[u^*\Omega_0]=[u^*\Omega_1]$ in $H^2(\Sigma,\RR)$, then
\[
\deg_{\omega_0,\zeta_0}^{\mathrm{tr}}(\Sigma,u)=\deg_{\omega_1,\zeta_1}^{\mathrm{tr}}(\Sigma,u).
\]

Suppose, moreover, that $f$ is proper with fibers of complex dimension $m$, that the $\omega_i$ are closed and restrict to K\"ahler forms on the fibers, and that the $\zeta_i$ are chosen by the mean-zero normalization \eqref{eq:zeta}, i.e., $\zeta_i=f_*(\omega_i^{m+1})/((m+1)V_i)$, where $V_i=f_*(\omega_i^m)$. Then the same equality holds whenever $[\omega_1]-[\omega_0]\in f^*H^2(S,\RR)$.
\end{lemma}
\begin{proof}
Fiber integration commutes with the exterior derivative, and hence
\[
[\pi_*(u^*\Omega_i^{k+1})]=\pi_*\bigl([u^*\Omega_i]^{k+1}\bigr)\in H^2(S,\RR).
\]
The two transgressed forms have the same de Rham class, and their degrees agree by Stokes' theorem and $\D\omega_S^{n-1}=0$.

For the last assertion, write $a_i:=[\omega_i]$ and $a_1=a_0+f^*b$ with $b\in H^2(S,\RR)$. Since fiber integration commutes with $\D$, the $V_i$ are positive constants and the $\zeta_i$ are closed. Restriction to each fiber gives $V_1=V_0=:V$. The projection formula gives
\[
f_*(a_1^{m+1})=f_*(a_0^{m+1})+(m+1)Vb,
\]
since the terms containing at least two factors pulled back from $S$ have insufficient vertical degree. Therefore, $[\zeta_1]-[\zeta_0]=b$, and then $[\Omega_1]=[\Omega_0]$. The first assertion applies.
\end{proof}

\begin{theorem}\label{thm:D'_D''u_identity_subfib_transgression}
Suppose $n\geq1$, Assumption \ref{assum:comoment}\,(1), (2), and (4) hold, and $(\Sigma,u)$ is a sub-fibration of relative dimension $k\geq0$ such that $\pi:=f\comp u$ is proper and is unramified when $k=0$. Let $\omega_S$ be a Hermitian form on $S$. Then
\begin{equation}\label{eq:D'_D''u_identity_subfib_transgression_real}
e_{D'}-e_{D''}=\Lambda_{\omega_S}\bigl(\hat\omega_{u,\zeta}+\widehat\mu_u^*\bigl(F_{\nabla^\RR}+\bigl[[\theta,\bar\theta]\bigr]_\RR\bigr)\bigr),
\end{equation}
where $e_{D'}(s):=\int_{\Sigma_s}|(D'u)_N|_{\omega_S,h_N}^2\nu_s$ and  $e_{D''}(s):=\int_{\Sigma_s}|(D''u)_N|_{\omega_S,h_N}^2\nu_s$. If moreover $\nabla^{1,0}$ is a K\"ahler connection, then
\begin{equation}\label{eq:D'_D''u_identity_subfib_transgression}
e_{D'}-e_{D''}=\Lambda_{\omega_S}\bigl(\hat\omega_{u,\zeta}+\widehat\mu_u^*\bigl(\bigl[F_D^{1,1}-G_{D'}^{1,1}-G_{D''}^{1,1}\bigr]_\RR\bigr)\bigr).
\end{equation}

Suppose in addition that $(f,\omega_{X/S},D'',D)$ is a harmonic bundle, $\nabla^\RR$ preserves $\mathfrak{k}_{X/S}^{\mathrm R}$, and $S$ is compact. Then every flat sub-fibration satisfies $\deg_{\omega_X,\zeta}^{\mathrm{tr}}(\Sigma,u)=0$, whereas every Higgs sub-fibration satisfies $\deg_{\omega_X,\zeta}^{\mathrm{tr}}(\Sigma,u)\geq0$ and is flat if and only if $\deg_{\omega_X,\zeta}^{\mathrm{tr}}(\Sigma,u)=0$.
\end{theorem}

\begin{proof}
Define $e_\partial,e_{\dbar},e_\theta,e_{\bar\theta}$ in the same way as $e_{D'},e_{D''}$, using the corresponding normal forms. Let $\Omega_A(V,W):=\Omega_{s,u_s}(A_V,A_W)$. Applying the identity $\Lambda_{\omega_S}(L_\sigma^*\omega_N)=|L_\sigma^{1,0}|^2-|L_\sigma^{0,1}|^2$ for $L_\sigma:T_sS^\RR\to N_{s,\sigma}^\RR$, $L_\sigma(V):=A_V(\sigma)$ and then integrating over $\Sigma_s$ gives
\begin{equation}\label{eq:normal_partial_dbar_transgression}
e_\partial-e_{\dbar}=\Lambda_{\omega_S}\Omega_A.
\end{equation}
Fix $V,W\in T_sS^\RR$. Along $\Sigma_s$ we may modify arbitrary lifts $\widetilde V,\widetilde W$ of $V,W$ by vertical tangent vectors so that $\D u(\widetilde V)-H_V=A_V^\perp$ and $\D u(\widetilde W)-H_W=A_W^\perp$. Fiber integration gives
\begin{align*}
\hat\omega_u(V,W)
&=\frac1{k+1}\int_{\Sigma_s}u_s^*\bigl(\iota_{H_W+A_W^\perp}\iota_{H_V+A_V^\perp}\omega_X^{k+1}\bigr)\\
&=\int_{\Sigma_s}\bigl(\omega_s(A_V^\perp,A_W^\perp)+\omega_X(H_V,H_W)\bigr)\nu_s\\
&=\Omega_A(V,W)+\int_{\Sigma_s}\bigl(\omega_X^H(V,W)\comp u_s\bigr)\nu_s.
\end{align*}
The minimal-coupling identity $\omega_X^H=-\mu^*F_{\nabla^\RR}+f^*\zeta$ therefore yields
\begin{equation}\label{eq:Omega_A_transgressed_form}
\Omega_A=\hat\omega_{u,\zeta}+\widehat\mu_u^*F_{\nabla^\RR}.
\end{equation}

Let $\xi,\eta\in\mathfrak h_s$, write $h_\xi:=\mu_s^*(\xi)$ and $h_\eta:=\mu_s^*(\eta)$, and take the orthogonal decompositions
\[
\xi\comp u_s=\D u_s(x_\xi)+\xi^\perp,
\qquad
\eta\comp u_s=\D u_s(x_\eta)+\eta^\perp.
\]
Similarly to \eqref{eq:ortho_ham}, we have $\D(h_\xi\comp u_s)=-\iota_{x_\xi}(u_s^*\omega_s)$ and $\D(h_\eta\comp u_s)=-\iota_{x_\eta}(u_s^*\omega_s)$. Thus $x_\xi,x_\eta$ are Hamiltonian vector fields on the compact manifold $(\Sigma_s,u_s^*\omega_s)$. The integral of their Poisson bracket vanishes, so $\int_{\Sigma_s}(u_s^*\omega_s)(x_\xi,x_\eta)\nu_s=0$. Equivariance of $\mu_s^*$ and the orthogonality of the tangent and normal components now give
\begin{equation}\label{eq:normal_comoment_equivariance}
\widehat\mu_{s,u_s}^*([\xi,\eta])=\Omega_{s,u_s}(\widehat\xi_{u_s},\widehat\eta_{u_s}).
\end{equation}

For $v\in T_s S$, write $\theta(v)=X_v+\I Y_v$ with $X_v,Y_v\in\mathfrak{k}_s$. Lemma \ref{lem:higgs_norm_identity}, applied after orthogonal projection to the normal bundle, gives
\[
|(\bar\theta u)_N(\bar v)|_{h_N}^2-|(\theta u)_N(v)|_{h_N}^2=4\omega_s\bigl((X_v^\RR)^\perp,(Y_v^\RR)^\perp\bigr).
\]
Combining this with \eqref{eq:normal_comoment_equivariance} and the calculation leading to \eqref{eq:diff_theta_fin}, and then summing in an $\omega_S$-unitary frame, gives
\begin{equation}\label{eq:normal_higgs_transgression}
e_{\bar\theta}-e_\theta=\Lambda_{\omega_S}\widehat\mu_u^*\bigl[[\theta,\bar\theta]\bigr]_\RR.
\end{equation}
By type orthogonality, $e_{D'}-e_{D''}=(e_\partial-e_{\dbar})+(e_{\bar\theta}-e_\theta)$. \eqref{eq:D'_D''u_identity_subfib_transgression_real} follows from \eqref{eq:normal_partial_dbar_transgression}, \eqref{eq:Omega_A_transgressed_form}, and \eqref{eq:normal_higgs_transgression}. If $\nabla^{1,0}$ is K\"ahler, \eqref{eq:symp_curv_11_curv} and \eqref{eq:curv11_decomp} give $\bigl(F_{\nabla^\RR}+\bigl[[\theta,\bar\theta]\bigr]_\RR\bigr)^{1,1}=\bigl[F_D^{1,1}-G_{D'}^{1,1}-G_{D''}^{1,1}\bigr]_\RR$, and \eqref{eq:D'_D''u_identity_subfib_transgression} follows. The last statement is clear after integration over $S$.
\end{proof}

\begin{lemma}\label{lem:combined_transgressed_degree}
Suppose Assumption \ref{assum:comoment}\,(1), (2), and (4) hold. Let $(\Sigma,u)$ be a compact sub-fibration of relative dimension $k\geq0$, let $\omega_S$ be a Hermitian form on $S$, and let $\omega_{\Sigma,\lambda}$ be as in Lemma \ref{lem:deg_theta_F_split}. Assume that $\pi$ is unramified when $k=0$. Then
\begin{equation}\label{eq:combined_transgressed_degree}
\deg_\zeta(\Sigma,u;\omega_{\Sigma,\lambda})=\binom{n+k-1}{k}\lambda^{n-1}\deg_{\omega_X,\zeta}^{\mathrm{tr}}(\Sigma,u;\omega_S).
\end{equation}
\end{lemma}
\begin{proof}
Put $\mathcal C_0:=F_{\nabla^\RR}+[[\theta,\bar\theta]]_\RR$. Since all normal forms are regarded as sections of $\pi^*T^*S^\CC\otimes N$, the volume expansion for the split metric and the $\lambda^{-1}$ scaling of their squared norms give
\begin{equation}\label{eq:split_normal_energy_fiber}
\int_\Sigma\bigl(|(D'u)_N|^2-|(D''u)_N|^2\bigr)(\omega_{\Sigma,\lambda})^{n+k}=\binom{n+k}{k}\lambda^{n-1}\int_S(e_{D'}-e_{D''})\omega_S^n.
\end{equation}
Integrating \eqref{eq:D'_D''u_identity_subfib_real} and \eqref{eq:D'_D''u_identity_subfib_transgression_real} gives
\begin{align*}
\int_\Sigma\bigl(|(D'u)_N|^2-|(D''u)_N|^2\bigr)(\omega_{\Sigma,\lambda})^{n+k}
&=(n+k)\deg_\zeta(\Sigma,u;\omega_{\Sigma,\lambda})\\
&\quad +(n+k)\int_\Sigma u^*\mu^*(\mathcal C_0)\wedge(\omega_{\Sigma,\lambda})^{n+k-1},\\
\int_S(e_{D'}-e_{D''})\omega_S^n&=n\deg_{\omega_X,\zeta}^{\mathrm{tr}}(\Sigma,u;\omega_S)+n\int_S\widehat\mu_u^*(\mathcal C_0)\wedge\omega_S^{n-1}.
\end{align*}
Fiber integration gives
\[
\int_\Sigma u^*\mu^*(\mathcal C_0)\wedge(\omega_{\Sigma,\lambda})^{n+k-1}=\binom{n+k-1}{k}\lambda^{n-1}\int_S\widehat\mu_u^*(\mathcal C_0)\wedge\omega_S^{n-1}.
\]
Substituting these three identities into \eqref{eq:split_normal_energy_fiber}, the $\mathcal C_0$-terms cancel because $(n+k)\binom{n+k-1}{k}=n\binom{n+k}{k}$, and \eqref{eq:combined_transgressed_degree} follows.
\end{proof}

The following pointwise identity applies to arbitrary sub-fibrations and does not require the properness of $\pi$.

\begin{theorem}\label{thm:Dc_Du_identity_subfib}
Suppose Assumption \ref{assum:comoment}\,(1)--(3) hold and $\nabla^{1,0}$ is a K\"ahler connection. Let $(\Sigma,u)$ be a sub-fibration, and let $\omega_\Sigma$ be a Hermitian form on $\Sigma$. Let $\alpha\in A^1(X,\RR)$ be the real $1$-form of Theorem \ref{thm:Dc_Du_identity}, defined by $\alpha(V):=\mu_{f(x)}^*(K_{\D f_x(V)})(x)$ for $V\in T_xX^\RR$. Then
\begin{equation}\label{eq:Dc_Du_identity_subfib_N}
|(D^c u)_N|^2-|(Du)_N|^2=\Lambda_{\omega_\Sigma}\bigl(-2\D(u^*\alpha)+2u^*\mu^*\bigl(\bigl[G_{D'}^{1,1}-G_{D''}^{1,1}\bigr]_\RR\bigr)\bigr)-\Phi_F(u),
\end{equation}
where
\begin{equation}\label{eq:Phi_F_def}
\Phi_F(u):=4\Re\bigl(\langle(\partial u)_F,(\theta u)_F\rangle_{\omega_\Sigma,h_F}+\langle(\dbar u)_F,(\bar\theta u)_F\rangle_{\omega_\Sigma,h_F}\bigr).
\end{equation}
In particular, if $(f,\omega_{X/S},D'',D)$ is a harmonic bundle, $\nabla^\RR$ preserves $\mathfrak{k}_{X/S}^{\mathrm{R}}$, and $(\Sigma,\omega_\Sigma)$ is compact semi-K\"ahler, then a flat sub-fibration $(\Sigma,u)$ is a Higgs sub-fibration if and only if $\int_\Sigma \Phi_F(u)\omega_\Sigma^{n+k}=0$.
\end{theorem}

\begin{proof}
By Lemma \ref{lem:pullback_comoment_subfib}, the pulled-back datum on $\pi^*X\to\Sigma$ satisfies Assumption \ref{assum:comoment}\,(1)--(3), with comoment map $\mu^*_{\pi^*\!f}=\tilde\pi^*\mu^*$. Similarly to \eqref{eq:def_KV}, use $\pi^*\theta$ to define $\widetilde K$. The $1$-form $\tilde\alpha\in A^1(\pi^*X)$ defined by $\tilde\alpha(V):=\mu^*_{\pi^*\!f}(\widetilde K_{\D(\pi^*\!f)(V)})$ satisfies $\tilde u^*\tilde\alpha=u^*\alpha$. Indeed, for $v\in T_\sigma\Sigma^\RR$, $\D(\pi^*\!f)(\D\tilde u(v))=v$,
\[
\tilde u^*\tilde\alpha(v)=\mu^*_{\pi^*\!f}(\widetilde K_v)|_{\tilde u(\sigma)}=\mu^*(K_{\D\pi(v)})|_{u(\sigma)}=u^*\alpha(v),
\]
using the fiberwise identification $\mathfrak k_{\pi^*\!f,\sigma}^{\CC}\cong\mathfrak k_{\pi(\sigma)}^{\CC}$. Together with $G_{\pi^*D''}^{1,1}=\pi^*G_{D''}^{1,1}$ and $G_{\pi^*D'}^{1,1}=\pi^*G_{D'}^{1,1}$ (Lemma \ref{lem:pullback_curv_subfib}), applying Theorem \ref{thm:Dc_Du_identity} to $\tilde u:\Sigma\to\pi^*X$ and using \eqref{eq:pullback_ident}, we obtain \begin{equation}\label{eq:Dc_Du_identity_subfib_total}
|D^cu|^2-|Du|^2 = \Lambda_{\omega_\Sigma}\bigl(-2\D(u^*\alpha)+2\,u^*\mu^*\bigl(\bigl[G_{D'}^{1,1}-G_{D''}^{1,1}\bigr]_\RR\bigr)\bigr).
\end{equation}
A direct computation gives \[ |(D^c u)_F|^2-|(Du)_F|^2=\Phi_F(u),\qquad |(D^c u)_N|^2-|(Du)_N|^2=(|D^cu|^2-|Du|^2)-\Phi_F(u), \]
and \eqref{eq:Dc_Du_identity_subfib_N} follows. The last statement is clear.
\end{proof}
In the following, we show that the obstruction $\int_\Sigma \Phi_F(u)\omega_\Sigma^{n+k}$ always vanishes for flat sub-fibrations under the above compactness and harmonicity hypotheses. Since $F=0$ when $k=0$, we only need to consider the case when $k\ge 1$.
\begin{theorem}\label{thm:Dc_Du_identity_subfib_transgression}
Suppose Assumption \ref{assum:comoment}\,(1)--(3) hold, and $(\Sigma,u)$ is a sub-fibration of relative dimension $k\geq0$ such that $\pi:=f\comp u$ is proper and is unramified when $k=0$. Let $\omega_S$ be a Hermitian form on $S$. For $V=v+\bar v\in TS^\RR$, let $K_V=(\bar\theta(\bar v)-\theta(v))^\RR$, $\kappa_V:=(K_V\comp u_s)_{N^\RR}$, and $\beta_u(V):=\widehat\mu_{s,u_s}^*(K_V)$. Then
\begin{equation}\label{eq:Dc_Du_identity_subfib_transgression}
e_{D^c}-e_D=-2\Lambda_{\omega_S}\D\beta_u+2\Lambda_{\omega_S}\widehat\mu_u^*\bigl((\D_{\mathfrak h}K)^{1,1}\bigr),
\end{equation}
where $e_D(s):=\int_{\Sigma_s}|(Du)_N|_{\omega_S,h_N}^2\nu_s$ and $e_{D^c}(s):=\int_{\Sigma_s}|(D^cu)_N|_{\omega_S,h_N}^2\nu_s$. If $\nabla^{1,0}$ is a K\"ahler connection, then
\begin{equation}\label{eq:Dc_Du_identity_subfib_transgression_curv}
e_{D^c}-e_D=-2\Lambda_{\omega_S}\D\beta_u+2\Lambda_{\omega_S}\widehat\mu_u^*\bigl(\bigl[G_{D'}^{1,1}-G_{D''}^{1,1}\bigr]_\RR\bigr).
\end{equation}
In particular, if $(S,\omega_S)$ is compact semi-K\"ahler, $(f,\omega_{X/S},D'',D)$ is a harmonic bundle, and $\nabla^\RR$ preserves $\mathfrak{k}_{X/S}^{\mathrm R}$, then every flat sub-fibration is a Higgs sub-fibration.
\end{theorem}

\begin{proof}
Fix $s\in S$ and $\xi\in\mathfrak h_s$, and put $h_\xi:=\mu_s^*(\xi)$. Let $u_{s,t}:\Sigma_s\to X_s$ be a smooth variation through immersions with $u_{s,0}=u_s$ and variational normal class $A$, and let $A^\perp$ be its orthogonal representative. Then
\[
\left.\frac{\D}{\D t}\right|_{t=0}u_{s,t}^*(h_\xi\omega_s^k)=u_s^*\bigl(\iota_{\dot u_{s,0}}\D(h_\xi\omega_s^k)\bigr)+\D\bigl(u_s^*\iota_{\dot u_{s,0}}(h_\xi\omega_s^k)\bigr).
\]
The second term integrates to zero by Stokes' theorem. Moreover, replacing $\dot u_{s,0}$ by $A^\perp$ does not change the first integral, since their difference lies pointwise in $\D u_s(T\Sigma_s)$, whose contraction with the $(2k+1)$-form $\D(h_\xi\omega_s^k)$ vanishes after pullback to the real $2k$-manifold $\Sigma_s$. Since $\D h_\xi=-\iota_\xi\omega_s$ and $\D\omega_s=0$, we have $\D(h_\xi\omega_s^k)=\D h_\xi\wedge\omega_s^k=-\iota_\xi\omega_s^{k+1}/(k+1)$. Consequently, Cartan's formula and Stokes' theorem give
\[
\begin{aligned}
\left.\frac{\D}{\D t}\right|_{t=0}\widehat\mu_{s,u_{s,t}}^*(\xi)
&=\int_{\Sigma_s}u_s^*\bigl(\iota_{A^\perp}\D(h_\xi\omega_s^k)\bigr)=-\frac1{k+1}\int_{\Sigma_s}u_s^*\bigl(\iota_{A^\perp}\iota_\xi\omega_s^{k+1}\bigr)\\
&=-\Omega_{s,u_s}(\widehat\xi_{u_s},A)=\Omega_{s,u_s}(A,\widehat\xi_{u_s}).
\end{aligned}
\]

We now compute $\D\beta_u$. Extend $V,W\in T_sS^\RR$ to local real vector fields, let $\gamma(t)$ be the integral curve of $V$ through $s$, and choose a $\pi$-projectable lift of $V$ whose local flow induces $\varphi_t:\Sigma_s\to\Sigma_{\gamma(t)}$. Let $P_t$ be the local parallel transport of $\nabla^\RR$ along $\gamma$. For sufficiently small $t$, it is defined on a neighborhood of the compact set $u_s(\Sigma_s)$ and satisfies $P_t^*\omega_{\gamma(t)}=\omega_s$. Set $j_t:=P_t^{-1}\comp u_{\gamma(t)}\comp\varphi_t:\Sigma_s\longrightarrow X_s$. The normal class of $\dot j_0$ is $A_V$. If $h_t:=\mu_{\gamma(t)}^*(K_{W_{\gamma(t)}})$, then
\[
\beta_u(W)(\gamma(t))=\int_{\Sigma_s}\bigl((P_t^*h_t)\comp j_t\bigr)(j_t^*\omega_s)^k.
\]
Assumption \ref{assum:comoment}\,(3) implies $\left.\frac{\D}{\D t}\right|_{t=0}P_t^*h_t=\mu_s^*([H_V,K_W])$. Then
\[
V\bigl(\beta_u(W)\bigr)=\Omega_{s,u_s}(A_V,\kappa_W)+\widehat\mu_{s,u_s}^*([H_V,K_W]).
\]
Antisymmetrizing and subtracting $\beta_u([V,W])$ gives
\begin{equation}\label{eq:dbeta_u_normal}
(\D\beta_u)(V,W)=\Omega_{s,u_s}(A_V,\kappa_W)-\Omega_{s,u_s}(A_W,\kappa_V)+\widehat\mu_{s,u_s}^*\bigl((\D_{\mathfrak h}K)(V,W)\bigr),
\end{equation}
where $(\D_{\mathfrak h}K)(V,W):=[H_V,K_W]-[H_W,K_V]-K_{[V,W]}$. If $\mu^*$ is defined only on $\mathfrak{k}_{X/S}^{\mathrm R}$, Assumption \ref{assum:comoment}\,(3) ensures that every vector field in the last expression remains in its domain.

The connections $\nabla_+:=\nabla^\RR+K$ and $\nabla_-:=\nabla^\RR-K$ have horizontal lifts $H_V+K_V$ and $H_V-K_V$. Correspondingly, $A_V^+:=A_V-\kappa_V$ and $A_V^-:=A_V+\kappa_V$. Define $\Omega_\pm(V,W):=\Omega_{s,u_s}(A_V^\pm,A_W^\pm)$. Extend $K$ complex-linearly to $TS^\CC$. By the proof of Theorem \ref{thm:Dc_Du_identity}, $K_v=-\theta(v)+\overline{\bar\theta(\bar v)}$ and $K_{\bar v}=\bar\theta(\bar v)-\overline{\theta(v)}$. Fix an $\omega_S$-unitary frame $v_1,\ldots,v_n$ of $T_s S$ and write $a_j:=(\partial u)_N(v_j)$, $b_j:=(\dbar u)_N(\bar v_j)$, $c_j:=(\theta u)_N(v_j)$, and $d_j:=(\bar\theta u)_N(\bar v_j)$. The type components of $(A^+)^{1,0}$ are $a_j+c_j,b_j-d_j$, those of $(A^-)^{1,0}$ are $a_j-c_j,b_j+d_j$, those of $(Du)_N$ are $a_j-c_j,b_j-d_j$, and those of $(D^cu)_N$ are $-(a_j+c_j),b_j+d_j$. Similarly to \eqref{eq:normal_partial_dbar_transgression},
\begin{equation}\label{eq:Dc_Du_Omega_pm}
e_{D^c}-e_D=\Lambda_{\omega_S}(\Omega_+-\Omega_-).
\end{equation}

Direct expansion and skew-symmetry give
\[
(\Omega_+-\Omega_-)(V,W)=-2\Omega_{s,u_s}(A_V,\kappa_W)+2\Omega_{s,u_s}(A_W,\kappa_V).
\]
Comparing with \eqref{eq:dbeta_u_normal}, we obtain
\[
\Omega_+-\Omega_-=-2\D\beta_u+2\widehat\mu_u^*(\D_{\mathfrak h}K).
\]
Taking the $(1,1)$-part and applying $\Lambda_{\omega_S}$ proves \eqref{eq:Dc_Du_identity_subfib_transgression}. If $\nabla^{1,0}$ is K\"ahler, the proof of Theorem \ref{thm:Dc_Du_identity} gives $(\D_{\mathfrak h}K)^{1,1}=\bigl[G_{D'}^{1,1}-G_{D''}^{1,1}\bigr]_\RR$, and \eqref{eq:Dc_Du_identity_subfib_transgression_curv} follows. The last statement follows from Stokes' theorem.
\end{proof}

\begin{remark}
Suppose $k\ge1$. If $(\Sigma,\omega_\Sigma)$ is compact semi-K\"ahler, then $S$ admits a semi-Kähler metric. If $n=1$, every Hermitian form $\omega_S$ is semi-K\"ahler. If $n\geq2$, set $\chi:=\pi_*(\omega_\Sigma^{n+k-1})\in A^{n-1,n-1}(S,\RR)$. Fiber integration commutes with $\D$, so $\D\chi=0$. Moreover, $\chi$ is strictly positive, so there exists a unique strictly positive $(1,1)$-form $\omega_S$ satisfying $\omega_S^{n-1}=\chi$ (see \cite[Eq.~(4.8)]{Michelsohn82}). Hence $(S,\omega_S)$ is compact semi-K\"ahler.
\end{remark}

\subsection{Projectivized subbundles and the transgressed degree}
\label{subsec:proj_subbundle_degree}
This subsection computes the transgressed degree of a projectivized holomorphic subbundle and then uses Lemma \ref{lem:combined_transgressed_degree} to recover its combined degree. The resulting formula generalizes Lemma \ref{lem:deg_line_subbundle}.

Throughout, $(E,\dbar_E,h)$ is a Hermitian holomorphic vector bundle of rank $m\geq2$ over a compact K\"ahler manifold $(S,\omega_S)$ of complex dimension $n\geq1$. We projectivize as in Section \ref{subsec:harm_vec_bun}: $f:X:=\PP(E)\to S$ carries the holomorphic fiber bundle structure and the relatively K\"ahler form $\omega_X$ induced by $h$, satisfying $[\omega_X]=c_1\bigl(\mathcal O_{\PP(E)}(1)\bigr)$, and $\zeta=-\frac1m c_1(E,h)$ in \eqref{eq:min-coupling}. Here $\PP(E)$ denotes the bundle of lines in $E$.

Let $\mathcal F\subset E$ be a holomorphic subbundle of rank $k+1\geq1$. Set $\Sigma:=\PP(\mathcal F)$, let $\pi:\Sigma\to S$ be the projection, and let $u:\Sigma=\PP(\mathcal F)\hookrightarrow\PP(E)=X$ be the natural inclusion. Then $(\Sigma,u)$ is a sub-fibration of relative dimension $k$.

\begin{proposition}\label{prop:proj_subbundle_transgressed_degree}
We have
\begin{equation}\label{eq:proj_subbundle_transgressed_degree}
\deg_{\omega_X,\zeta}^{\mathrm{tr}}(\Sigma,u)=\frac1m\deg(E)-\frac1{k+1}\deg(\mathcal F).
\end{equation}
\end{proposition}

\begin{proof}
Let $\xi:=c_1\bigl(\mathcal O_{\PP(\mathcal F)}(1)\bigr)\in H^2(\Sigma,\RR)$. Since the inclusion $u$ identifies $u^*\mathcal O_{\PP(E)}(1)$ with $\mathcal O_{\PP(\mathcal F)}(1)$, we have $[u^*\omega_X]=\xi$. Moreover, $f\comp u=\pi$ and $\zeta=-\frac1m c_1(E,h)$, so $\bigl[u^*(\omega_X-f^*\zeta)\bigr]=\xi+\frac1m\pi^*c_1(E)$. The tautological exact sequence $0\to \mathcal O_{\PP(\mathcal F)}(-1)\to \pi^*\mathcal F\to Q\to 0$ gives $c(Q)=\frac{\pi^*c(\mathcal F)}{1-\xi}$. Since $\operatorname{rank}Q=k$, we have $c_{k+1}(Q)=0$, implying $\xi^{k+1}+\pi^*c_1(\mathcal F)\xi^k+\cdots+\pi^*c_{k+1}(\mathcal F)=0$.
Together with the normalization $\int_{\PP(\mathcal F_s)}\xi^k=1$, this gives $\pi_*(\xi^k)=1$ and $\pi_*(\xi^{k+1})=-c_1(\mathcal F)$, while $\pi_*(\xi^j)=0$ for $j<k$.

By Definition \ref{def:transgressed_degree_subfib}, $[\hat\omega_{u,\zeta}]=\frac1{k+1}\pi_*\bigl(\xi+\frac1m\pi^*c_1(E)\bigr)^{k+1}=\frac1{k+1}\bigl(-c_1(\mathcal F)+\frac{k+1}{m}c_1(E)\bigr)$. Wedging with $\omega_S^{n-1}$ and integrating over $S$ proves \eqref{eq:proj_subbundle_transgressed_degree}.
\end{proof}

Let $\omega_{\mathrm{FS},\mathcal F}$ be the vertical Fubini--Study form on $\PP(\mathcal F)$, extended by zero on the horizontal distribution of the Chern connection of $(\mathcal F,h|_{\mathcal F})$ and normalized by $\int_{\PP(\mathcal F_s)}\omega_{\mathrm{FS},\mathcal F}^{k}=1$. For $\lambda>0$, define $\omega_{\Sigma,\lambda}:=\omega_{\mathrm{FS},\mathcal F}+\lambda\,\pi^*\omega_S$. Note that by Lemma \ref{lem:deg_theta_F_split}, $\deg(\Sigma,u;\omega_{\Sigma,\lambda})$ is independent of a Higgs field on $E$.

\begin{corollary}\label{cor:split_adiabatic_exact}
For every $\lambda>0$,
\begin{equation}\label{eq:split_adiabatic_exact}
\deg(\Sigma,u;\omega_{\Sigma,\lambda})=-\frac{\lambda^{n-1}}{k+1}\binom{n+k-1}{k}\deg(\mathcal F).
\end{equation}
\end{corollary}

\begin{proof}
The case $k=0$ is just Lemma \ref{lem:deg_line_subbundle}. Assume henceforth that $k\geq1$. The metric $\omega_{\Sigma,\lambda}$ is the split metric of Lemma \ref{lem:combined_transgressed_degree}. Consequently,
\[
\deg_\zeta(\Sigma,u;\omega_{\Sigma,\lambda})
=\binom{n+k-1}{k}\lambda^{n-1}\deg_{\omega_X,\zeta}^{\mathrm{tr}}(\Sigma,u).
\]
Since only the term containing $k$ vertical factors contributes to the fiber integral and $\int_{\PP(\mathcal F_s)}\omega_{\mathrm{FS},\mathcal F}^k=1$, we have
\[
\int_\Sigma\pi^*\zeta\wedge\omega_{\Sigma,\lambda}^{n+k-1}=\binom{n+k-1}{k}\lambda^{n-1}\int_S\zeta\wedge\omega_S^{n-1}=-\frac1m\binom{n+k-1}{k}\lambda^{n-1}\deg(E).
\]
By the definition of $\deg_\zeta$ and Proposition \ref{prop:proj_subbundle_transgressed_degree}, \eqref{eq:split_adiabatic_exact} follows.
\end{proof}

\section{Higgs morphisms and flat morphisms}\label{sec:morphisms}
In the classical nonabelian Hodge theory for vector bundles, the equivalence between Higgs morphisms and flat morphisms between two harmonic vector bundles $E_1$ and $E_2$ is governed by the auxiliary harmonic vector bundle $\Hom(E_1,E_2)=E_1^{*}\otimes E_2$, which reduces the correspondence for morphisms to the correspondence for sections. In the general nonlinear setting the fibers $X_{i,s}$ have no vector space structure, and the space of fiberwise holomorphic maps $X_{1,s}\to X_{2,s}$ is in general infinite-dimensional. Instead, we realize a morphism $\mathsf{F}:X_1\to X_2$ as the graph sub-fibration of $\mathsf{F}$ inside the fiber product $X_1\times_S X_2\to S$.

Throughout this section, $(S,\omega_S)$ is a connected Hermitian manifold of complex dimension $n\ge 1$. For $j=1,2$, let $(f_j:X_j\to S,T_{X_j/S})$ be a complex fiber bundle with connected fibers of complex dimension $m_j$, equipped with a $\dbar$-operator $\dbar_j$, a $\theta_j$-adapted fiberwise K\"ahler metric $\omega_{X_j/S}$, a symplectic connection $\nabla_j^{1,0}$ associated to $(\dbar_j,\omega_{X_j/S})$ and induced by a real $(1,1)$-form $\omega_j$ on $X_j$, and a relatively holomorphic almost Higgs field $\theta_j\in A^{1,0}(S,\mathfrak{k}_{X_j/S}^{\CC})$. We write $h_{X_j/S}$ for the Hermitian metric associated to $\omega_{X_j/S}$, abbreviate $\bar\theta_j:=\bar\theta_{\omega_{X_j/S}}$, and set $D_j'':=\dbar_j+\theta_j$, $D_j':=\partial_j+\bar\theta_j$, $D_j:=D_j'+D_j''$, where $\partial_j$ is the almost connection induced by $\nabla_j^{1,0}$.

\subsection{The product bundle}
\label{subsec:product_bundle}
Let $X:=X_1\times_S X_2$ and $f:=f_1\comp p_1=f_2\comp p_2:X\to S$, where $p_j:X\to X_j$ are the canonical projections, so that $X_s=X_{1,s}\times X_{2,s}$ has complex dimension $m:=m_1+m_2$ and the relative tangent bundle splits as
\begin{equation}\label{eq:T_X_split}
T_{X/S}=p_1^{*}T_{X_1/S}\oplus p_2^{*}T_{X_2/S}.
\end{equation}
We endow $f:X\to S$ with the product data $\dbar_X:=p_1^{*}\dbar_1\oplus p_2^{*}\dbar_2$, $\partial_X:=p_1^{*}\partial_1\oplus p_2^{*}\partial_2$ (as operators on sections), and
\begin{align}
\omega_{X/S}&:=p_1^{*}\omega_{X_1/S}+p_2^{*}\omega_{X_2/S},\label{eq:omegaX_S}\\
\omega_X&:=p_1^{*}\omega_1+p_2^{*}\omega_2,\label{eq:omegaX}\\
\theta_X(v)|_{(x_1,x_2)}&:=(\theta_1(v)|_{x_1},\theta_2(v)|_{x_2}),\label{eq:thetaX}
\end{align}
where $v\in T_sS$, $s=f_1(x_1)=f_2(x_2)$. We have $\theta_X\in A^{1,0}(S,\mathfrak{k}_{X/S}^\CC)$ with
\begin{equation}\label{eq:k_product}
\mathfrak{k}_{X/S,s}:=\mathfrak{k}_{X_1/S,s}\oplus\mathfrak{k}_{X_2/S,s}.
\end{equation}
Denote $D_X'':=\dbar_X+\theta_X$, $D_X':=\partial_X+\bar{\theta}_X$, and $D_X:=D_X'+D_X''$.

\begin{lemma}\label{lem:product_basic}
The product data make $f:X\to S$ a complex fiber bundle with a $\dbar$-operator, a fiberwise K\"ahler metric, a symplectic connection $\nabla_X^{1,0}$ induced by $\omega_X$, and a relatively holomorphic almost Higgs field $\theta_X$, with $\bar\theta_X(\bar v)=(\bar\theta_1(\bar v),\bar\theta_2(\bar v))$. The induced almost complex structure on $X$ is the fiber-product almost complex structure. If the $\dbar_j$ are integrable, then $\dbar_X$ is integrable and $f$ is a holomorphic fibration.
\end{lemma}
\begin{proof}
All structures are defined componentwise along \eqref{eq:T_X_split}; the $\dbar$-operator, almost connection, metric, Higgs field, and conjugate Higgs field are direct sums. The horizontal subbundle of $\nabla_X^\RR$ is the $\omega_X$-orthogonal complement of $T_{X/S}^\RR$ in $TX^\RR$, and this orthogonality decouples between the two factors because $\omega_X$ is a sum of pullbacks. Hence $\nabla_X^{1,0}=p_1^{*}\nabla_1^{1,0}\oplus p_2^{*}\nabla_2^{1,0}$ as an operator on sections, and $\nabla_X^{1,0}$ is a symplectic connection associated to $\omega_{X/S}$ and $\dbar_X$. Relative holomorphicity of $\theta_X$ and the decomposition of $\bar\theta_X$ follow from those of $\theta_j$, using $\bar\theta=-\phi(\theta)$ and $\phi_X=\phi_1\oplus\phi_2$ on $\mathfrak{k}_{X/S}^\CC$. Integrability is componentwise.
\end{proof}

\begin{lemma}\label{lem:product_harmonic}
Assume in addition that for $j=1,2$ the data $(f_j,\omega_{X_j/S},D_j'',D_j)$ satisfy the lifting conditions of Section \ref{sec:kahler_id}, so that $\nabla_j^{1,0}$ is a K\"ahler connection. Then $\nabla_X^{1,0}$ is a K\"ahler connection, and $F_{D_X}=p_1^{*}F_{D_1}\oplus p_2^{*}F_{D_2}$, $G_{D_X''}=p_1^{*}G_{D_1''}\oplus p_2^{*}G_{D_2''}$.
In particular, if both factors are nonlinear harmonic bundles, then so is the product. If Assumption \ref{assum:comoment}\,(i) \textup($i\in\{1,2,3,4\}$\textup) holds for both factors, then it holds for the product \footnote{Except that in Assumption \ref{assum:comoment}\,(2) for $X$, we define $\mu_{X,s}^*$ on $\mathfrak{h}_{X/S,s}:=\{(p_1^*\xi_1,p_2^*\xi_2)\,|\,
\xi_j\in\mathfrak h_{X_j/S,s}\}$, where $\mu_{j,s}^*$ is defined on $\mathfrak{h}_{X_j/S,s}$. In general, $\mathfrak{h}_{X/S}$ need not be $\mathrm{Ham}_{X/S}$ even when $\mathfrak{h}_{X_j/S}=\mathrm{Ham}_{X_j/S}$.}, with comoment map $\mu_X^{*}(\xi)=p_1^{*}\mu_1^{*}(\xi_1)+p_2^{*}\mu_2^{*}(\xi_2)$ for a split vector field $\xi=(p_1^*\xi_1,p_2^*\xi_2)$. If $\nabla_j^\RR$ preserves $\mathfrak{h}_{X_j/S}$ for $j=1,2$, then $\nabla_X^\RR$ preserves $\mathfrak{h}_{X/S}$. If the associated forms are $\zeta_1,\zeta_2$, then $\zeta_X=\zeta_1+\zeta_2$.
\end{lemma}
\begin{proof}
All local coefficients are direct sums, giving the K\"ahler-connection property. Also, the local formulas for the (pseudo-)curvature split componentwise. For Assumption \ref{assum:comoment}: (1) follows from the definition of $\mathfrak{k}_{X/S}$ in \eqref{eq:k_product}; (2) $\mu_X^{*}$ satisfies $\D_{X_s}\mu_X^{*}(\xi)=-\iota_\xi\omega_{X,s}$ and equivariance, as $\omega_{X,s}$ is a sum of pullbacks and the bracket on $\mathfrak{k}_{X/S}^\CC$ is block-diagonal; (3) parallelism and preservation of $\mathfrak{h}_{X/S}$ hold componentwise; (4) $\omega_X$ is closed by \eqref{eq:omegaX}. We also have $F_{\nabla_X^\RR}\in A^2(S,\mathfrak h_{X/S})$. Finally $\omega_X^H=p_1^{*}\omega_1^H+p_2^{*}\omega_2^H$, so substituting \eqref{eq:min-coupling} for each factor and using $f_j\comp p_j=f$ gives $\omega_X^H=-\mu_X^{*}F_{\nabla_X^\RR}+f^{*}(\zeta_1+\zeta_2)$, whence $\zeta_X=\zeta_1+\zeta_2$.
\end{proof}

\subsection{Morphisms and their graphs}\label{subsec:graph}
\begin{definition}\label{def:fiberwise_hol_morphism}
A \emph{smooth morphism over $S$} from $f_1$ to $f_2$ is a smooth map $\mathsf{F}:X_1\to X_2$ with $f_2\comp \mathsf{F}=f_1$. It is \emph{fiberwise holomorphic} if $\mathsf{F}_s:=\mathsf{F}|_{X_{1,s}}:X_{1,s}\to X_{2,s}$ is holomorphic for every $s\in S$, in which case its vertical differential $\D^{\mathrm{v}} \mathsf{F}:T_{X_1/S}\to \mathsf{F}^{*}T_{X_2/S}$ is complex-linear.
\end{definition}

A smooth morphism $\mathsf{F}$ determines the graph map
\begin{equation}\label{eq:u_def}
u:X_1\longrightarrow X,\qquad u(x):=(x,\mathsf{F}(x)),
\end{equation}
satisfying $p_1\comp u=\id_{X_1}$, $p_2\comp u=\mathsf{F}$, and $f\comp u=f_1$.

\begin{lemma}\label{lem:graph_subfib}
Assume $\dbar_1$ is integrable, so that $f_1:X_1\to S$ is a holomorphic fibration. Then $(\Sigma,u)$ with $\Sigma:=X_1$ is a sub-fibration of relative dimension $m_1$ for $f:X\to S$ in the sense of Definition~\ref{def:subfib_datum} if and only if $\mathsf{F}$ is fiberwise holomorphic. In this case:
\begin{enumerate}[label=(\roman*)]
\item $E$ has the canonical splitting
\begin{equation}\label{eq:E_split}
E=T_{X_1/S}\oplus \mathsf{F}^{*}T_{X_2/S},
\end{equation}
and the Hermitian metric on $E$ induced by $\omega_{X/S}$ is $h_E:=h_{X_1/S}\oplus\mathsf{F}^{*}h_{X_2/S}$;
\item the relative tangent subbundle $F\subset E$ of Definition~\ref{def:F_N} is
\begin{equation}\label{eq:F_sub_def}
F=\bigl\{\bigl(W,\D^{\mathrm{v}} \mathsf{F}(W)\bigr):W\in T_{X_1/S}\bigr\};
\end{equation}
\item the projection to the normal bundle is
\begin{equation}\label{eq:pr_N_explicit}
\mathrm{pr}_N:E\twoheadrightarrow N:=E/F\xrightarrow{\;\cong\;}\mathsf{F}^{*}T_{X_2/S},\quad \mathrm{pr}_N(W_1,W_2)=W_2-\D^{\mathrm{v}} \mathsf{F}(W_1);
\end{equation}
\item $F^\perp
=\{(-(\D^{\mathrm{v}} \mathsf{F})^*W,W):W\in\mathsf F^*T_{X_2/S}\}
\subset T_{X_1/S}\oplus\mathsf F^*T_{X_2/S}$, where $F^\perp$ is the $h_E$-orthogonal complement of $F$ and $(\D^{\mathrm{v}} \mathsf{F})^*:\mathsf{F}^{*}T_{X_2/S}\to T_{X_1/S}$ is the Hermitian adjoint of $\D^{\mathrm{v}} \mathsf{F}$ with respect to $h_{X_1/S}$ and $\mathsf{F}^{*}h_{X_2/S}$.
\end{enumerate}
\end{lemma}
\begin{proof}
The composition $\pi:=f\comp u=f_1$ is a holomorphic submersion of relative complex dimension $m_1$. The fiberwise map $u_s:X_{1,s}\to X_s$, $x\mapsto(x,\mathsf{F}_s(x))$ has differential $W\mapsto(W,\D \mathsf{F}_s(W))$, hence is an immersion, and is holomorphic if and only if $\mathsf{F}_s$ is. Parts (i)--(iii) follow from \eqref{eq:T_X_split} restricted along $u$ together with the verification that $\mathrm{pr}_N$ has kernel $F$. For (iv), $\langle(W_1,W_2),(W,\D^{\mathrm{v}} \mathsf{F}(W))\rangle_{h_E}=\langle W_1+(\D^{\mathrm{v}} \mathsf{F})^{*}(W_2),W\rangle_{h_{X_1/S}}$ vanishes for all $W$ if and only if $W_1=-(\D^{\mathrm{v}} \mathsf{F})^{*}(W_2)$.
\end{proof}
\begin{remark}
The isomorphism $N\cong\mathsf F^*T_{X_2/S}$ is not, in general, an isometry for the quotient metric. Under this identification, $h_N(y_1,y_2)=\langle (\id+\D^{\mathrm v}\mathsf F(\D^{\mathrm v}\mathsf F)^*)^{-1}y_1,y_2\rangle_{h_{X_2/S}}$.
\end{remark}

\subsection{Higgs morphisms and flat morphisms}\label{subsec:morphism_def}
The horizontal lift of $D_i$ is given by
\begin{equation}\label{eq:flat_lift_def}
H_{D_i}(v)|_x:=H_i(v)|_x+\bigl(\theta_i(v^{1,0})+\bar\theta_i(v^{0,1})\bigr)^\RR |_x\in T_xX_i^\RR,\quad v\in T_sS^\RR,\;x\in X_{i,s},
\end{equation}
where $H_i(v)|_x$ is the horizontal lift via $\nabla_i^\RR$ and $(\xi)^\RR:=\xi+\bar\xi$. A section of $f_i$ is $D_i$-flat
in the sense of Section \ref{sec:kahler_id} if and only if its
differential equals $H_{D_i}$ on $TS^\RR$.

\begin{definition}\label{def:Higgs_flat_morphism}
Let $\mathsf{F}:X_1\to X_2$ be a smooth morphism over $S$.
\begin{enumerate}[label=(\roman*)]
\item $\mathsf{F}$ is a \emph{Higgs morphism} if $\mathsf{F}$ is a pseudo-holomorphic map (in particular fiberwise holomorphic) with respect to the almost complex structures $J_j$ determined by $\dbar_j$, and for every $v\in T_sS$ and every $x\in X_{1,s}$,
\begin{equation}\label{eq:Higgs_morphism_def}
\D \mathsf{F}\bigl(\theta_1(v)|_x\bigr)=\theta_2(v)|_{\mathsf{F}(x)}.
\end{equation}
\item $\mathsf{F}$ is a \emph{flat morphism} if for every $v\in T_sS^\RR$ and every $x\in X_{1,s}$,
\begin{equation}\label{eq:flat_morphism_def}
\D \mathsf{F}\bigl(H_{D_1}(v)|_x\bigr)=H_{D_2}(v)|_{\mathsf{F}(x)}.
\end{equation}
Since $f_2\comp\mathsf F=f_1$, this is equivalent to $\mathsf F$ preserving the horizontal distributions.
\end{enumerate}
\end{definition}

\begin{proposition}\label{prop:morphism_subfib_correspondence}
Assume $\dbar_1$ is integrable. Let $\mathsf{F}:X_1\to X_2$ be a smooth morphism over $S$.
\begin{enumerate}[label=(\roman*)]
\item $\mathsf{F}$ is a Higgs morphism if and only if $(\Sigma,u)$ in Lemma \ref{lem:graph_subfib} is a Higgs sub-fibration.
\item $\mathsf{F}$ is fiberwise holomorphic and is a flat morphism if and only if $(\Sigma,u)$ is a flat sub-fibration.
\end{enumerate}
\end{proposition}

\begin{proof}
We work in adapted local holomorphic coordinates $(s^i,z_1^\alpha)$ on $X_1$ (in particular, $\Gamma_{1,\bar i}^{\alpha}=0$) and adapted local coordinates $(s^i,z_2^a)$ on $X_2$. In coordinates $(s^i,z_1^\alpha,z_2^a)$ on $X$, the connection coefficients of $\nabla_X^{1,0}$ are $\Gamma_{X,i}^\alpha=\Gamma_{1,i}^\alpha(s,z_1)$ and $\Gamma_{X,i}^a=\Gamma_{2,i}^a(s,z_2)$ by Lemma \ref{lem:product_basic}. Write $\mathsf{F}:(s,\bar s, z_1,\bar z_1)\mapsto(s,\mathsf{F}^a(s,\bar s,z_1,\bar z_1))$, $\theta_1(\partial_{s^i})=\theta_{1,i}^{\alpha}\partial_{z_1^\alpha}$, and $\theta_2(\partial_{s^i})=\theta_{2,i}^{a}\partial_{z_2^a}$.

(i) If $(\Sigma,u)$ is a Higgs sub-fibration, then $\mathsf{F}$ is fiberwise holomorphic and $\partial_{\bar z_1^\beta}\mathsf{F}^a=0$. By Definition \ref{def:dbar_on_section} applied to the section $\tilde u:\Sigma\to f_1^{*}X$ and \eqref{eq:pullback_ident},
\begin{equation}\label{eq:dbar_u_local}
\dbar u(\partial_{\bar s^i})=(\partial_{\bar s^i}\mathsf{F}^a-\Gamma_{2,\bar{i}}^a\comp\mathsf{F}) \partial_{z_2^a}|_{u(x)},\qquad \dbar u(\partial_{\bar z_1^\beta})=(\partial_{\bar z_1^\beta}\mathsf F^a)\partial_{z_2^a}|_{u(x)}=0.
\end{equation}
By Definition \ref{def:higgs_on_section},
\begin{equation}\label{eq:thetaX_u_local}
\theta_X u(\partial_{s^i})|_{u(x)}=\bigl(\theta_{1,i}^\alpha(s,z_1)\partial_{z_1^\alpha},\,\theta_{2,i}^a(s,\mathsf{F}(s,\bar s,z_1,\bar z_1))\partial_{z_2^a}\bigr).
\end{equation}
Projecting via \eqref{eq:pr_N_explicit},
\begin{align}
(\dbar u)_N(\partial_{\bar s^i})&=(\partial_{\bar s^i}\mathsf{F}^a-\Gamma_{2,\bar i}^a\comp\mathsf F)\partial_{z_2^a}|_{u(x)},\quad (\dbar u)_N(\partial_{\bar z_1^\beta})=0, \label{eq:dbar_u_N}\\
(\theta_X u)_N(\partial_{s^i})&=\bigl(\theta_{2,i}^a\comp \mathsf{F}-(\partial_{z_1^\alpha}\mathsf{F}^a)\theta_{1,i}^\alpha\bigr)\partial_{z_2^a}|_{u(x)}.\label{eq:thetaX_u_N}
\end{align}

By Definition \ref{def:partial_on_section} applied to $\tilde u$, for $\partial_{s^i},\partial_{z_1^\beta}\in T_xX_1$,
\begin{align*}
\partial u(\partial_{s^i})\big|_{u(x)}&=\bigl(\partial_{s^i}z_1^\alpha-\Gamma_{1,i}^\alpha(s,z_1)\bigr)\partial_{z_1^\alpha}+\bigl(\partial_{s^i}\mathsf{F}^a-\Gamma_{2,i}^a\comp \mathsf{F}\bigr)\partial_{z_2^a}\\
&=-\Gamma_{1,i}^\alpha\partial_{z_1^\alpha}+\bigl(\partial_{s^i}\mathsf{F}^a-\Gamma_{2,i}^a\comp \mathsf{F}\bigr)\partial_{z_2^a},\\
\partial u(\partial_{z_1^\beta})|_{u(x)}&=\partial_{z_1^\beta}+(\partial_{z_1^\beta}\mathsf{F}^a)\partial_{z_2^a}\in F_x.
\end{align*}
Projecting,
\begin{align}
(\partial u)_N(\partial_{s^i})&=\bigl(\partial_{s^i}\mathsf{F}^a+\Gamma_{1,i}^\alpha\partial_{z_1^\alpha}\mathsf{F}^a-\Gamma_{2,i}^a\comp \mathsf{F}\bigr)\partial_{z_2^a}|_{u(x)},\label{eq:partial_u_N}\\
(\partial u)_N(\partial_{z_1^\beta})&=0.\label{eq:partial_u_vert_zero}
\end{align}
Similarly to \eqref{eq:thetaX_u_N}, we have
\begin{equation}\label{eq:thetabarX_u_N}
(\bar\theta_Xu)_N(\partial_{\bar s^i})=\bigl(\bar\theta_{2,\bar i}^a\comp \mathsf{F}-(\partial_{z_1^\alpha}\mathsf{F}^a)\bar\theta_{1,\bar i}^\alpha\bigr)\partial_{z_2^a}|_{u(x)}.
\end{equation}

By \eqref{eq:dbar_u_N}, $(\dbar u)_N=0$ if and only if $\partial_{\bar s^i}\mathsf{F}^a=\Gamma_{2,\bar i}^a\comp\mathsf F$ for all $i,a$, which combined with fiberwise holomorphicity gives the pseudo-holomorphicity of $\mathsf{F}$. By \eqref{eq:thetaX_u_N}, $(\theta_Xu)_N=0$ if and only if $\theta_{2,i}^a\comp \mathsf{F}=(\partial_{z_1^\alpha}\mathsf{F}^a)\theta_{1,i}^\alpha$, the local form of \eqref{eq:Higgs_morphism_def}. The converse is immediate.

(ii) Suppose $\mathsf{F}$ is fiberwise holomorphic. We have $H_{D_j}(\partial_{s^i})=\partial_{s^i}+(\Gamma_{j,i}^\beta+\theta_{j,i}^\beta)\partial_{z_j^\beta}+\bigl(\overline{\Gamma_{j,\bar i}^\beta}+\overline{\bar\theta_{j,\bar i}^\beta}\bigr)\partial_{\bar z_j^\beta}$. Applying $\D \mathsf{F}$ to $H_{D_1}(\partial_{s^i})$ at $x_1$ and equating with $H_{D_2}(\partial_{s^i})$ at $\mathsf{F}(x_1)$ gives
\[\partial_{s^i}\mathsf{F}^a+(\Gamma_{1,i}^\alpha+\theta_{1,i}^\alpha) \partial_{z_1^\alpha}\mathsf{F}^a+(\overline{\Gamma_{1,\bar i}^\alpha}+\overline{\bar\theta_{1,\bar i}^\alpha})\partial_{\bar z_1^\alpha}\mathsf{F}^a=\Gamma_{2,i}^a\comp \mathsf{F}+\theta_{2,i}^a\comp \mathsf{F}.\]
Using $\partial_{\bar z_1^\alpha}\mathsf{F}^a=0$ (fiberwise holomorphicity) and rearranging,
\[ \partial_{s^i}\mathsf{F}^a+\Gamma_{1,i}^\alpha\partial_{z_1^\alpha}\mathsf{F}^a-\Gamma_{2,i}^a\comp \mathsf{F}=\theta_{2,i}^a\comp \mathsf{F}-\theta_{1,i}^\alpha\partial_{z_1^\alpha}\mathsf{F}^a, \]
i.e., $(\partial u)_N(\partial_{s^i})-(\theta_Xu)_N(\partial_{s^i})=0$ by \eqref{eq:thetaX_u_N} and \eqref{eq:partial_u_N}. Similarly, $(\dbar u)_N(\partial_{\bar s^i})-(\bar\theta_Xu)_N(\partial_{\bar s^i})=0$. Combined with \eqref{eq:dbar_u_N} and \eqref{eq:partial_u_vert_zero}, these yield $(D_Xu)_N=(\partial u)_N+(\dbar u)_N-(\theta_Xu)_N-(\bar\theta_Xu)_N=0$.

Conversely, suppose $(\Sigma,u)$ is a flat sub-fibration. In particular it is a sub-fibration, so $\mathsf{F}$ is fiberwise holomorphic by Lemma \ref{lem:graph_subfib}. Given fiberwise holomorphicity, the computation above shows that the vanishing of $(D_Xu)_N$ on $\partial_{s^i}$ and $\partial_{\bar s^i}$ is equivalent to the flat condition \eqref{eq:flat_morphism_def}.
\end{proof}

\begin{proposition}\label{prop:subfib_as_morphism}
Let $(\Sigma,u)$ be a sub-fibration for $f:X\to S$, with $\pi:=f\comp u$ being a complex fiber bundle, $E:=u^{*}T_{X/S}$, $F:=\D u(T_{\Sigma/S})\subset E$, $N:=E/F$, and $\D^{\mathrm v}u:=\D u|_{T_{\Sigma/S}}:T_{\Sigma/S}\xrightarrow{\;\cong\;}F$, and let $\dbar_\Sigma^{\mathrm{nat}}$ be the canonical $\dbar$-operator of the holomorphic fibration $\pi:\Sigma\to S$. Regard $u$ as a fiberwise holomorphic smooth morphism over $S$ from $(\pi:\Sigma\to S,T_{\Sigma/S})$ to $(f:X\to S,T_{X/S})$.
\begin{enumerate}[label=(\roman*)]
\item If $(\Sigma,u)$ is a Higgs sub-fibration then $u:(\Sigma,\dbar_\Sigma,\theta_\Sigma)\to(X,\dbar,\theta)$ is a Higgs morphism for
\begin{equation}\label{eq:induced_source_higgs}
\dbar_\Sigma:=\dbar_\Sigma^{\mathrm{nat}}-(\D^{\mathrm v}u)^{-1}(\dbar u), \qquad \theta_\Sigma:=(\D^{\mathrm v}u)^{-1}(\theta u)\in A^{1,0}\bigl(S,\pi_*T_{\Sigma/S}\bigr).
\end{equation}
Conversely, if $u:(\Sigma,\dbar_\Sigma,\theta_\Sigma)\to(X,\dbar,\theta)$ is a Higgs morphism, then $(\Sigma,u)$ is a Higgs sub-fibration and $(\dbar_\Sigma,\theta_\Sigma)$ are necessarily \eqref{eq:induced_source_higgs}. If $(\dbar,\theta)$ is a Higgs bundle then so is $(\dbar_\Sigma,\theta_\Sigma)$.

\item If $(\Sigma,u)$ is a flat sub-fibration of $(f,D)$, then $u:(\Sigma,D_\Sigma)\to (X,D)$ is a flat morphism for the Ehresmann connection $D_\Sigma$ on $\pi:\Sigma\to S$ whose horizontal lift is determined by
\begin{equation}\label{eq:induced_source_flat}
\D u\bigl(H_{D_\Sigma}(v)|_\sigma\bigr)=H_{D}(v)|_{u(\sigma)},
\qquad v\in T_{\pi(\sigma)}S^{\RR},\ \sigma\in\Sigma.
\end{equation}
Conversely, if $\Sigma$ carries an Ehresmann connection $D_\Sigma$ for which $u$ is a flat morphism, then $(\Sigma,u)$ is a flat sub-fibration and $D_\Sigma$ is necessarily \eqref{eq:induced_source_flat}. If $D$ is flat then so is $D_\Sigma$.
\end{enumerate}
\end{proposition}

\begin{proof}
(i) Throughout we use adapted local coordinates $(s^{i},t^{a})$ on $\Sigma$ and $(s^{i},z^{\alpha})$ on $X$, in which $\pi(s,t)=s$, $f(s,z)=s$ and $u(s,t)=(s,u^{\alpha}(s,t))$ with $u^{\alpha}$ holomorphic in $t$. We write $\dbar(\partial_{\bar i})=[\partial_{\bar i}+\Gamma_{\bar i}^{\alpha}\partial_{\alpha}]\bmod\widebar{T_{X/S}}$,
$\theta(\partial_{i})=\theta_{i}^{\alpha}\partial_{\alpha}$, $e_{\alpha}:=\partial_{\alpha}\comp u$, and
$V_{a}:=\D^{\mathrm v}u(\partial_{t^{a}})=(\partial_{t^{a}}u^{\alpha})\,e_{\alpha}$, so that $\{V_{a}\}$ is a local
frame of $F$.

For $W\in T_{\Sigma/S}$ one has $\D\pi(W)=0$, so $(\theta u)(W)=\theta(\D\pi(W))|_{u}=0$. The horizontal lift of $\partial_{\bar t^{a}}$ in $\pi^{*}X$ is $(\partial_{\bar t^{a}},0)$
by \eqref{eq:horizontal_lift_pullback}, whence
$\dbar u(\partial_{\bar t^{a}})=\mathrm{pr}^{1,0}\bigl(\D u(\partial_{\bar t^{a}})\bigr)
=\mathrm{pr}^{1,0}\bigl((\partial_{\bar t^{a}}\bar u^{\alpha})\partial_{\bar\alpha}\bigr)=0$,
using $\partial_{\bar t^{a}}u^{\alpha}=0$. Thus $\dbar u$ and $\theta u$ annihilate $T_{\Sigma/S}$ and lie in
$\pi^{*}\widebar{T^*S}\otimes E$ resp.\ $\pi^{*}T^*S\otimes E$.

Assume $(\Sigma,u)$ is a Higgs sub-fibration. Then $(\D^{\mathrm v}u)^{-1}(\dbar u)$ and $\theta_\Sigma:=(\D^{\mathrm v}u)^{-1}(\theta u)$ are defined. By the horizontality just proved, $(\D^{\mathrm v}u)^{-1}(\dbar u)\in C^\infty(\Sigma,\pi^{*}\widebar{T^*S}\otimes T_{\Sigma/S})$ and $\theta_\Sigma \in A^{1,0}(S,\pi_*T_{\Sigma/S})$ (clearly $\theta_\Sigma$ is relatively holomorphic). So $\dbar_\Sigma$ in \eqref{eq:induced_source_higgs} is a $\dbar$-operator on $(\pi,T_{\Sigma/S})$. We have $\dbar u(\partial_{\bar i})=\eta_{\bar i}^{\alpha}e_{\alpha}$ with $\eta_{\bar i}^{\alpha}=\partial_{\bar i}u^{\alpha}-\Gamma_{\bar i}^{\alpha}\comp u$. The hypothesis $(\dbar u)_N=0$ gives unique $\lambda_{\bar i}^{a}$ with $\eta_{\bar i}^{\alpha}=\lambda_{\bar i}^{a}\,\partial_{t^{a}}u^{\alpha}$, namely $\lambda_{\bar i}=(\D^{\mathrm v}u)^{-1}(\dbar u(\partial_{\bar i}))$. By \eqref{eq:induced_source_higgs}, $\dbar_\Sigma(\partial_{\bar i})=[\partial_{\bar i}-\lambda_{\bar i}^{a}\partial_{t^{a}}]$. Let $J_\Sigma,J$ be the almost complex structures determined by $\dbar_\Sigma,\dbar$. Then $\widebar{T_\Sigma^{J_\Sigma}}$ is spanned by $\partial_{\bar t^{a}}$ and $\widehat H_{\bar i}:=\partial_{\bar i}-\lambda_{\bar i}^{a}\partial_{t^{a}}$, while $\widebar{T_X^{J}}$ is spanned by $\partial_{\bar z^{\alpha}}$ and $\partial_{\bar i}+\Gamma_{\bar i}^{\alpha}\partial_{\alpha}$. Now $\D u(\partial_{\bar t^{a}})=(\partial_{\bar t^{a}}\bar u^{\alpha})\partial_{\bar\alpha}\in\widebar{T_X^{J}}$, and
\[
\D u(\widehat H_{\bar i})=\partial_{\bar i}+\bigl(\partial_{\bar i}u^{\alpha}-\lambda_{\bar i}^{a}\partial_{t^{a}}u^{\alpha}\bigr)\partial_{\alpha}+(\cdots)\,\partial_{\bar\alpha}=\partial_{\bar i}+(\Gamma_{\bar i}^{\alpha}\comp u)\,\partial_{\alpha}+(\cdots)\,\partial_{\bar\alpha}\in\widebar{T_X^{J}}.
\]
Hence $\D u\bigl(\widebar{T_\Sigma^{J_\Sigma}}\bigr)\subset\widebar{T_X^{J}}$, i.e.\ $u$ is pseudo-holomorphic. For $v\in T_sS$ and $\sigma\in\Sigma_s$, $\theta_\Sigma(v)|_\sigma=(\D^{\mathrm v}u)^{-1}\bigl(\theta(v)|_{u(\sigma)}\bigr)\in T_{\Sigma/S,\sigma}$, so $\D u\bigl(\theta_\Sigma(v)|_\sigma\bigr)=\D^{\mathrm v}u\bigl((\D^{\mathrm v}u)^{-1}(\theta(v)|_{u(\sigma)})\bigr)=\theta(v)|_{u(\sigma)}$, which is precisely \eqref{eq:Higgs_morphism_def}. Therefore $u$ is a Higgs morphism.

Suppose $u$ is a Higgs morphism for some $(\dbar_\Sigma',\theta_\Sigma')$. Pseudo-holomorphicity forces $\D u(\widehat H_{\bar i}')\in\widebar{T_X^{J}}$ for $\widehat H_{\bar i}'=\partial_{\bar i}-\lambda_{\bar i}^{\prime a}\partial_{t^{a}}$. Then $\partial_{\bar i}u^{\alpha}-\lambda_{\bar i}^{\prime a}\partial_{t^{a}}u^{\alpha}=\Gamma_{\bar i}^{\alpha}\comp u$, i.e.\ $\dbar u(\partial_{\bar i})=\lambda_{\bar i}^{\prime a}V_{a}\in F$. Thus $(\dbar u)_N=0$ and $\dbar_\Sigma'=\dbar_\Sigma$. The identity $\D^{\mathrm v}u(\theta_\Sigma'(v))=\theta(v)|_{u}$ forces $\theta(v)|_{u}\in F$, i.e.\ $(\theta u)_N=0$, and $\theta_\Sigma'=(\D^{\mathrm v}u)^{-1}(\theta u)=\theta_\Sigma$. Hence $(\Sigma,u)$ is a Higgs sub-fibration and $(\dbar_\Sigma',\theta_\Sigma')$ coincide with \eqref{eq:induced_source_higgs}.

Suppose $(\dbar,\theta)$ is a Higgs bundle. Since $\dbar$ is integrable, the Nijenhuis tensor $N_{J_X}$ vanishes. Let $V,W$ be vector fields on $\Sigma$. As $u$ is an immersion, choose local vector fields $V',W'$ on $X$ that are $u$-related to $V,W$. Pseudo-holomorphicity $\D u\comp J_\Sigma=J \comp\D u$ makes $J_\Sigma V,J_\Sigma W$ $u$-related to $J_X V',J_X W'$, so every term of the Nijenhuis tensor matches and
\[
\D u\bigl(N_{J_\Sigma}(V,W)\bigr)=N_{J_X}(V',W')\big|_{u}=0.
\]
As $\D u$ is injective, $N_{J_\Sigma}=0$, so $\dbar_\Sigma$ is integrable. Hence $\Sigma$ is a complex manifold, $\pi$ is holomorphic, and $u:\Sigma\to X$ is a holomorphic map. By $G^{1,1}_{D''}=0$, the field $\theta\in H^{0}(X,f^{*}\Omega_S\otimes T_{X/S})$ is holomorphic. Pulling back along the holomorphic map $u$, $E=u^{*}T_{X/S}$ is a holomorphic vector bundle over $(\Sigma,J_\Sigma)$, $F=\D u(T_{\Sigma/S})\subset E$ is a holomorphic subbundle, and the restriction $\theta u=u^{*}\theta\in H^{0}(\Sigma,\pi^{*}\Omega_S\otimes E)$ is holomorphic. Therefore $\theta_\Sigma=(\D^{\mathrm v}u)^{-1}(\theta u)\in H^{0}(\Sigma,\pi^{*}\Omega_S\otimes T_{\Sigma/S})$ is a holomorphic Higgs field. Fix $v,w\in T_sS$ and extend them to local vector fields on $S$. By construction, $\theta_\Sigma(v),\theta_\Sigma(w)$ are $u$-related to $\theta(v),\theta(w)$, hence
\[
\D u\bigl([\theta_\Sigma(v),\theta_\Sigma(w)]\bigr)=[\theta(v),\theta(w)]\big|_{u}=0,
\]
using $G^{2,0}_{\theta}=0$. As $\D u$ is injective, $[\theta_\Sigma(v),\theta_\Sigma(w)]=0$. Since $v,w$ are
arbitrary, $G^{2,0}_{\theta_\Sigma}=\tfrac12[\theta_\Sigma,\theta_\Sigma]=0$.

\smallskip
(ii) Assume $(\Sigma,u)$ is a flat sub-fibration. For $v\in T_sS^{\RR}$ and $\sigma\in\Sigma_s$, choose any $\hat v\in T_\sigma\Sigma^{\RR}$ with
$\D\pi(\hat v)=v$. Since $\D f\bigl(\D u(\hat v)\bigr)=\D\pi(\hat v)=v=\D f\bigl(H_{D}(v)|_{u(\sigma)}\bigr)$, we have $\D u(\hat v)-H_{D}(v)|_{u(\sigma)}\in\ T_{X/S,u(\sigma)}^{\RR}$, whose $(1,0)$-part has image under $\mathrm{pr}_N$ equal to $(Du)_N(\hat v)|_\sigma=0$ by hypothesis. Hence $\D u(\hat v)-H_{D}(v)|_{u(\sigma)}\in F^{\RR}=\D u(T_{\Sigma/S}^{\RR})$, which equals $\D u(w)$ for a unique $w\in T_{\Sigma/S,\sigma}^{\RR}$, and we set $H_{D_\Sigma}(v)|_\sigma:=\hat v-w$. Then $\D u(H_{D_\Sigma}(v))=H_{D}(v)|_{u(\sigma)}$ and $\D\pi(H_{D_\Sigma}(v))=v$. Replacing $\hat v$ by $\hat v+w'$ \textup($w'$ vertical\textup) replaces $w$ by $w+w'$ and leaves $H_{D_\Sigma}(v)$ unchanged, so $H_{D_\Sigma}$ is well-defined, $\RR$-linear and smooth in $v$, and defines an Ehresmann connection $D_\Sigma$ on $\pi$. By construction \eqref{eq:induced_source_flat} holds for all $v\in T_sS^{\RR}$, which is exactly \eqref{eq:flat_morphism_def}. Thus $u$ is a flat morphism.

If $u$ is a flat morphism for an Ehresmann connection $D_\Sigma$, then for $v\in T_sS^{\RR}$ and a lift $\hat v$ of $v$, $\D u(\hat v)-H_{D}(v)|_{u}=\D u\bigl(\hat v-H_{D_\Sigma}(v)\bigr)\in F^\RR$, whence $(Du)_N(\hat v)=\mathrm{pr}_N\bigl(\mathrm{pr}^{1,0}\bigl(\D u(\hat v)-H_{D}(v)|_{u}\bigr)\bigr)=0$. On vertical directions $(Du)_N=0$ automatically. So $(Du)_N=0$ and $D_\Sigma$ is \eqref{eq:induced_source_flat}.

Suppose $D$ is flat, and extend $v,w$ to local vector fields on $S$. By \eqref{eq:induced_source_flat}, $H_{D_\Sigma}(v)$ and $H_{D}(v)$ are $u$-related, hence $\D u\bigl([H_{D_\Sigma}(v),H_{D_\Sigma}(w)]\bigr)=[H_{D}(v),H_{D}(w)]\comp u$. The curvature of $D$ is the vertical part of $[H_{D}(v),H_{D}(w)]-H_{D}([v,w])$, which vanishes. Therefore $\D u\bigl([H_{D_\Sigma}(v),H_{D_\Sigma}(w)]-H_{D_\Sigma}([v,w])\bigr)=\bigl([H_{D}(v),H_{D}(w)]-H_{D}([v,w])\bigr)\comp u=0$. As $u$ is an immersion, $[H_{D_\Sigma}(v),H_{D_\Sigma}(w)]=H_{D_\Sigma}([v,w])$, i.e., $D_\Sigma$ is flat.
\end{proof}

\subsection{The correspondence between morphisms} \label{subsec:morphism_correspondence}
\begin{definition}\label{def:deg_anomaly_morphism}
Assume $X_1$ is compact and $\dbar_1$ is integrable. Fix a Hermitian metric $\omega_\Sigma$ on $\Sigma:=X_1$ (e.g., $\omega_{\Sigma,\lambda}:=\omega_1+\lambda\,f_1^{*}\omega_S$, $\lambda\gg1$). For a fiberwise holomorphic smooth morphism $\mathsf{F}:X_1\to X_2$ with graph $u:X_1\to X=X_1\times_SX_2$, define
\begin{align}
\deg(\mathsf{F})&:=\deg(\Sigma,u),\quad \deg_\zeta(\mathsf F):=\deg_{\zeta_X}(\Sigma,u)= \deg(\mathsf F)-\int_{X_1}f_1^*(\zeta_1+\zeta_2)\wedge\omega_\Sigma^{n+m_1-1},\label{eq:deg_F_def}\\
\Phi(\mathsf{F})&:=\int_\Sigma\Phi_{F}(u)\,\omega_\Sigma^{n+m_1},\label{eq:anomaly_F_def}
\end{align}
where $\deg_{\zeta_X}(\Sigma,u)$ is the combined degree of Definition \ref{def:deg_subfib}, $\zeta_X=\zeta_1+\zeta_2$, $\zeta_i\in A^2(S,\RR)$, and $\Phi_{F}(u)$ is given by \eqref{eq:Phi_F_def}, both computed for the sub-fibration $(\Sigma,u)$ (Lemma \ref{lem:graph_subfib}). In the applications below, when Assumption \ref{assum:comoment}\,(2) and (4) hold for $f_i$, we let $\zeta_i$ be the form in \eqref{eq:min-coupling}.
\end{definition}

Applying Theorems \ref{thm:D'_D''u_identity_subfib}, \ref{thm:Dc_Du_identity_subfib}, and \ref{thm:Dc_Du_identity_subfib_transgression} to $(\Sigma,u)$, together with Lemmas \ref{lem:product_basic} and \ref{lem:product_harmonic}, we immediately obtain the following.
\begin{theorem}\label{thm:morphism_correspondence}
Let $\mathsf{F}:X_1\to X_2$ be a fiberwise holomorphic smooth morphism over $S$ with $X_1$ compact. Suppose $(f_j,\omega_{X_j/S},D_j'',D_j)$ ($j=1,2$) are nonlinear harmonic bundles.
\begin{enumerate}[label=(\roman*)]
\item Suppose that $(f_j,\omega_{X_j/S},D_j'',D_j)$ satisfy Assumption \ref{assum:comoment}\,(1),(2),(4), and that $\nabla_j^\RR$ preserves $\mathfrak{k}_{X_j/S}^{\mathrm R}$. Then
\begin{equation}\label{eq:int_id_AB}
\int_\Sigma|D_{\mathsf F}'|^2\,\omega_\Sigma^{n+m_1}-\int_\Sigma|D_{\mathsf F}''|^2\,\omega_\Sigma^{n+m_1}=(n+m_1)\deg_\zeta(\mathsf{F}),
\end{equation}
where $D_{\mathsf F}'=(D_X'u)_N$ and $D_{\mathsf F}''=(D_X''u)_N$. If $\mathsf{F}$ is a flat morphism, then $\deg_\zeta(\mathsf{F})=0$. A Higgs morphism satisfies $\deg_\zeta(\mathsf F)\ge 0$ and is flat if and only if $\deg_\zeta(\mathsf{F})=0$.
\item Suppose $(f_j,\omega_{X_j/S},D_j'',D_j)$ satisfy Assumption \ref{assum:comoment}\,(1)--(3), and suppose each $\nabla_j^\RR$ preserves $\mathfrak{k}_{X_j/S}^{\mathrm R}$. If $\omega_\Sigma$ is semi-K\"ahler, then
\begin{equation}\label{eq:int_id_anomaly}
\int_\Sigma|D_{\mathsf F}'-D_{\mathsf F}''|^2\,\omega_\Sigma^{n+m_1}-\int_\Sigma|D_{\mathsf F}'+D_{\mathsf F}''|^2\,\omega_\Sigma^{n+m_1}=-\Phi(\mathsf{F}).
\end{equation}
Every Higgs morphism satisfies $\Phi(\mathsf{F})=0$. Every flat morphism is a Higgs morphism.
\end{enumerate}
\end{theorem}
\begin{remark}
By Theorem \ref{thm:Dc_Du_identity_subfib_transgression} (for $m_1\ge 1$) and Theorem \ref{thm:Dc_Du_identity_subfib} (for $m_1=0$), the last statement of Theorem \ref{thm:morphism_correspondence}\,(ii) also holds when the assumption that $\omega_\Sigma$ is semi-K\"ahler is replaced by the assumption that $\omega_S$ is semi-K\"ahler.
\end{remark}

\begin{example}\label{ex:higgs_morphism_not_flat}
Let $(f:X:=S\times Y\to S,\theta=0,\omega_X=\mathrm{pr}_2^*\omega_Y)$ be the harmonic bundle as in Example \ref{ex:ramified_multisection_general}, where $S=Y=\CC/(\ZZ\oplus\I\ZZ)$ is an elliptic curve with the standard flat K\"ahler metric $\frac{\I}{2}\D z\wedge\D\bar z$. Consider the automorphism
\[\mathsf{F}:X\to X,\qquad \mathsf{F}(s,z)=(s,z+\psi(s)),\]
where $\psi:S\to Y$ is any holomorphic map. By Definition \ref{def:Higgs_flat_morphism}, $\mathsf{F}$ is a Higgs morphism, and it is a flat morphism if and only if $\psi$ is a constant map. For both trivial factors, take $\mathfrak k=0$ as in Example \ref{ex:ramified_multisection_general}. Then $\zeta_1=\zeta_2=0$, and $\deg_\zeta(\mathsf F)=\deg(\mathsf F)$.

We compute $\deg(\mathsf{F})$ explicitly from Definition \ref{def:deg_subfib}. Write the fiber coordinates of $X_1$ and $X_2$ as $z_1$ and $z_2$, so that $X=X_1\times_S X_2$ has coordinates $(s,z_1,z_2)$ and the graph of $\mathsf{F}$ is
\[
u:\Sigma=X_1=S\times Y\longrightarrow X,\qquad u(s,z_1)=(s,z_1,z_1+\psi(s)),
\]
with $\pi=f\comp u:\Sigma\to S$, $(s,z_1)\mapsto s$ of relative dimension $1$. By \eqref{eq:omegaX_S}, \eqref{eq:omegaX}, \eqref{eq:thetaX},
\[
\omega_X=\tfrac{\I}{2}\bigl(\D z_1\wedge\D\bar z_1+\D z_2\wedge\D\bar z_2\bigr), \qquad \theta_X=0,
\]
and we choose $\omega_\Sigma=\tfrac{\I}{2}\D z_1\wedge\D\bar z_1+\lambda\,\tfrac{\I}{2}\D s\wedge\D\bar s$.

Since $\theta_X=0$, $\deg_{\theta,F}(\Sigma,u)=0$. Since $u$ is holomorphic, $\dbar u=0$ and $\deg_{F,\dbar}(\Sigma,u)=0$. Thus by \eqref{eq:combined_deg_subfib},
\[
\deg(\mathsf{F})=\deg_{\omega_X}(\Sigma,u)-\deg_{F,\partial}(\Sigma,u)=\int_\Sigma\bigl(u^*\omega_X-\omega_{F,\partial}\bigr)\wedge\omega_\Sigma.
\]
From $u^*\D z_1=\D z_1$ and $u^*\D z_2=\D z_1+\psi'(s)\,\D s$,
\[
u^*\omega_X=\tfrac{\I}{2}(2\,\D z_1\wedge\D\bar z_1+\overline{\psi'}\,\D z_1\wedge\D\bar s+\psi'\,\D s\wedge\D\bar z_1+|\psi'|^2\,\D s\wedge\D\bar s).
\]
By Lemma \ref{lem:graph_subfib}, $E=T_{X_1/S}\oplus\mathsf{F}^*T_{X_2/S}$ has the $h_E$-orthonormal frame $\{\partial_{z_1},\partial_{z_2}\}$, the relative tangent bundle is $F=\CC \,(\partial_{z_1}+\partial_{z_2})$ (as $\D^{\mathrm{v}}\mathsf{F}(\partial_{z_1})=\partial_{z_2}$), and $\mathrm{pr}_F(a\partial_{z_1}+b\partial_{z_2})=\tfrac{a+b}{2}(\partial_{z_1}+\partial_{z_2})$.
By Definition \ref{def:partial_on_section}, $\partial u(\partial_{z_1})=\partial_{z_1}+\partial_{z_2}$ and $\partial u(\partial_s)=\psi'(s)\,\partial_{z_2}$, so
\[
(\partial u)_F(\partial_{z_1})=\partial_{z_1}+\partial_{z_2},\qquad (\partial u)_F(\partial_s)=\tfrac{1}{2}\psi'\bigl(\partial_{z_1}+\partial_{z_2}\bigr).
\]
Since $\omega_{X/S}\bigl(\partial_{z_1}+\partial_{z_2},\overline{\partial_{z_1}+\partial_{z_2}}\bigr)=\I$, Definition \ref{def:omega_F} yields
\[
\omega_{F,\partial}=\tfrac{\I}{4}(4\D z_1\wedge\D\bar z_1+2\overline{\psi'}\,\D z_1\wedge\D\bar s+2\psi'\,\D s\wedge\D\bar z_1+|\psi'|^2\,\D s\wedge\D\bar s).
\]
Since $\psi^*\omega_Y=|\psi'|^2\,\tfrac{\I}{2}\D s\wedge\D\bar s$, we obtain
\[
u^*\omega_X-\omega_{F,\partial}=\tfrac{\I}{4}|\psi'|^2\,\D s\wedge\D\bar s=\tfrac12\pi^*(\psi^*\omega_Y).
\]
Then we have
\[
\deg(\mathsf{F})=\frac12\int_{S\times Y}\pi^*\bigl(\psi^*\omega_Y\bigr)\wedge\omega_\Sigma=\frac12\Bigl(\int_S\psi^*\omega_Y\Bigr)\Bigl(\int_Y\omega_Y\Bigr)=\frac{1}{2}\deg(\psi)\ge 0,
\]
with equality if and only if $\psi^*\omega_Y= 0$, i.e.\ $\psi$ is constant, in agreement with Theorem \ref{thm:morphism_correspondence}\,(i).
\end{example}
The following example shows that the compactness hypothesis on $X_1$ in the flat-to-Higgs correspondence of Theorem \ref{thm:morphism_correspondence}\,(ii) cannot be omitted.
\begin{example}\label{ex:noncompact_flat_morphism_not_higgs}
Let $S=\CC/(\ZZ+\I\ZZ)$ with the periodic coordinate $s$ from the coordinate $\tilde{s}$ of $\widetilde{S}=\CC$ and set $\omega_S:=\frac{\I}{2}\D s\wedge\D\bar s$. Fix $0<q<1$. Let $\gamma_1,\gamma_2$ generate $\pi_1(S)\cong\ZZ^2$ along the periods $1,\I$, and set $\rho_1:\pi_1(S)\to G_1:=\CC^{*}$, $m\gamma_1+n\gamma_2\mapsto q^{m}$, so $\mathrm{im}(\rho_1)=q^{\ZZ}$. On the fiber
\[
(Y_1,\omega_{Y_1})=(\CC^{*},\tfrac{\I}{2}\,\D z\wedge\D\bar z),
\]
we have $\Stab_{G_1}(\omega_{Y_1})=S^1$, and the $S^1$-action by multiplication is Hamiltonian: the infinitesimal action is generated by $v=\I z\partial_z$, which satisfies $\iota_v\omega_{Y_1}=-\dbar(\frac{1}{2}|z|^2)$. Then $\mathfrak k_{Y_1}:=\RR\,\I z\partial_z\subset\mathrm{Ham}^{1,0}(Y_1)$. Let
\[
X_1:=\widetilde S\times_{\rho_1}\CC^{*}\longrightarrow S
\]
be the associated flat bundle, equipped with the harmonic metric induced by the $\rho_1$-equivariant harmonic map $\phi:\widetilde S\to\CC^{*}/S^{1}$, $\tilde{s}\mapsto (\log q)\mathrm{Re}(\tilde{s})$, where $\CC^{*}/S^{1}\cong\RR$ via $\log|\cdot|$ and $\rho_1(\gamma)$ acts on $\RR$ by translation by $\log|\rho_1(\gamma)|$. The resulting nonlinear harmonic bundle satisfies Assumption \ref{assum:comoment}\,(1)--(4), with Higgs field
\[
\theta_1(\partial_s)=-(\partial_s\phi)\,z\partial_z=-\tfrac{1}{2}(\log q) z\partial_z.
\]
Let $X_2:=S\times\CC P^1$ be the trivial harmonic bundle
with $\theta_2=0$.

Let $E_q:=\CC^{*}/q^{\ZZ}$ with coordinate $z$. In $w=\log z$, it is $\CC/(\log q\,\ZZ+2\pi\I\,\ZZ)$, an elliptic curve. Choose a non-constant holomorphic map $\tilde{\psi}:E_q\to\CC P^1$, and let $\psi:\CC^{*}\to\CC P^1$ be its pullback, so $\psi(qz)=\psi(z)$. As $\mathrm{im}(\rho_1)=q^{\ZZ}$, $\psi$ is $\rho_1$-invariant, so
\[
\widetilde{\mathsf F}:\widetilde S\times\CC^{*}\to\widetilde S\times\CC P^1,\qquad
(\tilde s,z)\mapsto(\tilde s,\psi(z))
\]
descends to a fiberwise holomorphic morphism $\mathsf F:X_1\to X_2$ over $S$. In the flat trivializations $\widetilde{\mathsf F}$ is independent of $\tilde s$, so $\mathsf F$ is a flat morphism. $\mathsf{F}$ is not a Higgs morphism, since $\theta_2=0$, whereas
\[
\D\mathsf F\bigl(\theta_1(\partial_s)\bigr)=-\tfrac{1}{2}(\log q)\,\D\psi(z\partial_z)\not\equiv0.
\]
The condition \eqref{eq:Higgs_morphism_def} fails.
\end{example}

\subsection{Degree of the morphism associated to a sub-fibration} \label{subsec:morphism_deg_subfib}
In this subsection, we compare the transgressed degree of a sub-fibration with the degree of its associated morphism, i.e., the transgressed degree of the graph sub-fibration associated to this morphism.

Let $(S,\omega_S)$ be a connected compact Hermitian manifold of complex dimension $n\geq1$, let $f\colon X\to S$ be a complex fiber bundle equipped with a real $2$-form $\omega_X$ whose restriction to every fiber is K\"ahler, and fix $\zeta\in A^2(S,\RR)$. Let $(\Sigma,u)$ be a sub-fibration of relative dimension $k\geq0$ such that $\pi:=f\comp u$ is proper and is unramified when $k=0$, and set $\Omega_\Sigma:=u^*\omega_X$. The restriction $\Omega_\Sigma|_{\Sigma_s}=u_s^*(\omega_X|_{X_s})$ is K\"ahler for every $s\in S$. Regard $\pi\colon\Sigma\to S$ as a complex fiber bundle with a real $2$-form $\Omega_\Sigma$ which restricts to a fiberwise K\"ahler metric, and choose a form $\zeta_\Sigma\in A^2(S,\RR)$. For the fiber-integration identities in this subsection, the fibers are allowed to be disconnected. Define
\[
\tau:=\pi_*(\Omega_\Sigma^{k+1})/(k+1),\qquad V:=\pi_*(\Omega_\Sigma^k),\qquad\hat\omega_{u,\zeta}=\tau-V\zeta,\qquad \hat\omega_{\id_\Sigma,\zeta_\Sigma}=\tau-V\zeta_\Sigma.
\]

Let $\widehat X:=\Sigma\times_S X$, write $p_1\colon\widehat X\to\Sigma$ and $p_2\colon\widehat X\to X$ for the projections, and set $\widehat f:=\pi\comp p_1=f\comp p_2$. Equip $\widehat X$ with the $2$-form $\hat\omega_X:=p_1^*\Omega_\Sigma+p_2^*\omega_X$ and the base form $\zeta_\Sigma+\zeta$. The graph $\hat u:=(\id_\Sigma,u)\colon\Sigma\to\widehat X$ is a sub-fibration of relative dimension $k$, since $\widehat f\comp\hat u=\pi$ and $\hat u_s=(\id_{\Sigma_s},u_s)$ is a holomorphic immersion.

\begin{proposition}\label{prop:morphism_subfib_transgressed_degree}
With the notation above, $\hat\omega_{\hat u,\zeta_\Sigma+\zeta}=2^k\bigl(\hat\omega_{\id_\Sigma,\zeta_\Sigma}+\hat\omega_{u,\zeta}\bigr)$,
and consequently
\[
\deg_{\hat\omega_X,\zeta_\Sigma+\zeta}^{\mathrm{tr}}(\Sigma,\hat u)=2^k\bigl(\deg_{\Omega_\Sigma,\zeta_\Sigma}^{\mathrm{tr}}(\Sigma,\id_\Sigma)+\deg_{\omega_X,\zeta}^{\mathrm{tr}}(\Sigma,u)\bigr).
\]
\end{proposition}
\begin{proof}
Since $\hat u^*\hat\omega_X=2\Omega_\Sigma$, the raw transgressed form and the fiberwise volume of the graph are $2^{k+1}\tau$ and $2^kV$, respectively. Hence
\[
\hat\omega_{\hat u,\zeta_\Sigma+\zeta}=2^{k+1}\tau-2^kV(\zeta_\Sigma+\zeta)=2^k\bigl((\tau-V\zeta_\Sigma)+(\tau-V\zeta)\bigr),
\]
which proves the identity of forms. Integration against $\omega_S^{n-1}$ proves the identity of degrees.
\end{proof}
\begin{remark}
Choosing $\zeta_\Sigma:=\tau/V$ gives $\hat{\omega}_{\id_\Sigma,\zeta_\Sigma}=0$ and hence $\deg_{\Omega_\Sigma,\zeta_\Sigma}^{\mathrm{tr}}(\Sigma,\id_\Sigma)=0$. If $\Omega_\Sigma$ is closed and $\pi$ has connected fibers, this choice agrees with the form in \eqref{eq:min-coupling} for the mean-zero fiberwise comoment map, by \eqref{eq:zeta}.
\end{remark}

\bibliographystyle{amsalpha}
\bibliography{hol}

@incollection {bott72lectures,
    AUTHOR = {Bott, Raoul},
     TITLE = {Lectures on characteristic classes and foliations},
 BOOKTITLE = {Lectures on algebraic and differential topology ({S}econd
              {L}atin {A}merican {S}chool in {M}ath., {M}exico {C}ity,
              1971)},
    SERIES = {Lecture Notes in Math.},
    VOLUME = {279},
     PAGES = {1--94},
      NOTE = {Notes by Lawrence Conlon, with two appendices by J. Stasheff},
 PUBLISHER = {Springer, Berlin-New York},
      YEAR = {1972},
   MRCLASS = {57D30},
  MRNUMBER = {362335},
MRREVIEWER = {D.\ B.\ Fuchs},
}

@article {GM18,
    AUTHOR = {Greb, Daniel and Miebach, Christian},
     TITLE = {Hamiltonian actions of unipotent groups on compact {K}\"ahler
              manifolds},
   JOURNAL = {\'Epijournal G\'eom. Alg\'ebrique},
  FJOURNAL = {\'Epijournal de G\'eom\'etrie Alg\'ebrique. EPIGA},
    VOLUME = {2},
      YEAR = {2018},
     PAGES = {Art. 10, 30},
      ISSN = {2491-6765},
   MRCLASS = {32M05 (14L24 32M10 32Q15 53D20)},
  MRNUMBER = {3894859},
MRREVIEWER = {Wojciech\ Lisiecki},
       DOI = {10.46298/epiga.2018.volume2.4486},
       URL = {https://doi.org/10.46298/epiga.2018.volume2.4486},
}

@article {Michelsohn82,
    AUTHOR = {Michelsohn, M. L.},
     TITLE = {On the existence of special metrics in complex geometry},
   JOURNAL = {Acta Math.},
  FJOURNAL = {Acta Mathematica},
    VOLUME = {149},
      YEAR = {1982},
    NUMBER = {3-4},
     PAGES = {261--295},
      ISSN = {0001-5962,1871-2509},
   MRCLASS = {53C55 (32C30)},
  MRNUMBER = {688351},
MRREVIEWER = {Steven\ M.\ Zucker},
       DOI = {10.1007/BF02392356},
       URL = {https://doi.org/10.1007/BF02392356},
}

@book {Kobayashi87,
    AUTHOR = {Kobayashi, Shoshichi},
     TITLE = {Differential geometry of complex vector bundles},
    SERIES = {Publications of the Mathematical Society of Japan},
    VOLUME = {15},
      NOTE = {Kan\^o{} Memorial Lectures, 5},
 PUBLISHER = {Princeton University Press, Princeton, NJ},
      YEAR = {1987},
     PAGES = {xii+305},
      ISBN = {0-691-08467-X},
   MRCLASS = {53C55 (32-02 32L05 32L10 32L20)},
  MRNUMBER = {909698},
MRREVIEWER = {Daniel\ M.\ Burns, Jr.},
       DOI = {10.1515/9781400858682},
       URL = {https://doi.org/10.1515/9781400858682},
}

@article {lichnerowicz69applications,
    AUTHOR = {Lichnerowicz, Andr\'e},
     TITLE = {Applications harmoniques et vari\'et\'es {K}\"ahleriennes},
   JOURNAL = {Rend. Sem. Mat. Fis. Milano},
  FJOURNAL = {Rendiconti del Seminario Matematico e Fisico di Milano},
    VOLUME = {39},
      YEAR = {1969},
     PAGES = {186--195},
      ISSN = {0370-7377},
   MRCLASS = {53.72},
  MRNUMBER = {270305},
MRREVIEWER = {J.\ D.\ Zund},
       DOI = {10.1007/BF02924136},
       URL = {https://doi.org/10.1007/BF02924136},
}

@article {hattori24smooth,
    AUTHOR = {Hattori, Kota},
     TITLE = {Smooth maps minimizing the energy and the calibrated geometry},
   JOURNAL = {J. Geom. Anal.},
  FJOURNAL = {Journal of Geometric Analysis},
    VOLUME = {34},
      YEAR = {2024},
    NUMBER = {1},
     PAGES = {Paper No. 9, 22},
      ISSN = {1050-6926,1559-002X},
   MRCLASS = {53C38},
  MRNUMBER = {4659436},
MRREVIEWER = {Kotaro\ Kawai},
       DOI = {10.1007/s12220-023-01452-1},
       URL = {https://doi.org/10.1007/s12220-023-01452-1},
}

@misc{demailly12CADG,
    author = {Demailly, Jean-Pierre},
    title = {Complex Analytic and Differential Geometry},
    year = {2012},
    note = {\url{https://www-fourier.univ-grenoble-alpes.fr/~demailly/manuscripts/agbook.pdf}},
}

@article {BiquardGarciaPradaRubio2017,
    AUTHOR = {Biquard, Olivier and Garc\'ia-Prada, Oscar and Rubio, Roberto},
     TITLE = {Higgs bundles, the {T}oledo invariant and the {C}ayley
              correspondence},
   JOURNAL = {J. Topol.},
  FJOURNAL = {Journal of Topology},
    VOLUME = {10},
      YEAR = {2017},
    NUMBER = {3},
     PAGES = {795--826},
      ISSN = {1753-8416,1753-8424},
   MRCLASS = {14H60 (57R57 58D29)},
  MRNUMBER = {3797597},
MRREVIEWER = {Zhenbo\ Qin},
       DOI = {10.1112/topo.12023},
       URL = {https://doi.org/10.1112/topo.12023},
}

@article {OgiueTakagi1981,
    AUTHOR = {Ogiue, Koichi and Takagi, Ryoichi},
     TITLE = {A geometric meaning of the rank of {H}ermitian symmetric
              spaces},
   JOURNAL = {Tsukuba J. Math.},
  FJOURNAL = {Tsukuba Journal of Mathematics},
    VOLUME = {5},
      YEAR = {1981},
    NUMBER = {1},
     PAGES = {33--37},
      ISSN = {0387-4982,2423-821X},
   MRCLASS = {53C35 (32M15)},
  MRNUMBER = {630519},
MRREVIEWER = {Bang-yen\ Chen},
       DOI = {10.21099/tkbjm/1496159317},
       URL = {https://doi.org/10.21099/tkbjm/1496159317},
}

@article {yau78schwarz,
    AUTHOR = {Yau, Shing Tung},
     TITLE = {A general {S}chwarz lemma for {K}\"ahler manifolds},
   JOURNAL = {Amer. J. Math.},
  FJOURNAL = {American Journal of Mathematics},
    VOLUME = {100},
      YEAR = {1978},
    NUMBER = {1},
     PAGES = {197--203},
      ISSN = {0002-9327,1080-6377},
   MRCLASS = {32H25 (53C55)},
  MRNUMBER = {486659},
MRREVIEWER = {Samuel\ I.\ Goldberg},
       DOI = {10.2307/2373880},
       URL = {https://doi.org/10.2307/2373880},
}

@article {royden80schwarz,
    AUTHOR = {Royden, H. L.},
     TITLE = {The {A}hlfors-{S}chwarz lemma in several complex variables},
   JOURNAL = {Comment. Math. Helv.},
  FJOURNAL = {Commentarii Mathematici Helvetici},
    VOLUME = {55},
      YEAR = {1980},
    NUMBER = {4},
     PAGES = {547--558},
      ISSN = {0010-2571,1420-8946},
   MRCLASS = {32H20},
  MRNUMBER = {604712},
MRREVIEWER = {Marcus\ Wright},
       DOI = {10.1007/BF02566705},
       URL = {https://doi.org/10.1007/BF02566705},
}

@article {calabi79,
    AUTHOR = {Calabi, E.},
     TITLE = {M\'etriques k\"ahl\'eriennes et fibr\'es holomorphes},
   JOURNAL = {Ann. Sci. \'Ecole Norm. Sup. (4)},
  FJOURNAL = {Annales Scientifiques de l'\'Ecole Normale Sup\'erieure.
              Quatri\`eme S\'erie},
    VOLUME = {12},
      YEAR = {1979},
    NUMBER = {2},
     PAGES = {269--294},
      ISSN = {0012-9593},
   MRCLASS = {32L15 (32C10 53C55)},
  MRNUMBER = {543218},
MRREVIEWER = {Y.-T.\ Siu},
       URL = {http://www.numdam.org/item?id=ASENS_1979_4_12_2_269_0},
}

@article {AJS04,
    AUTHOR = {Abe, Makoto and Jin, Teisuke and Shima, Tadashi},
     TITLE = {Ramified coverings over a complete {K}\"ahler space},
   JOURNAL = {Arch. Math. (Basel)},
  FJOURNAL = {Archiv der Mathematik},
    VOLUME = {83},
      YEAR = {2004},
    NUMBER = {2},
     PAGES = {154--158},
      ISSN = {0003-889X,1420-8938},
   MRCLASS = {32Q15 (32C15)},
  MRNUMBER = {2104943},
MRREVIEWER = {Finnur\ L\'arusson},
       DOI = {10.1007/s00013-004-1072-5},
       URL = {https://doi.org/10.1007/s00013-004-1072-5},
}

@ARTICLE{LS26,
       author = {{Li}, Nianzi and {Sheng}, Mao},
        title = "{Connections, metrics and Higgs fields on complex fiber bundles}",
      journal = {arXiv e-prints},
         year = 2026,
        month = feb,
          eid = {arXiv:2602.13838},
        pages = {arXiv:2602.13838},
          doi = {10.48550/arXiv.2602.13838},
archivePrefix = {arXiv},
       eprint = {2602.13838},
 primaryClass = {math.DG},
       adsurl = {https://ui.adsabs.harvard.edu/abs/2026arXiv260213838L}
}

@book {GLS96,
    AUTHOR = {Guillemin, Victor and Lerman, Eugene and Sternberg, Shlomo},
     TITLE = {Symplectic fibrations and multiplicity diagrams},
 PUBLISHER = {Cambridge University Press, Cambridge},
      YEAR = {1996},
     PAGES = {xiv+222},
      ISBN = {0-521-44323-7},
   MRCLASS = {58F06 (17B99 22E45 58F05 81S10)},
  MRNUMBER = {1414677},
MRREVIEWER = {Alejandro\ Uribe},
       DOI = {10.1017/CBO9780511574788},
       URL = {https://doi.org/10.1017/CBO9780511574788},
}

@article {Si1,
    AUTHOR = {Simpson, Carlos T.},
     TITLE = {Higgs bundles and local systems},
   JOURNAL = {Inst. Hautes \'Etudes Sci. Publ. Math.},
  FJOURNAL = {Institut des Hautes \'Etudes Scientifiques. Publications
              Math\'ematiques},
    NUMBER = {75},
      YEAR = {1992},
     PAGES = {5--95},
      ISSN = {0073-8301,1618-1913},
   MRCLASS = {32G13 (14D07 53C07 58D27 58E15)},
  MRNUMBER = {1179076},
MRREVIEWER = {William\ Goldman},
       URL = {http://www.numdam.org/item?id=PMIHES_1992__75__5_0},
}

@article {uhlenbeck_yau1986existence,
    AUTHOR = {Uhlenbeck, Karen and Yau, Shing-Tung},
     TITLE = {On the existence of {H}ermitian-{Y}ang-{M}ills connections in
              stable vector bundles},
   JOURNAL = {Comm. Pure Appl. Math.},
  FJOURNAL = {Communications on Pure and Applied Mathematics},
    VOLUME = {39},
      YEAR = {1986},
     PAGES = {S257--S293},
      ISSN = {0010-3640,1097-0312},
   MRCLASS = {58G05 (32L15 53C05 58E15)},
  MRNUMBER = {861491},
MRREVIEWER = {Daniel\ S.\ Freed},
       DOI = {10.1002/cpa.3160390714},
       URL = {https://doi.org/10.1002/cpa.3160390714},
}

@article {simpson1988construct,
    AUTHOR = {Simpson, Carlos T.},
     TITLE = {Constructing variations of {H}odge structure using
              {Y}ang-{M}ills theory and applications to uniformization},
   JOURNAL = {J. Amer. Math. Soc.},
  FJOURNAL = {Journal of the American Mathematical Society},
    VOLUME = {1},
      YEAR = {1988},
    NUMBER = {4},
     PAGES = {867--918},
      ISSN = {0894-0347,1088-6834},
   MRCLASS = {58E15 (32L15 53C25 53C55)},
  MRNUMBER = {944577},
       DOI = {10.2307/1990994},
       URL = {https://doi.org/10.2307/1990994},
}

@ARTICLE{mccarthy2022canonical,
       author = {{McCarthy}, John Benjamin},
        title = "{Canonical metrics on holomorphic fibre bundles}",
      journal = {arXiv e-prints},
         year = 2022,
        month = feb,
          eid = {arXiv:2202.11630},
        pages = {arXiv:2202.11630},
          doi = {10.48550/arXiv.2202.11630},
archivePrefix = {arXiv},
       eprint = {2202.11630},
 primaryClass = {math.DG},
       adsurl = {https://ui.adsabs.harvard.edu/abs/2022arXiv220211630M}
}

@ARTICLE {She25b,
       author = {{Sheng}, Mao},
        title = "{Nonlinear Hodge correspondence in positive characteristic}",
      journal = {arXiv e-prints},
         year = 2025,
        month = oct,
          eid = {arXiv:2510.05578},
        pages = {arXiv:2510.05578},
          doi = {10.48550/arXiv.2510.05578},
archivePrefix = {arXiv},
       eprint = {2510.05578},
 primaryClass = {math.AG},
       adsurl = {https://ui.adsabs.harvard.edu/abs/2025arXiv251005578S}
}

@article {dervan_sektnan2021optimal,
    AUTHOR = {Dervan, Ruadha\'i{} and Sektnan, Lars Martin},
     TITLE = {Optimal symplectic connections on holomorphic submersions},
   JOURNAL = {Comm. Pure Appl. Math.},
  FJOURNAL = {Communications on Pure and Applied Mathematics},
    VOLUME = {74},
      YEAR = {2021},
    NUMBER = {10},
     PAGES = {2132--2184},
      ISSN = {0010-3640,1097-0312},
   MRCLASS = {53C55 (32Q15 53C07)},
  MRNUMBER = {4303016},
MRREVIEWER = {R.\ P.\ Albuquerque},
       DOI = {10.1002/cpa.21930},
       URL = {https://doi.org/10.1002/cpa.21930},
}

@book{sardanashvili2013advanced,
  title={Advanced differential geometry for theoreticians: fiber bundles, jet manifolds and Lagrangian theory},
  author={Sardanashvili, Gennadi{\u\i}},
  year={2013},
  publisher={LAP Lambert Academic Publishing}
}

@article {donaldson87twisted,
    AUTHOR = {Donaldson, S. K.},
     TITLE = {Twisted harmonic maps and the self-duality equations},
   JOURNAL = {Proc. London Math. Soc. (3)},
  FJOURNAL = {Proceedings of the London Mathematical Society. Third Series},
    VOLUME = {55},
      YEAR = {1987},
    NUMBER = {1},
     PAGES = {127--131},
      ISSN = {0024-6115,1460-244X},
   MRCLASS = {58E20 (32L15 53C05)},
  MRNUMBER = {887285},
MRREVIEWER = {Mitsuhiro\ Itoh},
       DOI = {10.1112/plms/s3-55.1.127},
       URL = {https://doi.org/10.1112/plms/s3-55.1.127},
}

@article {hitchin87selfdual,
    AUTHOR = {Hitchin, N. J.},
     TITLE = {The self-duality equations on a {R}iemann surface},
   JOURNAL = {Proc. London Math. Soc. (3)},
  FJOURNAL = {Proceedings of the London Mathematical Society. Third Series},
    VOLUME = {55},
      YEAR = {1987},
    NUMBER = {1},
     PAGES = {59--126},
      ISSN = {0024-6115,1460-244X},
   MRCLASS = {32G13 (14F05 14H15 32L10 53C05 58E99 81E13)},
  MRNUMBER = {887284},
MRREVIEWER = {Mitsuhiro\ Itoh},
       DOI = {10.1112/plms/s3-55.1.59},
       URL = {https://tlink.lib.tsinghua.edu.cn:443/https/443/org/doi/yitlink/10.1112/plms/s3-55.1.59},
}

@article {corlette88,
    AUTHOR = {Corlette, Kevin},
     TITLE = {Flat {$G$}-bundles with canonical metrics},
   JOURNAL = {J. Differential Geom.},
  FJOURNAL = {Journal of Differential Geometry},
    VOLUME = {28},
      YEAR = {1988},
    NUMBER = {3},
     PAGES = {361--382},
      ISSN = {0022-040X,1945-743X},
   MRCLASS = {58E20 (32L99 53C10)},
  MRNUMBER = {965220},
MRREVIEWER = {John\ C.\ Wood},
DOI = {10.4310/jdg/1214442469},
URL = {https://doi.org/10.4310/jdg/1214442469},
}

\end{document}